\documentclass[11pt,a4paper,english,reqno,a4paper]{amsart}
\usepackage{amsmath,amssymb,amsthm, comment,graphicx}
\usepackage{tikz}
\usepackage{fancyhdr}
\usepackage{epsfig}
\usepackage{mathrsfs}
\usepackage{multirow}
\usepackage{tabu}
\newtheorem{theorem}{Theorem}[section]
\newtheorem{definition}{Definition}[section]
\newtheorem{lemma}{Lemma}[section]
\newtheorem{remark}{Remark}[section]
\newtheorem{proposition}{Proposition}[section]

\numberwithin{equation}{section}

\allowdisplaybreaks[4]

\begin{document}
\title[Global well-posedness to the combustion piston problem]
{Global well-posedness of shock front solutions to one-dimensional piston problem for combustion Euler flows}

\author{Kai Hu}
\address{School of Mathematics and Statistics, Southwest University, Chongqing, 400715, China}
\email{hukaimath@swu.edu.cn}

\author{Jie Kuang }
\address{ Innovation Academy for Precision Measurement Science and
Technology, and Wuhan Institute of Physics and Mathematics,
Chinese Academy of Sciences Chinese Academy of Sciences, Wuhan 430071, China}
\email{jkuang@wipm.ac.cn,\ jkuang@apm.ac.cn }


\date{\today}

\begin{abstract}
This paper is devoted to the well-posedness theory of piston problem for compressible {combustion} Euler flows with physical ignition condition.
A significant combustion phenomena called detonation will occur provided the reactant is compressed and ignited by a leading shock.
Mathematically, the problem can be formulated as an initial-boundary value problem for hyperbolic balance laws with a large shock front as free boundary.
In present paper, we establish the global well-posedness of entropy solutions via wave front tracking scheme within the framework of $BV\cap L^1$ space.
The main difficulties here stem from the discontinuous source term without uniform dissipation structure, and from the characteristic-boundary associated with degenerate characteristic field.
In dealing with the obstacles caused by ignition temperature,
we develop a modified Glimm-type functional to control the oscillation growth of combustion waves, even if the exothermic source fails to uniformly decay.
As to the characteristic boundary, the degeneracy of contact discontinuity is fully employed to get elegant stability estimates near the piston boundary.
Meanwhile, we devise a weighted Lyapunov functional to balance the nonlinear effects arising from large shock, characteristic boundary and exothermic reaction,
then obtain the $L^1-$stability of combustion wave solutions. Our results reveal that
one dimensional \emph{ZND} detonation waves {supported} by a forward piston are indeed nonlinearly stable under small perturbation in $BV$ sense.
This is the first work on well-posedness of inviscid reacting Euler fluids dominated by ignition temperature.
\end{abstract}

\date{}
\maketitle

\tableofcontents

\section{Introduction and main result}\setcounter{equation}{0}

In Eulerian coordinates, the one-dimensional compressible Euler equations for combustible fluids can be written in the following form
\begin{eqnarray}\label{eq:1.1}
\left\{
\begin{array}{llll}
\partial_t\rho + \partial_x(\rho u) =0, \\[5pt]
\partial_t(\rho u) + \partial_x(\rho u^2 + p) =0 ,\\[5pt]
\partial_t\big(\rho (\frac{1}{2}u^2+e)\big) + \partial_x\big(\rho u(\frac{1}{2}u^2+e)+pu\big) = \mathfrak{q}_0 \rho Y\phi (T) , \\[5pt]
\partial_t(\rho Y) + \partial_x(\rho u Y) = - \rho Y\phi (T),	
\end{array}
     \right.
\end{eqnarray}
where $\rho$, $u$, $p$, $T$ and $e$ are respectively the density, velocity, pressure, temperature and specific internal energy of the fluid, $Y$ denotes the mass fraction of the reactant in mixed gas, and function $\phi(T)$ stands for the combustion reaction rate. Positive constant $\mathfrak{q}_0$ is the specific binding energy of reactant.

If we choose the density $\rho$ and entropy $S$ as two independent thermodynamical variables, then $p$, $e$ and $T$ can be seen as the functions of $(\rho,S)$, \emph{i.e.},
$p=p(\rho, S), \ e=e(\rho, S), \ T=T(\rho, S)$ through
thermodynamical relation:
\begin{eqnarray}\label{eq:1.2}
TdS=de-\frac{p}{\rho^2}d\rho.
\end{eqnarray}
For ideal polytropic gas, the constitutive relations are represented by
\begin{eqnarray}\label{eq:1.3}
p=\kappa\rho^{\gamma}\textrm{e}^{\frac{S}{c_{\textrm{v}}}},\qquad e=\frac{\kappa}{\gamma-1}\rho^{\gamma-1}\textrm{e}^{\frac{S}{c_{\textrm{v}}}},  \qquad T=\frac{\kappa}{(\gamma-1)c_{\textrm{v}}}\rho^{\gamma-1}\textrm{e}^{\frac{S}{c_{\textsc{v}}}},
\end{eqnarray}
where $\gamma>1$ is the adiabatic exponent, $c_{\textrm{v}}>0$ is the specific heat at constant volume, and $\kappa>0$ is any constant under scaling.

Throughout this paper, we suppose that reaction rate $\phi(T)$ for combustion flows satisfies the following \emph{Arrhenius Law} with physical ignition condition, namely
\begin{equation}\label{eq:1.4}
\phi(T)= \left\{
\begin{aligned}
		& T^{\vartheta}\textrm{e}^{-\frac{\mathfrak{R}}{T}}\qquad &\mbox{if}\quad T>T_{\textrm{i}},\\[5pt]
		& 0 \qquad                         & \mbox{if}\quad T \leq T_{\textrm{i}},
\end{aligned}
\right.
\end{equation}
where $T_{\textrm{i}}>0$ stands for the ignition temperature of reactant, and constants $\vartheta\geq0$, $\mathfrak{R} >0$.
In terms of (\ref{eq:1.4}), the combustion process is switched on-off by reactant temperature.

System \eqref{eq:1.1} with the reaction rate \eqref{eq:1.4} is also called \emph{Zeldovich-von Neumann-D\"{o}ring} model (\emph{i.e.} \emph{ZND}) which is widely applied in combustion theory and related numerical simulation (see \cite{Williams-1985} for more details).
In this paper, we are concerned with the combustion process caused by piston motion in a tube, and aim to investigate the well-posedness of combustion waves for the \emph{ZND} model.
Specifically, a fuel-filled long tube is closed by a piston at one end, and open at the other end.
If we push the piston with a speed $u_{\mathfrak{p}}(t)>0$ towards the combustible gas at rest, then a shock front appears and departs from the piston (see Fig.\ref{fig1}). Since the temperature of the flow increases across the shock front,
chemical reaction will be initiated behind the leading shock. Consequently, a kind of fierce combustion phenomena, called detonation, may occur in the tube.
Its dynamics research has a significant application of detonation engine for hypersonic aircrafts.

\begin{figure}[ht]
\begin{minipage}[t]{0.45\textwidth}
\centering
\begin{tikzpicture}[scale=0.7]
\draw [thick][->] (-2, -1)--(4, -1);
\draw [thick][->] (-1.5, -1.5)--(-1.5, 3.3);
\draw [line width=0.05cm](-1.5,-1)to[out=65,in=-120](2.2, 2.8);
\draw [line width=0.02cm][red](-1.5,-1)to [out=40,in=-140](2.8,1.2);
\%draw [thin][<-] (-2.5,0) --(-0.5,0.5);
\node at (4,-1.3) {$x$};
\node at (-1.8, 3.2) {$t$};
\node at (-1.8, -1.3) {$O$};
\node at (2.8, 2.8) {$\Gamma_\mathfrak{p}$};
\node at (3.3, 1.1) {$\Gamma_{\rm{s}}$};
\node at (1.5, -0.3) {$\Omega_{\mathfrak{p}}$};
\end{tikzpicture}
\caption{Shock front for combustion wave to piston problem}\label{fig1}
\end{minipage}
\begin{minipage}[t]{0.5\textwidth}
\centering
\begin{tikzpicture}[scale=0.7]
\draw [thick][->] (-2, -1)--(4, -1);
\draw [thick][->] (-1.5, -1.5)--(-1.5, 3);
\draw [line width=0.05cm](-1.5,-1)to (2.0,2.7);
\draw [line width=0.02cm][red](-1.5,-1)to(3.2,1.1);
\node at (4,-1.3) {$x$};
\node at (-1.8, 2.9) {$t$};
\node at (-1.8, -1.3) {$O$};
\node at (3.0, 2.8) {$x=\bar{u}_{\mathfrak{p}}t$};
\node at (4.2, 1.1) {$x=\bar{s}_{\mathfrak{p}}t$};
\node at (1.9, -0.5) {$\bar{U}_{\mathfrak{b}, \textrm{r}}$};
\node at (1.9, 1.5) {$\bar{U}_{\mathfrak{b}, \textrm{l}}$};
\end{tikzpicture}
\caption{Shock front for piston problem with constant speed}\label{fig2}
\end{minipage}
\end{figure}


Mathematically, for given piston velocity $u_{\mathfrak{p}}(t)$, the path of moving piston can be determined by $\chi_\mathfrak{p}(t)=\int_0^t u_\mathfrak{p}(\tau) d\tau$.
Thus, we can define the domain and its boundary respectively by
\begin{eqnarray*}
\Omega_{\mathfrak{p}}\doteq\{(t,x)\in\mathbb{R}^2:t>0, x>\chi_{\mathfrak{p}}(t)\},\quad
\Gamma_{\mathfrak{p}}\doteq\{(t,x)\in\mathbb{R}^{2}:t>0, x=\chi_{\mathfrak{p}}(t)\}.
\end{eqnarray*}
See Fig.\ref{fig1}. Set $U=(\rho, u,p,Y)^{\top}$. Then we prescribe the initial state of the flow in the tube as
\begin{eqnarray}\label{eq:1.5}
U(0,x)=U_{0}(x)=(\rho_0, u_0, p_0,Y_0)^{\top}(x)\qquad \mbox{for}\ x>0,
\end{eqnarray}
and along the path of piston, the flow satisfies
\begin{equation}\label{eq:1.6}
u(t,x)=u_\mathfrak{p}(t) \qquad \mbox{on}\  \Gamma_\mathfrak{p}.
\end{equation}
With this set-up, the piston problem described above can be formulated as the initial-boundary value problem to equations \eqref{eq:1.1} with shock front as a free boundary.
In this paper, we study the entropy solutions to problem \eqref{eq:1.1} and \eqref{eq:1.5}-\eqref{eq:1.6} defined as follows.
\begin{definition}[Entropy solution]\label{D:1.1}
A vector-valued function $U=(\rho, u,p,Y)^{\top}\in BV_{\textnormal{loc}}(\Omega_\mathfrak{p})$ is called an entropy solution to the initial-boundary value problem \eqref{eq:1.1} and \eqref{eq:1.5}-\eqref{eq:1.6}
provided that

\noindent{\textnormal{(i)}} $U(t, x)$ is a weak solution to equations \eqref{eq:1.1} in the distribution sense in $\Omega_\mathfrak{p}$, and satisfies \eqref{eq:1.5}-\eqref{eq:1.6} in the trace sense;

\noindent{\textnormal{(ii)}} $U(t, x)$ satisfies the entropy inequality
\begin{align*}\label{eq:1.7}
\partial_{t}(\rho S)+\partial_{x}(\rho u S)\geq \frac{\mathfrak{q}_0\rho Y\phi(T)}{T}
\end{align*}
in the distribution sense in $\Omega_\mathfrak{p}$.

\end{definition}


When the piston moves with a constant speed, that is $u_{\mathfrak{p}}(t)=\bar{u}_{\mathfrak{p}}>0$, into the static gas with constant state $\bar{U}_{\mathfrak{b},\textrm{r}}=(\bar{\rho}_{\textrm{r}}, 0,\bar{p}_{\textrm{r}},0)^{\top}$  in the tube (see Fig.\ref{fig2}), then we follow the arguments in \cite{Courant-Friedrichs-1948} and know that the initial-boundary value problem \eqref{eq:1.1} and \eqref{eq:1.5}-\eqref{eq:1.6}
admits a unique piecewise constant solution $\bar{U}_{\mathfrak{b}}=(\bar{\rho}_{\mathfrak{b}}, \bar{u}_{\mathfrak{b}}, \bar{p}_{\mathfrak{b}}, \bar{Y}_{\mathfrak{b}})^{\top}$ which contains a strong 4-shock front $x=\bar{\chi}_{\rm{s}}(t)$  with constant speed $\dot{\bar{\chi}}_{\rm{s}}(t)=\bar{\rm{s}}_{\mathfrak{b}}$ for $t\geq 0$. Precisely,
\begin{equation}\label{eq:1.8}
\begin{aligned}
\bar{U}_{\mathfrak{b}}(t,x)&= \left\{
\begin{array}{llll}
\bar{U}_{\mathfrak{b},\textrm{l}}\doteq (\bar{\rho}_{\textrm{l}}, \bar{u}_{\textrm{l}}, \bar{p}_{\textrm{l}},0)^{\top}&\quad \mbox{for}\quad x\in( \bar{u}_{\mathfrak{p}}t,\bar{\rm{s}}_{\mathfrak{b}}t),\\[5pt]
\bar{U}_{\mathfrak{b},\textrm{r}}\doteq (\bar{\rho}_{\textrm{r}}, 0, \bar{p}_{\textrm{r}},0)^{\top} & \quad  \mbox{for}\quad x\in(\bar{\rm{s}}_{\mathfrak{b}}t, +\infty),
\end{array}
\right.
\end{aligned}
\end{equation}
where $(\bar{\rho}_{\textrm{l}}, \bar{u}_{\textrm{l}}, \bar{p}_{\textrm{l}},0)$ and $(\bar{\rho}_{\textrm{r}}, 0, \bar{p}_{\textrm{r}},0)$ are both constant states satisfying
\begin{eqnarray*}\label{eq:1.9}
\begin{split}
\bar{\rho}_{\textrm{l}}>\bar{\rho}_{\textrm{r}}>0,\quad \bar{u}_{\textrm{l}}=\bar{u}_{\mathfrak{p}}, \quad \bar{\rm{s}}_{\mathfrak{b}}>\bar{u}_{\textrm{l}}>0.
\end{split}
\end{eqnarray*}
Thus, the corresponding temperatures $\bar{T}_{\textrm{l}}$ and $\bar{T}_{\rm{r}}$ (determined by the relations in \eqref{eq:1.3}) respectively for $\bar{U}_{\textrm{l}}$ and $\bar{U}_{\textrm{r}}$
satisfy
\begin{eqnarray*}\label{eq:1.10}
\begin{split}
\bar{T}_{\textrm{l}} > \bar{T}_{\textrm{r}}.
\end{split}
\end{eqnarray*}
In this case, we call the solution $\bar{U}_{\mathfrak{b}}$ with shock front $x=\bar{\rm{s}}_{\mathfrak{b}}t$ as the \emph{background solution} to problem \eqref{eq:1.1} and \eqref{eq:1.5}-\eqref{eq:1.6}. Then, the main purpose in this paper is to investigate the well-posedness of entropy solutions with more general data $(U_0, u_{\mathfrak{p}})$ near the \emph{background} state $(\bar{U}_{\mathfrak{b},\textrm{r}}, \bar{u}_{\mathfrak{p}})$ in the $BV\cap L^{1}$ framework. We begin with some basic assumptions on the initial data $U_{0}(x)$ and velocity $u_{\mathfrak{p}}(t)$ of moving piston.
Assume that
\begin{enumerate}
\item[$(\mathbf{A1})$] the initial data $U_{0}$ defined  in \eqref{eq:1.5} satisfies $U_{0}-\bar{U}_{\mathfrak{b},\textrm{r}}\in (BV\cap L^{1})(\mathbb{R}_{+}; \mathbb{R}^{4})$,
$\inf_{x\in\mathbb{R}_{+}}\rho_0(x)>0$ and $Y_0\in [0,1]$;
\item[$(\mathbf{A2})$] the velocity $u_{\mathfrak{p}}$ of the moving piston defined on $\Gamma_{\mathfrak{p}}$ in \eqref{eq:1.6} satisfies $u_{\mathfrak{p}}-\bar{u}_{\mathfrak{p}}\in (BV\cap L^{1})(\mathbb{R}_{+}; \mathbb{R})$.
\end{enumerate}
For notational convenience, denote the  initial-boundary value problem \eqref{eq:1.1} and \eqref{eq:1.5}-\eqref{eq:1.6} by (\textbf{\emph{IBVP}}).
Our main result in this paper is stated as follows.

\begin{theorem}\label{thm:1.1}
Under the assumptions $(\mathbf{A1})$--$(\mathbf{A2})$, there exists a constant $\epsilon>0$ depending solely on $\bar{U}_{\mathfrak{b}}$
such that the ${(IBVP)}$ has a global entropy solution $U(t,x)$ if the initial-boundary data $(U_{0}(x),u_{\mathfrak{p}}(t))$ satisfy one of the following conditions:
\begin{enumerate}
\item[\textrm{(C1)}] $\bar{T}_{\rm{l}}<T_{\rm{i}}$ and
\begin{eqnarray}\label{eq:1.11}
T.V.\{U^{\textsc{e}}_{0}(\cdot)-\bar{U}^{\textsc{e}}_{\mathfrak{b},\rm{r}}; \mathbb{R}_{+}\}+T.V.\{u_{\mathfrak{p}}(\cdot); \mathbb{R}_{+}\}<\epsilon,
\end{eqnarray}
where $U^{\textsc{e}}_0=(\rho_0,u_0,p_0)^{\top}$ and $\bar{U}^{\textsc{e}}_{\mathfrak{b},\rm{r}}=(\bar{\rho}_{\rm{r}}, 0,\bar{p}_{\rm{r}})^{\top}$;

\item[\textrm{(C2)}] $\bar{T}_{\rm{r}}> T_{\rm{i}}$ and
\begin{eqnarray}\label{eq:1.12}
T.V.\{U_{0}(\cdot)-\bar{U}_{\mathfrak{b},\rm{r}}; \mathbb{R}_{+}\}+T.V.\{u_{\mathfrak{p}}(\cdot); \mathbb{R}_{+}\}<\epsilon;
\end{eqnarray}

\item[\textrm{(C3)}]  $\bar{T}_{\rm{l}}> T_{\rm{i}}>\bar{T}_{\rm{r}}$ and
\begin{eqnarray}\label{eq:1.12}
T.V.\{U_{0}(\cdot)-\bar{U}_{\mathfrak{b},\rm{r}}; \mathbb{R}_{+}\}+T.V.\{u_{\mathfrak{p}}(\cdot); \mathbb{R}_{+}\}+\bar{u}^{-1}_{\mathfrak{p}}\|Y_{0}\|_{L^1(\mathbb{R}_{+})}<\epsilon.
\end{eqnarray}

\end{enumerate}
The solution $U(t,x)$ is the unique limit of approximations obtained by fractional-step wave front tracking algorithm, which contains a strong shock front $x=\chi_{\rm{s}}(t)\in Lip(\mathbb{R}_{+})$ such that $\chi_{\rm{s}}(0)=0$, $\chi_{\mathfrak{p}}(t)<\chi_{\rm{s}}(t)$
and $\dot{\chi}_{\rm{s}}(t)\in BV(\mathbb{R}_{+})$. Furthermore, there holds
\begin{eqnarray}\label{eq:1.14}
\begin{split}
&\sup_{\chi_\mathfrak{p}(t)<x<\chi_{\rm{s}}(t)}|U(t,\cdot)-\bar{U}_{\mathfrak{b},\rm{l}}|+\sup_{x>\chi_{\rm{s}}(t)}|U(t,\cdot)-\bar{U}_{\mathfrak{b}, \rm{r}}|
+\sup_{t>0}|\dot{\chi}_{\rm{s}}(t)-\bar{\rm{s}}_{\mathfrak{b}}|\\
&\quad +T.V.\{U(t,\cdot);(\chi_\mathfrak{p}(t),\chi_{\rm{s}}(t)) \}+T.V.\{U(t,\cdot);(\chi_{\rm{s}}(t), +\infty)\}
 \leq C_{0}\epsilon,
\end{split}
\end{eqnarray}
where the constant $C_0>0$ depends solely on $\bar{U}_{\mathfrak{b}}$.

Assume $V(t,x)=(\tilde{\rho}, \tilde{u}, \tilde{p}, \tilde{Y})^{\top}(t,x)$ is another global entropy solution to ${(IBVP)}$ with the initial-boundary data $({V_{0}(x)},v_{\mathfrak{p}}(t))$ satisfying the similar condition to $(U_{0}(x),u_{\mathfrak{p}}(t))$ from \textnormal{(C1)-(C3)} for $V_0(x)=(\tilde{\rho}_0, \tilde{u}_0, \tilde{p}_0, \tilde{Y}_0)^{\top}(x)$.
Then for any $t>\tilde{t}\geq 0$, there holds the $L^1$-stability estimate
\begin{align}\label{eq:1.16}
\begin{split}
\big\|U(t,\cdot)-V(t,\cdot)\big\|_{L^{1}(\mathbb{R})}
\leq L\Big(\|U(\tilde{t},\cdot)-V(\tilde{t},\cdot)\|_{L^{1}(\mathbb{R})}+\|u_{\mathfrak{p}}(\cdot)-v_{\mathfrak{p}}(\cdot)\|_{L^{1}({[\tilde{t},t]})}\Big),
\end{split}
\end{align}
where $L>0$ is a Lipschitz constant dependent of $\bar{U}_{\mathfrak{b}}$ and {$\|Y_0\|_{L^1(\mathbb{R}_{+})}$, $\|\tilde{Y}_0\|_{L^1(\mathbb{R}_{+})}$}.

\end{theorem}

\begin{remark}\label{rem:1.1}
Throughout this paper, $\|U(t,\cdot)\|_{L^1(\mathbb{R})}$ is defined by
\begin{equation*}
\|U(t,\cdot)\|_{L^1(\mathbb{R})} \doteq \int_{\mathbb{R}} |U_{\rm{ex}}(t,x)| dx,
\end{equation*}
where the extension $U_{\rm{ex}}(t,x)$ is given by
\begin{equation*}
U_{\rm{ex}} (t,x)=
\left\{
\begin{array}{llll}
 U(t,x) \qquad\qquad &\textnormal{if}\quad  x\geq \chi_\mathfrak{p}(t),\\[5pt]
0   \qquad  &\textnormal{if}\quad  x<\chi_\mathfrak{p}(t).
\end{array}
\right.
\end{equation*}

\end{remark}

\begin{remark}\label{rem:1.2}
The solutions under conditions \textnormal{(C1)-(C3)} correspond to entirely unburnt flow, partially ignited flow and completely ignited flow, respectively.
For non-reacting fluid, the variation of $Y_0$ can be arbitrarily large, because the reactant is just transported along particle path, but never amplifies the oscillation of fluid.
For reacting fluid, we require only smallness condition on $T.V.\{Y_0;\mathbb{R}_{+}\}$, but allow the total reactant $\|Y_0\|_{L^1(\mathbb{R}_{+})}$ and $\|\rho_0 Y_0\|_{L^1(\mathbb{R}_{+})}$ suitably large. Thus Theorem \ref{thm:1.1} shows that one-dimensional ZND detonation waves are nonlinearly stable provided the leading shock is strong enough.
\end{remark}

\begin{remark}\label{rem:1.3}
Comparing with the results in \cite{Ding-2018, Kuang-Zhao-2020} that require the adiabatic constant $\gamma$ satisfies $1<\gamma \leq 3$, we only need
 $\gamma>1$ for Theorem \ref{thm:1.1} in this paper. Besides, we also consider more realistic reaction rate $\phi(T)$ which has a discontinuity at the ignition temperature $T_{\rm{i}}$, and establish the global well-posedness of the entropy solution that contains a strong combustion wave front. Therefore, the results in \cite{Ding-2018, Kuang-Zhao-2020} can be seen as special cases for Theorem \ref{thm:1.1} from this point of view.

\end{remark}

%

There are many literatures on the mathematical theories of piston problem for compressible inviscid flow, which involve the existence and structural stability of shock waves or rarefaction waves in the framework of $BV$ or $L^{\infty}$ space. When the piston moves to static gas in a tube, the global existence of shock front solutions to one-dimensional piston problem
for Euler equations was established in $BV$ space by Wang \cite{Wang-2005} and Ding \cite{Ding-2018}.
For multidimensional case, the authors in \cite{Chen-Wang-Zhang-2004, Chen-Wang-Zhang-2008} considered the axially symmetric Euler flow,
and established the global stability of multidimensional shock front solutions for both weak and strong shocks induced by axially symmetric piston motion.
On the other hand, when the piston is withdrawn in the tube, the global stability of rarefaction wave solutions to one-dimensional piston problem was completed in \cite{Ding-Kuang-Zhang-2017} through a modified  wave-front tracking scheme. In the $L^{\infty}$ framework,  the global entropy solution  for spherically symmetric piston problem was constructed  in \cite{Chen-Chen-Wang-Wang-2005} by means of compensated compactness and shock capturing.

For exothermically reacting Euler flow, by developing a fractional-step Glimm scheme, Chen and Wagner \cite{Chen-Wagner-2003} firstly established the global existence of entropy solutions for one-dimensional Cauchy problem without smallness assumption on initial data in $BV$ norm. We refer to \cite{Liu-1977, Schochet-1991} for more details about the $BV$ solutions in gas dynamics with large data. After that, Ding \cite{Ding-2017} considered the piston problem of reacting fluid and obtained
the global existence and asymptotic behavior of the entropy solutions that contain a strong rarefaction wave.
Recently, the authors in \cite{Kuang-Zhao-2020} also studied shock front solution to one-dimensional piston problem with reaction, and established the structural stability of strong shock wave as well as its asymptotic behavior by fractional-step wave front tracking scheme.
Particularly, all the previous literatures on combustion equations require that the reaction rate $\phi(T)$ is continuous and positive everywhere
in order to derive the exponential decay of exothermic source.

Finally, we remark that there are a lot of important results on well-posedness theory for general hyperbolic conservation or balance laws in one dimension situation.
For more details, we can refer to \cite{Lewicka-2004, Lewicka-Trivisa-2002} for Cauchy problem with initial data containing large profile, \cite{Amadori-1997, Donadello-Marson-2007} for initial-boundary value problem of conservation laws, \cite{Amadori-Goss-Guerra-2002, Amadori-Guerra-2002, Christoforou-2006, Colombo-Guerra-2010, Dafermos-Hsiao-1982, Liu-1979} for Cauchy problem or initial-boundary value problem with inhomogeneous term, and the references therein.

In this paper, we establish the well-posedness theory of one-dimensional piston problem for compressible and combustible Euler flows within the $BV\cap L^1$ framework. To the best of our knowledge, it is the first work on $BV\cap L^1$ theory for inviscid combustible fluids with physical ignition condition \eqref{eq:1.4}.
Mathematically, this problem can be formulated as an initial-boundary value problem for one-dimensional hyperbolic balance laws with a discontinuous source.
There are two main technical difficulties arising from ignition temperature and characteristic boundary. For the presence of ignition temperature $T_{\text{i}}$ in \eqref{eq:1.4}, the partial reaction phenomena will occur such that the decrease rate of exothermic source varies greatly within reaction zone. It leads to the
loss of uniform dissipation structures (for instance exponential decay), so that the arguments in \cite{Chen-Wagner-2003,Ding-2017, Hu-2018, Hu-2019, Kuang-Zhao-2020} are invalid here.
Besides, nonuniform heat release definitely amplifies the oscillation of combustion waves, and probably induces the instability of free boundary of flame.
The characteristic boundary of multiple degenerate fields is another difficulty of proving well-posedness. Concerning the time derivative of stability functional near characteristic boundary, the linear terms of distance index $q_i$ of two solutions (see Definition \ref{Def: distance index}) cannot be completely canceled by means of direct weight manipulation as in \cite{Bressan-2000,Bressan-Liu-Yang-1999,Lewicka-2004,Colombo-Guerra-2010}.
This is the major reason why standard Lyapunov functional does not work for the characteristic boundary value problems of many hyperbolic systems.

In dealing with the obstacles stated above, we employ a fractional-step wave front tracking scheme to construct approximate solutions, and then devise a modified Glimm-type functional $\mathcal{G} (t)$ for seeking their uniform $BV$ bound.
Observe that the product of reflection coefficients on the strong shock and on the piston boundary is less than one (see Propositions \ref{prop:2.2}--\ref{prop:2.3} below).
Thus we can choose appropriate weights in $\mathcal{G} (t)$ to prove its monotonicity in non-reaction process, which indicates the global existence of the entropy solutions to (\emph{IBVP}) without combustion.
Based on this, we further consider the reacting flows dominated by ignition temperature. The crux of this matter is how to calculate the growth of oscillation within reaction zone if the source fails to uniformly decay.
To achieve this, we utilize nonuniform spatial estimates in present paper instead of the temporal decay of dissipation in earlier framework. Insert a special consumption term $\mathcal{L}_{\textsc{by}}(t)$ into functional $\mathcal{G} (t)$, which accounts for the consumption of reactant along the strong shock front and piston boundary. This crucial term finally offsets the growth of combustion waves, and enables us to establish the existence of combustion solutions.

We further apply the technique of Lyapunov functional and quasi-characteristics analysis to $L^1$-stability argument. Precisely, take two approximate solutions $U^{\varepsilon}$ and $V^{\varepsilon}$ with data $(U^{\varepsilon}_{0},u^{\varepsilon}_{\mathfrak{p}})$ and $(V^{\varepsilon}_{0},v^{\varepsilon}_{\mathfrak{p}})$ respectively, and consider the maximal curve $x=\chi^{\textsc{m}}_{\mathfrak{p}}(t) $ of distinct piston paths for  $U^{\varepsilon}$ and $V^{\varepsilon}$. We design
a weighted Lyapunov functional $\mathscr{L}(U^{\varepsilon}, V^{\varepsilon})$, which is a equivalent metric to $||U^{\varepsilon}-V^{\varepsilon}||_{L^1([\chi^{\textsc{m}}_{\mathfrak{p}} , +\infty))}$.
For the purpose of stability estimates, a concise weights distribution is proposed to balance the nonlinear effects arising from large shock, characteristic boundary and exothermic reaction.
Then, we will show that $\mathscr{L}(U^{\varepsilon}, V^{\varepsilon})$ is strictly decreasing in reaction process. To cope with the difficulty caused by characteristic boundary, we first establish a new accurate estimate on different distance indices of genuinely nonlinear fields (see \eqref{estimate q3+ for case 4.3.7} in section 5.2) by employing the degeneracy feature of contact discontinuity.
As a result, all of the linear terms of index $q_i$ can be successfully canceled near the piston boundary. It implies that the boundary parts of $\frac{d}{d t}\mathscr{L}(U^{\varepsilon}, V^{\varepsilon})$ are dominated by the velocities difference of $U^{\varepsilon}$ and $V^{\varepsilon}$ at $\chi^{\textsc{m}}_{\mathfrak{p}}$.
Subsequently, by introducing a tool of quasi-characteristic curve,  we establish the velocities comparison estimate along the characteristic boundary and maximal curve.
Based on linear terms cancellation and velocities comparison, we obtain the stability estimate on piston boundaries, then eventually establish the $L^1-$stability
result for both non-reacting and reacting flows.


The remainder of this paper will be arranged as follows.
In section 2, we consider the elementary wave curves for homogeneous system, and then study the Riemann problems including perturbations of strong shock front and piston boundary.
Moreover, some local estimates of interaction between weak wave and strong shock, as well as the weak wave reflection on piston boundary are also presented.
We then construct the approximate solutions to (\emph{IBVP}) in section 3 by proposing a fractional-step wave front tracking algorithm.
Section 4 is devoted to the global existence of entropy solutions for both non-reacting and reacting flows by means of modified Glimm-type functional. In addition, the notion of quasi-characteristic curve is introduced, and the velocities comparison in $L^1$ norm on distinct curves is shown.
In section 5, we further develop the weighted Lyapunov functional and derive $L^1$-stability estimates of entropy solutions.
As a byproduct, we can show the uniqueness of entropy solution for non-reacting flow.
In the end, section 6 exhibits the uniqueness of limit solution for reacting flow obtained by fractional-step wave front tracking scheme.

\section{Homogenous system}\setcounter{equation}{0}

In this section, as a preliminary, we first present some elementary wave curves and their properties for the homogeneous system reduced from \eqref{eq:1.1}, and then study the Riemann problems which involve the piston boundary
as well as the strong shock front. Based on these facts, we can further show some local estimates on interaction between weak waves, weak wave and strong shock, as well as the weak wave reflection on the boundary.

\subsection{Elementary wave curves for homogeneous system}

Let
\begin{eqnarray*}
\begin{split}
&E(U) =\Big(\rho, \rho u, \rho\big(\frac{1}{2}u^2+e\big), \rho Y\Big)^\top,\\
&F(U)=\Big( \rho u, \rho u^2 +p, \rho u\big(\frac{1}{2}u^2+e+\frac{p}{\rho}\big), \rho uY\Big)^\top,\\
&G(U) = \Big(0,0, \mathfrak{q}_0\rho Y\phi (T), -\rho Y\phi (T) \Big)^\top.
 \end{split}
\end{eqnarray*}
Then equations  \eqref{eq:1.1} can be rewritten as the following hyperbolic balance laws:
\begin{equation}\label{eq:2.1}
\partial_t E(U) + \partial_x F(U)=G(U).
\end{equation}
When $G(U)\equiv0$, equations \eqref{eq:2.1} become one-dimensional hyperbolic system of conservation laws, \emph{i.e.},
\begin{equation}\label{eq:2.2}
\partial_t E(U) + \partial_x F(U)=0.
\end{equation}
The system \eqref{eq:2.2} admits four eigenvalues:
\begin{eqnarray*}\label{eq:2.3}
\begin{split}
&\lambda_1(U) = u-c, \quad\ \lambda_2(U)= \lambda_3(U) =u,\quad\ \lambda_4(U) =u+c,
\end{split}
\end{eqnarray*}
and the corresponding linearly independent eigenvectors are
\begin{eqnarray*}\label{eq:2.4}
\begin{split}
&\textbf{r}_1 (U)= \frac{2}{(\gamma+1)c}\big(-\rho,\  c,\ -\gamma p,\ 0\big)^\top,\quad\ \textbf{r}_2 (U) = (1,0,0,0)^\top,\\[5pt]
& \textbf{r}_3 (U) = (0,0,0,1)^\top, \quad\ \textbf{r}_4 (U) =\frac{2}{(\gamma+1)c} \big(\rho,\ c,\ \gamma p,\ 0 )^\top,
\end{split}
\end{eqnarray*}
where $c$ is the sound speed defined by $c=\sqrt{\frac{\gamma p}{\rho}}$. Moreover, a direct computation shows that
\begin{eqnarray*}
\begin{split}
\nabla_{U}\lambda_i(U)\cdot \textbf{r}_i(U)\equiv1 \quad \mbox{for}\  i=1,4, \quad\ \nabla_{U}\lambda_j(U)\cdot \textbf{r}_j(U)\equiv0\quad \mbox{for}\  j=2,3.
\end{split}
\end{eqnarray*}
This means that the $i$th characteristic fields are genuinely nonlinear for $i=1,4$, while the $j$th characteristic fields are linearly degenerate for $j=2,3$.

Given constant state $U_*=(\rho_*, u_*, p_*, Y_*)^{\top}$ away from vaccum,
the \emph{Rankine-Hugoniot} condition with wave speed ${\textrm{s}}$, i.e.,
\begin{eqnarray}\label{eq:2.6}
\begin{split}
F(U)-F(U_*)={\textrm{s}}\big(E(U)-E(U_*)\big),
\end{split}
\end{eqnarray}
for system \eqref{eq:2.2} determines the \emph{Hugoniot loci} through $U_*$ in state space as follows:

\begin{eqnarray}\label{eq:2.5}
\begin{split}
&{\mathscr{S}}_1(U_*):\ \frac{p}{p_*}=\frac{\eta \rho- \rho_*}{\eta \rho_*- \rho},\quad
\frac{u-u_*}{c_* }=-\sqrt{\frac{2}{\gamma-1}}\frac{\rho-\rho_*}{\sqrt{(\eta\rho_*-\rho)\rho}}, \quad Y=Y_*, \quad \rho \neq \rho_*;\\[5pt]
&\mathscr{S}_2(U_*):\ \rho\neq \rho_* ,\qquad   u=u_*,\qquad p=p_*,\qquad Y=Y_*;\\[5pt]
&\mathscr{S}_3(U_*):\ \rho= \rho_* ,\qquad   u=u_*,\qquad  p=p_*,\qquad Y\neq Y_*;\\[5pt]
&{\mathscr{S}}_4(U_*):\ \frac{p}{p_*}=\frac{\eta \rho- \rho_*}{\eta \rho_*- \rho},\quad
\frac{u-u_*}{c_* }=\sqrt{\frac{2}{\gamma-1}}\frac{\rho-\rho_*}{\sqrt{(\eta\rho_*-\rho)\rho}}, \quad Y=Y_*, \quad \rho \neq \rho_*,
\end{split}
\end{eqnarray}
where constant $\eta=\frac{\gamma+1}{\gamma-1}$.
Let the Hugoniot locus ${\mathscr{S}}_i(U_*)$ be parameterized by mapping $\alpha_i \mapsto S_i (\alpha_i)(U_*)$  with
\begin{equation}\label{parameterization for Si}
  \alpha_i= \left\{
  \begin{aligned}
    & \lambda_i (S_i (\alpha_i)(U_*)) - \lambda_i(U_*) \ \ \  && \text{if } i=1,4,\\
    & \rho- \rho_*  &&  \text{if } i=2,\\
    & Y- Y_*  &&  \text{if } i=3.
  \end{aligned}
  \right.
\end{equation}
The isentropes through $U_*$ for system \eqref{eq:2.2} read
\begin{eqnarray}\label{eq:2.7}
\begin{split}
&\mathscr{R}_{1}(U_*):\ u+\frac{2c}{\gamma-1}=u_*+\frac{2c_*}{\gamma-1},\qquad \frac{p}{\rho^{\gamma}}=\frac{p_*}{\rho^{\gamma}_*},
\qquad Y=Y_*;\\[5pt]
&\mathscr{R}_{4}(U_*):\ u-\frac{2c}{\gamma-1}=u_*-\frac{2c_*}{\gamma-1},\qquad \frac{p}{\rho^{\gamma}}=\frac{p_*}{\rho^{\gamma}_*},
\qquad Y=Y_*,
\end{split}
\end{eqnarray}
which are parameterized by mapping  $\alpha_i \mapsto R_i (\alpha_i)(U_*)$ with
\begin{equation*}
\alpha_i= \lambda_i (R_i (\alpha_i)(U_*)) - \lambda_i(U_*),\qquad  i=1,4.
\end{equation*}

Notice that the last equation for $Y$ in \eqref{eq:2.2} can be decoupled from the first three equations (\emph{i.e.}, non-isentropic Euler equations).
Henceforth, we can rewrite any state
$$U=(U^{\textsc{e}}, Y)^{\top}\ \ \text{with} \ \ U^{\textsc{e}}\doteq(\rho, u, p)^{\top}.$$
Then the states along parameterized curves ${\mathscr{S}}_i(U_*)$ and ${\mathscr{R}}_i(U_*)$ in \eqref{eq:2.5}\eqref{eq:2.7} are decomposed into
Eulerian part and mass fraction part as follows,
\begin{align*}\label{eq:2.9}
& S_i(\alpha_i)(U_*)=\big(S^{\textsc{e}}_i(\alpha_i)(U^{\textsc{e}}_*), Y_*\big)^{\top} \quad  (i=1,2,4), \\
& R_i(\alpha_i)(U_*)=\big(R^{\textsc{e}}_i(\alpha_i)(U^{\textsc{e}}_*), Y_*\big)^{\top} \quad (i=1,4), \\
& S_3(\alpha_3)(U_*)=(U^{\textsc{e}}_*, S^{\textsc{y}}_3(\alpha_3)(Y_*))^{\top}.
\end{align*}

Treat $U_{\rm{l}}=(U^{\textsc{e}}_{\rm{l}}, Y_{\rm{l}})^{\top}$ and
$U_{\rm{r}}=(U^{\textsc{e}}_{\rm{r}}, Y_{\rm{r}})^{\top}$ as the left and right states of Riemann problem $(U_{\rm{l}}, U_{\rm{r}})$.
Thus, according to \emph{Lax's entropy condition}, the elementary wave curves projected into $U^{\textsc{e}}-$subspace are defined by
\begin{equation*}\label{eq:2.12}
\begin{split}
&\Phi^{\textsc{e}}_i(\alpha_i)( U^{\textsc{e}}_{\rm{l}})= \left\{
\begin{aligned}
S^{\textsc{e}}_i(\alpha_i)(U^{\textsc{e}}_{\rm{l}}) \qquad & \mbox{for}\ \alpha_i<0,\\[5pt]
R^{\textsc{e}}_i(\alpha_i)(U^{\textsc{e}}_{\rm{l}}) \qquad & \mbox{for}\  \alpha_i\geq0,
\end{aligned}
\right.\ \ \ \ \ (i=1,4) \\[5pt]
&\Phi^{\textsc{e}}_{2}(\alpha_2)(U^{\textsc{e}}_{\rm{l}})=S^{\textsc{e}}_{2}(\alpha_{2})(U^{\textsc{e}}_{\rm{l}}).
 \end{split}
 \end{equation*}
It is clear that
\begin{equation*}\label{eq:2.16}
\Phi^{\textsc{e}}_i(\alpha_i)( U^{\textsc{e}}_{\rm{l}})\big|_{\alpha_i=0}=U^{\textsc{e}}_{\rm{l}},\qquad \frac{d\Phi^{\textsc{e}}_i(\alpha_i)( U^{\textsc{e}}_{\rm{l}})}{d\alpha_i}\Big|_{\alpha_i=0}=\textbf{r}^{\textsc{e}}_i (U^{\textsc{e}}_{\rm{l}})\qquad  (i=1,2,4).
\end{equation*}
Here $\textbf{r}^{\textsc{e}}_i (U^{\textsc{e}})$ consists of the first three components of eigenvector $\textbf{r}_i (U)$.
Accordingly, denote the projected elementary curves through $U^{\textsc{e}}_{\rm{r}}$ by
\begin{equation*}
\begin{split}
&\Psi^{\textsc{e}}_i(\alpha_i)( U^{\textsc{e}}_{\rm{r}})= \left\{
\begin{aligned}
S^{\textsc{e}}_i(\alpha_i)(U^{\textsc{e}}_{\rm{r}}) \qquad & \mbox{for}\ \alpha_i>0,\\[5pt]
R^{\textsc{e}}_i(\alpha_i)(U^{\textsc{e}}_{\rm{r}}) \qquad & \mbox{for}\  \alpha_i\leq0,
\end{aligned}
\right.\ \ \ \ \ (i=1,4) \\
&\Psi^{\textsc{e}}_{2}(\alpha_2)( U^{\textsc{e}}_{\rm{r}})=S^{\textsc{e}}_{2}(\alpha_{2})(U^{\textsc{e}}_{\rm{r}}).
\end{split}
\end{equation*}
With regard to background shock in \eqref{eq:1.8}, we claim that
\begin{equation}\label{eq:2.12}
  \bar{U}^{\textsc{e}}_{\mathfrak{b},\textrm{r}}=\Phi^{\textsc{e}}_{4}(\bar{\alpha}^{\rm{s}}_{4})(\bar{U}^{\textsc{e}}_{\mathfrak{b},\textrm{l}}), \quad
  \bar{U}^{\textsc{e}}_{\mathfrak{b},\textrm{l}}=\Psi^{\textsc{e}}_{4}(-\bar{\alpha}^{\rm{s}}_{4})(\bar{U}^{\textsc{e}}_{\mathfrak{b},\textrm{r}})
\end{equation}
for a negative $\bar{\alpha}^{\rm{s}}_{4}= \lambda_4(\bar{U}^{\textsc{e}}_{\mathfrak{b},\textrm{r}})-  \lambda_4(\bar{U}^{\textsc{e}}_{\mathfrak{b},\textrm{l}})$.

Let
$\textbf{r}^{\textsc{e}}_i (U^{\textsc{e}}_*, \alpha_i)$ be the tangent vector for curve $S^{\textsc{e}}_i(\alpha_i)(U^{\textsc{e}}_*)$, while
$\textbf{r}^{\textsc{e}}_i ({R^{\textsc{e}}_{i}}(\alpha_i)(U^{\textsc{e}}_*))$ be the tangent vector for curve ${R^{\textsc{e}}_{i}}(\alpha_i)(U^{\textsc{e}}_*)$.
Taking $U^{\textsc{e}}_*=U^{\textsc{e}}_{\rm{r}}$ when $i=1,4$, we define the tangent field along curve $\Psi^{\textsc{e}}_i (\alpha_i)( U^{\textsc{e}}_{\rm{r}})$ by
\begin{equation}\label{def:tangent vector}
\tilde{\textbf{r}}^{\textsc{e}}_i (U^{\textsc{e}}_{\rm{r}}, \alpha_i)\doteq \left\{
\begin{aligned}
&\textbf{r}^{\textsc{e}}_i (U^{\textsc{e}}_{\rm{r}},\alpha_i) \quad & \mbox{for}\ \alpha_i>0,\\[5pt]
&\textbf{r}^{\textsc{e}}_i (R^{\textsc{e}}_i(\alpha_i)(U^{\textsc{e}}_{\rm{r}})) \quad & \mbox{for}\ \alpha_i \leq 0.
\end{aligned}
\right.
\end{equation}
In particular, one has $\textbf{r}^{\textsc{e}}_i (U^{\textsc{e}}_{\rm{r}},0)=\textbf{r}^{\textsc{e}}_i (U^{\textsc{e}}_{\rm{r}})$.

Henceforth, we write the derivatives $\displaystyle\dot\ \doteq\frac{d}{dt}$ and $\displaystyle '\doteq  \frac{d}{d\alpha_4}$ for brevity.

The previous parametrization of elementary curves will provide us some technical advantages to analyze the solvability,  wave-interaction and stability, etc.
For instance, the following lemma exhibits the monotonicity of fluid state along the shock curve.

\begin{lemma}\label{lem:2.1a}
Assume state $U^{\textsc{e}}_{\rm{r}}=(\rho_{\rm{r}},u_{\rm{r}},p_{\rm{r}})^{\top}$ is away from vacuum.
Along the shock curve corresponding to the compressive branch of $\Psi^{\textsc{e}}_4(\alpha_4)( U^{\textsc{e}}_{\rm{r}})$, the density, velocity, pressure and shock speed
are all increasing with respect to parameter $\alpha_4$; i.e.,
$$
\rho'(\alpha_4)>0, \quad u'(\alpha_4)>0, \quad  p'(\alpha_4)>0, \quad  {\rm{s}}'(\alpha_4)>0,
$$
where $(\rho, u, p) =U^{\textsc{e}}= S^{\textsc{e}}_4 (\alpha_4)(U^{\textsc{e}}_{\rm{r}})$ with $\alpha_4>0$.
\end{lemma}

{
\begin{proof}
Suppose state $U=(\rho ,u , p, Y)^{\top}=S_4 (\alpha_4)(U_{\rm{r}})$.
One has $\alpha_4= \lambda_4(U) - \lambda_4(U_{\rm{r}})>0$ iff
$\rho\in (\rho_{\rm{r}}, \eta \rho_{\rm{r}})$, where constant $\eta=(\gamma+1)/(\gamma-1)$.
Clearly, Hugoniot locus ${\mathscr{S}}_4(U_{\rm{r}})$ in \eqref{eq:2.5} can be treated as a curve of single parameter $\rho$.
Since $U\in {\mathscr{S}}_4(U_{\rm{r}})$, we differentiate $u$ and $p$ with respect to $\rho$ and obtain

\begin{equation}\label{eq:A1}
\begin{aligned}
  & \frac{dp}{d\rho}  = p_{\rm{r}}\rho_{\rm{r}} \cdot\frac{\eta^2-1}{(\eta\rho_{\rm{r}}- \rho)^2}>0 ,\\
  & \frac{du}{d\rho}   = c_{\rm{r}}\rho_{\rm{r}} \cdot\frac{\sqrt{\eta-1}}{2} \cdot \frac{(\eta-2)\rho+ \eta \rho_{\rm{r}}}{(\eta \rho_{\rm{r}}\rho- \rho^2)^{3/2}} >0
\end{aligned}
\end{equation}
for any $\rho\in (\rho_{\rm{r}}, \eta \rho_{\rm{r}})$; furthermore,
\begin{align}
\notag
  & \frac{dc}{d\rho}   = c_{\rm{r}} \sqrt{\rho_{\rm{r}}} \cdot\frac{(\rho-\rho_{\rm{r}})^2 + (\eta-1)(\rho^2+\rho_{\rm{r}}^2)}{2\rho^{3/2}(\eta\rho_{\rm{r}}- \rho)^{3/2}(\eta \rho- \rho_{\rm{r}})^{1/2}}>0 ,\\
\label{eq:A2}
  & \frac{d^2}{d\rho^2} (\rho u) = c_{\rm{r}}\rho_{\rm{r}}^2 \cdot \frac{\sqrt{\eta-1}}{4} \cdot\frac{\eta \rho \big((3\eta -4)\rho + \eta \rho_{\rm{r}}\big)}{(\eta \rho_{\rm{r}}\rho- \rho^2)^{5/2}} >0 .
\end{align}
They immediately give
\begin{equation}\label{eq:A3}
\frac{d\alpha_4}{d\rho} = \frac{du}{d\rho}+\frac{dc}{d\rho} >0.
\end{equation}
From \eqref{eq:A1} and \eqref{eq:A3}, we see that
$$\rho'>0, \qquad   p'=\frac{dp}{d\rho} \cdot \frac{d\rho}{d\alpha_4}>0, \qquad
u'=\frac{du}{d\rho} \cdot \frac{d\rho}{d\alpha_4}>0. $$

The Hugoniot-Rankine condition $\textrm{s}\big(E(U)-E(U_{\rm{r}})\big) = F(U)-F(U_{\rm{r}})$ determines the $4-$shock speed   by $\textrm{s}=[\rho u]/ [\rho]$,
where $[\rho]= \rho-\rho_{\rm{r}}>0$. Note that inequality \eqref{eq:A2} implies the fact momentum $\rho u$ is convex in $\rho$.
Hence we have
$$\textrm{s}'  = \frac{d\textrm{s}}{d\rho} \cdot \frac{d\rho}{d\alpha_4} =  \frac{\rho'}{[\rho]} \cdot
 \left\{ \frac{d(\rho u )}{d\rho} -\frac{[\rho u]}{[\rho]}  \right\} >0.$$

\end{proof}
}

\subsection{Riemann problem and local interaction estimates for homogeneous system }

First, we consider the Riemann problem of \eqref{eq:2.2} with the data
\begin{eqnarray}\label{eq:2.18}
U|_{t=\hat{t}}=\left\{
\begin{array}{llll}
U_{\rm{l}}=(U^{\textsc{e}}_{\rm{l}}, Y_{\rm{l}})^{\top}, \qquad x<\hat{x},\\[5pt]
U_{\rm{r}}=(U^{\textsc{e}}_{\rm{r}}, Y_{\rm{r}})^{\top}, \qquad x>\hat{x},
\end{array}
\right.
\end{eqnarray}
where $U^{\textsc{e}}_{\rm{l}}=(\rho_{\rm{l}}, u_{\rm{l}},p_{\rm{l}})^{\top}$ and
$U^{\textsc{e}}_{\rm{r}}=(\rho_{\rm{r}}, u_{\rm{r}},p_{\rm{r}})^{\top}$.
Then, its Riemann solver involving only weak waves is given by the following lemma.

\begin{lemma}\label{lem:2.1}
There exists $\delta>0$ sufficiently small such that for $U_{\rm{l}}, U_{\rm{r}}\in \mathcal{N}_{\delta}
(\bar{U}_{\mathfrak{b}, \rm{l}})$ $($or $\mathcal{N}_{\delta}(\bar{U}_{\mathfrak{b}, \rm{r}}))$, the Riemann problem \eqref{eq:2.2}\eqref{eq:2.18}
admits a unique admissible solution that consists of at most four constant states separated by four elementary waves. Moreover, $U_{\rm{l}}$ and $U_{\rm{r}}$ satisfy
\begin{eqnarray*}\label{2.19}
U^{\textsc{e}}_{\rm{r}}=\Phi^{\textsc{e}}_{4}(\alpha_{4})\circ \Phi^{\textsc{e}}_{2}(\alpha_{2})\circ
\Phi^{\textsc{e}}_{1}(\alpha_{1})( U^{\textsc{e}}_{\rm{l}}),\quad Y_{\rm{r}}=Y_{\rm{l}}+ \alpha_3,
\end{eqnarray*}
and
\begin{eqnarray*}\label{2.20}
\sum^{4}_{j=1}|\alpha_{j}|\leq \mathcal{O}(1)\Big(|U^{\textsc{e}}_{\rm{r}}-U^{\textsc{e}}_{\rm{l}}|+|Y_{\rm{r}}-Y_{\rm{l}}|\Big).
\end{eqnarray*}
\end{lemma}
Throughout this paper, symbol $\mathcal{O}(1)$ denotes the bounded quantity that depends solely on $\bar{U}_{\mathfrak{b}}$.

By Lemma \ref{lem:2.1} and the arguments in \cite{Bressan-2000}, we have the following the local estimates on interaction between weak waves.
\begin{proposition}\label{prop:2.1}
Assume that the solution $U(t,x)$ to system \eqref{eq:2.2} contains three adjacent constant states $U_{\rm{l}}=(U^{\textsc{e}}_{\rm{l}}, Y_{\rm{l}})^{\top}, U_{\rm{m}}=(U^{\textsc{e}}_{\rm{m}}, Y_{\rm{m}})^{\top},
U_{\rm{r}}=(U^{\textsc{e}}_{\rm{r}}, Y_{\rm{r}})^{\top} \in \mathcal{N}_{\delta}(\bar{U}_{\mathfrak{b}, \rm{r}})$ $($or $\mathcal{N}_{\delta}(\bar{U}_{\mathfrak{b}, \rm{l}}))$ with $\delta>0$ sufficiently small. They are separated by two incoming waves with strengths $\alpha^*_{i}, \alpha^*_{j}$ respectively.

\noindent{\textnormal{(i)}}\ If
$1 \leq j\leq i\leq 4$ and $i,j \neq3 $ such that
$U^{\textsc{e}}_{\rm{m}}= \Phi^{\textsc{e}}_{i}(\alpha^*_{i})( U^{\textsc{e}}_{\rm{l}}), \
U^{\textsc{e}}_{\rm{r}}=\Phi^{\textsc{e}}_{j}(\alpha^*_{j})( U^{\textsc{e}}_{\rm{m}}), \
 Y_{\rm{l}}=Y_{\rm{m}}=Y_{\rm{r}}$,
then for the outgoing waves determined by equality
$U^{\textsc{e}}_{\rm{r}}=\Phi^{\textsc{e}}_{4}(\alpha_{4})\circ \Phi^{\textsc{e}}_{2}(\alpha_{2})\circ
\Phi^{\textsc{e}}_{1}(\alpha_{1})( U^{\textsc{e}}_{\rm{l}})$,
there hold
\begin{eqnarray}\label{eq:2.21}
\begin{split}
&|\alpha_{i}-\alpha^*_{i}|+|\alpha_{j}-\alpha^*_{j}|+\sum_{k\neq i,j,3}|\alpha_{k}|=\mathcal{O}(1)|\alpha^*_{i}\alpha^*_{j}| \qquad &\text{if }\ i> j,\\[5pt]
&|\alpha_{i}-\alpha^*_{i}-\alpha^*_{j}|+\sum_{k\neq i,3}|\alpha_{k}|=\mathcal{O}(1)|\alpha^*_{i}\alpha^*_{j}| \qquad &\text{if }\ i=j.
\end{split}
\end{eqnarray}

\noindent{\textnormal{(ii)}}\ If
$i=1$ and $j=3 $ such that
$U^{\textsc{e}}_{\rm{m}}=U^{\textsc{e}}_{\rm{l}}$, $Y_{\rm{m}}=Y_{\rm{l}}+ \alpha_3^*$ and
$U^{\textsc{e}}_{\rm{r}}=  \Phi^{\textsc{e}}_{1}(\alpha^*_{1})( U^{\textsc{e}}_{\rm{m}})$, $ Y_{\rm{r}}=Y_{\rm{m}}$, then for
$\alpha_k$ $(1\leq k\leq 4)$ determined by
\begin{equation}\label{eq:2.22}
  U^{\textsc{e}}_{\rm{r}}=\Phi^{\textsc{e}}_{4}(\alpha_{4})\circ \Phi^{\textsc{e}}_{2}(\alpha_{2})\circ
\Phi^{\textsc{e}}_{1}(\alpha_{1})( U^{\textsc{e}}_{\rm{l}}), \ \ \ \ Y_{\rm{r}}=Y_{\rm{l}}+ \alpha_3,
\end{equation}
there hold
\begin{eqnarray*}\label{eq:2.23}
\begin{split}
\alpha_{1}=\alpha^*_{1},\quad \alpha_{3}=\alpha^*_{3},\quad \alpha_{2}=\alpha_{4}=0 .
\end{split}
\end{eqnarray*}
Similarly, if
$i= 4$ and $j=3 $ such that
$U^{\textsc{e}}_{\rm{m}} =\Phi^{\textsc{e}}_{4}(\alpha^*_{4})( U^{\textsc{e}}_{\rm{l}})$, $Y_{\rm{m}}=Y_{\rm{l}}$ and
$U^{\textsc{e}}_{\rm{r}}= U^{\textsc{e}}_{\rm{m}}$, $ Y_{\rm{r}}=Y_{\rm{m}}+ \alpha_3^*$, then for
$\alpha_k$ $(1\leq k\leq 4)$ determined by \eqref{eq:2.22},
there hold
\begin{eqnarray*}\label{eq:2.24}
\begin{split}
\alpha_{1}=\alpha_{2}=0, \quad \alpha_{3}=\alpha^*_{3},\quad \alpha_{4}=\alpha^*_{4}.
\end{split}
\end{eqnarray*}

\end{proposition}

\begin{figure}[ht]
\begin{minipage}[t]{0.45\textwidth}
\centering
\begin{tikzpicture}[scale=1.1]
\draw [line width=0.04cm](-1.0,-0.8)to (-0.3,0.5)to(-0.7,1.9);
\draw [thin][blue](-0.3,0.5)to(0.6,1.2);
\draw [thin][blue](-0.3,0.5)to(0.7,1.1);
\draw [thin][blue](-0.3,0.5)to(0.7,1.0);

\node at (1.1, 1.2) {$\alpha_{4}$};
\node at (0.0, 1.3) {$U_{\rm{l}}$};
\node at (0.5, 0.4) {$U_{\rm{r}}$};

\node at (-1.2, 0.5) {$(\hat{t}, \chi^{\varepsilon}_{\mathfrak{p}}(\hat{t}))$};
\end{tikzpicture}
\caption{Riemann problem near the corner point of the piston}\label{fig3}
\end{minipage}
\begin{minipage}[t]{0.5\textwidth}
\centering
\begin{tikzpicture}[scale=1.1]
\draw [line width=0.04cm](-1.0,-0.8)--(0.0,2.0);
\draw [thin][blue](0.4,-0.6)--(-0.5,0.6);
\draw [thin](-0.5,0.6)--(0.9,1.2);

\node at (0.3, 1.4) {$U^{+}_{\mathfrak{b}}$};
\node at (0.4, 0.4) {$U_{\rm{r}}$};
\node at (-0.3, -0.4) {$U^{-}_{\mathfrak{b}}$};
\node at (-1.3, 0.5) {$(\hat{t}, \chi^{\varepsilon}_{\mathfrak{p}}(\hat{t}))$};

\node at (1.3, 1.25) {$\alpha_{4}$};
\node at (0.7, -0.7) {$\alpha^{*}_{1}$};

\end{tikzpicture}
\caption{Weak wave reflection on the piston}\label{fig4}
\end{minipage}
\end{figure}

Next we turn to study the Riemann problem for the system \eqref{eq:2.2} near the piston with the following initial-boundary data:
\begin{eqnarray}\label{eq:2.27}
\left\{
\begin{array}{llll}
u=u_{\mathfrak{p}}(\hat{t}+0) &\qquad \mbox{on}\quad \big\{(t,x): x=\chi^{\varepsilon}_{\mathfrak{p}}(t), t>\hat{t}\big\},\\[5pt]
U=U_{\rm{r}} &\qquad \mbox{on}\quad \big\{(t,x): x>\chi^{\varepsilon}_{\mathfrak{p}}(t), t=\hat{t}\big\},
\end{array}
\right.
\end{eqnarray}
where $U_{\rm{r}}=(U^{\textsc{e}}_{\rm{r}}, Y_{\rm{r}})^{\top}$ for $U^{\textsc{e}}_{\rm{r}}=(\rho_{\rm{r}}, u_{\rm{r}}, p_{\rm{r}})^{\top}$, and
$\chi^{\varepsilon}_{\mathfrak{p}}(t)$ represents the polygonal approximate boundary that will be defined in section 3.
Then we have the following solvability of Riemann problem \eqref{eq:2.2} and \eqref{eq:2.27}.

\begin{lemma}\label{lem:2.2}
There exists a sufficiently small parameter $\delta>0$ such that if $u_{\rm{r}}=u_{\mathfrak{p}}(\hat{t}-0)$ and $u_{\rm{l}}=u_{\mathfrak{p}}(\hat{t}+0)$ with $u_{\mathfrak{p}}(\hat{t}\pm0)\in \mathcal{N}_{\delta}(\bar{u}_{\mathfrak{p}})$
and $U_{\rm{r}}\in \mathcal{N}_{\delta}(\bar{U}_{\mathfrak{b}, \rm{l}})$,
then the Riemann problem \eqref{eq:2.2} and \eqref{eq:2.27} admits a unique admissible solution which includes a $4$-wave with strength $\alpha_4$ (see Fig.\ref{fig3}).
Moreover, there holds the estimate
\begin{eqnarray}\label{eq:2.28}
\alpha_4=\mathcal{O}(1)|u_{\mathfrak{p}}(\hat{t}+0)-u_{\mathfrak{p}}(\hat{t}-0)|.
\end{eqnarray}
\end{lemma}

The result below shows us an estimate for weak wave reflection on the boundary.

\begin{proposition}\label{prop:2.2}
Suppose that constant states $U_{\rm{r}}=(U^{\textsc{e}}_{\rm{r}}, Y_{\rm{r}})^{\top}$, $U^{-}_{\mathfrak{b}}=(U^{\textsc{e},-}_{\mathfrak{b}}, Y^{-}_{\mathfrak{b}})^{\top}\in \mathcal{N}_{\delta}(\bar{U}_{\mathfrak{b},\rm{l}})$ and velocity $u_{\mathfrak{p}}\in \mathcal{N}_{\delta}(\bar{u}_{\mathfrak{p}})$ satisfy
\begin{eqnarray}\label{eq:2.29}
U^{\textsc{e},-}_{\mathfrak{b}}=\Psi^{\textsc{e}}_{1}(-\alpha^*_{1})( U^{\textsc{e}}_{\rm{r}}),\quad Y^{-}_{\mathfrak{b}}=Y_{\rm{r}},\quad u^{-}_{\mathfrak{b}}=u_{\mathfrak{p}},
\end{eqnarray}
and that constant state $U^+_{\mathfrak{b}}=(U^{\textsc{e},+}_{\mathfrak{b}}, Y^{+}_{\mathfrak{b}})^{\top}\in \mathcal{N}_{\delta}(\bar{U}_{\mathfrak{b},\rm{l}})$
satisfies
\begin{eqnarray}\label{eq:2.30}
U^{\textsc{e},+}_{\mathfrak{b}}=\Psi^{\textsc{e}}_{4}(-\alpha_{4})( U^{\textsc{e}}_{\rm{r}}),\quad Y^{+}_{\mathfrak{b}}=Y_{\rm{r}},\quad u^{+}_{\mathfrak{b}}=u_{\mathfrak{p}},
\end{eqnarray}
where $U^{\textsc{e},\pm}_{\mathfrak{b}}=(\rho^{\pm}_{\mathfrak{b}},u^{\pm}_{\mathfrak{b}}, p^{\pm}_{\mathfrak{b}})^{\top}$,
$\alpha^*_{1}= \lambda(U_{\rm{r}})- \lambda(U^-_{\mathfrak{b}})$ and $\alpha_{4}= \lambda(U_{\rm{r}})- \lambda(U^+_{\mathfrak{b}})$ (see Fig.\ref{fig4}).
Then, for $\delta>0$ sufficiently small, it holds the estimate
\begin{eqnarray}\label{eq:2.31}
\alpha_4=\alpha^{*}_1+\mathcal{O}(1)|\alpha^{*}_{1}|^2.
\end{eqnarray}

\begin{proof}
By \eqref{eq:2.29} and \eqref{eq:2.30}, we have the relation
$ \Psi^{\textsc{e}}_{4}(-\alpha_{4})( U^{\textsc{e}}_{\rm{r}}) \cdot \mathbf{n}
=\Psi^{\textsc{e}}_{1}(-\alpha^*_{1})( U^{\textsc{e}}_{\rm{r}}) \cdot \mathbf{n}$,
where vector $\mathbf{n}=(0,1,0)^\top $.
To evaluate $\alpha_{4}$, let's consider the function
\begin{eqnarray*}
f_{\mathfrak{b}}(\alpha_{4},\alpha^*_{1}) \doteq\Psi^{\textsc{e}}_{4}(-\alpha_{4})( U^{\textsc{e}}_{\rm{r}})\cdot \mathbf{n}
 -\Psi^{\textsc{e}}_{1}(-\alpha^*_{1})( U^{\textsc{e}}_{\rm{r}})\cdot \mathbf{n}.
\end{eqnarray*}
Notice that $f_{\mathfrak{b}}(0,0)=0$ and
$$\frac{\partial f_{\mathfrak{b}}}{\partial\alpha_{4}} \Big|_{\alpha_{4}=\alpha^*_{1}=0}=-\mathbf{r}^{\textsc{e}}_{4}(U^{\textsc{e}}_{\rm{r}})\cdot \mathbf{n} \neq0 .$$
Then, by implicit function theorem, we know that $\alpha_{4}$ can be solved as
$C^2$-function of $\alpha^{*}_1$ from the equation $f_{\mathfrak{b}}(\alpha_{4},\alpha^*_{1})=0$ provided $\delta>0$ sufficiently small. Moreover, a direct computation shows that
\begin{eqnarray*}
\frac{\partial \alpha_4}{\partial \alpha^{*}_1}\Big|_{\alpha^*_{1}=0}=\frac{\mathbf{r}^{\textsc{e}}_{1}(U^{\textsc{e}}_{\rm{r}})\cdot \mathbf{n}}
{\mathbf{r}^{\textsc{e}}_{4}(U^{\textsc{e}}_{\rm{r}})\cdot \mathbf{n} }=1.
\end{eqnarray*}
By Taylor formula, we finally derive the estimate \eqref{eq:2.31}.
\end{proof}

\end{proposition}

We further consider the Riemann problem for system \eqref{eq:2.2} with the data \eqref{eq:2.18} in different regions,
and obtain the following solvability of this problem involving strong shock.
\begin{lemma}\label{lem:2.4}
Given two constant sates $U_{\rm{l}}, U_{\rm{r}}$ in \eqref{eq:2.18} satisfying $U_{\rm{l}}\in \mathcal{N}_{\delta}
(\bar{U}_{\mathfrak{b}, \rm{l}})$ and $U_{\rm{r}}\in \mathcal{N}_{\delta}(\bar{U}_{\mathfrak{b}, \rm{r}})$, there exists a small constant $\delta>0$ such that the Riemann problem \eqref{eq:2.2}\eqref{eq:2.18}
admits a unique admissible solution which consists of at most four constant states separated by three weak waves $\alpha_k$ $(k=1,2,3)$ and a strong shock $\alpha^{\rm{s}}_4$.
Moreover, $U_{\rm{l}}$ and $U_{\rm{r}}$ satisfy
\begin{eqnarray}\label{eq:2.32}
U^{\textsc{e}}_{\rm{r}}=S^{\textsc{e}}_{4}({\alpha^{\rm{s}}_{4}}) \circ \Phi^{\textsc{e}}_{2}(\alpha_{2})\circ  \Phi^{\textsc{e}}_{1}(\alpha_{1})(U^{\textsc{e}}_{\rm{l}}),
\quad Y_{\rm{r}}=Y_{\rm{l}}+ \alpha_3.
\end{eqnarray}
\end{lemma}

\begin{proof}
It is obvious that $\alpha_{3}=Y_{\rm{r}}-Y_{\rm{l}}$ by definition \eqref{parameterization for Si}. To prove \eqref{eq:2.32} for $U^{\textsc{e}}$, we need to investigate the equivalent equation
\begin{eqnarray*}
U^{\textsc{e}}_{\rm{l}}= \Psi^{\textsc{e}}_1 (-\alpha_1)\circ \Psi^{\textsc{e}}_2(-\alpha_2)\circ S^{\textsc{e}}_4 (-{\alpha^{\rm{s}}_{4}})(U^{\textsc{e}}_{\rm{r}}).
\end{eqnarray*}
Set the intermediate state $U^{\textsc{e}}_{\rm{m}} \doteq  \Psi^{\textsc{e}}_2 (-\alpha_2)\circ S^{\textsc{e}}_4 (-{\alpha^{\rm{s}}_{4}})(U^{\textsc{e}}_{\rm{r}})$. Then define a function
\begin{eqnarray*}
f_{\rm{s}}(\alpha_1, \alpha_2, \alpha^{\rm{s}}_{4}; U^{\textsc{e}}_{\rm{l}}, U^{\textsc{e}}_{\rm{r}})
= U^{\textsc{e}}_{\rm{r}} +\int^{-{\alpha^{\rm{s}}_{4}}}_0 \textbf{r}^{\textsc{e}}_4( U^{\textsc{e}}_{\rm{r}}, \sigma) d\sigma-\alpha_2\textbf{r}^{\textsc{e}}_2
+ \int^{-\alpha_1}_0 \tilde{\textbf{r}}^{\textsc{e}}_1(U^{\textsc{e}}_{\rm{m}}, \sigma) d\sigma -U^{\textsc{e}}_{\rm{l}},
\end{eqnarray*}
where vector $\tilde{\textbf{r}}^{\textsc{e}}_1$ is given by \eqref{def:tangent vector}. The fact \eqref{eq:2.12} and Lemma \ref{lem:2.1a} directly yield that $f_{\rm{s}}(0, 0, \bar{\alpha}^{\rm{s}}_{4}; \bar{U}^{\textsc{e}}_{\mathfrak{b},\rm{l}}, \bar{U}^{\textsc{e}}_{\mathfrak{b},\rm{r}})\equiv0$ and
\begin{eqnarray*}
\begin{split}
\det \big(\nabla_{(\alpha_1, \alpha_2, \alpha^{\rm{s}}_{4})} f_{\rm{s}} (0, 0, \bar{\alpha}^{\rm{s}}_{4}; \bar{U}^{\textsc{e}}_{\mathfrak{b},\rm{l}}, \bar{U}^{\textsc{e}}_{\mathfrak{b},\rm{r}})\big)
&=-\det \big(\textbf{r}^{\textsc{e}}_1(\bar{U}^{\textsc{e}}_{\mathfrak{b},\rm{l}}), \textbf{r}^{\textsc{e}}_2(\bar{U}^{\textsc{e}}_{\mathfrak{b},\rm{l}}),  \textbf{r}^{\textsc{e}}_4( \bar{U}^{\textsc{e}}_{\mathfrak{b},\rm{r}}, -\bar{\alpha}^{\rm{s}}_{4})\big)\\[5pt]
&= \frac{2\big(c_{\rm{l}} p'(-\bar{\alpha}^{\rm{s}}_{4}) + \gamma p_{\rm{l}} u'(-\bar{\alpha}^{\rm{s}}_{4})\big)}{(\gamma+1)c_{\rm{l}}}>0.
\end{split}
\end{eqnarray*}
If choosing $\delta>0$ small enough, then by implicit function theorem, $f_{\rm{s}}=0$ has a unique solution $(\alpha_1, \alpha_2, \alpha^{\rm{s}}_{4})$ near
$(0, 0, \bar{\alpha}^{\rm{s}}_{4})$. Finally, relation \eqref{eq:2.32} is a direct conclusion of $f_{\rm{s}}=0$.

\end{proof}


\begin{figure}[ht]
\begin{minipage}[t]{0.45\textwidth}
\centering
\begin{tikzpicture}[scale=1.1]
\draw [line width=0.04cm][red](-0.7,-0.7)--(-0.3,0.5);
\draw [thin](-1.6,-0.4)--(-0.3,0.5);

\draw [thin](-0.3,0.5)--(-0.6,1.8);
\draw [dashed][thin][blue](-0.3,0.5)to(0.5,1.7);
\draw [line width=0.04cm][red](-0.3,0.5)to(0.8, 1.0);

\node at (1.1, 1.1) {$\alpha^{\rm{s}}_{4}$};
\node at (0.55, 1.95) {$\alpha_{2(3)}$};
\node at (-0.6, 2.0) {$\alpha_{1}$};
\node at (-1.8, -0.7) {$\alpha^*_{4}$};
\node at (-0.7, -1.0) {$\alpha^{\rm{s},*}_{4}$};

\node at (-0.9, -0.3) {$U_{\rm{m}}$};
\node at (-0.1, 1.5) {$U_{\rm{m1}}$};
\node at (0.3, -0.1) {$U_{\rm{r}}$};
\node at (-1.2, 0.7) {$U_{\rm{l}}$};
\end{tikzpicture}
\caption{Weak-strong \protect\\ waves interaction}\label{fig5}
\end{minipage}
\begin{minipage}[t]{0.45\textwidth}
\centering
\begin{tikzpicture}[scale=1.1]
\draw [line width=0.04cm][red](-1.0,-0.8)--(-0.5,0.6);
\draw [thin](0.4,-0.6)--(-0.5,0.6);

\draw [line width=0.04cm][red](-0.5,0.6)--(0.9,1.2);
\draw [dashed][thin](-0.5,0.6)--(0.3, 1.7);
\draw [thin](-0.5,0.6)--(-0.7,1.9);

\node at (0.4, 0.4) {$U_{\rm{r}}$};
\node at (-0.3, -0.4) {$U_{\rm{m}}$};
\node at (-0.25, 1.6) {$U_{\rm{m1}}$};
\node at (-1.3, 0.5) {$U_{\rm{l}}$};

\node at (1.25, 1.25) {$\alpha^{\rm{s}}_{4}$};
\node at (0.4, 2.0) {$\alpha_{2(3)}$};
\node at (-0.7, 2.1) {$\alpha_{1}$};

\node at (-1.0, -1.1) {$\alpha^{\rm{s},*}_{4}$};
\node at (0.7, -0.7) {$\alpha^{*}_{k}$};

\end{tikzpicture}
\caption{Strong-weak \protect\\ waves interaction }\label{fig6}
\end{minipage}
\end{figure}

\begin{proposition}\label{prop:2.3}
Suppose that constant states $U_{\rm{l}}, U_{\rm{m}}\in \mathcal{N}_{\delta}(\bar{U}_{\mathfrak{b},\rm{l}})$
and $U_{\rm{r}}\in \mathcal{N}_{\delta}(\bar{U}_{\mathfrak{b},\rm{r}})$ (see Fig.\ref{fig5}) with $\delta>0$ sufficiently small, and they satisfy
\begin{eqnarray}\label{eq:2.33}
U^{\textsc{e}}_{\rm{m}}=\Phi^{\textsc{e}}_4(\alpha^*_4)(U^{\textsc{e}}_{\rm{l}}),\quad
U^{\textsc{e}}_{\rm{r}}=S^{\textsc{e}}_4(\alpha^{\rm{s},*}_4)(U^{\textsc{e}}_{\rm{m}}),\quad
Y_{\rm{r}}=Y_{\rm{m}}=Y_{\rm{l}}.
\end{eqnarray}
If
\begin{eqnarray}\label{eq:2.34}
U^{\textsc{e}}_{\rm{r}}=S^{\textsc{e}}_4(\alpha^{\rm{s}}_4)\circ \Phi^{\textsc{e}}_2(\alpha_2) \circ \Phi^{\textsc{e}}_1(\alpha_1)(U^{\textsc{e}}_{\rm{l}}),\qquad
Y_{\rm{r}}=Y_{\rm{l}} +\alpha_3 ,
\end{eqnarray}
then there hold that
\begin{eqnarray}\label{eq:2.35}
\alpha_1  = \kappa_{\rm{s}} \alpha^*_4 + \mathcal{O}(1)|\alpha^*_4 |^2,\quad   \alpha_2 =\mathcal{O}(1)|\alpha^*_4|,\quad \alpha_3=0,
\end{eqnarray}
and
\begin{eqnarray}\label{eq:2.36}
&\alpha^{\rm{s}}_4=\alpha^{\rm{s},*}_4+\mathcal{O}(1)|\alpha^{*}_4|,
\end{eqnarray}
where the reflection coefficient $\kappa_{\rm{s}}$ satisfies $|\kappa_{\rm{s}}|<1$.
\end{proposition}
\begin{remark}\label{rem:2.2}
The fact $|\kappa_{\rm{s}}|<1$ in \eqref{eq:2.35} is a key point in proving the global existence of entropy solution to $($\textit{IBVP}$)$.
\end{remark}

\begin{proof}
First, by relations \eqref{eq:2.33} and \eqref{eq:2.34}, we have $\alpha_{3}=0$ and
\begin{eqnarray}\label{eq:2.37}
\Psi^{\textsc{e}}_1(-\alpha_1)\circ \Psi^{\textsc{e}}_2(-\alpha_2)\circ  S^{\textsc{e}}_4(-\alpha^{\rm{s}}_4)( U^{\textsc{e}}_{\rm{r}})
={\Psi^{\textsc{e}}_4(-\alpha^*_4)\circ S^{\textsc{e}}_4(-\alpha^{\rm{s},*}_4)(U^{\textsc{e}}_{\rm{r}}).}
\end{eqnarray}
Then, from Lemma \ref{lem:2.4}, we know that $\alpha_1$, $\alpha_2$ and $\alpha^{\rm{s}}_4$ can be solved from equation \eqref{eq:2.37} as $C^2$-functions of variables $\alpha^*_4$
and $\alpha^{\rm{s},*}_4$ in the vicinity of point $(0,\bar{\alpha}^{\rm{s}}_4)$.

 To estimate $\alpha_1$, $\alpha_2$ and $\alpha^{\rm{s}}_4$, we rewrite the equation \eqref{eq:2.37} as
\begin{eqnarray*}
\begin{split}
&U^{\textsc{e}}_{\rm{r}}+\int^{-\alpha^{\rm{s}}_4}_0 \mathbf{r}^{\textsc{e}}_{4} (U^{\textsc{e}}_{\rm{r}}, \sigma) d\sigma-\alpha_2\mathbf{r}^{\textsc{e}}_2
	+ \int^{-\alpha_1}_0 \tilde{\mathbf{r}}^{\textsc{e}}_{1} ( U^{\textsc{e}}_{\rm{m1}}, \sigma) d\sigma \\[5pt]
&\qquad\qquad\quad  =U^{\textsc{e}}_{\rm{r}}+\int^{-\alpha^{\rm{s},*}_4}_0 \mathbf{r}^{\textsc{e}}_{4} ( U^{\textsc{e}}_{\rm{r}}, \sigma) d\sigma+ \int^{-\alpha_4^*}_0 \tilde{\mathbf{r}}^{\textsc{e}}_{4} ( U^{\textsc{e}}_{\rm{m}}, \sigma) d\sigma,
\end{split}
\end{eqnarray*}
where $U^{\textsc{e}}_{\rm{m1}} = \Psi^{\textsc{e}}_2(-\alpha_2)\circ  S^{\textsc{e}}_4(-\alpha^{\rm{s}}_4)( U^{\textsc{e}}_{\rm{r}})$.
Differentiating this with respect to $\alpha_4^*$, and letting $\alpha_4^* =0, \alpha^{\rm{s},*}_4=\bar{\alpha}^{s}_{4}$ and $(U^{\textsc{e}}_{\rm{m}},U^{\textsc{e}}_{\rm{r}})=(\bar{U}^{\textsc{e}}_{\mathfrak{b},\rm{l}},\bar{U}^{\textsc{e}}_{\mathfrak{b},\rm{r}})$,
one derives that
\begin{equation*}
\mathbf{r}^{\textsc{e}}_1 (\bar{U}^{\textsc{e}}_{\mathfrak{b},\rm{l}}) \frac{\partial \alpha_1}{\partial \alpha^*_4}\Big|_{\alpha_4^* =0, \alpha^{\rm{s},*}_4=\bar{\alpha}^{\rm{s}}_{4}}
	+ \mathbf{r}^{\textsc{e}}_2 (\bar{U}^{\textsc{e}}_{\mathfrak{b},\rm{l}})\frac{\partial \alpha_2}{\partial \alpha^*_4}\Big|_{\alpha_4^* =0, \alpha^{\rm{s},*}_4=\bar{\alpha}^{\rm{s}}_{4}}
	+ \mathbf{r}^{\textsc{e}}_4 ( \bar{U}^{\textsc{e}}_{\mathfrak{b},\rm{r}},  -\bar{\alpha}^{\rm{s}}_4) \frac{\partial \alpha^{\rm{s}}_4}{\partial \alpha_4^*}\Big|_{\alpha_4^* =0, \alpha^{\rm{s},*}_4=\bar{\alpha}^{\rm{s}}_{4}}
	= \mathbf{r}^{\textsc{e}}_4 (\bar{U}^{\textsc{e}}_{\mathfrak{b},\rm{l}}),
\end{equation*}
which gives
$$
\frac{\partial \alpha_1}{\partial \alpha^*_4}\Big|_{\alpha_4^* =0, \alpha^{\rm{s},*}_4=\bar{\alpha}^{\rm{s}}_{4}}
= \frac{\det \big(\mathbf{r}^{\textsc{e}}_4 (\bar{U}^{\textsc{e}}_{\mathfrak{b},\rm{l}}), \mathbf{r}^{\textsc{e}}_2(\bar{U}^{\textsc{e}}_{\mathfrak{b},\rm{l}}), \mathbf{r}^{\textsc{e}}_4 (\bar{U}^{\textsc{e}}_{\mathfrak{b},\rm{r}},  -\bar{\alpha}^{\rm{s}}_4) \big)}
{\det\big(\mathbf{r}^{\textsc{e}}_1 (\bar{U}^{\textsc{e}}_{\mathfrak{b},\rm{l}}), \mathbf{r}^{\textsc{e}}_2(\bar{U}^{\textsc{e}}_{\mathfrak{b},\rm{l}}), \mathbf{r}^{\textsc{e}}_4 ( \bar{U}^{\textsc{e}}_{\mathfrak{b},\rm{r}}, -\bar{\alpha}^{\rm{s}}_4)\big)}
=\frac{\gamma \bar{p}_{\rm{l}}u'(-\bar{\alpha}^{\rm{s}}_{4}) -\bar{c}_{\rm{l}}p'(-\bar{\alpha}^{\rm{s}}_{4})}{\gamma\bar{p}_{\rm{l}}u'(-\bar{\alpha}^{\rm{s}}_{4}) + \bar{c}_{\rm{l}}p'(-\bar{\alpha}^{\rm{s}}_{4})}.
$$
With the help of Lemma \ref{lem:2.1a}, we can further deduce that
\begin{eqnarray*}
\begin{split}
\Bigg|\frac{\partial \alpha_1}{\partial \alpha^*_4}\Big|_{\alpha_4^* =0, \alpha^{\rm{s},*}_4=\bar{\alpha}^{\rm{s}}_{4}}\Bigg|
=\left|\frac{\gamma \bar{p}_{\rm{l}}u'(-\bar{\alpha}^{\rm{s}}_{4}) -\bar{c}_{\rm{l}}p'(-\bar{\alpha}^{\rm{s}}_{4})}{\gamma \bar{p}_{\rm{l}}u'(-\bar{\alpha}^{\rm{s}}_{4}) + \bar{c}_{\rm{l}}p'(-\bar{\alpha}^{\rm{s}}_{4})}\right|<1,
\end{split}
\end{eqnarray*}
and
\begin{eqnarray*}
\begin{split}
\frac{\partial \alpha^{\rm{s}}_4}{\partial \alpha^*_4}\Big|_{\alpha_4^* =0, \alpha^{\rm{s},*}_4=\bar{\alpha}^{\rm{s}}_{4}}
&=\frac{\det \big(\mathbf{r}^{\textsc{e}}_1 (\bar{U}^{\textsc{e}}_{\mathfrak{b},\rm{l}}), \mathbf{r}^{\textsc{e}}_2(\bar{U}^{\textsc{e}}_{\mathfrak{b},\rm{l}}), \mathbf{r}^{\textsc{e}}_4 (\bar{U}^{\textsc{e}}_{\mathfrak{b},\rm{l}}) \big)}
{\det\big(\mathbf{r}^{\textsc{e}}_1 (\bar{U}^{\textsc{e}}_{\mathfrak{b},\rm{l}}), \mathbf{r}^{\textsc{e}}_2(\bar{U}^{\textsc{e}}_{\mathfrak{b},\rm{l}}), \mathbf{r}^{\textsc{e}}_4 (\bar{U}^{\textsc{e}}_{\mathfrak{b},\rm{r}}, -\bar{\alpha}^{\rm{s}}_4)\big)}\\[5pt]
&=\frac{4\gamma\bar{c}_{\rm{l}}\bar{p}_{\rm{l}}}{(\gamma+1)\bar{c}_{\rm{l}}\big(\gamma \bar{p}_{\rm{l}}u'(-\bar{\alpha}^{\rm{s}}_{4})+\bar{c}_{\rm{l}}p'(-\bar{\alpha}^{\rm{s}}_{4})\big)}>0.
\end{split}
\end{eqnarray*}
By Taylor formula and the regularity of $\alpha_{1}$, we finially obtain that for any $U_{\rm{l}}, U_{\rm{m}}\in \mathcal{N}_{\delta}(\bar{U}_{\mathfrak{b},\rm{l}})$ and $U_{\rm{r}}\in \mathcal{N}_{\delta}(\bar{U}_{\mathfrak{b},\rm{r}})$
with $\delta$ sufficiently small, there always holds
\begin{eqnarray*}
\begin{split}
\alpha_{1}(\alpha^{*}_{4},\alpha^{\rm{s},*}_{4})&=\alpha_{1}(0,\alpha^{\rm{s}}_{4})+\frac{\partial \alpha_1}{\partial \alpha^*_4}\Big|_{\alpha_4^* =0}\cdot \alpha^*_4+\mathcal{O}(1)|\alpha^*_4 |^2\\[5pt]
&=\kappa_{\rm{s}}\alpha^*_4+\mathcal{O}(1)|\alpha^*_4 |^2
\end{split}
\end{eqnarray*}
with coefficient $|\kappa_{\rm{s}}|<1$. This gives the estimate for $\alpha_{1}$ in \eqref{eq:2.35}. The other estimates in \eqref{eq:2.35}\eqref{eq:2.36} can be obtained by analogous argument.

\end{proof}

Similarly, if a weak wave interacts with the strong $4$-shock from right, we also have

\begin{proposition}\label{prop:2.4}
Suppose constant states $U_{\rm{l}}\in \mathcal{N}_{\delta}(\bar{U}_{\mathfrak{b},\rm{l}})$
and $ U_{\rm{m}}, U_{\rm{r}}\in \mathcal{N}_{\delta}(\bar{U}_{\mathfrak{b},\rm{r}})$ (see Fig.\ref{fig6}) with $\delta>0$ sufficiently small.

\noindent
{\rm{(i)}} If these states satisfy
\begin{eqnarray*}\label{eq:2.38a}
U^{\textsc{e}}_{\rm{m}}=S^{\textsc{e}}_4(\alpha^{\rm{s},*}_4)(U^{\textsc{e}}_{\rm{l}}),\quad
U^{\textsc{e}}_{\rm{r}}=\Phi^{\textsc{e}}_k(\alpha^*_k)(U^{\textsc{e}}_{\rm{m}}),\quad
Y_{\rm{r}}=Y_{\rm{m}}=Y_{\rm{l}},\ \ k\neq 3,
\end{eqnarray*}
then, for $\alpha_i$ $(1 \leq i \leq 3)$ and $\alpha^{\rm{s}}_4$ determined by
\begin{eqnarray}\label{eq:2.38}
U^{\textsc{e}}_{\rm{r}}=S^{\textsc{e}}_4(\alpha^{\rm{s}}_4)\circ \Phi^{\textsc{e}}_2(\alpha_2) \circ \Phi^{\textsc{e}}_1(\alpha_1)( U^{\textsc{e}}_{\rm{l}}),\quad Y_{\rm{r}}=Y_{\rm{l}}+\alpha_{3},
\end{eqnarray}
there holds
\begin{eqnarray*}\label{eq:2.39}
\alpha_1=\mathcal{O}(1)|\alpha^*_k|, \quad  \alpha_2=\mathcal{O}(1)|\alpha^*_k|,\quad \alpha_3=0, \quad \alpha^{\rm{s}}_4=\alpha^{\rm{s},*}_4+\mathcal{O}(1)|\alpha^{*}_k|;
\end{eqnarray*}

\noindent
{\rm{(ii)}} If these states satisfy
\begin{eqnarray*}\label{eq:2.38b}
U^{\textsc{e}}_{\rm{r}}=U^{\textsc{e}}_{\rm{m}}=S^{\textsc{e}}_4(\alpha^{\rm{s},*}_4)(U^{\textsc{e}}_{\rm{l}}),\quad
Y_{\rm{m}}=Y_{\rm{l}}, \quad
Y_{\rm{r}}=Y_{\rm{m}}+\alpha^*_{3},
\end{eqnarray*}
then, for $\alpha_i$ $(1 \leq i \leq 3)$ and $\alpha^{\rm{s}}_4$ determined by \eqref{eq:2.38},
there holds
\begin{eqnarray*}\label{eq:2.40}
\alpha_1=\alpha_2=0,\quad \alpha_3=\alpha^{*}_3, \quad \alpha^{\rm{s}}_4=\alpha^{\rm{s},*}_4.
\end{eqnarray*}

\end{proposition}

For the purpose of $L^1$-stability analysis in section 5, we state the following result on solvability of Riemann problems in the sense of states connection by Hugoniot loci.

\begin{lemma}\label{lem:2.5}
Assume constant states $U_{\rm{l}}, U_{\rm{r}} \in \mathcal{N}_{\delta}(\bar{U}_{\mathfrak{b},\rm{l}}) \cup \mathcal{N}_{\delta}(\bar{U}_{\mathfrak{b},\rm{r}})$ with $\delta>0$ sufficiently small.
Then there exists a unique vector $\mathbf{q}=(q_1,q_2,q_3,q_4)$ such that
$$
U_{\rm{r}}=S_{4}(q_4)\circ S_{3}(q_3)\circ S(q_2)\circ S_{1}(q_1)(U_{\rm{l}}).
$$
When $U_{\rm{l}} \in \mathcal{N}_{\delta}(\bar{U}_{\mathfrak{b},\rm{l}})$ and  $U_{\rm{r}} \in \mathcal{N}_{\delta}(\bar{U}_{\mathfrak{b},\rm{r}})$,
the component $q_4$ is very close to $\bar{\alpha}^{\rm{s}}_4$.
\end{lemma}

The proof of above lemma is completely based on the technique of tangent field analysis and implicit function theorem, as done in Lemma \ref{lem:2.4}.
So we skip the details here.

\begin{definition}\label{Def: distance index}
For given states $U_{\rm{l}}$ and $U_{\rm{r}}$, if there is a unique vector $\mathbf{q}=(q_1,q_2,q_3,q_4)$ such that
\begin{equation*}
  U_{\rm{r}} =\mathcal{H}(\mathbf{q})(U_{\rm{l}}) \doteq S_{4}(q_4)\circ S_{3}(q_3)\circ S(q_2)\circ S_{1}(q_1)(U_{\rm{l}}),
\end{equation*}
we say that $q_i$ $(1\leq i\leq 4)$ are the distance indices of two states $U_{\rm{l}}$ and $U_{\rm{r}}$.
\end{definition}

\section{Construction of the approximate solutions to (\emph{IBVP})}
In this section, we will construct the approximate solutions to (\emph{IBVP})
via a modified fractional-step wave front tracking scheme. The detailed algorithm is stated as follows.

\begin{figure}[ht]
\begin{center}
\begin{tikzpicture}[scale=0.8]
\draw [thin][->] (-2.3, -1)--(4.4, -1);
\draw [thin][->] (-1.5, -1.7)--(-1.5, 3.3);
\draw [line width=0.03cm](-1.5,-1)to (-1.0,0);
\draw [line width=0.03cm](-1.0,0)to(-0.8,1);
\draw [line width=0.03cm](-0.8,1)to(-0.2,2);
\draw [line width=0.03cm](-0.2,2)to(-0.3,2.9);

\draw  [dashed][thin](-1.0,0)to(3.8,0);
\draw  [dashed][thin](-0.8,1)to(3.8,1);
\draw  [dashed][thin](-0.2,2)to(3.8,2);

\draw [red][line width=0.02cm](-1.5,-1)to(-0.7,-0.5);
\draw [thin](-0.3,-1)to(-0.7,-0.5);
\draw [red][line width=0.02cm](-0.7,-0.5)to(0.3,0);
\draw [dashed][thin](-0.7,-0.5)to(-0.2,0);
\draw [thin](-0.7,-0.5)to(-0.7,0);

\draw [thin](1.0,-1)to(1.5,-0.5);
\draw [thin](2.0,-1)to(1.5,-0.5);

\draw [thin](1.5,-0.5)to(1.0,0);
\draw [dashed][thin](1.5,-0.5)to(1.6,0);
\draw [thin](1.5,-0.5)to(2.4,0);

\draw [thin](-1.0,0)to(-0.6,0.45);
\draw [thin](-0.7,0)to(-0.6,0.45);
\draw [thin](-0.6,0.45)to(-0.85, 0.7);
\draw [thin](-0.85, 0.7)to(-0.3,1)to(-0.5,1.5)to(-0.1,1.5);

\draw [thin](-0.2,0)to(0,0.5);
\draw [thin](0.3,0)to(0,0.5)to(-0.04,1)to(-0.1,1.5)to(-0.1,2);
\draw [thin](-0.04,1)to(0.4,1.5);
\draw [dashed][thin](-0.1,1.5)to(0.2,2);
\draw [thin](-0.1,1.5)to(0.4,2);

\draw [red][line width=0.03cm](0.3,0)to(1.0,0.6)to(1.8, 1);
\draw [thin](1.0,0)to(1.0,0.6)to(1.1,1)to(1.1,1.5);
\draw [dashed][thin](0.3,0)to(0.4,0.6)to(0.5,1)to(0.4,1.5);
\draw [dashed][thin](1.0,0.6)to(1.4, 1)to(1.6,1.5);
\draw [thin](1.8, 1)to(1.6,1.5);

\draw [thin](1.1,1.5)to(0.8,2);
\draw [dashed][thin](1.1,1.5)to(1.1,2);
\draw [thin](1.1,1.5)to(1.3,2);
\draw [thin](0.4,1.5)to(0.6,2);
\draw [thin](1.6,1.5)to(1.4,2);
\draw [dashed][thin](1.6,1.5)to(1.5,2);
\draw [thin](1.6,1.5)to(1.8,2);

\draw [red][line width=0.02cm](1.8, 1)to(2.0,1.5);
\draw [thin](1.6,0)to(2.0,0.5)to(2.2,1)to(2.0,1.5);

\draw [red][line width=0.03cm](2.0,1.5)to(2.3,2)to(2.8,2.5);
\draw [dashed][thin](2.0,1.5)to(2.0,2);
\draw [thin](2.0,1.5)to(1.8, 2);

\draw [dashed][thin](2.0,0.5)to(2.0,1);
\draw [thin](2.4,0)to(2.0,0.5)to(1.9,1);

\draw [thin](0,0.5)to(0.4,0.6)to(0.8,1)to(1.1,1.5);

\node at (4.4,-1.25) {$x$};
\node at (-1.7, 3.2) {$t$};
\node at (-1.75, -1.25) {$O$};

\node at (-0.2, 3.2) {$\Gamma^{\varepsilon}_{\mathfrak{p}}$};
\node at (2.85, 2.8) {$\Gamma^{\varepsilon}_{s}$};
\node at (3.2, 0.4) {$\Omega^{\varepsilon}_{\mathfrak{p}}$};
\end{tikzpicture}
\end{center}
\caption{Approximate solutions to the (\emph{IBVP})}\label{fig7}
\end{figure}
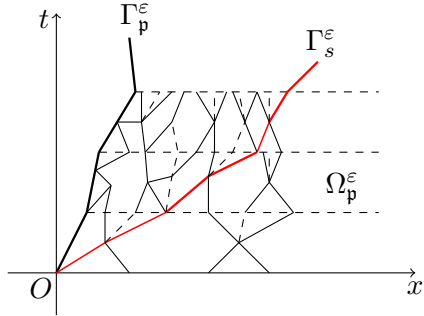

Under assumptions $(\mathbf{A1})$$(\mathbf{A2})$, we can construct the approximations for given initial-boundary data $(U_{0},u_{\mathfrak{p}})$. More precisely,
choose a small parameter $\varepsilon>0$, and then take two piecewise constant functions $U^{\varepsilon}_{0}(x)$ and $u^{\varepsilon}_{\mathfrak{p}}(t)$
such that:
\begin{itemize}
  \item both $U^{\varepsilon}_{0}(x)$ and $u^{\varepsilon}_{\mathfrak{p}}(t)$ have finitely many discontinuities;
  \item $\|U^{\varepsilon}_{0}(\cdot)-U_{0}(\cdot)\|_{L^{1}(\mathbb{R}_{+})}<\varepsilon,\quad T.V.\{U^{\varepsilon}_{0}(\cdot);\mathbb{R}_{+}\}\leq T.V.\{U_{0}(\cdot);\mathbb{R}_{+}\},$ \\
$\|u^{\varepsilon}_{\mathfrak{p}}(\cdot)-u_{\mathfrak{p}}(\cdot) \|_{L^{1}(\mathbb{R}_{+})}<\varepsilon,
\quad T.V.\{u^{\varepsilon}_{\mathfrak{p}}(\cdot);\mathbb{R}_{+}\}\leq T.V.\{u_{\mathfrak{p}}(\cdot);\mathbb{R}_{+}\}.$
\end{itemize}
Accordingly, the approximate piston boundary for $\Gamma_{\mathfrak{p}}$  and the approximate domain for $\Omega_{\mathfrak{p}}$ are defined by
\begin{eqnarray*}\label{eq:3.3}
\Gamma^{\varepsilon}_{\mathfrak{p}}=\{(t,x)\in \mathbb{R}^{2}: t>0,\ x=\chi^{\varepsilon}_{\mathfrak{p}}(t)\},\quad
\Omega^{\varepsilon}_{\mathfrak{p}}=\{(t,x)\in \mathbb{R}^{2}: t>0,\ x>\chi^{\varepsilon}_{\mathfrak{p}}(t)\},
\end{eqnarray*}
where
$x=\chi^{\varepsilon}_{\mathfrak{p}}(t)\doteq\int^{t}_{0}u^{\varepsilon}_{\mathfrak{p}}(\tau)d\tau$.
We divide time interval $[0, \varepsilon^{-1}]$ into $N$ subintervals, such that division points
\begin{eqnarray*}
t_0= 0,\quad t_1= \varepsilon,\quad  t_2= 2\varepsilon,\quad \cdot\cdot\cdot\quad  t_N= N\varepsilon,
\end{eqnarray*}
and additionally set $t_{N+1}=+\infty$.
Then the whole domain $\Omega^{\varepsilon}_{\mathfrak{p}}$ is divided into some strips
\begin{eqnarray*}\label{eq:3.4}
\Omega^{\varepsilon}_{k}=\{(t,x)\in \mathbb{R}^{2}: t_{k-1} \leq t< t_{k},\ x>\chi^{\varepsilon}_{\mathfrak{p}}(t)\} \ \ \ \
\text{for } k=1,\cdots,N.
\end{eqnarray*}

We start to construct the $\varepsilon$-approximate solution $U^{\varepsilon}$ in $\Omega^{\varepsilon}_{\mathfrak{p}}$ in the following way.

At time $t_{0}$, the Riemann problem for homogenous system \eqref{eq:2.2} at each discontinuous point of $U^{\varepsilon}_{0}$ is solved as stated in section 2, such that the solution consists of weak or strong shocks, rarefaction waves and contact discontinuities. As done in \cite{Bressan-2000}, rarefaction waves need to be partitioned into several small central rarefaction fans with strength less than $\hat{\delta}$, which is called the accurate Riemann solver denoted by $(\textsc{ARS})$ for simplicity. Another is the simplified Riemann solver denoted by $(\textsc{SRS})$, in which all the new waves are lumped into a single non-physical wave traveling with a fixed speed $\hat{\lambda}$ larger than all the characteristic speeds (see \cite[pp.129--pp.132]{Bressan-2000} for more details). At initial time $t_{0}$, we always use the $(\textsc{ARS})$.

The piecewise constant approximate solution $U^{\varepsilon}$ can be prolonged until waves interactions occur in $\Omega^{\varepsilon}_1$. This refers to (i) interactions between weak waves; (ii) interactions between weak waves and strong shock; (iii) wave reflections at non-corner points on $\Gamma^{\varepsilon}_{\mathfrak{p}}$ (see Fig.\ref{fig4}); (iv) weak wave appearance from the boundary corner (see Fig.\ref{fig3}). Then we can solve the Riemann problem again as stated in section 2. For cases (iii)(iv) on the boundary, we always use the $(\textsc{ARS})$ as in \cite{Amadori-1997, Colombo-Guerra-2010}.
To decide which Riemann solver is used for cases (i)(ii), we introduce a threshold parameter $\varrho>0$.
If the strengths $\alpha$, $\beta$ of two weak waves satisfy $|\alpha\beta|>\varrho$, the $(\textsc{ARS})$ is used; otherwise, we use the $(\textsc{SRS})$.
If a weak physical wave front with strength $|\alpha|>\varrho$ hits the strong shock, then we use the $(\textsc{ARS})$; otherwise, we use the $(\textsc{SRS})$.

At time $t_1$, we treat the combustion process by linear approximation, and hence prescribe the fluid state $U^{\varepsilon}(t_1,x)$ by the following equations
\begin{eqnarray*}\label{eq:3.5}
E(U^{\varepsilon}(t_1,x))=E(U^{\varepsilon}(t_1-0,x))+\varepsilon G(U^{\varepsilon}(t_1-0,x)),
\end{eqnarray*}
where $U^{\varepsilon}(t_1-0,x)=\lim_{t\to t_1^-} U^\varepsilon(t,x)$.
After that, solve the homogeneous system \eqref{eq:2.2} with initial data $U^{\varepsilon}(t_1,x)$.
In this process, if a physical weak wave or strong shock wave hits the line $t=t_1$, we use the $(\textsc{ARS})$;
if a non-physical wave hits the line $t=t_1$, we use the $(\textsc{SRS})$.

Inductively, assume that the approximate solution $U^{\varepsilon}$ has been constructed for $0<t\leq t_{k-1}$ ($k>1$), and contains the jumps of weak shock fronts, strong shock front, contact discontinuities, rarefaction fronts and non-physical fronts, which are denoted by $$\mathcal{J}(U^{\varepsilon})\doteq\mathcal{S}\cup\mathcal{S}_{\rm b}\cup\mathcal{C}\cup\mathcal{R}\cup\mathcal{NP}.$$
Then, in strip $\Omega^\varepsilon_k$, we construct the approximate solution $U^{\varepsilon}$ for homogeneous system \eqref{eq:2.2} with initial data $U^{\varepsilon}(t_{k-1},x)$ by repeating the above processes.
Finally, considering the combustion process at time $t_k$, we approximately solve the nonhomogeneous
system \eqref{eq:2.1} with initial data $U^{\varepsilon}(t_{k}-0,x)$ by the algebraic euqations
\begin{eqnarray}\label{eq:3.6}
E(U^{\varepsilon}(t_{k},x))=E(U^{\varepsilon}(t_{k}-0,x))+\varepsilon G(U^{\varepsilon}(t_{k}-0,x)).
\end{eqnarray}

\begin{remark}\label{rem:3.1}
In each step of construction, we can slightly perturb the speed of a single wave front
with its strength error less than $\varepsilon$, so that every interaction exactly involves two waves, and that every reflection on boundary $\Gamma^{\varepsilon}_{\mathfrak{p}}$ includes only one incident wave.
In addition to the four characteristic fields of system \eqref{eq:2.2}, we treat $\mathcal{NP}$ waves as the fifth family.
Then assign their speed $\lambda_5 \doteq \hat{\lambda}> \lambda_i$ for $1\leq i \leq 4$.
\end{remark}

\begin{remark}\label{rem:3.2}
 The parameters $\hat{\delta}$ and $\varrho$ are functions of $\varepsilon$, \emph{i.e.}, $\hat{\delta}=\hat{\delta}(\varepsilon)$, $\varrho=\varrho(\varepsilon)$,
 and satisfy that $\hat{\delta}(\varepsilon)\rightarrow0, \varrho(\varepsilon)\rightarrow0$ as $\varepsilon\rightarrow0$.
\end{remark}

Based on the above algorithm, we have the following properties for $\varepsilon$-approximate solution $U^\varepsilon$,  whose proofs are similar to those for Lemmas 7.1-7.2 in \cite{Bressan-2000}.

\begin{proposition}\label{prop:3.1}
Suppose $\alpha$ is a wave front of $\varepsilon$-approximation $U^{\varepsilon}$ to the system \eqref{eq:2.2}, which is located at $(t, x_\alpha(t))$.
Then the following local error estimates are satisfied:
\begin{eqnarray*}\label{eq:3.7}
\begin{split}
&\dot{x}_\alpha(t) [E(U^{\varepsilon})] - [F(U^{\varepsilon})]= \mathcal{O}(1) \varepsilon |\alpha|, \qquad  &&\alpha \in \mathcal{J}(U^{\varepsilon})\setminus \mathcal{NP}, \\[5pt]
&\dot{x}_\alpha(t) [E(U^{\varepsilon})] - [F(U^{\varepsilon})]  = \mathcal{O}(1)|\alpha| , \qquad  &&\alpha \in  \mathcal{NP},
\end{split}
\end{eqnarray*}
where $U^{\varepsilon,\pm}=U^{\varepsilon}(t,x_\alpha(t)\pm0)$ and
$$[E(U^{\varepsilon})]= E(U^{\varepsilon,-}) - E(U^{\varepsilon,+}), \ \ [F(U^{\varepsilon})] = F(U^{\varepsilon,-}) - F(U^{\varepsilon,+}).$$
\end{proposition}

\begin{proposition}\label{prop:3.2}
Suppose that  $U^{\varepsilon}$ contains three adjacent constant states $U_{\rm{l}}=(U^{\textsc{e}}_{\rm{l}}, Y_{\rm{l}})^{\top}$, $U_{\rm{m}}=(U^{\textsc{e}}_{\rm{m}}, Y_{\rm{m}})^{\top}$,
$U_{\rm{r}}=(U^{\textsc{e}}_{\rm{r}}, Y_{\rm{r}})^{\top} \in \mathcal{N}_{\delta}(\bar{U}_{\mathfrak{b}, \rm{r}})$ $($or $\mathcal{N}_{\delta}(\bar{U}_{\mathfrak{b}, \rm{l}}))$ with $\delta>0$ sufficiently small, and they are separated by two incoming waves $\alpha^*_{i}, \alpha^*_{j} \in \mathcal{J}(U^{\varepsilon})$.
Then the \textnormal{(SRS)} generated from Riemann problem $(U_{\rm{l}}, U_{\rm{r}})$ satisfies the following properties.

\noindent{\textnormal{(i)}}
If $\alpha^*_{i}, \alpha^*_{j} \notin \mathcal{NP}$ for $i\geq j$ but $i,j \neq3 $, such that
$U^{\textsc{e}}_{\rm{m}}= \Phi^{\textsc{e}}_{i}(\alpha^*_{i})( U^{\textsc{e}}_{\rm{l}}), \
U^{\textsc{e}}_{\rm{r}}=\Phi^{\textsc{e}}_{j}(\alpha^*_{j})( U^{\textsc{e}}_{\rm{m}}), $
 $Y_{\rm{l}}=Y_{\rm{m}}=Y_{\rm{r}}$,
then the \textnormal{(SRS)} includes an auxiliary state $\hat{U}_{\rm{r}} = (\hat{U}^{\textsc{e}}_{\rm{r}}, Y_{\rm{r}} )^{\top}$ with
\begin{equation*}
  \hat{U}^{\textsc{e}}_{\rm{r}} \doteq \left\{
  \begin{aligned}
    & \Phi^{\textsc{e}}_{i}(\alpha^*_{i}) \circ \Phi^{\textsc{e}}_{j}(\alpha^*_{j})( U^{\textsc{e}}_{\rm{l}}) \qquad  & \text{for } i>j, \\
    & \Phi^{\textsc{e}}_{i}(\alpha^*_{i}+\alpha^*_{j})( U^{\textsc{e}}_{\rm{l}})   & \text{for } i=j, \\
  \end{aligned}
  \right.
\end{equation*}
and non-physical $\alpha_{5} \doteq |\hat{U}^{\textsc{e}}_{\rm{r}}-U^{\textsc{e}}_{\rm{r}}|$. Moreover, there holds
$$\alpha_{5}=\mathcal{O}(1)|\alpha^*_{i}\alpha^*_{j}|. $$

\noindent{\textnormal{(ii)}}
If $\alpha^*_{i} \notin \mathcal{NP}$ and $\alpha^*_{j} \in \mathcal{NP}$ for $i\neq3, j=5$, such that
$U^{\textsc{e}}_{\rm{r}}=\Phi^{\textsc{e}}_{i}(\alpha^*_{i})( U^{\textsc{e}}_{\rm{m}})$, $ Y_{\rm{l}}=Y_{\rm{m}}=Y_{\rm{r}}$ and
$\alpha^*_{5} \doteq |U^{\textsc{e}}_{\rm{m}}-U^{\textsc{e}}_{\rm{l}}|$,
then the \textnormal{(SRS)} includes an auxiliary state
$\hat{U}_{\rm{r}} \doteq (\Phi^{\textsc{e}}_{i}(\alpha^*_{i}) ( U^{\textsc{e}}_{\rm{l}}), Y_{\rm{r}})^{\top}$
and non-physical  $\alpha_{5} \doteq |\Phi^{\textsc{e}}_{i}(\alpha^*_{i}) ( U^{\textsc{e}}_{\rm{l}}) -U^{\textsc{e}}_{\rm{r}}|$. It satisfies that
$$\alpha_{5}= \alpha^*_{5}+ \mathcal{O}(1)|\alpha^*_{i}\alpha^*_{5}|. $$

\noindent{\textnormal{(iii)}}
If $\alpha^*_{i} \notin \mathcal{NP}$ and $\alpha^*_{j} \in \mathcal{NP}$  for $i=3, j=5$, such that
$U^{\textsc{e}}_{\rm{r}}= U^{\textsc{e}}_{\rm{m}}$, $ Y_{\rm{r}}=Y_{\rm{m}} + \alpha^*_{3}$ and
$\alpha^*_{5} \doteq |U^{\textsc{e}}_{\rm{l}}- U^{\textsc{e}}_{\rm{m}}|$, $ Y_{\rm{m}}=Y_{\rm{l}}$,
then the \textnormal{(SRS)} includes an auxiliary state
$\hat{U}_{\rm{r}} \doteq  ( U^{\textsc{e}}_{\rm{l}}, Y_{\rm{l}} + \alpha^*_{3})^{\top}$
and non-physical $\alpha_{5} \doteq |U^{\textsc{e}}_{\rm{l}}-U^{\textsc{e}}_{\rm{r}}|$. It is obvious that
$$\alpha_{5}=\alpha^*_{5}. $$

\end{proposition}

\begin{proposition}\label{prop:3.3}
Suppose that  $U^{\varepsilon}$ contains three adjacent constant states $U_{\rm{l}}=(U^{\textsc{e}}_{\rm{l}}, Y_{\rm{l}})^{\top}$ and $U_{\rm{m}}=(U^{\textsc{e}}_{\rm{m}}, Y_{\rm{m}})^{\top}
 \in \mathcal{N}_{\delta}(\bar{U}_{\mathfrak{b}, \rm{l}})$, while $U_{\rm{r}}=(U^{\textsc{e}}_{\rm{r}}, Y_{\rm{r}})^{\top} \in\mathcal{N}_{\delta}(\bar{U}_{\mathfrak{b}, \rm{r}})$ with $\delta>0$ sufficiently small, and they are separated by two incoming waves $\alpha^*_5\in\mathcal{NP}$, $\alpha^{\rm{s},*}_4 \in \mathcal{S}_{\rm b}$.
If $U^{\textsc{e}}_{\rm{r}}=S^{\textsc{e}}_4(\alpha^{\rm{s},*}_4 )( U^{\textsc{e}}_{\rm{m}})$, $ Y_{\rm{l}}=Y_{\rm{m}}=Y_{\rm{r}}$ and
$\alpha^*_{5} \doteq |U^{\textsc{e}}_{\rm{l}}-U^{\textsc{e}}_{\rm{m}}|$,
then the \textnormal{(SRS)} of Riemann problem $(U_{\rm{l}}, U_{\rm{r}})$  includes an auxiliary state
$\hat{U}_{\rm{r}} \doteq (S^{\textsc{e}}_4(\alpha^{\rm{s},*}_4 )( U^{\textsc{e}}_{\rm{l}}), Y_{\rm{r}})^{\top}$
and non-physical  $\alpha_{5} \doteq |S^{\textsc{e}}_4(\alpha^{\rm{s},*}_4 )( U^{\textsc{e}}_{\rm{l}}) -U^{\textsc{e}}_{\rm{r}}|$. It satisfies that
$$\alpha_{5}= \mathcal{O}(1)|\alpha^*_{5}|. $$

\end{proposition}

\section{Global existence of entropy solutions to (\emph{IBVP})}

On account of ignition condition \eqref{eq:1.4}, the global solutions to combustion equations \eqref{eq:2.1} may evolve in
different trajectories in state space. We make a heuristic analysis on the background solution $\bar{U}_{\mathfrak{b}}$.  If temperature behind the leading shock
is below ignition one (\emph{i.e.} $\bar{T}_{\rm{l}}< T_{\text{i}} $), then reaction rate $\phi(T)=0$ everywhere; the entirely unburnt fluid is dominated by homogeneous Euler equations and reactant transport equation.
If $\bar{T}_{\rm{r}}> T_{\text{i}} $, then $\phi(T)>0$ everywhere; and combustion reaction may take place in the whole tube. We claim that the chemical dissipative term $G(U)$
decays exponentially as $t\to + \infty$.
It implies the global well-posedness of combustion solutions. However, if $\bar{T}_{\rm{l}}> T_{\text{i}} > \bar{T}_{\rm{r}}$, then only the fuel behind the large shock is ignited; an exothermic reaction occurs within the region enclosed by piston and shock front. Thus any uniform decay rate of reaction source fails to hold. The global well-posedness of combustion solutions becomes very delicate.
In this section, we intend to establish the global existence of \emph{BV} solutions for non-reacting and reacting flows respectively.

Since the perturbation of initial-boundary data is sufficiently small,
there exist positive $\underline{d}$ and $\overline{d} $ such that
\begin{equation}\label{eq:4.1}
\begin{aligned}
&\underline{d}<|\lambda_i(U) - \lambda_{i+1}(V)| <\overline{d} \quad  &&(i=1,3),\\[5pt]
&|\lambda_j(U)-\lambda_j(V)|< \underline{d},&\\[5pt]
&\underline{d}<|\lambda_j(U) - \mathrm{s}|<\overline{d}   \quad &&(1\leq j \leq 4),
\end{aligned}
\end{equation}
for states $U, V$ in the identical domain $\mathcal{N}_{\delta}(\bar{U}_{\mathfrak{b}, \rm{l}})$ or  $\mathcal{N}_{\delta}(\bar{U}_{\mathfrak{b}, \rm{r}})$.
Recall that $\mathrm{s}$ is the speed of a large $4$-shock wave connecting states from $\mathcal{N}_{\delta}(\bar{U}_{\mathfrak{b}, \rm{l}})$ to $\mathcal{N}_{\delta}(\bar{U}_{\mathfrak{b}, \rm{r}})$.
Suppose $\alpha_i \ (i=1,4)$ is an $i$-wave which connects states $U^+$, $U^- \in \mathcal{N}_{\delta}(\bar{U}_{\mathfrak{b}, \rm{l}})$. Then there is a positive $a$ such that
\begin{eqnarray*}
\frac{1}{a}|\alpha_i| \leq \big|[u]\big| + \big|[p]\big| \leq a|\alpha_i|,
\end{eqnarray*}
where $[u]$ and $[p]$ represent the jumps of velocity and pressure across the wave $\alpha_i$, respectively.

\subsection{Interaction potential and modified Glimm-type functional}

\begin{definition}[Approaching waves]\label{Def appraoching wave}
We say that $k_\alpha$-wave $\alpha$ and $k_\beta$-wave $\beta$ in the approximate solution $U^{\varepsilon}$ are approaching at time $t$, if one of the following conditions holds.
\begin{itemize}
\item  $k_\alpha > k_\beta$ and $x_\alpha < x_\beta$.
\item  $k_\alpha = k_\beta$ and the $k_\alpha$-family field is genuinely nonlinear . At least one of $\alpha$ and $\beta$ is shock wave.
\end{itemize}
\end{definition}
Introduce the sets
\begin{align*}
&\mathcal{A}= \big\{\ (\alpha, \beta) \ \big| \ \text{neither } \alpha \text{ nor } \beta \text{ is the large 4-shock, and they are approaching}  \ \big\}, \\[5pt]
& \mathcal{A}(\alpha) = \big\{ \ \beta\  \big|\  \beta \text{ is not the large 4-shock}, \text{and it is approaching given wave } \alpha \ \big\}.
\end{align*}
Obviously $\mathcal{A} $ includes all couples of  approaching small waves at time $t$. In particular, let symbol  $\alpha^{\rm{s}}_4$ denote the large 4-shock wave in solution $U^\varepsilon(t,x)$, and curve $x=\chi^\varepsilon_{\rm{s}} (t)$ be the path of this large shock.

Then we define the modified Glimm-type functional for approximate solution $U^\varepsilon$ to (\emph{IBVP})
by
\begin{equation}\label{eq:4.2}
\mathcal{G} (t) \doteq \mathcal{L}(t) +K \mathcal{Q}(t) + K\mathcal{L}_{\textsc{b}}(t)
\end{equation}
where functionals
\begin{eqnarray*}\label{eq:4.2a}
\begin{split}
\mathcal{L}(t) &=\mathcal{L}_{\textsc{e}}(t)  + \mathcal{L}_Y (t)\\[5pt]
	     &\doteq \bigg(|\alpha^{\rm{s}}_{4}-\bar{\alpha}^{\rm{s}}_{4}|+\sum_{i\neq 3} |\alpha_i| \bigg) + \theta_{3} \sum |\alpha_{3}|, \\[5pt]
\mathcal{Q}(t)&= \mathcal{Q}_w(t) + \mathcal{Q}_{\mathfrak{p}}(t) + \mathcal{Q}_{\rm{s}}(t)\\[5pt]
	&\doteq\sum_{\mathcal{A}} |\alpha_i \alpha_j| + \theta_1 \sum |\alpha_1|
	+  \sum_{\mathcal{A}( \alpha^{\rm{s}}_4)}\bigg (\sum_{i<3}|\alpha_i|+\theta_4 |\alpha_4| +\theta_5 |\alpha_5|\bigg),\\[5pt]
\mathcal{L}_{\textsc{b}}(t) &=\mathcal{L}_{\textsc{b}\mathfrak{p}}(t) + \mathcal{L}_{\textsc{by}}(t)  \\[5pt]
	&\doteq T.V.\{u^\varepsilon_{\mathfrak{p}}(\tau); (t,+\infty)\}  + \bigg(\sum_{ k\varepsilon\geq t} \Upsilon^{\varepsilon,\mathfrak{p}}_k + \sum_{k\varepsilon\geq t} \Upsilon^{\varepsilon,\rm{s}}_k \bigg).
\end{split}
\end{eqnarray*}

We remark that functional $\mathcal{Q}(t)$ is a variant of quadratic term in classic Glimm-type functional, which includes the sum $\mathcal{Q}_{\mathfrak{p}}(t)$ of 1-waves approaching the piston boundary $\Gamma^{\varepsilon}_{\mathfrak{p}}$, and the sum $\mathcal{Q}_{\rm{s}}(t)$ of weak waves approaching the large $4$-shock.
Functional $\mathcal{L}_{\textsc{b}}(t)$ manifests the boundary effect caused by piston motion $\mathcal{L}_{\textsc{b}\mathfrak{p}}(t)$ and by reactant consumption $\mathcal{L}_{\textsc{by}}(t)$ at piston and large leading shock.
We require that constant $K$ in \eqref{eq:4.2} is suitably large according to the standard \emph{BV} theory in \cite{Bressan-2000, Smoller-1994}.
Coefficient $\theta_1$ is suitably small, and $\theta_4,  \theta_5$ satisfy
\begin{equation}\label{eq:4.3}
0< \theta_1 |\kappa_{\rm{s}}| < \theta_4 <  \theta_1, \quad 0<\theta_5 <\theta_1,
\end{equation}
where $\kappa_{\rm{s}}$ given by Proposition \ref{prop:2.3} is the coefficient of weak waves reflection on strong $4$-shock front.
We place $\theta_1$ and  $\theta_4$ in $\mathcal{Q}(t)$ for the purpose of controlling the reflection effect
arising from weak-strong waves interaction. The significant coefficient $\theta_3$ is specified according to different conditions in Theorem \ref{thm:1.1}.
Under condition ($\textrm{C1}$) for entirely unburnt fluid, choose $\theta_3$ so small that
\begin{equation*}\label{eq:4.4}
0< \theta_3 T.V.\{Y_0(\cdot);(0,+\infty)\} \ll 1.
\end{equation*}
It ensures that $\mathcal{G}(0)$ is small enough.
In contrast, for reacting fluid under condition ($\textrm{C2}$) or ($\textrm{C3}$), we demand
$\theta_3$ is suitably large such that the decrease of $3$-waves may offset the growth of nonlinear waves caused by exothermic reaction.
Hence we introduce the transit consumption of reactant at $\chi^{\varepsilon}_{\mathfrak{p}}$ and  $\chi^{\varepsilon}_{\rm{s}}$, respectively as
\begin{eqnarray*}\label{eq:4.5}
\Upsilon^{\varepsilon,\mathfrak{p}}_k = Y^{\varepsilon,\mathfrak{p}}_{k}\cdot\phi(T^{\varepsilon,\mathfrak{p}}_{k})\varepsilon , \quad
\Upsilon^{\varepsilon,\rm{s}}_k = Y^{\varepsilon,\rm{s}}_{k}\cdot\phi(T^{\varepsilon, \rm{s}}_{k}) \varepsilon,
\end{eqnarray*}
where
\begin{eqnarray*}\label{eq:4.6}
\begin{split}
Y^{\varepsilon,\mathfrak{p}}_{k}=Y^{\varepsilon}(t_k-0, \chi^{\varepsilon}_\mathfrak{p}(t_k)+0),\qquad  T^{\varepsilon,\mathfrak{p}}_{k}=T^{\varepsilon}(t_k-0, \chi^{\varepsilon}_{\mathfrak{p}}(t_k)+0),\\[5pt]
Y^{\varepsilon, \rm{s}}_{k}=Y^{\varepsilon}(t_k-0, \chi^{\varepsilon}_{\rm{s}}(t_k)-0),\qquad  T^{\varepsilon,\rm{s}}_{k}=T^{\varepsilon}(t_k-0, \chi^{\varepsilon}_{\rm{s}}(t_k)-0),
\end{split}
\end{eqnarray*}
for $k\geq 1$.

By the scheme in section 3 and definition of $\mathcal{G}(t)$, we observe the fact that for given small $\delta>0$, there exists a small constant $\tilde{\delta}>0$ such that if $\mathcal{G}(t)<\tilde{\delta}$ at some $t>0$,
then the solution $U^\varepsilon (t,x)$ contains a unique large shock $\alpha^{\rm{s}}_4$.
Since  $|\alpha^{\rm{s}}_4- \bar{\alpha}^{\rm{s}}_4|< \tilde{\delta}$ and
$$T.V.\{U^\varepsilon (t,\cdot); (\chi^\varepsilon_{\mathfrak{p}} (t), \chi^\varepsilon_{\rm{s}} (t))\}
+T.V.\{U^\varepsilon (t,\cdot);  (\chi^\varepsilon_{\rm{s}} (t), +\infty)\}
\leq \mathcal{O}(1)\mathcal{G}(t), $$
there hold
\begin{equation}\label{eq:4.7}
 \begin{aligned}
   & U^\varepsilon (t,x)  \in\mathcal{N}_{\delta}(\bar{U}_{\mathfrak{b}, \rm{l}}), \qquad \forall \ x\in(\chi^\varepsilon_{\mathfrak{p}} (t), \chi^\varepsilon_{\rm{s}} (t)); \\
  & U^\varepsilon (t,x)  \in\mathcal{N}_{\delta}(\bar{U}_{\mathfrak{b}, \rm{r}}), \qquad \forall \ x\in(\chi^\varepsilon_{\rm{s}} (t), +\infty).
\end{aligned}
\end{equation}

Our aim is  to show that functional $\mathcal{G}(t)$ is decreasing with respect to $t$. If this monotonicity holds, we naturally suppose $\mathcal{G}(0)$ is a priori upper bound of $\mathcal{G}(t)$.
The proof involves two parts respectively for non-reacting and reacting {fluids}.

\subsection{Existence of the entropy solutions for non-reaction process}

We first investigate the non-reacting flow for equations \eqref{eq:1.1} under condition $(\textrm{C1})$, which is the basis of further
discussion on combustion phenomena in section 4.3.

Since $\bar{T}_{\rm{r}} <\bar{T}_{\rm{l}} < T_\textrm{i}$ under condition $(\textrm{C1})$, combustion does not occur at all. Thus
\begin{eqnarray*}
\Delta\mathcal{L}_{Y}\doteq\mathcal{L}_{Y} (t+0)- \mathcal{L}_{Y} (t-0)=0, \quad \mathcal{L}_{\textsc{by}} (t)\equiv 0,  \qquad \mbox{for}\ \mbox{all}\quad  t>0.
\end{eqnarray*}
We will prove the monotonicity of $\mathcal{G}(t)$ in different cases of local wave interactions. Henceforth, employ the notation
$$\alpha_i^*+\alpha_j^* \mapsto \alpha_i + \alpha_j +\alpha_k$$
to denote the process that collision of waves $\alpha_i^*$ and $\alpha_j^*$ produces a set of outgoing waves $\alpha_i , \alpha_j ,\alpha_k$.
Similarly, process $\alpha_1^*   \mapsto   \alpha_4$ stands for the wave reflection on piston boundary. We have the following lemma on monotonicity.

\begin{lemma}\label{lem:4.1}
Assume that positive $t$ is a time when interaction occurs in domain $\Omega^\varepsilon_\mathfrak{p}$ or on boundary $\Gamma^\varepsilon_\mathfrak{p}$,
and that constant $\tilde{\delta}$  is sufficiently small.
If $\mathcal{G}(t-0)<\tilde{\delta}$,
then there holds
\begin{eqnarray}\label{eq:4.8}
\mathcal{G}(t+0)<\mathcal{G}(t-0).
\end{eqnarray}
\end{lemma}
\begin{proof}
The argument consists of several cases as follows.

\textbf{Case 4.1.1}. \emph{Reflection $\alpha_1^*   \mapsto   \alpha_4$ on boundary.} Suppose that a incident weak $\alpha^*_1$ hits the piston boundary, and produces a reflection weak wave $\alpha_4$. Then the estimate \eqref{eq:2.31} in Proposition \ref{prop:2.2} gives
\begin{equation*}
\alpha_4= \alpha_1^* + \mathcal{O}(1) |\alpha_1^*|^2.
\end{equation*}
Based on this, we obtain
\begin{eqnarray*}
\begin{split}
\Delta\mathcal{L}(t)& = |\alpha_4| - |\alpha^*_1| = \mathcal{O}(1) |\alpha^*_1|^2 \leq \mathcal{O}(1) \mathcal{G}(t-0)|\alpha^*_1|,\\[5pt]
\Delta\mathcal{Q}(t)& = \Delta \mathcal{Q}_w(t) + \Delta \mathcal{Q}_{\mathfrak{p}}(t) + \Delta\mathcal{Q}_{\rm{s}}(t)\\[5pt]
	& = \Big(|\alpha_4| \sum_{\mathcal{A}(\alpha_4)} |\beta|  - |\alpha^*_1|\sum_{\mathcal{A}(\alpha^*_1)} |\beta| \Big)  - \theta_1 |\alpha^*_1| + \theta_4 |\alpha_4 |\\[5pt]
	& \leq \Big(\theta_4 -  \theta_1   +\mathcal{G}(t-0) + \mathcal{O}(1)\mathcal{G}(t-0) \big(\mathcal{G}(t-0) + \theta_4\big) \Big) |\alpha^*_1|  \\[5pt]
	& \leq - \frac{1}{2}(\theta_1 -  \theta_4)|\alpha^*_1| ,
\end{split}
\end{eqnarray*}
which leads to
\begin{eqnarray}\label{eq:4.9}
\begin{split}
\Delta\mathcal{G}(t)&= \Delta \mathcal{L}(t) + K\Delta \mathcal{Q}(t) \\[5pt]
& \leq \Big(\mathcal{O}(1) \mathcal{G}(t-0)- \frac{1}{2}K(\theta_1 - \theta_4)\Big) |\alpha^*_1| \\[5pt]
& < - \frac{1}{4}K (\theta_1 -  \theta_4) |\alpha_1^*| < 0,
\end{split}
\end{eqnarray}
provided $\tilde{\delta}$ is sufficiently small.

\textbf{Case 4.1.2}. \emph{Appearance of single $\alpha_4$ from boundary.}
If a weak wave $\alpha_4$ issues from a corner point $(t, \chi^{\varepsilon}_{\mathfrak{p}}(t))$ with the left and right velocities satisfying $u_{\rm{l}}=u^{\varepsilon}_{\mathfrak{p}}(t+0)$ and $u_{\rm{r}}=u^{\varepsilon}_{\mathfrak{p}}(t-0)$, then by estimate \eqref{eq:2.28} in Lemma \ref{lem:2.2}, we know that
the strength $|\alpha_4|$ is equivalent to $|u^{\varepsilon}_{\mathfrak{p}}(t+0) - u^{\varepsilon}_{\mathfrak{p}}(t-0)|$, \emph{i.e.},
\begin{equation*}
\frac{1}{a} |\alpha_4| \leq |u^{\varepsilon}_{\mathfrak{p}}({t+0}) - u^{\varepsilon}_{\mathfrak{p}}({t-0})|\leq a |\alpha_4|.
\end{equation*}
Thus it implies that
\begin{eqnarray*}
\begin{split}
&\Delta\mathcal{L}(t)= |\alpha_4|,\qquad
\Delta \mathcal{L}_{\textsc{b}}(t)= -|u^{\varepsilon}_{\mathfrak{p}}({t+0}) - u^{\varepsilon}_{\mathfrak{p}}({t-0})|\leq -\frac{1}{a} |\alpha_4|, \\[5pt]
&\Delta \mathcal{Q}(t) = \Delta \mathcal{Q}_w(t)  + \Delta \mathcal{Q}_{\rm{s}}(t)=|\alpha_4| \sum_{\mathcal{A}(\alpha_4)} |\beta| + \theta_4 |\alpha_4|
	\leq \big(\theta_1 + \mathcal{G}(t-0) \big)|\alpha_4|,
\end{split}
\end{eqnarray*}
and
\begin{align}
\notag
\Delta\mathcal{G}(t)& = \Delta \mathcal{L}(t) + K\Delta \mathcal{Q}(t) + K\Delta \mathcal{L}_{\textsc{b}}(t) \\[5pt]
\notag
&\leq |\alpha_4|  + K\Big(\theta_1 + \mathcal{G}(t-0) -\frac{1}{a} \Big)|\alpha_4|  \\[5pt]
\notag
&\leq |\alpha_4| - \frac{1}{2a} K|\alpha_4| \\[5pt]
\label{eq:4.10}
&\leq - \frac{1}{4a} K|\alpha_4|<0,
\end{align}
provided both $\theta_1$ and $\tilde{\delta}$ are small, while constant $K$ suitably large.

\textbf{Case 4.1.3}. \emph{Interaction $\alpha^*_4 + \alpha^{\rm{s},*}_4  \mapsto \alpha_1+ \alpha_2+ \alpha^{\rm{s}}_4$}.
If a weak wave $\alpha^*_{4}$ interacts with strong shock $\alpha^{\rm{s},*}_4$ from left,
then Proposition \ref{prop:2.3} directly gives the estimates
$$\alpha_1  = \kappa_{\rm{s}} \alpha^*_4 + \mathcal{O}(1) |\alpha^*_4 |^2 ,\qquad  \alpha_2  = \mathcal{O}(1) |\alpha^*_4|, \qquad
\alpha^{\rm{s}}_4=\alpha^{\rm{s},*}_4+\mathcal{O}(1)|\alpha^{*}_4|.$$
From these and coefficient condition \eqref{eq:4.3}, we deduce that
\begin{align*}
\Delta \mathcal{L}(t) & = |\alpha_1| + |\alpha_2| +|\alpha^{\rm{s}}_4- \bar{\alpha}^{\rm{s}}_4|- |\alpha^*_4|-|\alpha^{\rm{s},*}_4- \bar{\alpha}^{\rm{s}}_4|= \mathcal{O}(1)|\alpha^*_4|, \\[5pt]
\Delta \mathcal{Q}(t) & =  \Delta \mathcal{Q}_w(t) +\Delta\mathcal{Q}_{\mathfrak{p}}(t) +\Delta\mathcal{Q}_{\rm{s}}(t) \\[5pt]
	        & = \Big( |\alpha_1|\sum_{\mathcal{A}(\alpha_1)} |\beta|
	+ |\alpha_2|  \sum_{\mathcal{A}(\alpha_2)} |\beta|
	- |\alpha_4^*|  \sum_{\mathcal{A}(\alpha_4^*)} |\beta| \Big)
	+ \theta_1  |\alpha_1|  -  \theta_4|\alpha_4^*|    \\[5pt]
	& \leq  \big(|\kappa_{\rm{s}} \alpha_4^*| + \mathcal{O}(1) |\alpha_4^*|^2 \big)\big(\mathcal{G}(t-0)+ \theta_1\big)
	+ \mathcal{O}(1)\mathcal{G}(t-0) |\alpha_4^*|  - \theta_4 |\alpha_4^*|\\[5pt]
	& = \big(\theta_1|\kappa_{\rm{s}}| - \theta_4  + \mathcal{O}(1)\mathcal{G}(t-0) \big) |\alpha_4^*|\\[5pt]
	& \leq - \frac{1}{2}( \theta_4 - \theta_1|\kappa_{\rm{s}}| ) |\alpha_4^* |,
\end{align*}
which gives
\begin{equation}\label{eq:4.11}
\Delta\mathcal{G}(t)=\Delta\mathcal{L}(t)+ K\Delta \mathcal{Q}(t)\leq - \frac{1}{4} K( \theta_4 - \theta_1|\kappa_{\rm{s}}| ) |\alpha_4^*| <0,
\end{equation}
provided that $K>0$ is sufficiently large. Similarly, if a weak wave $\alpha^*_{4}$ interacts with strong shock $\alpha^{\rm{s},*}_4$ from right, then it follows from the Proposition \ref{prop:2.4} that $\Delta\mathcal{G}(t)<0$.

\textbf{Case 4.1.4}. \emph{Interaction} $ \alpha_i^* + \alpha^{\rm{s},*}_4  \mapsto \alpha_1+ \alpha_2+ \alpha^{\rm{s}}_4$ ($i=1,2$).
A weak $i$-wave $\alpha^*_{i}$ interacts with strong $4$-shock $\alpha^{\rm{s},*}_4$ from right for $i=1,2$. In this case, by Proposition \ref{prop:2.4}, we have that
\begin{eqnarray*}
|\alpha_1| + |\alpha_2| = \mathcal{O}(1) |\alpha_i^*|,\qquad \alpha^{\rm{s}}_{4}=\alpha^{\rm{s},*}_{4}+\mathcal{O}(1)|\alpha^{*}_{i}|.
\end{eqnarray*}
Then it follows that
\begin{eqnarray*}
\begin{split}
\Delta\mathcal{L}(t) & = |\alpha_1| + |\alpha_2| +|\alpha^{\rm{s}}_4- \bar{\alpha}^{\rm{s}}_4|- |\alpha^*_i|-|\alpha^{\rm{s},*}_4- \bar{\alpha}^{\rm{s}}_4|=  \mathcal{O}(1) |\alpha_i^* |,  \\[5pt]
\Delta \mathcal{Q}(t) & =\Delta\mathcal{Q}_w(t) +\Delta\mathcal{Q}_{\mathfrak{p}}(t) +\Delta\mathcal{Q}_{\rm{s}}(t) \\[5pt]
	& =\Big(|\alpha_1| \sum\limits_{\mathcal{A}(\alpha_1)} |\beta|+|\alpha_2|  \sum_{\mathcal{A}(\alpha_2)} |\beta|  - |\alpha_i^*|  \sum_{\mathcal{A}(\alpha_i^*)} |\beta| \Big) \\[5pt]
	&\qquad +\theta_1 \big( |\alpha_1| - (2-i) |\alpha_i^*|  \big)  -  |\alpha_i^*| \\[5pt]
	& \leq - |\alpha_i^*| + \mathcal{O}(1)(\mathcal{G}(t-0)+ \theta_1) |\alpha_i^*| \\[5pt]
	& \leq - \frac{1}{2}|\alpha_i^*|.
\end{split}
\end{eqnarray*}
Thus we can get
\begin{eqnarray}\label{eq:4.12}
\begin{split}
\Delta \mathcal{G}(t) &=\Delta\mathcal{L}(t)+ K\Delta \mathcal{Q}(t)
	\leq \mathcal{O}(1) |\alpha_i^* | - \frac{1}{2} K |\alpha_i^*|  <0 ,
\end{split}
\end{eqnarray}
by choosing $K$ sufficiently large.

\textbf{Case 4.1.5}. \emph{Interaction} $ \alpha_5^* + \alpha^{\rm{s},*}_4  \mapsto \alpha^{\rm{s}}_4 +\alpha_5 $ .
Suppose a $\mathcal{NP}$ wave $\alpha^*_5$ interacts with strong $4$-shock $\alpha^{\rm{s},*}_4$ from left. It follows from Proposition \ref{prop:3.3} that
\begin{eqnarray*}
\begin{split}
\Delta\mathcal{L}(t) & = |\alpha_5| +|\alpha^{\rm{s}}_4- \bar{\alpha}^{\rm{s}}_4|- |\alpha^*_5|-|\alpha^{\rm{s},*}_4- \bar{\alpha}^{\rm{s}}_4|=  \mathcal{O}(1) |\alpha_5^* |,  \\[5pt]
\Delta \mathcal{Q}(t) & =\Delta\mathcal{Q}_w(t) +\Delta\mathcal{Q}_{\mathfrak{p}}(t) +\Delta\mathcal{Q}_{\rm{s}}(t) \\[5pt]
	& =\Big(|\alpha_5| \sum\limits_{\mathcal{A}(\alpha_5)} |\beta|  - |\alpha_5^*|  \sum_{\mathcal{A}(\alpha_5^*)} |\beta| \Big)  - \theta_5 |\alpha_5^*| \\[5pt]
	& \leq \Big( \mathcal{O}(1)\mathcal{G}(t-0)-\theta_5 \Big) |\alpha_5^*| \\[5pt]
	& \leq - \frac{1}{2}\theta_5 |\alpha_5^*|.
\end{split}
\end{eqnarray*}
Thus we find that
\begin{eqnarray}\label{eq:4.12}
\begin{split}
\Delta \mathcal{G}(t) &=\Delta\mathcal{L}(t)+ K\Delta \mathcal{Q}(t)
	\leq \mathcal{O}(1) |\alpha_5^* | - \frac{1}{2} K \theta_5 |\alpha_i^*|  <0 ,
\end{split}
\end{eqnarray}
if $K$ is sufficiently large.

\textbf{Case 4.1.6}. \emph{Interactions between weak waves $\alpha^*_{i}$ and $\alpha^*_{j}$ $(1\leq i,j\leq 5)$.}
For instance, consider the interaction $\alpha_i^* + \alpha_j^* \mapsto \alpha_1+ \alpha_2+ \alpha_4$ for $i,j \in\{1,2,4\}$. Estimates in \eqref{eq:2.21} directly give
\begin{align*}
&|\alpha_i- \alpha_i^*| + |\alpha_j- \alpha_j^*| + |\alpha_k| = \mathcal{O}(1) |\alpha_i^*\alpha_j^*| \quad  \mbox{for} \quad i\neq j,  \\[5pt]
&|\alpha_i- \alpha_i^*- \alpha_j^*| + \sum\limits_{k\neq i}|\alpha_k| = \mathcal{O}(1) |\alpha_i^*\alpha_j^*|\quad  \mbox {for} \quad i= j.
\end{align*}
From these estimates, we deduce that
\begin{eqnarray*}
\begin{split}
\Delta \mathcal{L}(t) &= \mathcal{O}(1) |\alpha_i^*\alpha_j^*| ,\\[5pt]
\Delta \mathcal{Q}(t) & \leq \Big( -1 + \mathcal{O}(1) \big(\theta_1+ \theta_4 + \mathcal{G}(t-0) \big) \Big)  |\alpha_i^*\alpha_j^*|\\[5pt]
	                  &\leq -\frac{1}{2} |\alpha_i^*\alpha_j^*|,
\end{split}
\end{eqnarray*}
provided $\theta_1,\theta_4$ and $\tilde{\delta}$ are taken suitably small. It further implies that
\begin{eqnarray}\label{eq:4.13}
\begin{split}
\Delta \mathcal{G}(t) \leq  \mathcal{O}(1) |\alpha_i^*\alpha_j^*|  -\frac{1}{2}K |\alpha_i^*\alpha_j^*| <0,
\end{split}
\end{eqnarray}
by choosing $K>0$ sufficiently large.

We turn to interaction $\alpha_i^* +  \alpha_j^* \mapsto \square + \alpha_5$ for $i,j\neq 3$.
Here symbol $\square$ stands for a pair of outgoing waves of family $i$ or $j$.
Recall Proposition \ref{prop:3.2} on interactions involving $\mathcal{NP}$ waves.
If $i= j<5$, then $\square$ is actually a coalesced wave $\alpha_i$ such that $\alpha_i= \alpha_i^* +  \alpha_j^*$.
If $i\neq j$ and $i, j<5$, then $\square$ contains two waves $\alpha_i$ and $\alpha_j$ such that $\alpha_i=\alpha_i^*, \alpha_j=\alpha_j^*$.
If $i=5$ and $j<5$, then $\square$ preserves a single $\alpha_j$ after interaction, which satisfies $\alpha_j= \alpha_j^* $.
In any case, we always have
\begin{eqnarray*}
\begin{split}
\Delta \mathcal{L}(t)& =  \mathcal{O}(1) |\alpha_i^* \alpha_j^*|,\\[5pt]
\Delta \mathcal{Q}(t) & \leq  - |\alpha_i^* \alpha_j^*| +\mathcal{O}(1)(\mathcal{G}(t-0) + \theta_1) |\alpha_i^* \alpha_j^*|
	\leq - \frac{1}{2}  |\alpha_i^* \alpha_j^*|.
\end{split}
\end{eqnarray*}
It follows that
\begin{eqnarray}\label{eq:4.14}
\begin{split}
\Delta \mathcal{G}(t)  & \leq \mathcal{O}(1) |\alpha_i^* \alpha_j^*| - \frac{1}{2} K |\alpha_i^* \alpha_j^*|  <0.
\end{split}
\end{eqnarray}


Consider $\alpha_3^* +  \alpha_j^* \mapsto \alpha_3 + \alpha_j$ for $j\neq 2,3$. In this case, transport of reactant does not influence the motion of non-reacting fluid.
So $\alpha_3= \alpha_3^*$ and $\alpha_j = \alpha_j^*$, which leads to
\begin{eqnarray*}
\Delta\mathcal{L}(t)= 0,\quad   \Delta\mathcal{Q}(t)= -|\alpha^*_3  \alpha^*_j|.
\end{eqnarray*}
It gives that
\begin{eqnarray}\label{eq:4.15}
\Delta \mathcal{G}(t) \leq -K |\alpha_3^*  \alpha_j^* |<0.
\end{eqnarray}

Finally, combining the estimates \eqref{eq:4.9}-\eqref{eq:4.15} altogether, we derive the conclusion \eqref{eq:4.8}, and complete the proof.
\end{proof}

Lemma \ref{lem:4.1} implies the fact that $\mathcal{G}(t)\leq \mathcal{G}(0) \leq \mathcal{O}(1)\epsilon$ for every $t>0$,
provided $\epsilon$ is small enough.
Consequently, we obtain the following properties on the approximate solutions.

\begin{proposition}\label{prop:4.1}
Under the assumptions $(\mathbf{A1})$-$(\mathbf{A2})$, there exist constants $C_1>0$, $C_2>0$ depending solely on $\bar{U}_{\mathfrak{b}}$ such that if  condition $(\rm{C1})$ holds for $\epsilon>0$ sufficiently small, then the approximate solution $U^{\varepsilon}$ to $(\textit{IBVP})$ satisfies the estimates
\begin{eqnarray}\label{eq:4.16}
\begin{split}
&\sup_{x\in(\chi^{\varepsilon}_{\mathfrak{p}}(t),\chi^{\varepsilon}_{\rm{s}}(t))}|U^{\varepsilon}(t,\cdot)-\bar{U}_{\mathfrak{b},\rm{l}}|
+\sup_{x\in(\chi^{\varepsilon}_{\rm{s}}(t),+\infty)}|U^{\varepsilon}(t,\cdot)-\bar{U}_{\mathfrak{b},\rm{r}}|
+\sup_{t>0}|\dot{\chi}^\varepsilon_{\rm{s}}(t)-\bar{\rm{s}}_{\mathfrak{b}}|\\[5pt]
&\quad\ +T.V.\{U^{\varepsilon}(t,\cdot);(\chi^{\varepsilon}_{\mathfrak{p}}(t),\chi^{\varepsilon}_{\rm{s}}(t))\}
+T.V.\{U^{\varepsilon}(t,\cdot);(\chi^{\varepsilon}_{\rm{s}}(t),+\infty)\}<C_1\epsilon,
\end{split}
\end{eqnarray}
and
\begin{eqnarray}\label{eq:4.17}
\|U^{\varepsilon}(t,\cdot)-U^{\varepsilon}(\tilde{t},\cdot)\|_{L^1(\mathbb{R})}+|\chi^{\varepsilon}_{\rm{s}}(t)-\chi^{\varepsilon}_{\rm{s}}(\tilde{t})|\leq C_1(t-\tilde{t})\quad \mbox{for}\quad t>\tilde{t}>0.
\end{eqnarray}
Moreover, there holds that
\begin{eqnarray}\label{eq:4.18}
\|Y^{\varepsilon}(t,\cdot)\|_{L^1(\mathbb{R})}\leq \frac{\sup\rho^{\varepsilon}}{\inf\rho^{\varepsilon}}\|Y^{\varepsilon}(\tilde{t},\cdot)\|_{L^1(\mathbb{R})}+C_2(t-\tilde{t}){ \varepsilon}
\quad \mbox{for}\quad t>\tilde{t}>0.
\end{eqnarray}
\end{proposition}

\begin{proof}
Estimate \eqref{eq:4.16} is a direct result based on assumption $(\mathbf{A1})$, Lemma \ref{lem:2.1a}  and the fact
$$|\dot{\chi}^\varepsilon_{\rm{s}}(t)-\bar{\rm{s}}_{\mathfrak{b}}|+
T.V.\{U^{\varepsilon}(t,\cdot);(\chi^{\varepsilon}_{\mathfrak{p}}(t),\chi^{\varepsilon}_{\rm{s}}(t))\}
+T.V.\{U^{\varepsilon}(t,\cdot);(\chi^{\varepsilon}_{\rm{s}}(t),+\infty)\}\leq \mathcal{O}(1)\mathcal{G}(t).$$
See also \cite{Kuang-Zhao-2020} for detailed argument.
One can readily verify \eqref{eq:4.17} according to finite propagation speed for homogenous system \eqref{eq:2.2} under condition $(\rm{C1})$.
Now we continue to consider the estimate \eqref{eq:4.18}.
To this end, let's define the amount of reactant by
\begin{eqnarray*}
Z^\varepsilon(t) = \int_{\chi^\varepsilon_{\mathfrak{p}}(t)}^{+\infty} \rho^\varepsilon(t,x) Y^\varepsilon(t,x) dx.
\end{eqnarray*}
It is clear that $Z^\varepsilon(t)$ is globally Lipshcitz continuous, and differentiable everywhere except finitely many times related to  wave-interaction and angular points on piston boundary.
Let wave front $\alpha\in\mathcal{J}(U^{\varepsilon})$ be located at $x_\alpha(t)$ in $(\chi^\varepsilon_{\mathfrak{p}}(t),+\infty)$. Then, by \emph{Rankine-Hugoniot} conditions
and Proposition \ref{prop:3.1}, we get
\begin{equation}\label{eq:4.19}
\begin{split}
\dot{Z}^\varepsilon (t)&= \sum_{\alpha\in \mathcal{J}(U^{\varepsilon})} [\rho^\varepsilon Y^\varepsilon] \dot{x}_\alpha(t) - \rho^\varepsilon(t,\chi^{\varepsilon}_{\mathfrak{p}}(t)) Y^\varepsilon(t,\chi^{\varepsilon}_{\mathfrak{p}}(t)) u^{\varepsilon}_{\mathfrak{p}}(t)\\[5pt]
	& =\sum_{\alpha \in \mathcal{J}(U^{\varepsilon})\setminus \mathcal{NP}} \Big([\rho^\varepsilon u^\varepsilon Y^\varepsilon] + \mathcal{O}(1) \varepsilon|\alpha| \Big)
	+ \sum_{\alpha\in \mathcal{NP}} \Big([\rho^\varepsilon u^\varepsilon Y^\varepsilon] + \mathcal{O}(1)  |\alpha|  \Big) \\[5pt]
	& \qquad\quad- \rho^\varepsilon(t,\chi^\varepsilon_{\mathfrak{p}}(t)) Y^\varepsilon (t,\chi^\varepsilon_{\mathfrak{p}}(t)) u^\varepsilon_{\mathfrak{p}}(t)\\[5pt]
	& = \mathcal{O}(1) \varepsilon \sum_{\alpha\in \mathcal{J}(U^{\varepsilon})\setminus \mathcal{NP}}  |\alpha| + \mathcal{O}(1) \sum_{\alpha\in \mathcal{NP}} |\alpha| \\[5pt]
	&= \mathcal{O}(1) \varepsilon,
\end{split}
\end{equation}
where the notation $[f]=f(t,x_\alpha(t)-0)-f(t,x_\alpha(t)+0)$ for $f=\rho^\varepsilon Y^\varepsilon$ or $f=\rho^\varepsilon u^\varepsilon Y^\varepsilon$.
Integrate \eqref{eq:4.19} from $\tilde{t}$ to $t$ , then derive that
\begin{equation*}\label{reactant conservation}
\int^{+\infty}_{\chi^\varepsilon_{\mathfrak{p}}(t)} (\rho^\varepsilon Y^\varepsilon)(t,x) dx -  \int_{\chi^\varepsilon_{\mathfrak{p}}(\tilde{t})}^{+\infty} (\rho^\varepsilon Y^\varepsilon)(\tilde{t},x) dx
= \mathcal{O}(1) \varepsilon (t-\tilde{t}),\quad \mbox{for}\quad t>\tilde{t}>0.
\end{equation*}
Since $\bar{\rho}_{\rm{l}}>\bar{\rho}_{\rm{r}}>0$, it follows from \eqref{eq:4.16} and the smallness of $\epsilon$ that
\begin{equation*}
\|Y^\varepsilon(t,\cdot)\|_{L^1(\mathbb{R})} \leq \frac{\sup \rho^\varepsilon}{\inf \rho^\varepsilon} \|Y^\varepsilon(\tilde{t},\cdot)||_{L^1(\mathbb{R})}+ \mathcal{O}(1) \varepsilon (t-\tilde{t}),
\end{equation*}
which leads to the estimate \eqref{eq:4.18}.
\end{proof}

\emph{Proof of Theorem \ref{thm:1.1} under condition $(\rm{C1})$}. By Helly's compactness theorem and Proposition \ref{prop:4.1}, there is a sequence $\{\varepsilon_k\}$
such that as $\varepsilon_k \to 0^+$, the approximation $U^{\varepsilon_k}$ converges almost everywhere in $\mathbb{R}_+ \times \mathbb{R}$ to some function $U$ satisfying \eqref{eq:1.14}.
Moreover, following the standard methods as in \cite{Bressan-2000}, we can show that $U$ is an entropy solution to (\emph{IBVP}), and complete the proof on existence and structural stability under condition $(\textrm{C1})$.
\hfill$\Box$

\subsection{Existence of the entropy solutions for reaction process}
In this subsection, we are concerned with the existence of large amplitude combustion solutions to piston problems under condition $(\textrm{C2})$ or $(\textrm{C3})$. Assume that $U^\varepsilon$ is an $\varepsilon$-approximate solution constructed by fractional-step front tracking scheme as presented in section 3.
We aim to establish the uniform bound on the total variation of $U^\varepsilon$ for reacting flow.
Suppose $\mathcal{G}(t)<\tilde{\delta}$ for every $t\in (0,t_k)$. Hence \eqref{eq:4.7} holds. The piecewise continuity of $\phi(T)$ implies that there exist two positive constants $\underline{\phi}$ and $\bar{\phi}$ such that
\begin{equation}\label{eq:4.20}
	\underline{\phi} \leq \phi(T^\varepsilon (t,x)) \leq \bar{\phi} \qquad  \text{for }  t\in (0,t_k) , x\in \mathcal{I}.
\end{equation}
Here interval $\mathcal{I} \doteq (\chi^{\varepsilon}_{\mathfrak{p}}(t),+\infty)$ under condition (\textrm{C2}),
while $\mathcal{I} \doteq (\chi^{\varepsilon}_{\mathfrak{p}}(t),\chi^{\varepsilon}_{\rm{s}}(t))$ under condition (\textrm{C3}).
It suffices to prove $\mathcal{G}(t_k-0)<\mathcal{G}(t_k+0)$. Then the decrease of $\mathcal{G}(t)$ ensures that the properties
\eqref{eq:4.7}\eqref{eq:4.20} further hold for $t\in (0,t_{k+1})$.

For brevity, we always use superscript $^*$ in this subsection to mark the states before reaction.
According to the scheme \eqref{eq:3.6}, the states transition in combustion process is given by
\begin{eqnarray}\label{reaction step}
U=U^*+\varepsilon G_{0}(U^*)\qquad  \text{at time } t_k,
\end{eqnarray}
where $U\doteq U^{\varepsilon}(t_k,x) , U^* \doteq U^{\varepsilon}(t_{k}-0,x) $ and
\begin{eqnarray*}
G_{0}(U^*)=\big(0,0, (\gamma-1)\mathfrak{q}_0\rho^*\phi(T^*)Y^*, -\phi(T^*)Y^*\big)^{\top}.
\end{eqnarray*}
When a wave $ \alpha_i^*$ hits the grid line $t=t_k$, there are diverse cases in which wave $ \alpha_i^*$ splits into several new fronts after combustion (see Fig \ref{fig9}).
Here  again we use the notation
$$ \alpha_i^* \mapsto  \alpha_i+\alpha_j+ \alpha_k + \alpha_m \quad  \text{at } x_\alpha$$
to represent the process of wave splitting caused by combustion reaction at position $x_\alpha$.

\begin{lemma}\label{lem:4.2}
Suppose that $\tilde{\delta}$ and $\varepsilon$ are small enough, and that $\mathcal{G}(t)<\tilde{\delta}$ for $t\in (0,t_k)$.
Then, in reaction process at $t=t_k$, we have the estimates on wave strengths as follows.

${\rm{(i)}}$\ If
$ \alpha_3^* \mapsto  \alpha_1+\alpha_2+ \alpha_3 + \alpha_4 $ at $x_\alpha$,
 then there hold
\begin{eqnarray*}
 |\alpha_3| \leq |\alpha_3^*|  - |\alpha_3^*| \underline{\phi}\varepsilon,\quad
\sum_{j\neq  3} |\alpha_j|    = \mathcal{O}(1) |\alpha_3^*|\varepsilon.
\end{eqnarray*}

${\rm{(ii)}}$\ If $\alpha_i^* \mapsto  \alpha_1+\alpha_2+ \alpha_3 + \alpha_4 $ at $x_\alpha$ for $i\neq 3$,
 then there holds
\begin{eqnarray*}
 |\alpha_i -\alpha_i^*| +  \sum_{j\neq  i} |\alpha_j|=\mathcal{O}(1) \varepsilon|\alpha_i^*| Y^{*}_\alpha,
\end{eqnarray*}
where $Y^{*}_{\alpha}=Y^{\varepsilon}(t_{k}-0,x_{\alpha})$.

\end{lemma}

\begin{proof}
In order to distinguish the states on distinct sides of $x_\alpha$, we set
\begin{align*}
&Y^*_{\alpha\pm} \doteq Y^{\varepsilon} (t_k-0, x_\alpha\pm0),\quad  Y_{\alpha\pm} \doteq Y^{\varepsilon} (t_k, x_\alpha\pm0),\\
&T^*_{\alpha\pm} \doteq T^{\varepsilon} (t_k-0, x_\alpha\pm0),\quad T_{\alpha\pm} \doteq T^{\varepsilon} (t_k, x_\alpha\pm0), \\ &\Upsilon_{\alpha\pm}\doteq |Y_{\alpha\pm}-Y^{*}_{\alpha\pm}|.
\end{align*}

(i)\
Since $T^{*}_{\alpha+}=T^{*}_{\alpha-}$ in this case, the definitions of $\alpha_3$ and $\Upsilon_{\alpha\pm}$ directly yield that
\begin{eqnarray}\label{estimate 1 for reaction step}
\begin{split}
&|\alpha_3| = |(Y^{*}_{\alpha+}- Y^{*}_{\alpha-}) (1-\phi(T^{*}_{\alpha-})\varepsilon)|\leq   |\alpha_3^*| (1 -  \underline{\phi}\varepsilon ),\\
&|\Upsilon_{\alpha+} - \Upsilon_{\alpha-}|=|Y^{*}_{\alpha+}- Y^{*}_{\alpha-} |\phi(T^{*}_{\alpha-})\varepsilon \leq |\alpha_3^*| \bar{\phi} \varepsilon.
\end{split}
\end{eqnarray}
Notice that the outgoing waves $\alpha_1, \cdots, \alpha_4$ are uniquely determined by $\alpha_3^*,\Upsilon_{\alpha-}$ and $\Upsilon_{\alpha+}$.
Thus set function $\alpha_j= f_j(\alpha_3^*,\Upsilon_{\alpha-},\Upsilon_{\alpha+})$. It follows from (\ref{estimate 1 for reaction step}) that
\begin{eqnarray}\label{wave estimate for tk}
\begin{split}
	\alpha_j  & =  f_j(\alpha_3^*,\Upsilon_{\alpha-},\Upsilon_{\alpha-}) + \mathcal{O}(1) |\Upsilon_{\alpha-}-\Upsilon_{\alpha+}| \\[5pt]
	&  = f_j(\alpha_3^*, 0, 0) + \mathcal{O}(1) (|\alpha_3^*\Upsilon_{\alpha-}|+ |\Upsilon_{\alpha-} - \Upsilon_{\alpha+}|)  \\[5pt]
	&  = \delta_{3j}\cdot \alpha_3^* + \mathcal{O}(1)  |\alpha_3^*|\varepsilon,
\end{split}
\end{eqnarray}
where $\delta_{ij}$ is the Kronecker delta.  The above estimate in fact gives the conclusion in \rm(i).

\rm (ii)\ 
In this case, we have $Y^{*}_{\alpha +}=Y^{*}_{\alpha -}=Y^{*}_{\alpha}$. Hence
\begin{eqnarray*}
\begin{split}
|\alpha_3 |=  |\Upsilon_{\alpha+} - \Upsilon_{\alpha-}| =  Y_\alpha^* |\phi(T^{*}_{\alpha+})-\phi(T^{*}_{\alpha -})|\varepsilon = \mathcal{O}(1) |\alpha_i^*| Y^*_\alpha \varepsilon.
\end{split}
\end{eqnarray*}
We employ the similar argument in (\ref{wave estimate for tk}) to derive that
\begin{eqnarray*}
\begin{split}
\alpha_j &=  \delta_{ij}\cdot \alpha_i^* +  \big(|\alpha_i^*\Upsilon_{\alpha-}|+ |\Upsilon_{\alpha-} - \Upsilon_{\alpha+}| \big)\\[5pt]
&= \delta_{ij}\cdot \alpha_i^* +  \mathcal{O}(1) |\alpha_i^*| Y_\alpha^{*} \varepsilon.
\end{split}
\end{eqnarray*}

\end{proof}

Next, according to the fractional-step scheme, we drop the exothermic reaction terms if $t\in (t_{k-1}, t_k)$. Thus,
\begin{eqnarray*}
\Delta \mathcal{L}_{Y} (t)=  \mathcal{L}_{Y}(t+0) - \mathcal{L}_{Y}(t-0) =0, \qquad  \Delta \mathcal{L}_{\textsc{by}}(t)=\mathcal{L}_{\textsc{by}}(t+0) - \mathcal{L}_{\textsc{by}}(t-0) = 0.
\end{eqnarray*}
which implies that when $t\in (t_{k-1}, t_k)$, functional $\mathcal{G}(t)$ is decreasing as proved in section { 4.2.}
It suffices to check its monotonicity in the combustion process at time $t_{k}$ with $k\geq 1$.

\begin{lemma}\label{lem:4.3}
Suppose that $\tilde{\delta}$ and $\varepsilon$ are small enough, and that $\mathcal{G}(t)<\tilde{\delta}$ for $t\in (0,t_k)$.
It holds that
\begin{eqnarray}\label{eq:4.22}
\mathcal{G}(t_{k}+0)<\mathcal{G}(t_{k}-0).
\end{eqnarray}
\end{lemma}

\begin{proof}
We mainly investigate the partial reaction phenomenon under condition $\textrm{(C3)}$, since the global existence of completely ignited flow (i.e. $T(t,x)>T_\text{i}$ for every $x \geq \chi_{\mathfrak{p}}(t)$) under condition $\textrm{(C2)}$ can be seen as a byproduct which will be discussed at the end of this proof.

In fact the condition $\textrm{(C3)}$ and property \eqref{eq:4.7} imply that
\begin{eqnarray*}\label{temperature-(C3)}
\begin{split}
T^{\varepsilon}(t,x) > T_{\text{i}}\quad \ & \mbox{if  } \chi^{\varepsilon}_{\mathfrak{p}}(t)\leq x <\chi^{\varepsilon}_{\rm{s}}(t) , \qquad  T^{\varepsilon}(t,x) <T_{\textrm{i}}\quad  \mbox{if}\  x >\chi^{\varepsilon}_{\rm{s}}(t).
\end{split}
\end{eqnarray*}
Then, under the circumstance of partial reaction, we can show that the total consumption satisfies
\begin{align*}
& \sum_{ k=1 }^\infty \Upsilon^{\varepsilon, \mathfrak{p}}_k
\leq \sum_{ k=1 }^\infty \|Y_0\|_{\infty} \textrm{e}^{-\underline{\phi} k\varepsilon} \cdot \bar{\phi} \varepsilon
\leq \frac{2\bar{\phi} ||Y_0||_{\infty}}{\underline{\phi}} \qquad  \mbox{on}\   \chi^{\varepsilon}_{\mathfrak{p}},
\end{align*}
and
\begin{align*}
& \sum_{ k=1 }^\infty \Upsilon_k^{\varepsilon, \rm{s}} \leq \frac{2\bar{\phi} \|Y_0\|_{L^1(\mathbb{R})}}{\min \dot{\chi}^{\varepsilon}_{\rm{s}}(t)} < +\infty \qquad  \mbox{on}\ \chi^{\varepsilon}_{\rm{s}}.
\end{align*}
Based on these, we now focus on the pointwise changes of $\mathcal{G}(t_k)$ within the reaction zone $[\chi^{\varepsilon}_{\mathfrak{p}}(t_{k}), \chi^{\varepsilon}_{\rm{s}}(t_{k})]$.




\textbf{Case 4.2.1}. \emph{Combustion process at $\chi^{\varepsilon}_{\mathfrak{p}}(t_k)$.}
 Concerning the combustion process near the piston boundary, we have
\begin{equation}\label{eq:4.23}
\Delta \mathcal{G}(t_k)  \big|_{x=\chi^{\varepsilon}_{\mathfrak{p}}(t_k)} =K\Delta \mathcal{L}_{\textsc{b}}(t_k)\big|_{x=\chi^{\varepsilon}_{\mathfrak{p}}(t_k)}= -K\Upsilon^{\varepsilon,\mathfrak{p}}_k<0.
\end{equation}

\begin{figure}[ht]
\begin{minipage}[t]{0.45\textwidth}
\centering
\begin{tikzpicture}[scale=1]
\draw [dashed][thin](-1.4,1.1)--(1.2,1.1);
\draw [thin](-0.2,-0.2)--(-0.3,1.1);

\draw [thin](-0.3,1.1)--(-0.9,2.2);
\draw [dashed][thin][blue](-0.3,1.1)to(0,2.3);
\draw [thin](-0.3,1.1)to(0.6,2.2);

\node at (-2,1.1) {$t=t_k$};
\node at (0.9,2.4) {$\alpha_{4}$};
\node at (0,2.5) {$\alpha_{2(3)}$};
\node at (-0.9,2.4) {$\alpha_{1}$};
\node at (-0.1, -0.5) {$\alpha^*_{i}$};

\end{tikzpicture}
\caption{Combustion process in $(\chi^{\varepsilon}_{\mathfrak{p}}(t), \chi^{\varepsilon}_{\rm{s}}(t))$}\label{fig9}
\end{minipage}
\begin{minipage}[t]{0.45\textwidth}
\centering
\begin{tikzpicture}[scale=1.0]
\draw [line width=0.04cm][red](-1.0,-0.7)--(-0.5,0.5);
\draw [dashed][thin](-1.7,0.5)--(1.1,0.5);

\draw [line width=0.04cm][red](-0.5,0.5)--(0.7,1.5);
\draw [dashed][thin][blue](-0.5,0.5)--(-0.1, 1.7);
\draw [thin](-0.5,0.5)--(-1.0,1.7);

\node at (-2.3,0.5) {$t=t_k$};
\node at (1.0, 1.7) {$\alpha^{\rm{s}}_{4}$};
\node at (-0.1, 1.85) {$\alpha_{2(3)}$};
\node at (-1.1, 1.9) {$\alpha_{1}$};

\node at (-1.0, -1.1) {$\alpha^{\rm{s},*}_{4}$};

\end{tikzpicture}
\caption{Combustion process near strong shock}\label{fig10}
\end{minipage}
\end{figure}

\textbf{Case 4.2.2}. \emph{Combustion process at $ x_\alpha \in (\chi^{\varepsilon}_{\mathfrak{p}}(t_k), \chi^{\varepsilon}_{\rm{s}}(t_k))$.}
 Suppose a front of wave $\alpha_i^*$ hits the grid line $t=t_k$. See Fig \ref{fig9}.
Then combustion process gives rise to
$$ \alpha_i^* \mapsto  \alpha_1+\alpha_2+ \alpha_3+ \alpha_4 \ \ \ \ \ \
\text{at } x_\alpha \ .$$
Similarly,
$ \beta_j^* \mapsto \beta_1+\beta_2+ \beta_3+ \beta_4 $
represents waves splitting arising from the reaction at other position $x_\beta$.
First of all, consider the situation $i=3$.
It follows from Lemma \ref{lem:4.2} that
\begin{eqnarray}\label{eq:4.24}
\begin{split}
\Delta \mathcal{L}(t_k) \big|_{x=x_\alpha}
&= (\Delta { \mathcal{L}_{\textsc{e}}}(t_k) + \Delta \mathcal{L}_{Y}(t_k)) \Big|_{x=x_\alpha}\\[5pt]
&= \sum_{i\neq3} |\alpha_i|  + \theta_3 (|\alpha_3| - |\alpha_3 ^*|)\\[5pt]
&\leq -\frac{1}{2}\theta_3  \underline{\phi} \varepsilon  |\alpha_3^*|,
\end{split}
\end{eqnarray}
provided $\theta_3>0$ is sufficiently large.
Notice that
\begin{eqnarray*}
\begin{split}
&\sum_{i\neq 3} \Big(|\alpha_i| \sum_{\mathcal{A}(\alpha_i)} |\beta| \Big)\leq  \mathcal{O}(1) \mathcal{G}(t_{k}-0)|\alpha^*_3| \varepsilon,
\end{split}
\end{eqnarray*}
and
\begin{eqnarray*}
\begin{split}
&\big(|\alpha_3| -|\alpha_3 ^*|\big) \sum_{\mathcal{A}(\alpha_3)} |\beta| +  |\alpha^*_3|\Big(\sum_{\mathcal{A}(\alpha_3)} |\beta| - \sum_{\mathcal{A}(\alpha_3^*)} |\beta^*| \Big)\\[5pt]
&\quad \leq \mathcal{O}(1) |\alpha^*_3| \Big( \sum_{{x_\beta <\chi^{\varepsilon}_{\rm{s}}(t_{k})} \atop
{x_\beta \neq x_\alpha} }
\big(|\beta_3^*| + |\beta_j^*| Y_\beta^{*}\big)\varepsilon + \Upsilon^{\varepsilon,\rm{s}}_k \Big).
\end{split}
\end{eqnarray*}
Subsequently, we have that
\begin{align*}
\Delta \mathcal{Q}_w (t_k) \big|_{x=x_\alpha}
&=\sum_{i\neq3}\Big(|\alpha_i| \sum_{\mathcal{A}(\alpha_i)} |\beta|\Big) + |\alpha_3| \sum_{\mathcal{A}(\alpha_3)} |\beta|
- |\alpha_3 ^*|\sum_{\mathcal{A}(\alpha_3^*)} |\beta^*|\\[5pt]
& \leq   \mathcal{O}(1) \mathcal{G}(t_{k}-0)|\alpha_3^*| \varepsilon  +\mathcal{O}(1)|\alpha_3 ^*|
\Big( \sum_{
{x_\beta <\chi^{\varepsilon}_{\rm{s}}(t_{k})  \atop
x_\beta \neq x_\alpha}}
	\big(|\beta_3^*|+|\beta^*_j| Y_\beta^{*}\big)\varepsilon + \Upsilon^{\varepsilon,\rm{s}}_k \Big).
\end{align*}
On the other hand, by Lemma \ref{lem:4.2} (i) and \eqref{eq:4.3}, we also get
\begin{equation*}
\big(\Delta \mathcal{Q}_{\mathfrak{p}} + \Delta \mathcal{Q}_{\rm{s}}\big)(t_k)  \big|_{x=x_\alpha}=\theta_1 |\alpha_1| +   \theta_4 |\alpha_4 |
	\leq  \mathcal{O}(1)\theta_1 |\alpha_3^*| \varepsilon.
\end{equation*}
So it follows that
\begin{eqnarray}\label{eq:4.25}
\begin{split}
	\Delta \mathcal{Q} (t_k) \big|_{x=x_\alpha}& = (\Delta \mathcal{Q}_w + \Delta \mathcal{Q}_\mathfrak{p} + \Delta \mathcal{Q}_{\rm{s}})(t_k) \Big|_{x=x_\alpha}\\[5pt]
	&  \leq \mathcal{O}(1)\big(\mathcal{G}(t_{k}-0) + \theta_1\big)  |\alpha_3^*| \varepsilon \\[5pt]
&\quad + \mathcal{O}(1)|\alpha_3 ^*|
\Big( \sum_{
{x_\beta <\chi^{\varepsilon}_{\rm{s}}(t_{k})  \atop
x_\beta \neq x_\alpha}}
	\big(|\beta_3^*|+|\beta^*_j| Y_\beta^{*}\big)\varepsilon + \Upsilon^{\varepsilon,\rm{s}}_k \Big).
\end{split}
\end{eqnarray}
By \eqref{eq:4.24} and \eqref{eq:4.25}, we thus obtain the local estimate
\begin{eqnarray}\label{eq:4.26}
\begin{split}
\Delta \mathcal{G}(t_k)  \big|_{x=x_\alpha}& = ( \Delta\mathcal{L}+ K\Delta \mathcal{Q} )(t_k) \big|_{x=x_\alpha}\\[5pt]
		& \leq \Big(-\frac{1}{2}\theta_3  \underline{\phi}  + \mathcal{O}(1) (\mathcal{G}(t_{k}-0) + \theta_1)K \Big)  |\alpha_3 ^*| \varepsilon \\[5pt]
		&\quad +  \mathcal{O}(1)  K|\alpha_3 ^*|
\sum_{
{x_\beta <\chi^{\varepsilon}_{\rm{s}}(t_{k})  \atop
x_\beta \neq x_\alpha}} \big(|\beta_3^*| \varepsilon +  |\beta_j^*| Y_\beta^{*} \varepsilon \big)
       +  \mathcal{O}(1) K|\alpha_3 ^*|  \Upsilon^{\varepsilon,\rm{s}}_k.
\end{split}
\end{eqnarray}

Next, we investigate the situation $i\neq 3$, and take $i=1$ for instance.
It is deduced from Lemma \ref{lem:4.2} (ii) that
\begin{eqnarray*}
\begin{split}
\Delta \mathcal{L} (t_k) \big|_{x=x_\alpha}  = \sum_{i\neq3} |\alpha_i|  + \theta_3 |\alpha_3| - |\alpha_1 ^*|
\leq   \mathcal{O}(1) \big(1+\theta_3\big)|\alpha_1^*| Y^{*}_{\alpha}\varepsilon,
\end{split}
\end{eqnarray*}
and
\begin{eqnarray*}
\begin{split}
\Delta \mathcal{Q}_w (t_k) \big|_{x=x_\alpha}
& =\Big(|\alpha_1| \sum_{\mathcal{A}(\alpha_1)} |\beta| -  |\alpha_1 ^*|\sum_{\mathcal{A}(\alpha_1^*)} |\beta^*|\Big)  +  \sum_{i\neq1}\Big(|\alpha_i| \sum_{\mathcal{A}(\alpha_i)} |\beta|\Big)\\[5pt]
& \leq-|\alpha_1 ^*|\underline{\phi} \varepsilon  \sum_{ x_\beta <x_\alpha }  |\beta_3^*| +  \mathcal{O}(1) \mathcal{G}(t_{k}-0)|\alpha_1^*| Y_\alpha^{*} \varepsilon  \\[5pt]
	&\quad + \mathcal{O}(1)  |\alpha_1 ^*|\Big(
\sum_{
{x_\beta <\chi^{\varepsilon}_{\rm{s}}(t_{k})  \atop
x_\beta \neq x_\alpha}}
	\big(|\beta_3^*| \varepsilon +  |\beta_j^*| Y_\beta^{*} \varepsilon\big) + \Upsilon^{\varepsilon,\rm{s}}_k \Big), \\[5pt]
\big(\Delta \mathcal{Q}_{\mathfrak{p}} + \Delta \mathcal{Q}_{\rm{s}} \big)(t_k) \Big|_{x=x_\alpha} 	
	& = \theta_1\big(|\alpha_1|  -  |\alpha_1 ^*|\big) + \theta_4 |\alpha_4|
\leq   \mathcal{O}(1)\theta_1   |\alpha_1^*| Y^{*}_{\alpha}\varepsilon.
\end{split}
\end{eqnarray*}
We take advantage of a significant spatial estimate
\begin{eqnarray*}
\sum_{ x_\beta <x_\alpha }  |\beta_3^*|  \geq Y_\alpha^{*} - { Y(t_k-0, \chi^{\varepsilon}_{\mathfrak{p}}(t_k)+0)}.
\end{eqnarray*}
Substituting this into the estimate $\Delta \mathcal{Q}_w(t_k) \big|_{x=x_\alpha}$, we derive that
\begin{align*}
	\Delta \mathcal{Q} (t_k) \big|_{x=x_\alpha}
	& = \big(\Delta \mathcal{Q}_w + \Delta \mathcal{Q}_{\mathfrak{p}} + \Delta \mathcal{Q}_{\rm{s}}\big) (t_k) \big|_{x=x_\alpha}\\[5pt]
	& \leq -|\alpha_1 ^*|\underline{\phi}\varepsilon \Big(Y_\alpha^{*} - { Y(t_k-0, \chi^{\varepsilon}_{\mathfrak{p}}(t_k)+0)}\Big)
+\mathcal{O}(1) \mathcal{G}(t_{k}-0)|\alpha_1^*| Y_\alpha^{*}\varepsilon\\[5pt]
	&\quad + \mathcal{O}(1) |\alpha_1 ^*| \Big(
\sum_{
{x_\beta <\chi^{\varepsilon}_{\rm{s}}(t_{k})  \atop
x_\beta \neq x_\alpha}}\big(|\beta_3^*|+|\beta_j^*| Y_\beta^{*}\big)\varepsilon + \Upsilon^{\varepsilon,\rm{s}}_k \Big)
+ \mathcal{O}(1)\theta_1|\alpha_1^*| Y_\alpha^{*}\varepsilon \\[5pt]
	&  \leq \Big( \mathcal{O}(1)\big(\mathcal{G}(t_{k}-0) + \theta_1\big)  - \underline{\phi} \Big)  |\alpha_1^*| Y_\alpha^{*} \varepsilon
+ \mathcal{O}(1)  |\alpha_1 ^*|
\sum_{
{x_\beta <\chi^{\varepsilon}_{\rm{s}}(t_{k})  \atop
x_\beta \neq x_\alpha}}\big(|\beta_3^*|+|\beta_j^*| Y_\beta^{*} \big)\varepsilon
	\\[5pt]
	& \quad + |\alpha_1 ^*|\Upsilon^{\varepsilon,\mathfrak{p}}_k+ \mathcal{O}(1)|\alpha_1 ^*|  \Upsilon^{\varepsilon, \rm{s}}_k.
\end{align*}
Combining the estimates on $\Delta \mathcal{L}(t_k) \big|_{x=x_\alpha}$ and $\Delta \mathcal{Q} (t_k) \big|_{x=x_\alpha}$, we thus obtain
\begin{eqnarray}\label{eq:4.27}
\begin{split}
\Delta \mathcal{G} (t_k) \big|_{x=x_\alpha}
&\leq \Big(\mathcal{O}(1)\big(1+\theta_3\big) + \mathcal{O}(1)\big (\mathcal{G}(t_{k}-0) + \theta_1\big)K-K\underline{\phi}\Big) |\alpha_1^*| Y_\alpha^{*} \varepsilon \\[5pt]
& \quad\ +\mathcal{O}(1)K|\alpha_1 ^*|
\sum_{
{x_\beta <\chi^{\varepsilon}_{\rm{s}}(t_{k})  \atop
x_\beta \neq x_\alpha}} \big(|\beta_3^*|+|\beta_j^*| Y_\beta^{*} \big)\varepsilon \\[5pt]
& \quad\ + K|\alpha_1 ^*|\Upsilon^{\varepsilon, \mathfrak{p}}_k+ \mathcal{O}(1)K|\alpha_1 ^*|   \Upsilon^{\varepsilon, \rm{s}}_k.
\end{split}
\end{eqnarray}

Likewise, when wave $\alpha_i^*$ ($1<i\neq3$) hits the grid line $t=t_k$, we observe that
\begin{align*}
& \sum_{ x_\beta < x_\alpha }  |\beta_3^*|  \geq Y_\alpha^{*} -  { Y(t_k-0, \chi^{\varepsilon}_{\mathfrak{p}}(t_k)+0)}  \quad  \mbox{if}\quad i=2, \\[5pt]
& \sum_{ x_\beta > x_\alpha }  |\beta_3^*|  \geq Y_\alpha^{*} - { Y(t_k-0, \chi^{\varepsilon}_{\rm{s}}(t_k)-0)} \quad  \mbox{if}\quad i>3,
\end{align*}
and then figure out
\begin{equation}\label{eq:4.28}
\begin{aligned}
\Delta \mathcal{G} (t_k) \big|_{x=x_\alpha}& \leq \Big(\mathcal{O}(1) (1+\theta_3)+ \mathcal{O}(1)K\big(\mathcal{G}(t_{k}-0) + \theta_1\big) - K\underline{\phi}\Big)|\alpha_i^*| Y_\alpha^{*} \varepsilon \\[5pt]
& \qquad +\mathcal{O}(1)K|\alpha_i ^*|
\sum_{
{x_\beta <\chi^{\varepsilon}_{\rm{s}}(t_{k})  \atop
x_\beta \neq x_\alpha}}  \big(|\beta_3^*|\varepsilon +  |\beta_j^*| Y_\beta^{*} \varepsilon \big)\\[5pt]		
& \qquad + K|\alpha_i ^*|\Upsilon^{\varepsilon, \mathfrak{p}}_k+ \mathcal{O}(1)K|\alpha_i ^*|\Upsilon^{\varepsilon, \rm{s}}_k.
\end{aligned}
\end{equation}

\textbf{Case 4.2.3}. \emph{Combustion process at $\chi^{\varepsilon}_{\rm{s}}(t_{k})$.}
Suppose the strong shock $\alpha^{\rm{s},*}_{4}$ hits the grid line $t=t_k$. Then
the local reaction gives rise to
$$\alpha^{\rm{s},*}_{4} \mapsto  \alpha_1+\alpha_2+ \alpha_3+\alpha^{\rm{s}}_{4}\qquad
\text{at } \chi^{\varepsilon}_{\rm{s}}(t_{k}).$$
See Fig.\ref{fig10}.
According to the relations
\begin{eqnarray*}
|\alpha_1| + |\alpha_2| + |\alpha^{\rm{s}}_{4}- \alpha^{\rm{s},*}_{4}|= \mathcal{O}(1)\Upsilon^{\varepsilon, \rm{s}}_k,\qquad |\alpha_3| = \Upsilon^{\varepsilon, \rm{s}}_k.
\end{eqnarray*}
we deduce that
\begin{eqnarray}\label{eq:4.34}
\begin{split}
  \Delta \mathcal{L}(t_k) \big|_{x=\chi^{\varepsilon}_{\rm{s}}(t_{k})}
 & =|\alpha_1| + |\alpha_2| + \theta_3 |\alpha_3|
+|\alpha^{\rm{s}}_{4}- \bar{\alpha}^{\rm{s}}_{4}|- | \alpha^{\rm{s},*}_{4}-\bar{\alpha}^{\rm{s}}_4|\\[5pt]
 & = \big(\mathcal{O}(1)+\theta_3\big)\Upsilon^{\varepsilon,\rm{s}}_k,
\end{split}
\end{eqnarray}
and
\begin{eqnarray*}
\begin{split}
&\Delta \mathcal{Q}_w (t_k) \big|_{x=\chi^{\varepsilon}_{\rm{s}}(t_{k})} = \sum_{i=1}^3\Big(|\alpha_i| \sum_{\mathcal{A}(\alpha_i)} |\beta|\Big)  \leq \mathcal{O}(1) \mathcal{G}(t_{k}-0)\Upsilon^{\varepsilon,\rm{s}}_k,\\[5pt]
&\Delta \mathcal{Q}_{\mathfrak{p}}(t_k)\big|_{x=\chi^{\varepsilon}_{\rm{s}}(t_{k})} =\theta_1 |\alpha_1|=\mathcal{O}(1) \theta_1 \Upsilon^{\varepsilon,\rm{s}}_k,\\[5pt]
& \Delta \mathcal{Q}_{\rm{s}} (t_k) \big|_{x=\chi^{\varepsilon}_{\rm{s}}(t_{k})} = \theta_4 \sum_{\mathcal{A}( \alpha^{\rm{s}}_4)} |\beta_4|+   \theta_5 \sum_{x_\beta <\chi^{\varepsilon}_{\rm{s}}(t_{k})}|\beta_5| -\theta_4 \sum_{\mathcal{A}( \alpha^{\rm{s},*}_4)} |\beta_4^*|-  \theta_5 \sum_{x_\beta <\chi^{\varepsilon}_{\rm{s}}(t_{k})}|\beta_5^*|\\[5pt]
&\qquad\qquad\qquad\quad\ \leq  \mathcal{O}(1) \theta_1 \sum_{x_\beta <\chi^{\varepsilon}_{\rm{s}}(t_{k})} \big(|\beta_3^*| +  |\beta_j^*| Y_\beta^{*} \big)\varepsilon .
\end{split}
\end{eqnarray*}
Thus the previous estimates yield
\begin{eqnarray}\label{eq:4.35}
\begin{split}
\Delta\mathcal{Q}(t_k) \big|_{x=\chi^{\varepsilon}_{\rm{s}}(t_{k})}
&= \big(\Delta\mathcal{Q}_w + \Delta\mathcal{Q}_{\mathfrak{p}} +\Delta\mathcal{Q}_{\rm{s}} \big)(t_k)  \big|_{x=\chi^{\varepsilon}_{\rm{s}}(t_{k})}\\[5pt]
&\leq \mathcal{O}(1) \big( \mathcal{G}(t_{k}-0)+ \theta_1\big) \Upsilon^{\varepsilon, s}_k
+\mathcal{O}(1)\theta_1 \sum_{x_\beta <\chi^{\varepsilon}_{\rm{s}}(t_{k})} \big(|\beta_3^*| +  |\beta_j^*| Y^{*}_\beta\big) \varepsilon.
\end{split}
\end{eqnarray}
On the other hand, we obviously have
\begin{eqnarray}\label{eq:4.36}
\Delta \mathcal{L}_{\textsc{b}}(t_k)\big|_{x=\chi^{\varepsilon}_{\rm{s}}(t_{k})}=\Delta \mathcal{L}_{\textsc{b}\mathfrak{p}}(t_k)\big|_{x=\chi^{\varepsilon}_{\rm{s}}(t_{k})} = - \Upsilon^{\varepsilon,\rm{s}}_k.
\end{eqnarray}
Therefore, combining the estimates \eqref{eq:4.34}-\eqref{eq:4.36} altogether, we find that
\begin{eqnarray}\label{eq:4.29}
\begin{split}
\Delta \mathcal{G} (t_k) \big|_{x=\chi^{\varepsilon}_{\rm{s}}(t_{k})}
&\leq \Big( \mathcal{O}(1)+ \theta_3 + \mathcal{O}(1)K\big(\mathcal{G}(t_{k}-0) +\theta_1\big) -K \Big) \Upsilon^{\varepsilon, s}_k \\[5pt]
& \quad + \mathcal{O}(1)K\theta_1 \sum_{x_\beta <\chi^{\varepsilon}_{\rm{s}}(t_{k})} \big(|\beta_3^*|+  |\beta_j^*| Y_\beta^{*}\big)\varepsilon.
\end{split}
\end{eqnarray}

Finally, taking summation of the pointwise estimates \eqref{eq:4.23},
\eqref{eq:4.26}-\eqref{eq:4.28} and \eqref{eq:4.29} on $\Delta \mathcal{G}(t_k)$ over region $[\chi^{\varepsilon}_{\mathfrak{p}}(t_{k}), \chi^{\varepsilon}_{\rm{s}}(t_{k})]$, and using the symmetry of $\Delta \mathcal{G}(t_k)$ with respect to $x_\alpha$ and $x_\beta$, we conclude that
\begin{eqnarray}\label{eq:4.30}
\begin{split}
\Delta \mathcal{G}(t_k)
& = \Delta \mathcal{G} (t_k)  \big|_{x \leq \chi^{\varepsilon}_{\rm{s}}(t_{k})} \\[5pt]
& \leq K\Big(\frac{1}{2} \mathcal{G}(t_{k}-0) -1 \Big) \Upsilon^{\varepsilon,\mathfrak{p}}_k
+ \Big(\mathcal{O}(1) + \theta_3 + \mathcal{O}(1)K\big(\mathcal{G}(t_{k}-0) + \theta_1\big)-K\Big)\Upsilon^{\varepsilon, s}_k\\[5pt]
&\quad+ \Big( -\frac{1}{2} \theta_3 \underline{\phi} +  \mathcal{O}(1)K\big(\mathcal{G}(t_{k}-0) + \theta_1\big) \Big)\sum_{x_\alpha <\chi^{\varepsilon}_{\rm{s}}(t_{k})}|\alpha^*_3| \varepsilon \\[5pt]
&\quad + \Big(\mathcal{O}(1)\big(1+\theta_3\big) + \mathcal{O}(1)K\big(\mathcal{G}(t_{k}-0) + \theta_1\big)- \frac{1}{2} K\underline{\phi}\Big)
\sum_{
{i\neq 3  \atop
x_\alpha <\chi^{\varepsilon}_{\rm{s}}(t_{k})}}
|\alpha_i^*| Y_\alpha^{*} \varepsilon  \\[5pt]
& \leq -\frac{1}{2}K\Upsilon^{\varepsilon, \mathfrak{p}}_k  - \frac{1}{2}K\Upsilon^{\varepsilon,\rm{s}} _k-\frac{1}{4} \theta_3 \underline{\phi}\sum_{x_\alpha <\chi^{\varepsilon}_{\rm{s}}(t_{k})}|\alpha_3^*| \varepsilon
 -\frac{1}{4} K\underline{\phi}
 \sum_{
{i\neq 3   \atop
x_\alpha <\chi^{\varepsilon}_{\rm{s}}(t_{k})}}
 |\alpha_i^*| Y_\alpha^{*} \varepsilon \\[5pt]
&<0,
\end{split}
\end{eqnarray}
provided that $\tilde{\delta}$ and $\theta_1$ are sufficiently small, while $K$ is suitably large. This gives the proof of \eqref{eq:4.22} under  condition $(\textsc{C3})$.

In the light of previous argument on partial reaction, one can easily verify the situation of completely ignited flow under condition $(\textsc{C2})$, \emph{i.e.}, $T^{\varepsilon}(t,x) > T_{\rm{i}}$ for $x\geq\chi^{\varepsilon}_{\mathfrak{p}}(t)$. Since $\phi(T^{\varepsilon})>\underline{\phi}>0$ everywhere, the argument for $\chi^{\varepsilon}_{\mathfrak{p}}(t_{k})< x_\alpha<\chi^{\varepsilon}_{\rm{s}}(t_{k})$
in Case 4.2.2 remains valid for $ x_\alpha >\chi^{\varepsilon}_{\rm{s}}(t_{k})$.
Under condition $(\textsc{C2})$, we have the estimates
\begin{eqnarray*}
\sum_{i=1}^3|\alpha_i| =\mathcal{O}(1)|\alpha^{\rm{s},*}_4| Y^{*}\varepsilon =\mathcal{O}(1) \Upsilon^{\varepsilon, \rm{s}}_k \qquad \mbox{at}\ x=\chi^{\varepsilon}_{\rm{s}}(t_{k}).
\end{eqnarray*}
Then, similar to the proof of \eqref{eq:4.30}, we see that
\begin{eqnarray}\label{eq:4.31}
\begin{split}
\Delta \mathcal{G}(t_k)
&=\Delta \mathcal{G}(t_k)\big|_{x \geq \chi^{\varepsilon}_{\mathfrak{p}}(t_{k})} \\[5pt]
&\leq \Big( -\frac{1}{2} \theta_3 \underline{\phi} +  \mathcal{O}(1)K\big(\mathcal{G}(t_{k}-0) + \theta_1\big)\Big) \sum|\alpha_3^*|\varepsilon \\[5pt]
&\quad + \Big(\mathcal{O}(1)\big(1+\theta_3\big) + \mathcal{O}(1)K\big(\mathcal{G}(t_{k}-0)+\theta_1\big) - \frac{1}{2}K\underline{\phi}\Big)
\sum_{
{i\neq 3  \atop
x_{\alpha} \neq \chi^{\varepsilon}_{\rm{s}}(t_{k})}}
|\alpha_i^*| Y_\alpha^{*} \varepsilon \\[5pt]
&\quad +\Big(\mathcal{O}(1)\big(1+\theta_3\big) +\mathcal{O}(1)K\big(\mathcal{G}(t_{k}-0) +\theta_1\big)-K\Big) \Upsilon^{\varepsilon,\rm{s}}_k \\[5pt]
&\quad + \Big( \frac{1}{2}\mathcal{G}(t_{k}-0) -1 \Big)K\Upsilon^{\varepsilon,\mathfrak{p}}_k \\
&\leq -\frac{1}{2} K\Upsilon^{\varepsilon,\mathfrak{p}}_k -\frac{1}{2}K\Upsilon^{\varepsilon,\rm{s}}_k-\frac{1}{4} \theta_3 \underline{\phi} \sum |\alpha_3^*| \varepsilon
-\frac{1}{4} K\underline{\phi}
\sum_{
{i\neq 3  \atop
x_{\alpha} \neq \chi^{\varepsilon}_{\rm{s}}(t_{k})}}
|\alpha_i^*| Y_\alpha^{*} \varepsilon \\[5pt]
&<0,
\end{split}
\end{eqnarray}
This completes the proof of \eqref{eq:4.22} under condition $(\textsc{C2})$.
\end{proof}

By Lemma \ref{lem:4.3}, we can also obtain the following estimates on the approximate solution $U^{\varepsilon}$ for reacting flow.
\begin{proposition}\label{prop:4.2}
Under the assumptions $(\mathbf{A1})$-$(\mathbf{A2})$, there exist constants $C_3$, $C_4>0$ which depend only on $\bar{U}_{\mathfrak{b}}$, such that
if the condition $(\rm{C2})$ or $(\rm{C3})$ holds for sufficiently small $\epsilon$, then the fractional-step wave front tracking scheme yields a global in time approximate solution $U^{\varepsilon}$
to the $(\textit{IBVP})$ satisfying the estimates
\begin{eqnarray}\label{eq:4.32}
\begin{split}
&\sup_{x\in(\chi^{\varepsilon}_{\mathfrak{p}}(t),\chi^{\varepsilon}_{\rm{s}}(t))}|U^{\varepsilon}(t,\cdot)-\bar{U}_{\mathfrak{b},\rm{l}}|
+\sup_{x\in(\chi^{\varepsilon}_{\rm{s}}(t),+\infty)}|U^{\varepsilon}(t,\cdot)-\bar{U}_{\mathfrak{b},\rm{r}}|
+\sup_{t>0}|\dot{\chi}^\varepsilon_{\rm{s}}(t)-\bar{\rm{s}}_{\mathfrak{b}}|\\[5pt]
&\quad\ +T.V.\{U^{\varepsilon}(t,\cdot);(\chi^{\varepsilon}_{\mathfrak{p}}(t),\chi^{\varepsilon}_{\rm{s}}(t))\}+T.V.\{U^{\varepsilon}(t,\cdot);(\chi^{\varepsilon}_{\rm{s}}(t),+\infty)\}<C_3\epsilon,
\end{split}
\end{eqnarray}
for $t>0$, and
\begin{eqnarray}\label{eq:4.33}
\|U^{\varepsilon}(t,\cdot)-U^{\varepsilon}( \tilde{t} ,\cdot)\|_{L^1(\mathbb{R})}+|\chi^{\varepsilon}_{\rm{s}}(t)-\chi^{\varepsilon}_{\rm{s}}(\tilde{t} )|
\leq C_4\big( t -\tilde{t}+\varepsilon\big),
\end{eqnarray}
for any $ t >\tilde{t}>0$.
\end{proposition}

\begin{proof}
The argument of \eqref{eq:4.32} is completely analogous to that of \eqref{eq:4.16} under condition (C1).
It remains to verify the estimate \eqref{eq:4.33} for $U^{\varepsilon}$.
According to \eqref{eq:4.17} in Proposition \ref{prop:4.1}, we have
\begin{equation*}
\|{U}^\varepsilon(t,\cdot)- {U}^\varepsilon(\tilde{t},\cdot)\|_{L^1(\mathbb{R})} \leq \mathcal{O}(1)|t -\tilde{t}|,\quad \mbox{for}\ \mbox{any}\quad \tilde{t}, t \in [t_{k-1}, t_k).
\end{equation*}
On the other hand, by \eqref{reaction step}\eqref{eq:4.18}, we obtain that
\begin{eqnarray*}
\begin{split}
\|{U}^\varepsilon(t_k+0,\cdot)- {U}^\varepsilon(t_k-0,\cdot)\|_{L^1(\mathbb{R})}
&\leq \mathcal{O}(1)\varepsilon \int^{+\infty}_{\chi^{\varepsilon}_{\mathfrak{p}} (t_k)} Y^{\varepsilon}(t_{k}-0,x) \phi(T^{\varepsilon}(t_{k}-0,x))dx\\[5pt]
&\leq \mathcal{O}(1)\varepsilon \|Y^{\varepsilon}(t_{k}-0,\cdot)\|_{L^1(\mathbb{R})}\\[5pt]
&\leq \mathcal{O}(1) \big(\|Y_0\|_{L^1(\mathbb{R})} +1\big)  \varepsilon.
\end{split}
\end{eqnarray*}
Then, combining the above two estimates altogether, we can derive the approximate continuity of  $\|{U}^\varepsilon(t,\cdot)\|_{L^1(\mathbb{R})}$ with respect to $t$, \emph{i.e.},
\begin{eqnarray*}
\|{U}^\varepsilon( t )-{U}^\varepsilon(\tilde{t})\|_{L^1(\mathbb{R})} \leq \mathcal{O}(1) \big( t -\tilde{t}+ \lceil ( t -\tilde{t})/\varepsilon \rceil \cdot \varepsilon\big)\leq \mathcal{O}(1) \big( t -\tilde{t}+ \varepsilon\big),
\end{eqnarray*}
for any $ t >\tilde{t}> 0$.
\end{proof}

\begin{remark}\label{rem:4.1}
Review the existence argument by means of Glimm-type functional in this section. Basically, it is allowed that the initial positions of piston and large shock are separated,
namely $\chi^{\varepsilon}_{\mathfrak{p}}(0)< \chi^{\varepsilon}_{\rm{s}}(0)$. Therefore, the initial restriction on
$T.V.\{U_{0}(\cdot)-\bar{U}_{\mathfrak{b},\rm{r}}; \mathbb{R}_{+}\}$
in Theorem \ref{thm:1.1} can be relaxed to a smallness condition of
\begin{eqnarray*}
T.V.\{U_0(\cdot) - \bar{U}_{\mathfrak{b},\rm{l}}; (\chi^{\varepsilon}_{\mathfrak{p}}(0), \chi^{\varepsilon}_{\rm{s}}(0))\}
+ T.V.\{U_0(\cdot)-\bar{U}_{\mathfrak{b},\rm{r}}; (\chi^{\varepsilon}_{\rm{s}}(0), +\infty) \} .
\end{eqnarray*}
The global existence conclusion in Theorem \ref{thm:1.1} remains true. This observation will be used to analyze the trajectories of reacting flow in section 6.
\end{remark}

\emph{Proof of Theorem \ref{thm:1.1} under condition $(\rm{C2})$ or $(\rm{C3})$}.
Similar to the proof under condition $(\textrm{C1})$, we follow Proposition \ref{prop:4.2}
and apply Helly's compactness theorem to get the global existence of combustion solution $U$ to (\emph{IBVP})  under condition $(\textrm{C2})$ or $(\textrm{C3})$.
\hfill$\Box$


\subsection{Some further properties of the entropy solution}

Given initial-boundary data, one can derive the variation of the flow on any space-like curve by Glimm functional $\mathcal{G}(t) $.
Likewise, one can obtain the information of the flow on any time-like curve. For this purpose we devise a new geometric object
called quasi-characteristic, which is a generalization of characteristics in the classic theory of hyperbolic partial differential equations.
Furthermore, it becomes an essential tool to establish $L^1$ stability of weak solutions in sections 5.
\begin{definition}\label{def:5.1}
Assume that a generic hyperbolic system
\begin{equation}\label{eq:5.1}
\partial_t E({U}) + \partial_x F({U})=G({U})
\end{equation}
has $n$ eigenvalues such that
\begin{eqnarray*}
\lambda_{1}(U)  \leq  \cdots  \leq \lambda_{i}(U)  \leq \cdots \leq \lambda_{n}(U).
\end{eqnarray*}
Let ${U}(t,x)$ be a solution to system \eqref{eq:5.1}, which takes values in domain $\mathcal{D}_*$.
We say a Lipschitz continuous curve $x=\chi(t)$ is an $i$\emph{-quasi-characteristic} associated with ${U}(t,x)$,
if there exists a constant $d>0$ such that
\begin{eqnarray*}
\sup_{U\in\mathcal{D}_*} \lambda_{j}(U) +d<\dot{\chi}(t) <\inf_{U\in\mathcal{D}_*}\lambda_{k}(U) -d, \qquad \forall\ t>0,
\end{eqnarray*}
for any indices $j$, $k$ satisfying {$j<i<k$} and $\lambda_{j}(U)< \lambda_{i}(U) <\lambda_{k}(U)$.
\end{definition}

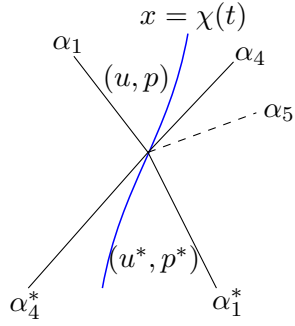
\begin{figure}[ht]
\begin{center}
\begin{tikzpicture}[scale=0.75]
\draw [blue][line width=0.02cm](-1.0,-1.5)to[out=80,in=-100](0.5,3.0);
\draw [thin](-2.3,-1.5)to(-0.2,0.9)to(1.3,2.5);
\draw [thin][dashed](-0.2,0.9)to(1.7,1.6);
\draw [thin](1.0,-1.5)to(-0.2,0.9)to(-1.5,2.6);
\node at (-1.6,2.8) {$\alpha_1$};
\node at (0.6,3.3) {$x=\chi(t)$};
\node at (1.6,2.6) {$\alpha_4$};
\node at (2.1,1.6) {$\alpha_5$};
\node at (-2.35,-1.8) {$\alpha^{*}_4$};
\node at (1.2,-1.75) {$\alpha^{*}_1$};
\node at (-0.1,-1.0) {$(u^{*},p^{*})$};
\node at (-0.4,2.2) {$(u,p)$};
\end{tikzpicture}
\end{center}
\caption{Quasi-characteristic curve }\label{fig11}
\end{figure}

As pointed out previously, the class of quasi-characteristics includes not only classic characteristics, but
their small perturbations in Lipschitz sense. Now we suppose $\chi(t)$ is a 2-quasi-characteristic  associated with (exact or approximate) solution
${U}$ to (\emph{IBVP}). Moreover, it is required that $\chi(t)$ satisfies
\begin{equation}\label{eq:5.2}
\chi_{\mathfrak{p}}(t) \leq \chi(t) < \chi_{\rm{s}} (t) \qquad \mbox{for any }\quad  t> 0,
\end{equation}
where $\chi_{\mathfrak{p}}(t)$ is the piston path and $\chi_{\rm{s}} (t)$ is the large shock front path.

Set $\chi(I)=\{(t,\chi(t)): x=\chi(t), t\in I\}$. Let notation $T.V.\{f; \chi(I)\}$ stand for the variation of $f$ along the Lipschitz continuous curve $x=\chi(t)$ over interval $I$.
Then we introduce the following functionals which are used to investigate the variation of $(u,p)$ on quasi-characteristic $x=\chi(t)$. Set
\begin{eqnarray*}\label{eq:5.3}
\begin{split}
\mathcal{L}^\chi_{\textsc{w}}(t) &= \sum_{
                                \begin{subarray}{c}
                                 i\neq 2,3\\
                                x_\alpha <\chi(t)
                                \end{subarray}}
                                l_i|\alpha_i| + \sum_{
                                \begin{subarray}{c}
                                 i\neq 2,3\\
                                x_\alpha >\chi(t)
                                \end{subarray}}r_i|\alpha_i|  + \theta_3 \sum |\alpha_{3}|,
\end{split}
\end{eqnarray*}
and
\begin{eqnarray*}\label{eq:5.4}
\begin{split}
\mathcal{L}^\chi_{\textsc{b}}(t) &= T.V.\{(u,p); \chi(I_{t})\} + a \sum_{ t_k\geq t } \Upsilon^{\chi+}_k + a\sum_{t_k\geq t} \Upsilon^{\chi-}_k, \quad
I_{t}=\text{interval }(0,t),
\end{split}
\end{eqnarray*}
where the transient consumption of reactant on $\chi(t)$ are
\begin{eqnarray*}\label{eq:5.5}
\begin{split}
&\Upsilon^{\chi+}_k=Y(t_k-0, \chi(t_k)+0) \phi(T((t_k-0, \chi(t_k)+0)))\varepsilon,\\[5pt]
&\Upsilon^{\chi-}_k = Y(t_k-0, \chi(t_k)-0) \phi(T((t_k-0, \chi(t_k)-0))) \varepsilon.
\end{split}
\end{eqnarray*}
Here the weights satisfy
\begin{equation}\label{eq:5.6}
r_1>l_1>l_4>r_4>0, \quad l_5>r_5>0.
\end{equation}

Define the redistributed Glimm-type functional for ${U}$ by
\begin{eqnarray*}\label{eq:5.7}
\mathcal{G}^{\chi}(t)= \mathcal{L}^\chi_{\textsc{w}}(t) + K\mathcal{Q}(t) + K \mathcal{L}_{\textsc{b}}(t) + \hat{\epsilon} \mathcal{L}^\chi_{\textsc{b}}(t)
\end{eqnarray*}
with $0<\hat{\epsilon} \ll 1$.
The assumption (\ref{eq:5.6}) on coefficients readily gives the fact that functional $\mathcal{G}^{\chi}(t)$ is decreasing when any single $\alpha_i\ (i\neq 2,3)$ goes across the curve $x=\chi(t)$.
Next, it suffices to prove the monotonicity of $\mathcal{G}^{\chi}(t)$ at time $t$ of waves interaction and of reaction step. Hence we show the following lemma that covers the situations of non-reacting and reacting flows.

\begin{lemma}\label{lem:5.1}
Redistributed Glimm-type functional $\mathcal{G}^{\chi}(t)$ is decreasing in $t$.
\end{lemma}

\begin{proof}
Since $\mathcal{G}^{\chi}(t)$ and $\mathcal{G} (t)$ have the same term $\mathcal{Q}(t)$, the argument here is highly analogous to that in sections 4.2-4.3.
So we only need to show some key estimates to prove the lemma. \\
(i) \emph{Non-reacting flow under condition $\textsc{(C1)}$}.

\textbf{Case 4.3.1}. \emph{Reflection $\alpha_1^*   \mapsto   \alpha_4$ at $\chi_\mathfrak{p}(t)$.} Recall that $\alpha_4 = \alpha_1^* + \mathcal{O}(1) |\alpha_1^*|^2$,
 then by the coefficients restriction \eqref{eq:5.6}, we have
\begin{eqnarray*}
\begin{split}
\Delta\mathcal{L}^\chi_{\textsc{w}}(t)
&= l_4 |\alpha_4| - l_1 |\alpha_1^*| =  l_4 \big(|\alpha_4| -  |\alpha_1^*|\big) + (l_4-l_1)|\alpha_1^*| \\[5pt]
&\leq \mathcal{O}(1) |\alpha_1^*|^2 \\[5pt]
&\leq \mathcal{O}(1) \mathcal{G} (t-0) |\alpha_1^*|.
\end{split}
\end{eqnarray*}
Therefore,
\begin{eqnarray*}
\begin{split}
\Delta \mathcal{G}^\chi(t) &
\leq \mathcal{O}(1) \mathcal{G} (t-0) |\alpha_1^*| - \frac{1}{2}K\big(\theta_1 - \theta_4\big)|\alpha_1^*|\\[5pt]
&\leq - \frac{1}{4}K(\theta_1 - \theta_4)|\alpha_1^*| <0 .
\end{split}
\end{eqnarray*}

\textbf{Case 4.3.2}. \emph{Appearance of single $\alpha_4$ from boundary.}
It is clear that
\begin{eqnarray*}
\begin{split}
\Delta \mathcal{G}^\chi(t)\leq l_4 |\alpha_4| - \frac{1}{2a}K |\alpha_4|\leq - \frac{1}{4a}K |\alpha_4| <0.
\end{split}
\end{eqnarray*}

\textbf{Case 4.3.3}. \emph{Interaction $\alpha_i^* + \alpha_j^* \mapsto \alpha_1+ \alpha_2+ \alpha_4 \ (i,j \neq 3)$ at $\chi(t)$.}
It suffices to verify the typical case $i=4, j=1$ at $x=\chi(t)$.
The other cases can be treated likewise. Since
\begin{eqnarray*}
\begin{split}
\Delta \mathcal{L}^\chi_{\textsc{w}}(t)
& =l_1 |\alpha_1| + r_4 |\alpha_4|  - r_1 |\alpha_1^*| -  l_4 |\alpha_4^*| \\[5pt]
&= (l_1-r_1)|\alpha_1^*| + (r_4 - l_4)|\alpha_4^*|+ \mathcal{O}(1) |\alpha_1^*\alpha_4^*|,\\[5pt]
\Delta \mathcal{L}^\chi_{\textsc{b}}(t)& =|u-u^*| + |p-p^*|\\[5pt]
&\leq a (|\alpha_1|+ |\alpha_4^*|) =  a (|\alpha_1^*|+ |\alpha_4^*| +\mathcal{O}(1) |\alpha_1^*\alpha_4^*|),
\end{split}
\end{eqnarray*}
it follows from \eqref{eq:5.6} that
\begin{eqnarray*}
\begin{split}
\Delta \mathcal{G}^\chi(t) 
	& \leq  \big(l_1-r_1 +\hat{\epsilon} a\big)|\alpha_1^*| + \big(r_4 - l_4 + \hat{\epsilon} a \big)|\alpha_4^*|+ \mathcal{O}(1) |\alpha_1^*\alpha_4^*|
	- \frac{1}{2}K |\alpha_1^*\alpha_4^*| \\[5pt]
	& \leq - \frac{1}{4}K |\alpha_1^*\alpha_4^*| <0 .
\end{split}
\end{eqnarray*}

\textbf{Case 4.3.4}. \emph{Interaction $\alpha_i^* +  \alpha_j^* \mapsto \square + \alpha_5\ (i,j \neq 3)$ at $\chi(t)$.}
When a non-physical wave $\alpha_5$ appears at $x=\chi(t)$, we also figure out the estimates for redistributed functionals similar to that in Case 4.1.6.
For instance, we consider the typical situation $i=4, \ j=1$. See Fig.\ref{fig11}. Notice that
\begin{eqnarray*}
\alpha_1=\alpha_1^*, \quad \alpha_4=\alpha_4^*,\quad  \alpha_5=\mathcal{O}(1) |\alpha_1^* \alpha_4^*|.
\end{eqnarray*}
Hence a direct computation shows that
\begin{eqnarray*}
\begin{split}
\Delta \mathcal{L}^\chi_{\textsc{w}}(t)&= \big(l_1 -r_1\big)|\alpha_1^*| +\big(r_4-l_4\big)|\alpha_4^*|+ \mathcal{O}(1) |\alpha_1^* \alpha_4^*|,\\[5pt]
\Delta \mathcal{L}^\chi_{\textsc{b}}(t) &  \leq  a\big(|\alpha_1|+ |\alpha_4^*|\big),
\end{split}
\end{eqnarray*}
which yield
\begin{align*}
\Delta \mathcal{G}^\chi(t) 
& \leq \big(l_1 -r_1 +\hat{ \epsilon} a \big)|\alpha_1^*| + \big(r_4-l_4 + \hat{\epsilon} a\big) |\alpha_4^*|+\mathcal{O}(1) |\alpha_i^* \alpha_j^*| - \frac{1}{2} K |\alpha_1^* \alpha_4^*| \\[5pt]
& < - \frac{1}{4} K |\alpha_1^* \alpha_4^*|<0.
\end{align*}
Analogously, for every $i,j<5$, there holds
\begin{eqnarray*}
\Delta \mathcal{G}^\chi(t) < - \frac{1}{4} K |\alpha_i^* \alpha_j^*| <0.
\end{eqnarray*}

If $i=5$ and $j=1$, then the estimates $\alpha_1=\alpha_1^*$ and $\alpha_5=\alpha_5^*+\mathcal{O}(1) |\alpha_1^* \alpha_5^*|  $ imply
\begin{eqnarray*}
\begin{split}
\Delta \mathcal{G}^\chi(t) & \leq \big(l_1 -r_1 + \hat{\epsilon} a \big)|\alpha_1^*| + \big(r_5 -l_5 + \hat{\epsilon} a\big) |\alpha_5^*|
+\mathcal{O}(1) |\alpha_1^* \alpha_5^*| - \frac{1}{2}K|\alpha_1^* \alpha_5^*| \\[5pt]
& < - \frac{1}{4}K|\alpha_1^* \alpha_5^*| <0.
\end{split}
\end{eqnarray*}
Similarly, if $i=5$ and $ j>1$, there holds
\begin{eqnarray*}
\mathcal{G}^\chi(t) < - \frac{1}{4} K|\alpha_i^* \alpha_j^*| <0.
\end{eqnarray*}

The rest of cases for waves collision at $x \neq\chi(t)$ are so trivial and easy to verify. This completes the proof for non-reacting flow.

(ii) \emph{Reacting flow under condition $\textsc{(C2)}$ or $\textsc{(C3)}$}.
Consider the $\varepsilon$-approximation solution $U^{\varepsilon}$ for partially ignited or completely ignited flow.
Based on the previous argument, functional $\mathcal{G}^\chi(t)$ is decreasing for $t \in (t_{k-1} , t_k)$.
Our purpose now is verifying the monotonicity for reaction process.
In fact, it suffices to establish the local estimate on $\Delta \mathcal{G}^{\chi}(t_k) \big|_{x=\chi(t_{k})}$.
We first consider  that wave $\alpha_1^*$ hits the grid line  $t=t_k$ and then splits into new waves by reaction, \emph{i.e.},
\begin{eqnarray*}
\alpha_1^* \mapsto \alpha_1+\alpha_2+ \alpha_3+\alpha_4  \quad \mbox{at}\  x=\chi(t_k).
\end{eqnarray*}
A straightforward calculation gives that
\begin{eqnarray*}
\begin{split}
\Delta \mathcal{L}^\chi_{\textsc{w}}(t_k) \big|_{x=\chi(t_k)}
& = l_1 \big(|\alpha_1| -|\alpha_1^*|\big) + \big(l_1 -r_1\big) |\alpha_1^*| + r_4|\alpha_4| + \theta_3 |\alpha_3| \\[5pt]
& \leq \mathcal{O}(1) |\alpha_1^*| Y_\alpha^{*}\varepsilon + (l_1 -r_1) |\alpha_1^*|, \\[5pt]
\Delta\mathcal{L}^{\chi}_{\textsc{b}} (t_k)\big|_{x=\chi(t_k)}
& \leq a \big(|\alpha_1| + \Upsilon^{\chi-}_k\big) - a\Upsilon^{\chi+}_k - a\Upsilon_k^{\chi-} \\[5pt]
&\leq a|\alpha_1^*| + \mathcal{O}(1)  |\alpha_1^*| Y_\alpha^{*} \varepsilon.
\end{split}
\end{eqnarray*}
Hence, we have
\begin{eqnarray*}
\begin{split}
\big( \Delta \mathcal{L}^\chi_{\textsc{w}} +\hat{\epsilon} \Delta\mathcal{L}^{\chi}_{\textsc{b}} \big) (t_k)\big|_{x=\chi(t_k)}
&\leq  \mathcal{O}(1) |\alpha_1^*| Y_\alpha^{*} \varepsilon + \big(l_1 -r_1 + \hat{\epsilon} a\big) |\alpha_1^*|\\[5pt]
&\leq   \mathcal{O}(1) |\alpha_1^*| Y_\alpha^{*} \varepsilon,
\end{split}
\end{eqnarray*}
if $\hat{\epsilon}$ is small enough.
Generally, for every $i\neq3$, reaction
$\alpha_i^* \mapsto \alpha_1+\alpha_2+ \alpha_3+\alpha_4 $ at $x=\chi(t_k)$
always implies
\begin{eqnarray*}
\big( \Delta \mathcal{L}^\chi_{\textsc{w}}  +\hat{\epsilon} \Delta\mathcal{L}^{\chi}_{\textsc{b}}\big) (t_k)\big|_{x=\chi(t_k)}
\leq  \mathcal{O}(1) |\alpha_i^*| Y_\alpha^{*} \varepsilon.
\end{eqnarray*}

When reaction $\alpha_3^* \mapsto \alpha_1+\alpha_2+ \alpha_3+\alpha_4 $ takes place at $x=\chi(t_k)$, we readily derive the estimate
\begin{eqnarray*}
\begin{split}
\big( \Delta \mathcal{L}^\chi_{\textsc{w}}  +\hat{\epsilon}\Delta\mathcal{L}^{\chi}_{\textsc{b}}\big) (t_k)\big|_{x=\chi(t_k)}
&\leq - \frac{1}{2}\theta_3 |\alpha_3^*|\underline{\phi} \varepsilon + \mathcal{O}(1) \hat{\epsilon} |\alpha_3^*|  \varepsilon \\[5pt]
&\leq  - \frac{1}{4}\theta_3 |\alpha_3^*|\underline{\phi} \varepsilon .
\end{split}
\end{eqnarray*}

Note that $\mathcal{L}^\chi_{\textsc{w}}(t)$ merely extracts part of waves from $\mathcal{L}(t)$, particularly preserves all the $3$-waves with respect to mass fraction.
Replace $\Delta\mathcal{L}(t)$ in sections 4.2-4.3 with estimates on $\Delta\mathcal{L}^\chi_{\textsc{w}}(t) +\hat{\epsilon} \Delta \mathcal{L}^\chi_{\textsc{b}}(t)$.
By the argument analogous to \eqref{eq:4.30}-\eqref{eq:4.31}, we can also show that $\Delta \mathcal{G}^{\chi} (t_k)<0$ regardless of whether $\bar{T}_{\rm{r}} > T_{\text{i}}$.
This concludes the proof.
\end{proof}

The decrease of $\mathcal{G}^\chi(t)$ implies
\begin{equation}\label{eq:5.8}
\hat{\epsilon} T.V.\{(u,p); \chi(I_{\infty})\}\leq \mathcal{G}^\chi(+\infty) \leq \mathcal{G}^\chi(0)
\leq \mathcal{O}(1) \mathcal{G}(0),
\end{equation}
where $I_{\infty}=(0,+\infty)$. Hence we conclude that the total variation of components $(u,p)$ in ${U}$ on any quasi-characteristic $x=\chi(t)$ satisfying \eqref{eq:5.2}
is bounded provided $ \mathcal{G}(0)$ is small enough.

We next compare the state distribution on distinct quasi-characteristics. To this end, suppose that ${U}=(\rho, u, p, Y)^{\top}$ is a solution to system \eqref{eq:2.1},
and $x=\tilde{\chi}(t)$ is its $2$-quasi-characteristic curve satisfying \eqref{eq:5.2} instead of $\chi(t)$. Then, we have the following property on the difference of fluid velocities.

\begin{proposition}[Velocity comparison]\label{prop:5.1}
For any $t>\tilde{t}\geq 0$, there holds
\begin{eqnarray}\label{eq:5.9}
\begin{split}
&\|u(\tau, \chi_{\mathfrak{p}}(\tau)) - u(\tau, \tilde{\chi}(\tau))\|_{L^{1}([\tilde{t}, t])}\\[5pt]
&\qquad\leq { \mathcal{O}(1) \Big(T.V.\{U_{0}(\cdot); \mathbb{R}_{+}\}+ T.V.\{u_{\mathfrak{p}}(\cdot); \mathbb{R}_{+}\} \Big)\|\chi_{\mathfrak{p}}(\tau)-\tilde{\chi}(\tau)\|_{C([\tilde{t},t])}},
\end{split}
\end{eqnarray}
where $x=\chi_{\mathfrak{p}} (t)$ is the path of piston.
\end{proposition}

\begin{proof}
It suffices to technically treat ${U}$ as the $\varepsilon$--approximate solution which includes finitely many constant states.
Suppose $x=\chi(t)$ is some 2-quasi-characteristic associated with ${U}$ and satisfies \eqref{eq:5.2}.
Let $h(t)$ be a non-negative function which is a small perturbation of $\chi(t)$ with $\|h\|_{C^{0,1}} \ll 1$ so that

$\bullet$ there exists a positive constant $d$ such that
\begin{eqnarray*}
\begin{split}
&\sup_{\mathcal{N}_{\delta}(\bar{U}_{\mathfrak{b},\rm{l}})} \lambda_1(U) + d < \dot{\chi}(t) < \inf_{\mathcal{N}_{\delta}(\bar{U}_{\mathfrak{b},\rm{l}})} \lambda_4(U) -d,\\[5pt]
&\sup_{\mathcal{N}_{\delta}(\bar{U}_{\mathfrak{b},\rm{l}})} \lambda_1(U) + d < \dot{\chi}(t) + \dot{h}(t)< \inf_{\mathcal{N}_{\delta}(\bar{U}_{\mathfrak{b},\rm{l}})} \lambda_4(U) -d;
\end{split}
\end{eqnarray*}

$\bullet$ there is not any wave-colliding point $(t,x)$ in the region $\chi(t) <x <\chi(t)+h(t) $ for every $t>0$;

$\bullet$ there are finitely many wave fronts $\{x_\alpha\}$ in ${U}$ which divide the region between $\chi(t)$ and $\chi(t)+h(t) $ into many subregions with constant states.

%
%
%

Next we only focus on the fronts of family $i$ with $i=1, 4, 5$. Assume $x=x_\alpha (t)$ is an $i$-wave front that moves from
point $(\tau_0, x_0) \in \chi$ (resp. $\chi +h  $) to  point $(\tau_1,x_1) \in\chi +h  $ (resp. $\chi $). Then it satisfies
\begin{eqnarray*}
|\dot{\chi}(t) - \dot{x}_\alpha (t)| \geq d ,\qquad  |\dot{\chi}(t)+\dot{h}(t) - \dot{x}_\alpha (t)| \geq d,
\end{eqnarray*}
for a.e. $t\in (\tau_0, \tau_1)$. This implies that
\begin{eqnarray*}
d\Delta t_\alpha \leq \int_{\tau_0}^{\tau_1} |\dot{\chi}(\tau) +\dot{h}(\tau) -\dot{x}_\alpha (\tau)| dt \leq \|h\|_{C([\tau_0,\tau_1])},
\end{eqnarray*}
where $\Delta t_\alpha \doteq \tau_1- \tau_0$.
Let $u^{\alpha+}, u^{\alpha-}$ stand for the velocities at two banks of discontinuity $x_\alpha$. Then from above inequality we can deduce for every $t>\tilde{t}\geq 0$ that
\begin{eqnarray}\label{eq:5.10}
\begin{split}
\int^t_{\tilde{t}} |u(\tau,\chi(\tau)) - u(\tau,\chi(\tau)+ h(\tau))|d\tau
& = \sum_{\alpha} |u^{\alpha+} - u^{\alpha-}| \Delta t_\alpha\\[5pt]
& \leq \frac{\ 1\ }{d}\|h\|_{C({ [\tilde{t},t]})}\sum_{\alpha} |u^{\alpha+} - u^{\alpha-}| \\[5pt]
&= \mathcal{O}(1)\|h\|_{C([\tilde{t},t])} T.V.\{(u,p); \chi((\tilde{t},t))\}.
\end{split}
\end{eqnarray}

Define a homotopic mapping
\begin{eqnarray*}
\chi_{\tilde{\theta}} (t) = \tilde{\theta} \chi_{\mathfrak{p}}(t) + (1-\tilde{\theta}) \tilde{\chi}(t) \quad \mbox{for}\ \tilde{\theta} \in [0,1].
\end{eqnarray*}
Decompose $[0,1]$ into finitely many small intervals $[\tilde{\theta}_i, \tilde{\theta}_{i+1}]$ with $0=\tilde{\theta}_0 < \tilde{\theta}_1 < \cdots <\tilde{\theta}_{n-1} < \tilde{\theta}_n=1$,
so that the wave interactions do not occur in each region between the curves $\chi_{\tilde{\theta}_i}$ and $\chi_{\tilde{\theta}_{i+1}}$.
Then, by \eqref{eq:5.8} and \eqref{eq:5.10}, we obtain
\begin{eqnarray*}
\begin{split}
&\quad \int^t_{\tilde{t}} |u(\tau, \chi_{\mathfrak{p}}(\tau)) - u(\tau, \tilde{\chi}(\tau)| d\tau\\[5pt]
	& \leq \sum_i  \int^t_{\tilde{t}} |u(\tau,\chi_{\tilde{\theta}_i}(\tau)) - u(\tau,\chi_{\tilde{\theta}_{i+1}}(\tau))| d\tau \\[5pt]
	& \leq \mathcal{O}(1)T.V.\{(u,p); \chi_{\tilde{\theta}_i}((\tilde{t},t))\}\sum_i \|\chi_{\tilde{\theta}_i}(\tau) - \chi_{\tilde{\theta}_{i+1}}(\tau) \|_{C([\tilde{t},t])} \\[5pt]
	& \leq \mathcal{O}(1) \Big( T.V.\{{U}_0(\cdot);\mathbb{R}_{+}\} + T.V.\{u_{\mathfrak{p}}(\cdot);\mathbb{R}_{+}\} \Big)\|\chi_{\mathfrak{p}}(\cdot)-\tilde{\chi}(\cdot)\|_{C([\tilde{t},t])},
\end{split}
\end{eqnarray*}
which gives the estimate \eqref{eq:5.9}.

\end{proof}

\section{$L^1$-stability for the entropy solutions to (\emph{IBVP})}

In this section, we will further study the $L^1$-stability of entropy solutions $U$ to (\emph{IBVP}). Suppose that
$U^{\varepsilon}=(U^{\textsc{e}, \varepsilon}, Y^{\varepsilon}_{1})$ and $V^{\varepsilon}=(V^{\textsc{e}, \varepsilon}, Y^{\varepsilon}_{2})$
with $U^{\textsc{e}, \varepsilon}=(\rho^{\varepsilon}_1, u^{\varepsilon}_1,p^{\varepsilon}_1)^{\top}$ and $V^{\textsc{e}, \varepsilon}=(\rho^{\varepsilon}_2, u^{\varepsilon}_2,p^{\varepsilon}_2)^{\top}$
are two $\varepsilon$-approximate solutions corresponding to the initial-boundary data $(U^{\varepsilon}_{0}, u^{\varepsilon}_{\mathfrak{p}})$
and $(V^{\varepsilon}_{0}, v^{\varepsilon}_{\mathfrak{p}})$, respectively. In the following discussion, we will drop the superscript $\varepsilon$ in ${U}^\varepsilon$, ${V}^\varepsilon$ as well as the corresponding data for simplicity.

For given boundary data $u_{\mathfrak{p}}$ and $v_{\mathfrak{p}}$, we define the corresponding boundary curves by
\begin{eqnarray*}
\chi_{\mathfrak{p},1}(t)\doteq \int_0^t u_{\mathfrak{p}}(\tau) d\tau, \qquad \chi_{\mathfrak{p},2}(t)\doteq \int_0^t v_\mathfrak{p}(\tau) d\tau
\end{eqnarray*}
for $t\geq 0$.
Then denote their minimal and maximal curves respectively by
\begin{equation*}\label{Def chi bar}
\chi^{\textrm{m}}_{\mathfrak{p}}(t) \doteq \min\big\{\chi_{\mathfrak{p},1}(t),\ \chi_{\mathfrak{p},2}(t)\big\}, \quad
\chi^{\textsc{m}}_{\mathfrak{p}}(t) \doteq \max\big\{\chi_{\mathfrak{p},1}(t),\ \chi_{\mathfrak{p},2}(t)\big\} .
\end{equation*}
Indeed, $\chi^{\textsc{m}}_{\mathfrak{p}}(t)$ is a $2$-quasi-characteristic associated with the solutions $U$ and $V$.
Besides, let's denote the paths of large $4$-shocks in $U$ and $V$ by $\chi_{\rm{s},1}(t)$, $\chi_{\rm{s},2}(t)$, respectively.
Then define another two curves
\begin{equation*}\label{Def chi3 bar}
\chi^{\textrm{m}}_{\rm{s}}(t) \doteq \min\big\{\chi_{\rm{s},1}(t),\ \chi_{\rm{s},2}(t)\big\}, \qquad \chi^{\textsc{m}}_{\rm{s}}(t) \doteq \max\big\{\chi_{\rm{s},1}(t),\ \chi_{\rm{s},2}(t)\big\}.
\end{equation*}

\subsection{Construction of the weighted Lyapunov functional }
For any $(t,x)\in\{(t,x)\in\mathbb{R}^2: x\geq \chi^{\textsc{m}}_{\mathfrak{p}}(t), t>0\} $, we make an orientation rule by
\begin{equation}\label{orientation rule}
(\tilde{{U}}_{\rm{l}}, \tilde{{U}}_{\rm{r}})=\left\{
\begin{aligned}
		&({U},{V}) &\qquad& \mbox{if}\ \mbox{both}\ U^{\textsc{e}}\ \mbox{and}\ V^{\textsc{e}} \in \mathcal{N}_{\delta}(\bar{U}^{\textsc{e}}_{\mathfrak{b},\rm{l}})\ \mbox{or}\ \mathcal{N}_{\delta}(\bar{U}^{\textsc{e}}_{\mathfrak{b},\rm{r}}),\\[5pt]
		&({U},{V}) &\qquad& \mbox{if}\ U^{\textsc{e}} \in \mathcal{N}_{\delta}(\bar{U}^{\textsc{e}}_{\mathfrak{b},\rm{l}})\ \mbox{and } V^{\textsc{e}}\in\mathcal{N}_{\delta}(\bar{U}^{\textsc{e}}_{\mathfrak{b},\rm{r}}),\\[5pt]
		&({V},{U}) &\qquad& \mbox{if}\ V^{\textsc{e}} \in \mathcal{N}_{\delta}(\bar{U}^{\textsc{e}}_{\mathfrak{b},\rm{l}})\ \mbox{and } U^{\textsc{e}}\in\mathcal{N}_{\delta}(\bar{U}^{\textsc{e}}_{\mathfrak{b},\rm{r}}).
\end{aligned}
\right.
\end{equation}
Then, according to Lemma \ref{lem:2.5}, states $\tilde{{U}}_{\rm{l}}$ and $\tilde{{U}}_{\rm{r}}$ can be connected  by the Hugoniot curves of system \eqref{eq:2.2}.
That is, there exists a group of distance indices $\mathbf{q}= (q_1, q_2, q_3, q_4)$ such that
\begin{eqnarray*}
\tilde{{U}}_{\rm{r}}=\mathcal{H}(\mathbf{q})(\tilde{{U}}_{\rm{l}})=S_4(q_4)\circ S_3(q_3)\circ S_2(q_2)\circ S_1(q_1)(\tilde{{U}}_{\rm{l}}).
\end{eqnarray*}
Set $\omega_0=\tilde{U}_{\rm{l}}$ and $\omega_4=\tilde{U}_{\rm{r}}$. Define the intermediate states
\begin{equation*}
\omega_i \doteq S_i(q_i)(\omega_{i-1}) \qquad  (i=1,2,3).
\end{equation*}
Obviously, the total strength $\mathop{\sum}_{i=1}^4 |q_i|$ is equivalent to $|{U}-{V}|$.
Henceforth, we will employ notation $\lambda_i \doteq \lambda_i(\omega_{i-1}, \omega_i)$ $(1\leq i\leq 4)$ for brevity.

Now we define Lyapunov functional for (\emph{IBVP}) by
\begin{equation*}
\mathscr{L}({U}(t), {V}(t))= \sum_{i=1}^4 \int^{+\infty}_{\chi^{\textrm{M}}_{\mathfrak{p}} (t)} \ell_i \mathcal{W}_i |q_i|dx,
\end{equation*}
where $\ell_i=\ell_i(t,x)$ is a function to balance the distinct speeds of large shocks. Assume further $\ell_i$ is piecewise constant in $x$ and assigned values as in Tables 1 and 2.


\begin{table}[!h]
\begin{center}
\begin{tabu}to 0.6\textwidth{|X[c]|X[c]|X[c]|X[c]|}
\hline			
$x$  & ($\chi^{\textsc{m}}_{\mathfrak{p}}, \ \chi^{\textrm{m}}_{\rm{s}}$) & ($\chi^{\textrm{m}}_{\rm{s}},\ \chi^{\textsc{m}}_{\rm{s}}$) & ($\chi^{\textsc{m}}_{\rm{s}}, \ +\infty$) \\[2pt]
\hline
$\ell_1$ & $\ell$ & $\ell^2$ & $\ell^3$ \\[2pt]
			$\ell_2$ & 1 & $\ell^2$ & $\ell^3$\\[2pt]
			$\ell_3$ & $1/\mathcal{K}$ & $\ell/\mathcal{K}$ & $\ell^2/\mathcal{K}$ \\[2pt]
			$\ell_4$ & $\kappa_0\ell$ & 1 & $\ell^3$ \\[2pt]
			\hline
		\end{tabu}
	\end{center}
\caption{Distribution of $\ell_i (t,x)$ under condition (C1)}
\label{ell distribution for (C1)}
\end{table}

\begin{table}[!h]
\begin{center}
		\begin{tabu}to 0.6\textwidth{|X[c]|X[c]|X[c]|X[c]|}
			\hline
			$x$  &  ($\chi^{\textsc{m}}_{\mathfrak{p}}, \ \chi^{\textrm{m}}_{\rm{s}}$) & ($\chi^{\textrm{m}}_{\rm{s}},\ \chi^{\textsc{m}}_{\rm{s}}$) & ($\chi^{\textsc{m}}_{\rm{s}}, \ +\infty$) \\[2pt]
			\hline
			$\ell_1$ & $\ell$ & $\ell^2$ & $\ell^3$ \\[2pt]
			$\ell_2$ & 1 & $\ell^2$ & $\ell^3$ \\[2pt]
			$\ell_3$ & $\ell^4$ & $\ell^5$ & $\ell^6$ \\[2pt]
			$\ell_4$ & $\kappa_0\ell$ & 1 & $\ell^3$ \\[2pt]
			\hline
		\end{tabu}
	\end{center}
\caption{Distribution of $\ell_i (t,x)$ under condition (C2) or (C3)} \label{ell distribution (C2) or (C3)}
\end{table}

We require that the constants $\mathcal{K}$ and $\ell$ are chosen so large that
\begin{equation}\label{condition kappa ell}
\mathcal{K} \gg \ell^6 \gg 1 ,\qquad  \ell > \frac{1}{\theta_3}.
\end{equation}
The different strategies of $\ell_3$ distribution in two tables are caused by lack of smallness restriction on $T.V.\{Y_0; \mathbb{R}_{+}\}$ in condition $\textrm{(C1)}$.
So we require  $\ell_3$ is small enough for non-reacting flow, while $\ell_3$ is much larger than other $\ell_i$ for reacting flow.
The positive coefficient $\kappa_0$ will be specified later. $\mathcal{W}_i$ is the key weight to retrieve the monotonicity of $\mathscr{L}({U}(t), {V}(t))$. It is originally introduced by
Bressan-Liu-Yang\cite{Bressan-Liu-Yang-1999} for general small \emph{BV} data and by \cite{Lewicka-2004} for large \emph{BV} data. Specifically,
\begin{equation}\label{Def Wi}
\mathcal{W}_i (x,t)\doteq 1 + \mathcal{K}\Big( A_i(x) +K\big(\mathcal{Q}({U}(t) )+\mathcal{Q}({V}(t))\big) +K\mathcal{L}_{\hat{\textsc{b}}}(t)\Big),
\end{equation}
where $K$ is the large constant from Glimm functional $\mathcal{G}(t)$, and the linear terms $A_i$ only includes part of waves selected from ${U}$ and ${V}$.
Precisely, let $\alpha$ be a wave of the $k_\alpha$th family, with location  $x_\alpha$ and strength $|\alpha|$.
Given time $t$, introduce notations
\begin{eqnarray*}
\begin{split}
&\mathcal{J}^{\textsc{e}}\doteq\mathcal{J}^{\textsc{e}}({U})\cup \mathcal{J}^{\textsc{e}}({V}) ,\qquad  \mathcal{J}^{\textsc{y}}_3 \doteq \mathcal{J}^{\textsc{y}}_3 ({U}) \cup  \mathcal{J}^{\textsc{y}}_3 ({V}),\\[5pt]
&\mathcal{J}^{\textsc{e}}_{\rm{l}}\doteq\mathcal{J}^{\textsc{e}}_{\rm{l}}({U})\cup \mathcal{J}^{\textsc{e}}_{\rm{l}}({V}),
\qquad \mathcal{J}^{\textsc{e}}_{\rm{r}}\doteq\mathcal{J}^{\textsc{e}}_{\rm{r}}({U})\cup \mathcal{J}^{\textsc{e}}_{\rm{r}}({V}),
\end{split}
\end{eqnarray*}
where
\begin{eqnarray*}
\begin{split}
&\mathcal{J}^{\textsc{e}}({U}) = \{ \alpha \ | \ \alpha \mbox{ is a small wave of }{U},\  k_\alpha \neq 3\}, \\[5pt]
&\mathcal{J}^{\textsc{e}}_{\rm{l}}({U}) = \{ \alpha \ | \ \alpha \in \mathcal{J}({U}) \mbox{ and component } U^{\textsc{e}}(t,x_\alpha \pm 0) \in \mathcal{N}_{\delta}(\bar{U}^{\textsc{e}}_{\mathfrak{b},\rm{l}}) \},\\[5pt]
&\mathcal{J}^{\textsc{e}}_{\rm{r}}({U}) = \{ \alpha \ | \ \alpha \in \mathcal{J}({U}) \mbox{ and component } U^{\textsc{e}}(t,x_\alpha \pm 0) \in \mathcal{N}_{\delta}(\bar{U}^{\textsc{e}}_{\mathfrak{b},\rm{r}}) \}, \\[5pt]
&\mathcal{J}^{\textsc{y}}_3 ({U})=\{  \alpha \ | \ \alpha \mbox{ is a 3-wave of }{U}\},
\end{split}
\end{eqnarray*}
and similarly define $ \mathcal{J}^{\textsc{e}}({V}), \mathcal{J}^{\textsc{e}}_{\rm{l}}({V})$, $ \mathcal{J}^{\textsc{e}}_{\rm{r}}({V})$ and $\mathcal{J}^{\textsc{y}}_3 ({V})$.
For $i=1,4 $, we define
\begin{eqnarray*}
	A_i(t,x)= &   \Bigg(\displaystyle\sum\limits_{
		\begin{subarray}{c}
			i<k_\alpha \leq 4, \ 	x_\alpha <x, \\
			\alpha \in \mathcal{J}^{\textsc{e}}
		\end{subarray}
	}
    +\sum\limits_{
        \begin{subarray}{c}
            1\leq k_\alpha <i,\  x_\alpha >x, \\
			\alpha \in \mathcal{J}^{\textsc{e}}
		\end{subarray}
	}
	\Bigg) |\alpha|
	+  \theta_3  \Bigg(\sum\limits_{
		\begin{subarray}{c}
			i=1, \ 	x_\alpha <x, \\
			\alpha \in \mathcal{J}^{\textsc{y}}_3
		\end{subarray}
	}
	+\sum\limits_{
		\begin{subarray}{c}
			i=4,\  x_\alpha >x, \\
			\alpha \in \mathcal{J}^{\textsc{y}}_3
		\end{subarray}
	}
	+ \sum\limits_{\alpha \in \mathcal{J}^{\textsc{y}}_3}
	\Bigg) |\alpha|   \\
& + \left\{
	\begin{aligned}
		& \Bigg(\sum\limits_{
			\begin{subarray}{c}
				k_\alpha =i ,  x_\alpha <x,\\
				\alpha \in \mathcal{J}^{\textsc{e}}({U})\\
			\end{subarray}
		}
		+ \sum\limits_{
			\begin{subarray}{c}
				k_\alpha =i ,  x_\alpha >x \\
				\alpha \in \mathcal{J}^{\textsc{e}}({V})
			\end{subarray}
		}
		\Bigg) |\alpha|
		\ \ \hspace{2cm} && \mbox{if   small }  q_i(x)<0 ,\\
		& \Bigg(\sum\limits_{
			\begin{subarray}{c}
				k_\alpha =i ,  x_\alpha <x,\\
				\alpha \in \mathcal{J}^{\textsc{e}}({V})\\
			\end{subarray}
		}
		+ \sum\limits_{
			\begin{subarray}{c}
				k_\alpha =i ,  x_\alpha >x \\
				\alpha \in \mathcal{J}^{\textsc{e}}({U})
			\end{subarray}
		}
		\Bigg) |\alpha|
		&& \mbox{if   small }  q_i(x) >0 ,\\
		& \Bigg(\sum\limits_{
			\begin{subarray}{c}
				k_\alpha =i ,  x_\alpha <x,\\
				\alpha \in \mathcal{J}^{\textsc{e}}_{\rm{l}}\\
			\end{subarray}
		}
		+ \sum\limits_{
			\begin{subarray}{c}
				k_\alpha =i ,  x_\alpha >x \\
				\alpha  \in \mathcal{J}^{\textsc{e}}_{\rm{r}}
			\end{subarray}
		}
		\Bigg) |\alpha| \ \ \
		&& \text{if } i=4 \mbox{ and } |q_i(x)|\ \mbox{is\ large },\\
		&\ \ 0   && \mbox{otherwise};
	\end{aligned}
	\right.
\end{eqnarray*}
while for $i=2,3$, we define
\begin{eqnarray*}
A_i(t,x)=   \Bigg(\sum\limits_{
	\begin{subarray}{c}
		k_\alpha =4, \  x_\alpha <x,  \\
		\alpha \in \mathcal{J}^{\textsc{e}}
	\end{subarray}
}
+\sum\limits_{
	\begin{subarray}{c}
		k_\alpha=1, \ x_\alpha >x, \\
		\alpha \in \mathcal{J}^{\textsc{e}}
	\end{subarray}
}
+ \theta_3 \sum\limits_{
	\alpha \in \mathcal{J}^{\textsc{y}}_3  }
\Bigg) |\alpha|.
\end{eqnarray*}
In addition, the term $\mathcal{Q}$ in \eqref{Def Wi} remains the amount of potential interactions in $\mathcal{G}(t)$; and
the term $\mathcal{L}_{\hat{\textsc{b}}}(t)$ is given by
\begin{align*}
\mathcal{L}_{\hat{\textsc{b}}}(t) & \doteq \mathcal{L}_{\textsc{b}}({U}(t)) +\mathcal{L}_{\textsc{b}}({V}(t))\\[5pt]
	&  = \Big(T.V.\{u_{\mathfrak{p}}; (t,+\infty)\}  + \sum_{ k\varepsilon \geq t } \Upsilon^{\mathfrak{p}}_k + \sum\limits_{k\varepsilon \geq t} \Upsilon^{\rm{s}}_k \Big) _{{U}}\\[5pt]
	& \quad \   + \Big( T.V.\{v_{\mathfrak{p}}; (t,+\infty)\} + \sum_{ k\varepsilon \geq t } \Upsilon^{\mathfrak{p}}_k+ \sum_{k\varepsilon \geq t} \Upsilon^{\rm{s}}_k\Big)  _{{V}}.
\end{align*}

Finally, we claim that if the initial-boundary data $({U}_0, u_{\mathfrak{p}})$ and $({V}_0, v_{\mathfrak{p}})$ are sufficiently small in \emph{BV} sense, then there holds
\begin{equation}\label{condition W}
1 \leq \mathcal{W}_i \leq \mathcal{W}  \qquad  (1\leq i\leq 4)
\end{equation}
for some constant $\mathcal{W}$ very close to 1.

\subsection{$L^1$-stability and uniqueness for non-reacting flow}
By means of Lyapunov functional $\mathscr{L}$, we now proceed to show the stability of global solutions for non-reacting flow.
Owing to condition $\textrm{(C1)}$, the functional $\mathscr{L}({U}(t), {V}(t))$ is Lipschitz continuous in $t$ and differentiable at any time except a finite subset of $(0,+\infty)$.
We will consider how $\mathscr{L}({U}(t), {V}(t))$ evolves at the time $t$ without waves interaction away from ${\chi}^{\textsc{m}}_{\mathfrak{p}}(t)$ and
 without waves emergence on ${\chi}^{\textsc{m}}_{\mathfrak{p}}(t)$. To this end, let's suppose that $\mathcal{J}$ denotes the set of all the waves in ${U} $ and ${V}$ at time $t$.
And assume that $\alpha\in\mathcal{J}$ which belongs to $k_\alpha$-familiy. Differentiating $\mathscr{L}$ with respect to $t$ gives that
\begin{eqnarray*}
\begin{split}
\dot{\mathscr{L}} ({U}(t) ,{V}(t) )
=&\mathop{\sum}_{\alpha \in\mathcal{J} }\mathop{\sum}_{i=1}^{4}
	\Big(|q_i^{\alpha-}| \ell_i^{\alpha-} \mathcal{W}_i ^{\alpha-}  -|q_i^{\alpha+}|
	\ell_i^{\alpha+} \mathcal{W}_i ^{\alpha+} \Big) \dot{x}_\alpha
	-\sum_{i=1}^4 |q_i^{\mathfrak{p}+}|
	\ell_i^{\mathfrak{p}+} \mathcal{W}_i ^{\mathfrak{p}+} \dot{\chi}^{\textsc{m}}_{\mathfrak{p}}\\[5pt]
	=&\mathop{\sum}_{\alpha \in \mathcal{J} }\mathop{\sum}_{i=1}^4
	\Big(|q_i^{\alpha+}| \ell_i^{\alpha+}\mathcal{ W}_i ^{\alpha+}(\lambda_i^{\alpha+}-\dot{x}_\alpha )
	- |q_i^{\alpha-}| \ell_i^{\alpha-} \mathcal{W}_i ^{\alpha-}(\lambda_i^{\alpha-}-\dot{x}_\alpha )\Big) \\[5pt]
	& \quad + \sum_{i=1}^4 |q_i^{\mathfrak{p}+}| \ell_i^{\mathfrak{p}+} \mathcal{W}_i ^{\mathfrak{p}+}\big(\lambda_i^{\mathfrak{p}+}-\dot{\chi}^{\textsc{m}}_{\mathfrak{p}} \big)\\[5pt]
	&\doteq \mathop{\sum}_{\alpha \in \mathcal{J} }\mathop{\sum}_{i=1}^4 \mathscr{E}_{i,\alpha}
	+\sum_{i=1}^4 \mathscr{E}_{i,\mathfrak{p}},
\end{split}
\end{eqnarray*}
where the symbols $q_i^{\alpha-}, \ell_i^{\alpha-},  \mathcal{W}_i ^{\alpha-}, \lambda_i^{\alpha-} $
(resp. $q_i^{\alpha+}, \ell_i^{\alpha+},  \mathcal{W}_i ^{\alpha+} , \lambda_i^{\alpha+}$)
stand for the quantities on the left-hand side (resp. right-hand side) of a front $x_\alpha$, while the terms
 $q_i^{\mathfrak{p}+}, \ell_i^{\mathfrak{p}+},  \mathcal{W}_i ^{\mathfrak{p}+}, \lambda_i^{\mathfrak{p}+}$
 on curve $x={\chi}^{\textsc{m}}_{\mathfrak{p}}(t)$ can be interpreted in the similar way.

In the sequel, our main purpose is to verify that
\begin{eqnarray}\label{eq:conclusion Ei}
\begin{split}
&\mathop{\sum}_{i=1}^4 \mathscr{E}_{i,\alpha}< \mathcal{O}(1)\varepsilon|\alpha|,\qquad     && \alpha\in \mathcal{J} \setminus\mathcal{NP}, \\[5pt]
&\mathop{\sum}_{i=1}^4 \mathscr{E}_{i,\alpha} \leq \mathcal{O}(1)|\alpha|,\qquad && \alpha\in \mathcal{NP}, \\[5pt]
&\sum_{i=1}^4 \mathscr{E}_{i,\mathfrak{p}} \leq \mathcal{O}(1) |\Delta u^{\textsc{m}}_{\mathfrak{p}}|,
\end{split}
\end{eqnarray}
where $\Delta u^{\textsc{m}}_{\mathfrak{p}}= u_1(t,{\chi}^{\textsc{m}}_{\mathfrak{p}}(t)+0) - u_2(t,{\chi}^{\textsc{m}}_{\mathfrak{p}}(t)+0)$.
Given wave $\alpha \in \mathcal{J}$, we rewrite $\mathscr{E}_{i,\alpha}$ in the concise form
\begin{equation}\label{Def Ei}
\mathscr{E}_i = |q_i^+| \ell_i^+ \mathcal{W}_i^+ (\lambda_i^+ - \dot{x}_\alpha) -  |q_i^-| \ell_i^- \mathcal{W}_i^- (\lambda_i^- - \dot{x}_\alpha), \quad 1\leq i\leq 4,
\end{equation}
and prove the estimates \eqref{eq:conclusion Ei} in different cases.

\begin{figure}[ht]
\begin{minipage}[t]{0.45\textwidth}
\centering
\begin{tikzpicture}[scale=0.7]
\draw [thin][->] (-4.3, -1)--(4, -1);
\draw [thin][line width=0.02cm](-3.5,0.7)--(-1, 0.7)--(-1,-0.3)--(3,-0.3);
\draw [thin][line width=0.02cm](-3.5,1.5)--(1.5, 1.5)--(1.5,0.2)--(3,0.2);
\draw[thin][dotted][line width=0.03cm](-1,-0.3)--(-1,-1);
\draw[thin][dotted][line width=0.03cm](1.5,0.2)--(1.5,-1);
\draw [thick][->] (-1, 0.2)--(0.2, 0.2);
\draw [thick][->] (1.5, 0.9)--(2.4, 0.9);
\node at (-0.9,-1.4) {$\chi^{\textrm{m}}_{\rm{s}}$};
\node at (1.6,-1.4) {$\chi^{\textsc{m}}_{\rm{s}}$};
\node at (4,-1.3) {$x$};
\node at (-0.4,0.6) {$\alpha^{\rm{s}}_{4}$};
\node at (-3.8,1.5) {$U$};
\node at (-3.8,0.7) {$V$};
\end{tikzpicture}
\caption{Flow comparison for Case 5.1.1}\label{Case 6.1.1: flow comparison}
\end{minipage}
\begin{minipage}[t]{0.5\textwidth}
\centering
\begin{tikzpicture}[scale=0.7]
\draw [thick][dashed](0,1) circle [radius=2];
\draw [thick][dashed](5.5,1) circle [radius=1.5];
\draw [thin][->](0.2,2.2) to [out=-30, in=180](5.7, 1.1);
\draw [thin][->] (-0.2, -0.2)to [out=40, in=260](0.8, 1.2)to[out=-30, in=180](5.7, 0.9);
\draw [thin][->] (-0.2, -0.2)to [out=160, in=-80](-1.4, 1.0)to[out=10, in=240](0.1, 2.15);

\node at (0.2,2.2) {$\bullet$};
\node at (-0.2,-0.2) {$\bullet$};
\node at (5.8,1) {$\bullet$};
\node at (0.2, 2.55) {$V^{\textsc{e},-}$};
\node at (-0.6, 1.8) {$\mathbf{q}^{-}$};
\node at (0, -0.55) {$U^{\textsc{e}}$};
\node at (3.0, 0.3) {$\mathbf{q}^{+}$};
\node at (3.0, 1.7) {$\alpha^{\rm{s}}_{4}$};
\node at (5.9, 1.55) {$V^{\textsc{e},+}$};
\node at (0.2,-1.7) {$\mathcal{N}_{\delta}(\bar{U}^{\textsc{e}}_{\mathfrak{b},\rm{l}})$};
\node at (5.8,-1.5) {$\mathcal{N}_{\delta}(\bar{U}^{\textsc{e}}_{\mathfrak{b},\rm{r}})$};
\end{tikzpicture}
\caption{Projected phase space for Case 5.1.1}\label{Case 6.1.1: projected phase space}
\end{minipage}
\end{figure}

\textbf{Case 5.1.1}. \emph{Front $x_\alpha = \chi^{\rm{m}}_{\rm{s}}(t)$}. Suppose that $\alpha=\alpha^{\rm{s}}_4$ is the large $4$--shock of ${V}$ which satisfies
$V^{\textsc{e},-}=(\rho, u, p)^{\top}= S^{\textsc{e}}_4(-\alpha^{\rm{s}}_4)(V^{\textsc{e},+})$ for $V^{\textsc{e},+}=(\rho^+, u^+, p^+)^{\top}$ (see Figs \ref{Case 6.1.1: flow comparison}-\ref{Case 6.1.1: projected phase space}).
For given strength $|\alpha^{\rm{s}}_4|$, the variables $q_1^-, q_2^-, q_4^+$ are uniquely determined by $q_1^+, q_2^+, q_4^-$, but independent of  $q_3^- , q_3^+$. In particular $q_3^- = q_3^+$. In projected phase space, set intermediate state
$\omega_2^{\textsc{e},-} = S^{\textsc{e}}_2(q^-_2)\circ S^{\textsc{e}}_1(q^-_1) (U^{\textsc{e}})$.
Thus the indices $q^\pm_i$ ($i \neq 3$) satisfy the relations
\begin{align*}
&U^{\textsc{e}}+ \int_0^{q_1^-} \textbf{r}^{\textsc{e}}_1(U^{\textsc{e}},\sigma) d\sigma +  q_2^- \textbf{r}^{\textsc{e}}_2 +   \int_0^{q_4^-} \textbf{r}^{\textsc{e}}_4(\omega_2^{\textsc{e},-},\sigma) d\sigma
=V^{\textsc{e},+} + \int_0^{-\alpha^{\rm{s}}_4} \textbf{r}^{\textsc{e}}_4(V^{\textsc{e},+},\sigma) d\sigma,
\end{align*}
and
\begin{align*}
&U^{\textsc{e}}+ \int_0^{q_1^+} \textbf{r}^{\textsc{e}}_1(U^{\textsc{e}},\sigma) d\sigma +  q_2^+\textbf{r}^{\textsc{e}}_2
=V^{\textsc{e},+} +\int_0^{-q_4^+} \textbf{r}^{\textsc{e}}_4(V^{\textsc{e},+},\sigma) d\sigma .
\end{align*}
Subtract to obtain that
\begin{align*}
\int_{{ q_1^+}}^{q_1^-} \textbf{r}^{\textsc{e}}_1(U^{\textsc{e}},\sigma) d\sigma + (q_2^- - q_2^+)\textbf{r}^{\textsc{e}}_2
+  \int_0^{q_4^-} \textbf{r}^{\textsc{e}}_4(\omega_2^{\textsc{e},-},\sigma) d\sigma
=  \int_{-q_4^+}^{-\alpha^{\rm{s}}_4}\textbf{r}^{\textsc{e}}_4(V^{\textsc{e},+},\sigma) d\sigma .
\end{align*}
Assume  $q_1^- = f_1 (q_1^+ , q_2^+ ,q_4^-)$. By direct computation, we immediately find that
\begin{equation}\label{partial q1- 1}
\frac{\partial }{\partial q_1^+} f_1(q_1^+ ,0,0) = 1, \qquad \frac{\partial }{\partial q_2^+} f_1(0,q_2^+ ,0) = 0.
\end{equation}
Let $q_1^+ = q_2^+=0$, then
\begin{align*}
\int_0^{q_1^-} \textbf{r}^{\textsc{e}}_1(U^{\textsc{e}},\sigma) d\sigma
+  q_2^- \textbf{r}^{\textsc{e}}_2
+  \int_0^{q_4^-}\textbf{r}^{\textsc{e}}_4(\omega_2^{\textsc{e},-},\sigma) d\sigma
=  \int_{-q_4^+}^{-\alpha^{\rm{s}}_4} \textbf{r}^{\textsc{e}}_4(V^{\textsc{e},+},\sigma) d\sigma .
\end{align*}
Differentiating this with respect to $q_4^-$ and taking $q_4^- =0$, it yields
\begin{equation*}
\textbf{r}^{\textsc{e}}_1 (V^{\textsc{e},-}) \frac{\partial f_1}{\partial q_4^-}
+ \textbf{r}^{\textsc{e}}_2 \frac{\partial q_2^-}{\partial q_4^-} +\textbf{r}^{\textsc{e}}_4 (V^{\textsc{e},-})
=  \textbf{r}^{\textsc{e}}_4 (V^{\textsc{e},+}, -\alpha^{\rm{s}}_4)\frac{\partial q_4^+}{\partial q_4^-},
\end{equation*}
where $\textbf{r}^{\textsc{e}}_4 (V^{\textsc{e},+}, -\alpha^{\rm{s}}_4)= (\rho', u', p')^{\top}$ with parameter $\sigma = -\alpha^{\rm{s}}_4$. From above equation, we can further derive that
\begin{equation}\label{partial q1- 2}
\kappa_{\rm{s}} \doteq\frac{\partial }{\partial q_4^-} f_1(0,0,0)
= -\frac{\det \big(\textbf{r}^{\textsc{e}}_4 (V^{\textsc{e},-}),\ \textbf{r}^{\textsc{e}}_2,\ \textbf{r}^{\textsc{e}}_4(V^{\textsc{e},+}, -\alpha^{\rm{s}}_4) \big)}
{\det\big(\textbf{r}^{\textsc{e}}_1 (V^{\textsc{e},-}),\ \textbf{r}^{\textsc{e}}_2,\ \textbf{r}^{\textsc{e}}_4(V^{\textsc{e},+}, -\alpha^{\rm{s}}_4) \big)},
\end{equation}
where
\begin{eqnarray*}
\begin{split}
\det \big(\textbf{r}^{\textsc{e}}_4 (V^{\textsc{e},-}),\ \textbf{r}^{\textsc{e}}_2,\ \textbf{r}^{\textsc{e}}_4(V^{\textsc{e},+}, -\alpha^{\rm{s}}_4) \big)
	&= \frac{2}{(\gamma+1)c}\big(\gamma pu' -cp'\big),\\[5pt]
\det\big(\textbf{r}^{\textsc{e}}_1 (V^{\textsc{e},-}),\ \textbf{r}^{\textsc{e}}_2,\ \textbf{r}^{\textsc{e}}_4(V^{\textsc{e},+}, -\alpha^{\rm{s}}_4) \big)
	& =\frac{2}{(\gamma+1)c}\big(-\gamma pu' -cp'\big).
\end{split}
\end{eqnarray*}
Then, from Lemma \ref{lem:2.1a}, it follows that
\begin{eqnarray*}
|\kappa_{\rm{s}}|=\left|\frac{\gamma pu' -cp'}{\gamma pu' + cp'}\right| <1.
\end{eqnarray*}
By (\ref{partial q1- 1})(\ref{partial q1- 2}) and Taylor formula, we obtain the estimate
\begin{equation}\label{estimate q1- case 4.3.1}
q_1^- = q_1^+ + \kappa_{\rm{s}} q_4^- +\mathcal{O}(1)\big( |q_1^+|^2 + |q_2^+|^2 + |q_4^-|^2\big)
\end{equation}
with $|\kappa_{\rm{s}}|<1$.

\begin{lemma}[Stability condition]\label{Lemma mu1<1}
Assume parameter $\epsilon$ in the conditions $(\rm{C1})$-$(\rm{C3})$ is small enough. Then there holds
\begin{equation}\label{estimate mu*speed ratio}
|\kappa_{\rm{s}}| \cdot \frac{|\lambda_1(V^-) - \dot{x}_\alpha|}{|\lambda_4(V^-) - \dot{x}_\alpha|} <1
\end{equation}
for the large $4$-shock $\alpha^{\rm{s}}_4$ joining two states  ${ V^{-}}\in\mathcal{N}_{\delta}({ \bar{U}}_{\mathfrak{b},\rm{l}})$
and ${ V^{+}}\in\mathcal{N}_{\delta}({ \bar{U}}_{\mathfrak{b},\rm{r}})$.
\end{lemma}

\begin{proof}
Let ${ \rm{s}} = \dot{x}_\alpha$ be the speed of large $4$-shock $\alpha^{\rm{s}}_4$. A straightforward computation gives
\begin{eqnarray*}
\begin{split}
|\kappa_{\rm{s}}| \cdot \frac{|\lambda_1(V^-) - \dot{x}_\alpha|}{|\lambda_4(V^-) - \dot{x}_\alpha|}
	& = \bigg| \frac{\gamma p u' - cp'}{\gamma p u' + cp'}\bigg| \cdot \bigg|\frac{c+({\rm{s}}-u)}{c-({ \rm{s}}-u)}\bigg|\\[5pt]
	& = \bigg| \frac{c\big(\gamma p u' -p'({ \rm{s}}-u)\big) - \big(c^2p' -\gamma p u' ({ \rm{s}}-u)\big)}{c\big(\gamma p u' -p'({ \rm{s}}-u)\big) + \big(c^2p' -\gamma p u' ({ \rm{s}}-u)\big)}\bigg|.
\end{split}
\end{eqnarray*}
Therefore, it suffices to prove
\begin{equation}\label{eq:key estimates}
\gamma p u' -({ \rm{s}}-u)p'>0 ,\qquad  c^2p' -\gamma ({ \rm{s}}-u)p u'>0.
\end{equation}
By the Rankine-Hugoniot condition $F({ V^{-}})-F({ V^{+}})={ \rm{s}}\big(E({ V^{-}})-E({ V^{+}})\big)$, we have

\begin{eqnarray}\label{eq:6.14}
\begin{split}
& [u]^2 = \frac{[\rho] [p]}{\rho\rho^+} ,\\[5pt]
& p'-\rho u'({ \rm{s}}-u)= { \rm{s}}'( [\rho u] -u[\rho]),\\[5pt]
&\gamma p u' -p'({ \rm{s}}-u) = { \rm{s}}'  \Big([p] - \frac{\gamma-1}{2}\cdot \rho^+ [u]^2\Big) .
\end{split}
\end{eqnarray}
The equality $\eqref{eq:6.14}_1$ together with $1< \rho/\rho^+  <  \frac{\gamma+1}{\gamma-1}$ gives
\begin{equation*}\label{RH fact 6}
[p] - \frac{\gamma-1}{2}\cdot \rho^+ [u]^2 = \frac{1}{2} [p]\Big( 3-\gamma   + (\gamma-1)\cdot \frac{\rho^+}{\rho }\Big) >0.
\end{equation*}
By $\eqref{eq:6.14}_2$-$\eqref{eq:6.14}_3$ and ${ \rm{s}}'>0$, we see that the first inequality in \eqref{eq:key estimates} holds. Equality $\eqref{eq:6.14}_1$ also yields
the second inequality in \eqref{eq:key estimates}, i.e.,
\begin{eqnarray*}
c^2p' -\gamma ({ \rm{s}}-u) p u'= c^2 \big(p' - ({ \rm{s}}-u)\rho u'\big)=c^2{ \rm{s}}' ([\rho u] - u [\rho])=c^2 { \rm{s}}' \rho^+ [u] >0.
\end{eqnarray*}
\end{proof}

Now we proceed to prove the estimate \eqref{eq:conclusion Ei} for $x_\alpha=\chi^{\rm{m}}_{\rm{s}}(t)$. Observe that
\begin{align*}
& \lambda_1^- - \dot{x}_\alpha=  \lambda_1(V^-) - \dot{x}_\alpha  + \mathcal{O}(1)\big(|q_1^+| + |q_2^+| +|q_4^-| \big),\\
& \lambda_4^- - \dot{x}_\alpha=   \lambda_4(V^-) - \dot{x}_\alpha  + \mathcal{O}(1)|q_4^-|,
\end{align*}
and then
\begin{equation*}\label{estimate m}
|\kappa_{\rm{s}} |\cdot |\lambda_1^- - \dot{x}_\alpha|
= \frac{|\kappa_{\rm{s}} | \cdot |\lambda_1(V^-) - \dot{x}_\alpha|}{|\lambda_4(V^-) - \dot{x}_\alpha|}
	\cdot |\lambda_4 ^- - \dot{x}_\alpha| + \mathcal{O}(1)\big(|q_1^+| + |q_2^+| +|q_4^-| \big).
\end{equation*}
This together with (\ref{estimate mu*speed ratio}) implies that
\begin{equation}\label{estimate mu1}
|\kappa_{\rm{s}} | \cdot |\lambda_1^- - \dot{x}_\alpha| <\kappa_1 |\lambda_4 ^- - \dot{x}_\alpha|,
\end{equation}
for some positive constant $\kappa_1<1$.
On the other hand, we obviously have
\begin{equation}\label{estimate q case 4.3.1}
	|q_2^+ - q_2^-|= \mathcal{O}(1)| q_4^-|,\ \ \ \ q_3^+ = q_3^-,\ \ \ \
	|\lambda_4^+ - \dot{x}_\alpha|= \mathcal{O}(1)| q_4^-|.
\end{equation}
From (\ref{Def Ei})(\ref{estimate q1- case 4.3.1})(\ref{estimate mu1})(\ref{estimate q case 4.3.1}), the assumption \eqref{eq:4.1} on distinct characteristics fields, weight condition (\ref{condition W}) and
the $\ell_i$-distribution in Table \ref{ell distribution for (C1)}, we deduce
\begin{align*}
\sum_{i=1}^4 \mathscr{E}_i
& \leq -|q_1^+| \cdot\ell_1^+ \mathcal{W}_1^+ |\lambda_1^+ - \dot{x}_\alpha|+ \kappa_1 |q_4^-|\ell_1^- \mathcal{W}_1^- |\lambda_4^- - \dot{x}_\alpha| \\[5pt]
	& \quad  + \Big( |q_1^+| + \mathcal{O}(1) \big(|q_1^+|^2 + |q_2^+|^2 + |q_4^-|^2\big)\Big) \ell_1^- \mathcal{W}_1^- |\lambda_1^- - \dot{x}_\alpha| \\[5pt]
	& \quad - |q_2^+|\ell_2^+ \mathcal{W}_2^+ |\lambda_2^+ - \dot{x}_\alpha| +  \big( |q_2^+|+  \mathcal{O}(1) |q_4^-| \big) \ell_2^- \mathcal{W}_2^- |\lambda_2^- - \dot{x}_\alpha|\\[5pt]
	& \quad -|q_3^+| \ell_3^+ \mathcal{W}_3^+ |\lambda_3^+ - \dot{x}_\alpha| +|q_3^-|\ell_3^- \mathcal{W}_3^- |\lambda_3^- - \dot{x}_\alpha| \\[5pt]
	& \quad + |q_4^+| \ell_4^+ \mathcal{W}_4^+\cdot \mathcal{O}(1) |q_4^-| -  |q_4^-|\ell_4^- \mathcal{W}_4^- |\lambda_4^- - \dot{x}_\alpha| \\[5pt]
	& \leq   |q_1^+| \ell_1^+ \cdot |\lambda_1^+ - \dot{x}_\alpha|\big( - 1 + \mathcal{O}(1)\ell^{-1} \big)\\[5pt]
	& \quad + |q_2^+|\ell_2^+ \cdot|\lambda_2^+ - \dot{x}_\alpha|\big( - 1 +\mathcal{O}(1)(\ell^{-1}+\ell^{-2})\big)\\[5pt]
	& \quad +  |q_3^+|\ell_3^+\cdot |\lambda_3^+ - \dot{x}_\alpha| \big( - 1 + \mathcal{O}(1)\ell^{-1} \big) \\[5pt]
	& \quad + |q_4^-|\ell_4^- \cdot|\lambda_4^- - \dot{x}_\alpha|\Big( - 1 +  \kappa_0^{-1} \kappa_1 \mathcal{W} +   \mathcal{O}(1)\big( \ell^{-1} + |q_4^-|\big)\Big) \\[5pt]
	&<0,
\end{align*}
provided that $\ell$ sufficiently large, the variation of initial-boundary data sufficiently small, and
\begin{equation}\label{mu0 condition 1}
\kappa_0> \kappa_1 \mathcal{W}.
\end{equation}

Suppose $\alpha=\alpha^{\rm{s}}_4$  is the large $4$--shock of ${U}$. Similar to estimates (\ref{estimate q1- case 4.3.1})(\ref{estimate mu1}), we have
\begin{align*}
& q_1^- = - q_1^+ + { \kappa_{\rm{s}}} q_4^- +\mathcal{O}(1) \big( |q_1^+|^2 + |q_2^+|^2 + |q_4^-|^2\big),\\
& |\kappa_{\rm{s}} | \cdot |\lambda_1^- - \dot{x}_\alpha| <\kappa_1 |\lambda_4 ^- - \dot{x}_\alpha|,
\end{align*}
for some positive constant $\kappa_1<1$.
Note that
\begin{equation*}
|q_2^+ - q_2^-|= \mathcal{O}(1)\big( |q_1^+| + |q_2^+|+ |q_4^-|\big),\quad |\lambda_4^+ - \dot{x}_\alpha|= \mathcal{O}(1)\big( |q_1^+| + |q_2^+| + |q_4^-|\big).
\end{equation*}
Hence we deduce
\begin{eqnarray*}
\begin{split}
\sum_{i=1}^4 \mathscr{E}_i
& \leq -|q_1^+| \cdot\ell_1^+ \mathcal{W}_1^+ |\lambda_1^+ - \dot{x}_\alpha|
	+ \kappa_1 | q_4^-|\ell_1^- \mathcal{W}_1^- |\lambda_4^- - \dot{x}_\alpha|\\[5pt]
	& \quad +\Big( |q_1^+|  +\mathcal{O}(1)\big( |q_1^+|^2 + |q_2^+|^2 + |q_4^-|^2\big) \Big)\ell_1^- \mathcal{W}_1^- \cdot|\lambda_1^- - \dot{x}_\alpha|\\[5pt]
	&\quad  -|q_2^+|\ell_2^+ \mathcal{W}_2^+ |\lambda_2^+ - \dot{x}_\alpha| +   \mathcal{O}(1)\big( |q_1^+| + |q_2^+|  + |q_4^-| \big) \ell_2^- \mathcal{W}_2^- |\lambda_2^- - \dot{x}_\alpha|\\[5pt]
	& \quad -|q_3^+|\ell_3^+ \mathcal{W}_3^+ |\lambda_3^+ - \dot{x}_\alpha| +|q_3^-|\ell_3^- \mathcal{W}_3^- |\lambda_3^- - \dot{x}_\alpha| \\[5pt]
	&\quad  +\mathcal{O}(1) |q_4^+| \ell_4^+ \mathcal{W}_4^+ \big( |q_1^+| + |q_2^+|  + |q_4^-|\big)  -   |q_4^-| \ell_4^- \mathcal{W}_4^- |\lambda_4^- - \dot{x}_\alpha| \\[5pt]
	& \leq \sum_{i=1}^2 |q_i^+|\ell_i^+ |\lambda_i^+ - \dot{x}_\alpha|\Big( - 1 +  \mathcal{O}(1) \big(\ell^{-1} + \ell^{-2} \big)  \Big) \\[5pt]
	& \quad +  |q_3^+|\ell_3^+ |\lambda_3^+ - \dot{x}_\alpha| \Big( - 1 + \mathcal{O}(1)\ell^{-1} \Big) \\[5pt]
	&\quad   +  |q_4^-|\ell_4^- |\lambda_4^- - \dot{x}_\alpha|\Big( - 1 + \kappa_0^{-1} \kappa_1 \mathcal{W} + \mathcal{O}(1)\big(\ell^{-1} + |q_4^-|\big)  \Big) \\[5pt]
	& <0,
\end{split}
\end{eqnarray*}
provided that $\ell \gg 1 \gg \epsilon$ and $\kappa_0> \kappa_1 \mathcal{W}$.


\begin{figure}[ht]
\begin{minipage}[t]{0.45\textwidth}
\centering
\begin{tikzpicture}[scale=0.7]
\draw [thin][->] (-4.3, -1)--(4, -1);
\draw [thin][line width=0.02cm](-3.5,0.7)--(-1, 0.7)--(-1,-0.3)--(3.5,-0.3);
\draw [thin][line width=0.02cm](-3.5,1.5)--(1.5, 1.5)--(1.5,0.2)--(3.5,0.2);
\draw[thin][dotted][line width=0.03cm](-1,-0.3)--(-1,-1);
\draw[thin][dotted][line width=0.03cm](1.5,0.2)--(1.5,-1);
\draw [thick][->] (-1, 0.2)--(0.2, 0.2);
\draw [thick][->] (1.5, 0.9)--(2.8, 0.9);
\node at (-0.9,-1.4) {$\chi^{\textrm{m}}_{\rm{s}}$};
\node at (1.6,-1.4) {$\chi^{\textsc{m}}_{\rm{s}}$};
\node at (4,-1.3) {$x$};
\node at (2.2,1.3) {$\alpha^{s}_{4}$};
\node at (-3.8,1.5) {$U$};
\node at (-3.8,0.7) {$V$};
\end{tikzpicture}
\caption{Flow comparison for Case 5.1.2}\label{Case 6.1.2: flow comparison}
\end{minipage}
\begin{minipage}[t]{0.5\textwidth}
\centering
\begin{tikzpicture}[scale=0.7]
\draw [thick][dashed](0,1) circle [radius=1.6];
\draw [thick][dashed](5.5,1) circle [radius=2.0];
\draw [thin][->](0.2,1.8) to [out=-30, in=190](5.4, 1.8);
\draw [thin][->] (0.2, 1.8)to [out=-110, in=30](-0.6, 0.8)to[out=-30, in=180](5.4, 0);
\draw [thin][->] (5.5,1.8)to [out=-30, in=120](6.5, 1.0)to[out=-120, in=30](5.6,0.0);

\node at (0.2,1.8) {$\bullet$};
\node at (5.5,1.8) {$\bullet$};
\node at (5.5,0.0) {$\bullet$};

\node at (7.0, 1.1) {$\mathbf{q}^{+}$};
\node at (3.0, -0.4) {$\mathbf{q}^{-}$};
\node at (2.7, 1.8) {$\alpha^{s}_{4}$};
\node at (0.1, 2.1) {$U^{\textsc{e},-}$};
\node at (5.9, 2.3) {$U^{\textsc{e},+}$};
\node at (5.5, -0.5) {$V^{\textsc{e}}$};
\node at (0.2,-1.7) {$\mathcal{N}_{\delta}(\bar{U}^{\textsc{e}}_{\mathfrak{b},\rm{l}})$};
\node at (5.8,-1.7) {$\mathcal{N}_{\delta}(\bar{U}^{\textsc{e}}_{\mathfrak{b},\rm{r}})$};
\end{tikzpicture}
\caption{Projected phase space for Case 5.1.2}\label{Case 6.1.2: projected phase space}
\end{minipage}
\end{figure}

\textbf{Case 5.1.2}. \emph{Front $x_\alpha = \chi^{\textsc{m}}_{\rm{s}}(t)$.} Assume that $\alpha=\alpha^{\rm{s}}_4$  is the large $4$--shock of ${U}$ or ${V}$
(see Figs \ref{Case 6.1.2: flow comparison}-\ref{Case 6.1.2: projected phase space}). In this case, we have the estimates
\begin{eqnarray*}
|q_1^-| +|q_2^-| = \mathcal{O}(1) \sum_{i\neq3} |q_i^+ |, \quad |q_3^+| = |q_3^-|, \quad |\lambda_4^- - \dot{x}_\alpha| = \mathcal{O}(1) \sum_{i\neq3} |q_i^+|,
\end{eqnarray*}
which yield that
\begin{align*}
\sum_{i=1}^4 \mathscr{E}_i
& \leq -|q_1^+|\ell_1^+ \mathcal{W}_1^+ |\lambda_1^+ - \dot{x}_\alpha| + \mathcal{O}(1) \sum_{i\neq3} |q_i^+|\ell_1^- \mathcal{W}_1^- |\lambda_1^- - \dot{x}_\alpha|\\[5pt]
	&\quad -|q_2^+|\ell_2^+ \mathcal{W}_2^+|\lambda_2^+ - \dot{x}_\alpha| + \mathcal{O}(1) \sum_{i\neq3} |q_i^+| \ell_2^- \mathcal{W}_2^- |\lambda_2^- - \dot{x}_\alpha|\\[5pt]
	& \quad -|q_3^+|\ell_3^+ \mathcal{W}_3^+ |\lambda_3^+ - \dot{x}_\alpha| +|q_3^-|\ell_3^- \mathcal{W}_3^- |\lambda_3^- - \dot{x}_\alpha| \\[5pt]
	&\quad  -|q_4^+|\ell_4^+ \mathcal{W}_4^+ |\lambda_4^+ - \dot{x}_\alpha| + \mathcal{O}(1) |q_4^-|\ell_4^- \mathcal{W}_4^-  \sum_{i\neq3} |q_i^+| \\[5pt]
	& \leq \sum_{i=1}^4 |q_i^+|\ell_i^+ |\lambda_i^+ - \dot{x}_\alpha|\Big( - 1 +  \mathcal{O}(1) \big(\ell^{-1} + \ell^{-3} \big)  \Big) \\[5pt]
	& <0.
\end{align*}


\begin{figure}[ht]
\begin{minipage}[t]{0.45\textwidth}
\centering
\begin{tikzpicture}[scale=0.7]
\draw [thin][->] (-4.3, -1)--(4, -1);
\draw [thin][line width=0.02cm](-3.5,0.7)--(-1.5, 0.7)--(-1.5,-0.3)--(3,-0.3);
\draw [thin][line width=0.02cm](-3.5,2.0)--(0,2.0)--(0,1.5)--(1.5, 1.5)--(1.5,0.2)--(3,0.2);
\draw[thin][dotted][line width=0.03cm](-1.5,-0.3)--(-1.5,-1);
\draw[thin][dotted][line width=0.03cm](0,1.5)--(0,-1);
\draw[thin][dotted][line width=0.03cm](1.5,0.2)--(1.5,-1);

\draw [thick][->] (-1.5, 0.2)--(-0.4, 0.2);
\draw [thick][->] (1.5, 0.9)--(2.4, 0.9);
\node at (-1.4,-1.4) {$\chi^{\textrm{m}}_{\rm{s}}$};
\node at (0.1,-1.45) {$x_{\alpha}$};
\node at (1.6,-1.4) {$\chi^{\textsc{m}}_{\rm{s}}$};
\node at (4,-1.3) {$x$};
\node at (0.1,2.3) {$\alpha$};
\node at (-3.8,2.0) {$U$};
\node at (-3.8,0.7) {$V$};
\end{tikzpicture}
\caption{Flow comparison for Case 5.1.3}\label{Case 6.1.3: flow comparison}
\end{minipage}
\begin{minipage}[t]{0.5\textwidth}
\centering
\begin{tikzpicture}[scale=0.7]
\draw [thick][dashed](0,1) circle [radius=2];
\draw [thick][dashed](5.5,1) circle [radius=1.5];
\draw [thin][->](-1.4, 1.0)to[out=10, in=240](0.2,2.2) to [out=-30, in=180](5.7, 1.1);
\draw [thin][->] (-0.2, -0.2)to [out=40, in=260](0.8, 1.2)to[out=-30, in=180](5.7, 0.9);
\draw [thin][->](-0.2, -0.2)to [out=160, in=-80](-1.4, 0.9);

\node at (-1.4, 1.0) {$\bullet$};
\node at (-0.2,-0.2) {$\bullet$};
\node at (5.8,1) {$\bullet$};
\node at (-1.2, 1.5) {$U^{\textsc{e},+}$};
\node at (0, -0.55) {$U^{\textsc{e},-}$};
\node at (3.0, 1.7) {$\mathbf{q}^{+}$};
\node at (3.0, 0.3) {$\mathbf{q}^{-}$};
\node at (-0.5, 0.5) {$\alpha$};
\node at (5.9, 1.55) {$V^{\textsc{e}}$};
\node at (0.2,-1.7) {$\mathcal{N}_{\delta}(\bar{U}^{\textsc{e}}_{\mathfrak{b},\rm{l}})$};
\node at (5.8,-1.5) {$\mathcal{N}_{\delta}(\bar{U}^{\textsc{e}}_{\mathfrak{b},\rm{r}})$};
\end{tikzpicture}
\caption{Projected phase space for Case 5.1.3}\label{Case 6.1.3: projected phase space}
\end{minipage}
\end{figure}

\textbf{Case 5.1.3}. \emph{Front $x_\alpha \in ({\chi}^{\rm{m}}_{\rm{s}}(t), {\chi}^{\textsc{m}}_{\rm{s}}(t))$}. We start with the case $k_\alpha \in \{1,2,4\}$
({see} Figs \ref{Case 6.1.3: flow comparison}-\ref{Case 6.1.3: projected phase space}). For $i<3$, there hold
\begin{eqnarray*}
|q_i^+ - q_i^-|=\mathcal{O}(1)|\alpha|, \quad |\lambda_i^+ - \lambda_i^-|=\mathcal{O}(1)|\alpha|,
\end{eqnarray*}
and
\begin{equation*}
|\mathcal{W}_i^+ - \mathcal{W}_i^-| = \left\{
	\begin{aligned}
		& \mathcal{K} |\alpha |  \qquad &&  \mbox{if}\ i\neq k_\alpha,\\[5pt]
		&  \mathcal{K} |\alpha |  \qquad &&\mbox{if}\ i= k_\alpha=1\ \mbox{and}\ q_i^+ q_i^->0,\\[5pt]
		&  0  \qquad &&\mbox{if}\ i= k_\alpha=2.
	\end{aligned}
	\right.
\end{equation*}
If $q_i^+ q_i^->0$, then
\begin{align*}
\mathscr{E}_i & =|q_i^+| \ell_i \big( \mathcal{W}_i^+ - \mathcal{W}_i^- \big) \big(\lambda_i^+ - \dot{x}_\alpha\big)
	+\ell_i \mathcal{W}_i^- \Big( \big(|q_i^+| -|q_i^-| \big) \big(\lambda_i^+ - \dot{x}_\alpha\big)  + |q_i^-|\big(\lambda_i^+ - \lambda_i^-\big)\Big) \\
	& \leq |q_i^+|\ell_i \mathcal{K} |\alpha| \bar{d} +\ell_i \mathcal{W} \big( |q_i^+- q_i^-|\overline{d} + |q_i^-||\lambda_i^+ - \lambda_i^-| \big) \\
	& \leq \mathcal{O}(1) \ell^2  |\alpha|.
\end{align*}
If $q_i^+ q_i^- \leq 0$, then the relation $|q_i^+| +|q_i^-| = |q_i^+ - q_i^-| =\mathcal{O}(1)|\alpha|$ implies
\begin{eqnarray*}
\begin{split}
\mathscr{E}_i & = |q_i^+| \ell_i  \mathcal{W}_i^+ (\lambda_i^+ - \dot{x}_\alpha)- |q_i^-|\ell_i \mathcal{W}_i^- (\lambda_i^+ - \dot{x}_\alpha) \\[5pt]
	& \leq  \ell^2  \mathcal{W} \big(|q_i^+| |\lambda_i^+ - \dot{x}_\alpha| + |q_i^-| |\lambda_i^- - \dot{x}_\alpha| \big) \\[5pt]
	& \leq \mathcal{O}(1) \ell^2  |\alpha|.
\end{split}
\end{eqnarray*}

For $i=3$, we have $\mathcal{W}_3^+ - \mathcal{W}_3^- = \mathcal{K} |\alpha |\text{sgn} ( k_\alpha -2)$ and $|q_3^+|=|q_3^-|$. Then
\begin{eqnarray}\label{estimate E3 in case 6.1.3}
\begin{split}
\mathscr{E}_3 & = \ell_3 |q_3^+| \Big( \big(\mathcal{W}_3^+ -\mathcal{W}_3^-\big)\big(\lambda_3^+ -\dot{x}_\alpha \big)+ \mathcal{W}_3^- \big(\lambda_3^+ -\lambda_3^-\big) \Big)\\[5pt]
	& \leq \ell_3 |q_3^+| \Big( \mathcal{K} \underline{d}|\alpha| + \mathcal{O}(1) \mathcal{W} |\alpha|\Big) \\[5pt]
&= \mathcal{O}(1) \ell |\alpha|,
\end{split}
\end{eqnarray}
because $\ell_3 \mathcal{K} = \ell$ according to Table \ref{ell distribution for (C1)}.

For $i=4$, the weight
\begin{equation*}
\mathcal{W}_4^+ - \mathcal{W}_4^- = \left\{
\begin{aligned}
		&  -\mathcal{K} |\alpha | \qquad  &&\mbox{if }\ k_\alpha<3,\\[5pt]
		&  \mathcal{K} |\alpha | \qquad  &&  \mbox{if }\ k_\alpha=4 \ \mbox{and }\ \alpha \in \mathcal{J}^{\textsc{e}}_{\rm{l}}, \\[5pt]
		&  -\mathcal{K} |\alpha | \qquad &&  \mbox{if }\ k_\alpha=4 \ \mbox{and }\ \alpha \in \mathcal{J}^{\textsc{e}}_{\rm{r}},
	\end{aligned}
	\right.
\end{equation*}
implies
\begin{eqnarray*}
\begin{split}
\mathscr{E}_4 &= |q_4^+| \ell_4 \big( \mathcal{W}_4^+ - \mathcal{W}_4^- \big) \big(\lambda_4^+ - \dot{x}_\alpha\big)
+\ell_4 \mathcal{W}_4^- \Big( \big(|q_4^+| -|q_4^-| \big) \big(\lambda_4^+ - \dot{x}_\alpha\big)  + |q_4^-|\big(\lambda_4^+ - \lambda_4^-\big)\Big) \\[5pt]
	& \leq -\mathcal{K}|\alpha||q_4^+| \underline{d} + \mathcal{W} \big( |q_4^+- q_4^-|\bar{d} + \mathcal{O}(1)|q_4^-| |\alpha| \big) \\[5pt]
	& \leq -\mathcal{K}|\alpha| |q_4^+| \underline{d} + \mathcal{O}(1)|\alpha|.
\end{split}
\end{eqnarray*}

In summary, we conclude that if $k_\alpha\in \{1,2,4\}$, there holds
\begin{equation*}\label{Estimate sum Ei for Euler case 5}
\sum_{i=1}^4 \mathscr{E}_i \leq  -\mathcal{K} \underline{d} |\alpha| |q_4^+|  + \mathcal{O}(1)\ell^2 |\alpha| <0.
\end{equation*}

When $k_\alpha =3$, it is clear that
\begin{eqnarray*}
q_i^+ = q_i^-,\quad \lambda_i^+ = \lambda_i^-,\quad  \mathcal{W}^+_i- \mathcal{W}^-_i = \mathcal{K} \theta_3 |\alpha|\text{sgn}(2-i),
\end{eqnarray*}
for every $i\neq 3$. Note that $|q_3^+- q_3^-|= |\alpha|$ and $\lambda_3^+ =\lambda_3^- $. Then we have
\begin{eqnarray*}
\begin{split}
\mathscr{E}_1 &  \leq  \ell_1 \mathcal{K} \theta_3  |\alpha||q_1^+| \bar{d} = \mathcal{O}(1) \ell_1  |\alpha|,\quad \mathscr{E}_2  =0,\\[5pt]
\mathscr{E}_3 & = \ell_3 \mathcal{W}_3^+ (|q_3^+|- |q_3^-|)(\lambda_3^+ - \dot{x}_\alpha) =  \mathcal{O}(1) \ell_3 |\alpha|,\\[5pt]
\mathscr{E}_4 &  \leq - \ell_4 \mathcal{K} \theta_3|\alpha||q_4^+| \underline{d}.
\end{split}
\end{eqnarray*}
From the above estimates, it follows that
\begin{eqnarray*}
\begin{split}
\sum_{i=1}^4 \mathscr{E}_i \leq \mathcal{O}(1) \ell^2 |\alpha|-\frac{\mathcal{K}}{\ell} \underline{d} |\alpha||q_4^+| <0,
\end{split}
\end{eqnarray*}
due to the coefficients condition (\ref{condition kappa ell}).

When $k_\alpha =5$, that is $\alpha\in \mathcal{NP}$, then
\begin{eqnarray*}
\mathcal{W}_i^+ = \mathcal{W}_i^- , \quad q_i^+- q_i^- = \mathcal{O}(1)|\alpha| , \quad  \lambda_i^+ - \lambda_i^- = \mathcal{O}(1)|\alpha|,
\end{eqnarray*}
and
\begin{eqnarray*}
\begin{split}
\mathscr{E}_i &= \ell_i \mathcal{W}_i \Big( (|q_i^+| -|q_i^-| ) (\lambda_i^+ - \dot{x}_\alpha)  + |q_i^-|  (\lambda_i^+ - \lambda_i^-)\Big)= \mathcal{O}(1)  |\alpha|
\end{split}
\end{eqnarray*}
for $1\leq i \leq 4$. Therefore
$$
\sum_{i=1}^4 \mathscr{E}_i \leq \mathcal{O}(1)  |\alpha|.
$$


\begin{figure}[ht]
\begin{minipage}[t]{0.45\textwidth}
\centering
\begin{tikzpicture}[scale=0.7]
\draw [thin][->] (-4.3, -1)--(4, -1);
\draw [thin][line width=0.02cm](-3.5,0.7)--(-0.5, 0.7)--(-0.5,-0.3)--(3,-0.3);
\draw [thin][line width=0.02cm](-3.5,2.0)--(-2,2.0)--(-2,1.5)--(1.5, 1.5)--(1.5,0.2)--(3,0.2);
\draw[thin][dotted][line width=0.03cm](-0.5,-0.3)--(-0.5,-1);
\draw[thin][dotted][line width=0.03cm](-2,1.5)--(-2,-1);
\draw[thin][dotted][line width=0.03cm](1.5,0.2)--(1.5,-1);

\draw [thick][->] (-0.5, 0.2)--(0.6, 0.2);
\draw [thick][->] (1.5, 0.9)--(2.4, 0.9);
\node at (-0.4,-1.4) {$\chi^{\textrm{m}}_{{\rm{s}}}$};
\node at (-2,-1.45) {$x_{\alpha}$};
\node at (-2,2.3) {$\alpha$};
\node at (1.6,-1.4) {$\chi^{\textsc{m}}_{{\rm{s}}}$};
\node at (4,-1.3) {$x$};
\node at (-3.8,2.0) {$U$};
\node at (-3.8,0.7) {$V$};
\end{tikzpicture}
\caption{Flow comparison for Case 5.1.4}\label{Case 6.1.4: flow comparison}
\end{minipage}
\begin{minipage}[t]{0.5\textwidth}
\centering
\begin{tikzpicture}[scale=0.7]
\draw [thick][dashed](1,1) circle [radius=2.5];
\draw [thin][->](2.2,1.2)to[out=190, in=10](1.0,0.8) to [out=160, in=-80](0, 2.2);
\draw [thin][->] (1.3, -0.3)to [out=180, in=0](-0.5, -0.2)to[out=100, in=210](-0.1, 2.3);
\draw [thin][->](1.3, -0.3)to [out=30, in=260](2.2, 1.1);

\node at (0, 2.3) {$\bullet$};
\node at (1.3,-0.3) {$\bullet$};
\node at (2.2,1.2) {$\bullet$};
\node at (1.6, -0.7) {$U^{\textsc{e},-}$};
\node at (2.9, 1.3) {$U^{\textsc{e},+}$};
\node at (0, 2.7) {${ V^{\textsc{e}}}$};
\node at (1.2, 1.5) {$\mathbf{q}^{+}$};
\node at (0.2, 0.3) {$\mathbf{q}^{-}$};
\node at (2.2, 0.1) {$\alpha$};

\node at (1.5,-2.3) {$\mathcal{N}_{\delta}(\bar{U}^{\textsc{e}}_{\mathfrak{b},\rm{l}})$};
\end{tikzpicture}
\caption{Projected phase space for Case 5.1.4}\label{Case 6.1.4: projected phase space}
\end{minipage}
\end{figure}

\textbf{Case 5.1.4}. \emph{Front $x_\alpha \in ({\chi}^{\textsc{m}}_{\mathfrak{p}}(t), { {\chi}^{\rm{m}}_{{\rm{s}}}}(t))\ \text{or}\  ({\chi}^{\textsc{m}}_{{\rm{s}}}(t), + \infty)$}.
This case is similar to the small initial data problems in \cite{Bressan-Liu-Yang-1999}, since projected states $U^{\textsc{e}}, V^{\textsc{e}}$ are located in
the identical domain $\mathcal{N}_{\delta}(\bar{U}^{\textsc{e}}_{\mathfrak{b},\rm{l}})$ or $\mathcal{N}_{\delta}(\bar{U}^{\textsc{e}}_{\mathfrak{b},\rm{r}})$
(see Figs \ref{Case 6.1.4: flow comparison}-\ref{Case 6.1.4: projected phase space}).

If $k_\alpha=1$ or $k_\alpha=4$, we deduce that
\begin{eqnarray*}
\begin{split}
\mathscr{E}_3 &= \ell_3 |q_3^+| \Big( \big(\mathcal{W}_3^+ - \mathcal{W}_3^-\big)\big(\lambda_3^+ - \dot{x}_\alpha\big)  + \mathcal{W}_3^- \big(\lambda_3^+ - \lambda_3^-\big) \Big)\\[5pt]
	&= \ell_3 |q_3^+| \big( -\mathcal{K} |\alpha| \underline{d} + \mathcal{O}(1)  |\alpha| \big) \\[5pt]
	&\leq -\frac{1}{2} \mathcal{K} \underline{d}\ell_3|\alpha| |q_3^+|  \\[5pt]
&<0 .
\end{split}
\end{eqnarray*}
By the same argument in \cite{Bressan-2000, Bressan-Liu-Yang-1999}, it is trivial that
$ \mathop{\sum}_{i\neq3} \mathscr{E}_{i}< \mathcal{O}(1) \varepsilon|\alpha|$ which leads to estimate $\eqref{eq:conclusion Ei}_1$.

If $k_\alpha =2 $, then $\lambda_3^+ - \lambda_3^- =  \mathcal{O}(1) |\alpha| \big(|q_1^+|+ |q_4^+|\big)$ which gives
\begin{eqnarray*}
\mathscr{E}_3 = \ell_3 |q_3^+|  \mathcal{W}_3^+ \big(\lambda_3^+ - \lambda_3^-\big)
\leq  \mathcal{O}(1)\ell_3 |\alpha| |q_3^+| \big(|q_1^+|+ |q_4^+|\big).
\end{eqnarray*}
On the other hand, based on the estimates $(8.59)$ and $(8.60)$ in \cite{Bressan-2000}, we see that
\begin{eqnarray*}
\sum_{i\neq3} \mathscr{E}_i  \leq  -\frac{1}{2}\mathcal{K} \underline{d} |\alpha|\big(|q_1^+| + |q_4^+|\big)
+ \mathcal{O}(1)\sum_{i\neq3} \ell_i |\alpha|\big(|q_1^+| + |q_4^+|\big)<0.
\end{eqnarray*}
Thus, combining the above two estimates altogether and taking $\mathcal{K}\gg\sum_{i=1}^4 \ell_i$, we obtain $\sum_{i=1}^4 \mathscr{E}_i <0$.

If $k_\alpha =3$, then there hold
\begin{eqnarray*}
&q_i^+ = q_i^-,\quad  \lambda_i^+ = \lambda_i^- ,\quad \mathcal{W}^+_i- \mathcal{W}^-_i= \mathcal{K} \theta_3 |\alpha| \text{sgn}(2-i) \quad \mbox{for}\  i\neq 3,
\end{eqnarray*}
and
\begin{eqnarray*}
&|q_3^+- q_3^-|= |\alpha|, \quad \lambda_3^+ = \lambda_3^-, \quad \lambda_3^+ - \dot{x}_\alpha =\mathcal{O}(1) \big(|q_1^+| + |q_4^+|\big).
\end{eqnarray*}
Owing to condition (\ref{condition kappa ell}) and above estimates, we have
\begin{align*}
&\mathscr{E}_1  \leq - \ell_1 \mathcal{K} \theta_3 |\alpha||q_1^+| \underline{d} \leq -  \mathcal{K}  |\alpha| |q_1^+| \underline{d},\qquad
\mathscr{E}_2  =0 ,\\[5pt]
&\mathscr{E}_3  \leq  \mathcal{O}(1) \ell_3 \mathcal{W}|\alpha| \big(|q_1^+| + |q_4^+|\big),\qquad
\mathscr{E}_4   \leq - \ell_4 \mathcal{K} \underline{d}\theta_3 |\alpha||q_4^+|\leq -  \mathcal{K}{{ \kappa_0}}  \underline{d}|\alpha||q_4^+|.
\end{align*}
Therefore, by choosing $\mathcal{K} \gg \ell_3$, it follows that
\begin{eqnarray*}
\begin{split}
\sum_{i=1}^4 \mathscr{E}_i \leq \mathcal{O}(1) \ell_3 |\alpha|\big(|q_1^+| + |q_4^+|\big)- \mathcal{K} {{ \kappa_0}} \underline{d}|\alpha|\big(|q_1^+| + |q_4^+|\big) <0 .
\end{split}
\end{eqnarray*}

If $k_\alpha=5$, the argument for $\alpha\in \mathcal{NP}$ can be repeated as in $\emph{Case 5.1.3}$. Hence we can also derive the estimate
 $\eqref{eq:conclusion Ei}_2$.

\textbf{Case 5.1.5}. \emph{Front $x={\chi}^{\textsc{m}}_{\mathfrak{p}}(t)$}.
Set ${U}^{\textsc{e}}_{\mathfrak{p}} \doteq U^{\textsc{e}}(t, \chi^{\textsc{m}}_{\mathfrak{p}}(t)+0)$ and  ${V}^{\textsc{e}}_{\mathfrak{p}} \doteq V^{\textsc{e}}(t, \chi^{\textsc{M}}_{\mathfrak{p}}(t)+0)$.
They satisfy
\begin{eqnarray}\label{eq:L1-boundary}
{V}^{\textsc{e}}_{\mathfrak{p}} =  S^{\textsc{e}}_4(q_4^{\mathfrak{p}+})\circ  S^{\textsc{e}}_2(q_2^{\mathfrak{p}+})\circ S^{\textsc{e}}_1(q_1^{\mathfrak{p}+})({U}^{\textsc{e}}_{\mathfrak{p}}).
\end{eqnarray}
Let the projected intermediate state $\omega^{\textsc{e}}_2 \doteq S^{\textsc{e}}_2(q_2^{\mathfrak{p}+})\circ S^{\textsc{e}}_1(q_1^{\mathfrak{p}+})({U}^{\textsc{e}}_{\mathfrak{p}})$. For fixed state ${U}^{\textsc{e}}_{\mathfrak{p}}\in \mathcal{N}_{\delta}({ \bar{U}^{\textsc{e}}_{\mathfrak{b},\rm{l}}})$, we find that
$q_4^{\mathfrak{p}+}$ can be uniquely determined by $q_1^{\mathfrak{p}+}, q_2^{\mathfrak{p}+}$ and $\Delta u^{\textsc{m}}_{\mathfrak{p}} \doteq u_{1}(t, \chi^{\textsc{m}}_{\mathfrak{p}}(t)+0) - u_{2}(t,\chi^{\textsc{m}}_{\mathfrak{p}}(t)+0) $. Next, we begin to evaluate $q_4^{\mathfrak{p}+}$. To this end, we can rewrite \eqref{eq:L1-boundary} as
\begin{eqnarray*}
\int_0^{q_1^{\mathfrak{p}+}} \textbf{r}^{\textsc{e}}_1(U^{\textsc{e}}_{\mathfrak{p}}, \sigma)d\sigma +q_2^{\mathfrak{p}+}\textbf{r}^{\textsc{e}}_2
+\int_0^{q_4^{\mathfrak{p}+}} \textbf{r}^{\textsc{e}}_4(\omega_2^{\textsc{e}} , \sigma) d\sigma=V^{\textsc{e}}_{\mathfrak{p}} - U^{\textsc{e}}_{\mathfrak{p}}.
\end{eqnarray*}
Multiply it by $\textbf{n}= (0,1,0)$ to get
\begin{equation}\label{Eq state on chi-}
\textbf{n} \cdot \int_0^{q_1^{\mathfrak{p}+}} \textbf{r}^{\textsc{e}}_1(U^{\textsc{e}}_{\mathfrak{p}}, \sigma)d\sigma
	+ \textbf{n} \cdot \int_0^{q_4^{\mathfrak{p}+}} \textbf{r}^{\textsc{e}}_4(\omega_2^{\textsc{e}} , \sigma) d\sigma
	= -\Delta u^{\textsc{m}}_{\mathfrak{p}}.
\end{equation}
By (\ref{Eq state on chi-}), we figure out that
\begin{eqnarray*}
\begin{split}
&\frac{\partial q_4^{\mathfrak{p}+}}{\partial q_1^{\mathfrak{p}+}} \bigg|_{q_1^{\mathfrak{p}+}=q_2^{\mathfrak{p}+}=0, \ \Delta u^{\textsc{m}}_{\mathfrak{p}}=0}= - \frac{\textbf{n}\cdot\text{r}^{\textsc{e}}_1(U^{\textsc{e}}_{\mathfrak{p}})}{\textbf{n}\cdot\textbf{r}^{\textsc{e}}_4(U^{\textsc{e}}_{\mathfrak{p}})} =-1 , \\ & \left(\frac{\partial }{\partial {q_2^{\mathfrak{p}+}} }\right) ^k q_4^{\mathfrak{p}+} \bigg|_{q_1^{\mathfrak{p}+}=q_2^{\mathfrak{p}+}=0, \ \Delta u^{\textsc{m}}_{\mathfrak{p}}=0} = 0,\quad k=1,2,\cdots,
\end{split}
\end{eqnarray*}
and
\begin{eqnarray*}
\frac{\partial q_4^{\mathfrak{p}+}}{\partial \Delta u^{\textsc{m}}_{\mathfrak{p}}} \bigg|_{q_1^{\mathfrak{p}+}=q_2^{\mathfrak{p}+}=0, \ \Delta u^{\textsc{m}}_{\mathfrak{p}}=0} = -\frac{1}{\textbf{n}\cdot\textbf{r}^{\textsc{e}}_4(U^{\textsc{e}}_{\mathfrak{p}})} = -\frac{\gamma+1}{2}.
\end{eqnarray*}
Hence there hold
\begin{eqnarray*}
q_4^{\mathfrak{p}+} = -q_1^{\mathfrak{p}+} -\frac{\gamma+1}{2} \Delta u^{\textsc{m}}_{\mathfrak{p}} + \mathcal{O}(1) \Big(|q_1^{\mathfrak{p}+}|^2 + |\Delta u^{\textsc{m}}_{\mathfrak{p}}|^2 + |q_1^{\mathfrak{p}+} q_2^{\mathfrak{p}+}|
+ |q_1^{\mathfrak{p}+}\Delta u^{\textsc{m}}_{\mathfrak{p}}|+ |q_2^{\mathfrak{p}+}\Delta u^{\textsc{m}}_{\mathfrak{p}}|\Big),
\end{eqnarray*}
and
\begin{equation}\label{estimate q3+ for case 4.3.7}
|q_4^{\mathfrak{p}+} |\leq |q_1^{\mathfrak{p}+}| +  \mathcal{O}(1)|\Delta u^{\textsc{m}}_{\mathfrak{p}}| + \mathcal{O}(1)\Big(|q_1^{\mathfrak{p}+}|+| q_2^{\mathfrak{p}+}|\Big)|q_1^{\mathfrak{p}+}|.
\end{equation}
Then it follows from (\ref{estimate q3+ for case 4.3.7}) that
\begin{align}
\notag
|\lambda_2^{\mathfrak{p}+} - \dot{\chi}^{\textsc{m}}_{\mathfrak{p}}(t)| & \leq |\lambda_2^{\mathfrak{p}+} - \lambda_2 ({U}^{\textsc{e}}_{\mathfrak{p}})|
+ |\lambda_2 ({U}^{\textsc{e}}_{\mathfrak{p}})- \dot{\chi}^{\textsc{m}}_{\mathfrak{p}}(t)| \\[5pt]
\notag		& \leq \mathcal{O}(1) \max\big\{|q_1^{\mathfrak{p}+} | , |q_4^{\mathfrak{p}+} | \big\} + |\Delta u^{\textsc{m}}_{\mathfrak{p}}|\\[5pt]
\label{estimate lamda2-chi}		& \leq \mathcal{O}(1) \big(|q_1^{\mathfrak{p}+}| + |\Delta u^{\textsc{m}}_{\mathfrak{p}}|\big).
\end{align}
We further find the relation
\begin{equation}\label{estimate lamda3-chi}
|\lambda_1^{\mathfrak{p}+}-\dot{\chi}^{\textsc{m}}_{\mathfrak{p}}(t)| - |\lambda_4^{\mathfrak{p}+} -\dot{\chi}^{\textsc{m}}_{\mathfrak{p}}(t)|
 = \mathcal{O}(1) \big(|q_1^{\mathfrak{p}+}| +|q_2^{\mathfrak{p}+}|+ |\Delta u^{\textsc{m}}_{\mathfrak{p}}|\big) .
\end{equation}
From (\ref{estimate q3+ for case 4.3.7})-(\ref{estimate lamda3-chi}), we deduce

\begin{eqnarray}\label{Estimate sum Ei on boundary}
\begin{split}
\sum_{i\neq3}\mathscr{E}_{i,\mathfrak{p}} &= \sum_{i\neq3} |q_i^{\mathfrak{p}+}| \ell_i^{\mathfrak{p}+} \mathcal{W}_i ^{\mathfrak{p}+}\big(\lambda_i^{\mathfrak{p}+}-\dot{\chi}^{\textsc{m}}_{\mathfrak{p}}(t)\big)\\[5pt]
	& \leq -|q_1^{\mathfrak{p}+}|\ell_1^{\mathfrak{p}+} \mathcal{W}_1^{\mathfrak{p}+}|\lambda_1^{\mathfrak{p}+} -\dot{\chi}^{\textsc{m}}_{\mathfrak{p}}(t)|
		+\mathcal{O}(1) |q_2^{\mathfrak{p}+}|\ell_2^{\mathfrak{p}+}\mathcal{W}_2^{\mathfrak{p}+} \big(|q_1^{\mathfrak{p}+}| + |\Delta u^{\textsc{m}}_{\mathfrak{p}}|\big)\\[5pt]
		&\quad+\Big( |q_1^{\mathfrak{p}+}| +  \mathcal{O}(1)|\Delta u^{\textsc{m}}_{\mathfrak{p}}| + \mathcal{O}(1)\big(|q_1^{\mathfrak{p}+}|+| q_2^{\mathfrak{p}+}|\big)|q_1^{\mathfrak{p}+}| \Big)
\ell_4^{\mathfrak{p}+} \mathcal{W}_4^{\mathfrak{p}+} \\[5pt]
		&\qquad\cdot \Big(|\lambda_1^{\mathfrak{p}+} - \dot{\chi}^{\textsc{m}}_{\mathfrak{p}}(t)| + \mathcal{O}(1) \big(|q_1^{\mathfrak{p}+}| +| q_2^{\mathfrak{p}+}| + |\Delta u^{\textsc{m}}_{\mathfrak{p}}|\big) \Big)\\[5pt]
&\leq \Big(-\ell_1^{\mathfrak{p}+} \mathcal{W}_1^{\mathfrak{p}+}+\ell_4^{\mathfrak{p}+} \mathcal{W}_4^{\mathfrak{p}+}\Big)|\lambda_1^{\mathfrak{p}+} - \dot{\chi}^{\textsc{m}}_{\mathfrak{p}}(t)||q_1^{\mathfrak{p}+}|\\[5pt]
&\quad +\mathcal{O}(1)\Big(|q_1^{\mathfrak{p}+}|+|q_2^{\mathfrak{p}+}|+|\Delta u^{\textsc{m}}_{\mathfrak{p}}|\Big)|q_1^{\mathfrak{p}+}|+\Big(\ell_4^{\mathfrak{p}+} \mathcal{W}_4^{\mathfrak{p}+}|\lambda_1^{\mathfrak{p}+} - \dot{\chi}^{\textsc{m}}_{\mathfrak{p}}(t)|+|q_2^{\mathfrak{p}+}|\Big)|\Delta u^{\textsc{m}}_{\mathfrak{p}}|
\\[5pt]
		& \leq \Big( - 1 + \kappa_0 \mathcal{W}+ \mathcal{O}(1)\big(\ell^{-1} + |q_1^{\mathfrak{p}+}| +|q_2^{\mathfrak{p}+}| + |\Delta u^{\textsc{m}}_{\mathfrak{p}}|  \big)  \Big) |\lambda_1^{\mathfrak{p}+} - \dot{\chi}^{\textsc{m}}_{\mathfrak{p}}(t)|\ell_1^{\mathfrak{p}+}|q_1^{\mathfrak{p}+}| \\[5pt]
		&\quad+  \mathcal{O}(1)(1+\ell) |\Delta u^{\textsc{m}}_{\mathfrak{p}}| \\[5pt]
		& <-\frac{1}{2} ( 1 - \kappa_0 \mathcal{W}) |\lambda_1^{\mathfrak{p}+} - \dot{\chi}^{\textsc{m}}_{\mathfrak{p}}(t)|\ell_1^{\mathfrak{p}+} |q_1^{\mathfrak{p}+}|
+ \mathcal{O}(1)(1+\ell) |\Delta u^{\textsc{m}}_{\mathfrak{p}}|,
\end{split}
\end{eqnarray}
 provided $\ell \gg 1 \gg \epsilon$ and
\begin{equation}\label{mu0 condition 2}
	\kappa_0  < \frac{1}{\mathcal{W}}.
\end{equation}
Recall that the constant $\mathcal{W}$ is sufficiently close to 1.  Now from (\ref{mu0 condition 1})(\ref{mu0 condition 2}), we can choose the appropriate $\kappa_0$ such that
\begin{equation*}\label{condition mu0}
\kappa_1 \mathcal{W}< \kappa_0  < \frac{1}{\mathcal{W}}.
\end{equation*}
Since $\ell_3 < 1$ under condition $\textrm{(C1)}$, it is deduced from the estimate
$\lambda_2^{\mathfrak{p}+} -\dot{\chi}^{\textsc{m}}_{\mathfrak{p}}(t) = \mathcal{O}(1)\big(|q_1^{\mathfrak{p}+}| + |\Delta u^{\textsc{m}}_{\mathfrak{p}}|\big)$ that
\begin{equation}\label{Estimate E3 on boundary}
\mathscr{E}_{3,\mathfrak{p}} = \ell_3 \mathcal{W}_3^{\mathfrak{p}+} |q_3^{\mathfrak{p}+}|\big(\lambda_2^{\mathfrak{p}+} -\dot{\chi}^{\textsc{m}}_{\mathfrak{p}}(t)\big)
	=  \mathcal{O}(1)\big(|q_1^{\mathfrak{p}+}| + |\Delta u^{\textsc{m}}_{\mathfrak{p}}|\big).
\end{equation}
Finally, by the estimates (\ref{Estimate sum Ei on boundary})(\ref{Estimate E3 on boundary}), we arrive at
\begin{eqnarray*}
\begin{split}
\sum_{i=1}^4 \mathscr{E}_{i,\mathfrak{p}}&\leq -\frac{1}{2}( 1 - \kappa_0 \mathcal{W} ) \ell \underline{d}|q_1^{\mathfrak{p}+}| +  \mathcal{O}(1)(1+\ell)|\Delta u^{\textsc{m}}_{\mathfrak{p}}|
+  \mathcal{O}(1)\big(|q_1^{\mathfrak{p}+}| + |\Delta u^{\textsc{m}}_{\mathfrak{p}}|\big)\\[5pt]
&\leq  \mathcal{O}(1) |\Delta u^{\textsc{m}}_{\mathfrak{p}}|,
\end{split}
\end{eqnarray*}
which shows the estimate $\eqref{eq:conclusion Ei}_3$.

We eventually verified all the conclusions in \eqref{eq:conclusion Ei}, and then derive the derivative estimate
\begin{eqnarray*}
\dot{\mathscr{L}} ({U}(t), {V}(t)) \leq \mathcal{O}(\varepsilon + |\Delta u^{\textsc{m}}_{\mathfrak{p}}|),\quad \mbox{for}\ \emph{a.e.}\ t \in (0,+\infty).
\end{eqnarray*}
Integrate it from $\tilde{t}$ to $t$ to get
\begin{equation}\label{Phi(t) inequality 1}
\mathscr{L}({U}(t) , {V}(t)) \leq \mathscr{L}({U}(\tilde{t}) , {V}(\tilde{t}))+\mathcal{O}(1)  \int_{\tilde{t}}^t |\Delta u^{\textsc{m}}_{\mathfrak{p}}(\tau)| d\tau
+\mathcal{O}(1)\varepsilon (t-\tilde{t}).
\end{equation}
In addition, it is obvious that the distance of two pistons satisfies
\begin{equation}\label{piston distance}
\|\chi_{\mathfrak{p},1}(\cdot)- \chi_{\mathfrak{p},2}(\cdot)\| _{C{ ([\tilde{t},t])}} \leq |\chi_{\mathfrak{p},1}(\tilde{t}) - \chi_{\mathfrak{p},2}(\tilde{t})|
+ \int^t_{\tilde{t}} |u_{\mathfrak{p}}(\tau) - v_{\mathfrak{p}}(\tau)| d\tau.
\end{equation}
Combining (\ref{Phi(t) inequality 1}) with (\ref{piston distance}), and employing Proposition \ref{prop:5.1}, we obtain
\begin{eqnarray*}
\begin{split}
	& \quad \int_{\chi^{\textsc{m}}_{\mathfrak{p}}(t)}^{+\infty} |{U}(t,x)-{V}(t,x)| dx  \\[5pt]
	& \leq \mathcal{O}(1) \Big( \int_{\chi^{\textsc{m}}_{\mathfrak{p}}(\tilde{t})}^{+\infty} |{U}(\tilde{t},x)- {V}(\tilde{t},x)| dx
+ \int_{\tilde{t}}^t |\Delta u^{\textsc{m}}_{\mathfrak{p}}(\tau)| d\tau  +\varepsilon (t-\tilde{t}) \Big)\\[5pt]
	& \leq \mathcal{O}(1) \Big( \int_{\chi^{\textsc{m}}_{\mathfrak{p}}(\tilde{t})}^{+\infty} |{U}(\tilde{t},x)- {V}(\tilde{t},x)| dx
+ \int_{\tilde{t}}^t |u_{\mathfrak{p}}(\tau) - u_1 (\tau, \chi^{\textsc{m}}_{\mathfrak{p}}(\tau))| d\tau \\[5pt]
	&\quad + \int_{\tilde{t}}^t |v_{\mathfrak{p}}(\tau) - u_2 (\tau, \chi^{\textsc{m}}_{\mathfrak{p}}(\tau))| d\tau
	+ \int_{\tilde{t}}^t |u_{\mathfrak{p}}(\tau) -v_{\mathfrak{p}}(\tau)| d\tau +\varepsilon (t-\tilde{t}) \Big) \\[5pt]
	& \leq \mathcal{O}(1) \Big( \int_{\chi^{\textsc{m}}_{\mathfrak{p}}(\tilde{t})}^{+\infty} |{U}(\tilde{t},x)-{V}(\tilde{t},x)|dx
	+   \int_{\tilde{t}}^t |u_{\mathfrak{p}}(\tau) -v_{\mathfrak{p}}(\tau)| d\tau +\varepsilon (t-\tilde{t})  \Big)  \\[5pt]
&\quad +\mathcal{O}(1)\|\chi_{\mathfrak{p},1}(\cdot) - \chi_{\mathfrak{p},2}(\cdot)\| _{C{ ([\tilde{t},t])}}\\[5pt]
	& \leq \mathcal{O}(1)\Big(  \int_{\chi^{\textsc{m}}_{\mathfrak{p}}(\tilde{t})}^{+\infty} |{U}(\tilde{t},x)-{V}^\varepsilon(\tilde{t},x)|dx
+ \int^t_{\tilde{t}} |u_{\mathfrak{p}}(\tau) - v_{\mathfrak{p}}(\tau)| d\tau+ \varepsilon (t-\tilde{t})\Big)\\[5pt]
&\quad +\mathcal{O}(1)|\chi_{\mathfrak{p},1}(\tilde{t}) - \chi_{\mathfrak{p},2}(\tilde{t})|.
\end{split}
\end{eqnarray*}
We thus deduce that
\begin{eqnarray}\label{estimate stability for Euler approximation}
\begin{split}
	& \quad\ \|{U}(t,\cdot)- {V}(t,\cdot)\|_{L^1(\mathbb{R})} \\[5pt]
		& = \int_{\chi^{\textrm{m}}_{\mathfrak{p}}(t)}^{\chi^{\textsc{m}}_{\mathfrak{p}}(t)}  |{U}(t,x)- {V}(t,x)| dx
+ \int_{\chi^{\textsc{m}}_{\mathfrak{p}}(t)}^{+\infty} |{U}(t,x)- {V}(t,x)| dx\\[5pt]
		& \leq  \mathcal{O}(1) \Big(\int_{\chi^{\textsc{m}}_{\mathfrak{p}}(\tilde{t})}^{+\infty} |{U}(\tilde{t},x)-{V}(\tilde{t},x)| dx
+ \int^t_{\tilde{t}} |u_{\mathfrak{p}}(\tau) - v_{\mathfrak{p}}(\tau)| d\tau\Big) \\[5pt]
		& \quad + \mathcal{O}(1)\Big( |\chi_{\mathfrak{p},1}(t) - \chi_{\mathfrak{p},2}(t)|+|\chi_{\mathfrak{p},1}(\tilde{t}) - \chi_{\mathfrak{p},2}(\tilde{t})|\Big)+\mathcal{O}(1)\varepsilon (t-\tilde{t})\\[5pt]
& \leq \mathcal{O}(1) \Big( \|{U}(\tilde{t},\cdot)- {V}(\tilde{t},\cdot) \|_{L^1(\mathbb{R})}
		+\|u_{\mathfrak{p}}(\cdot) -v_{\mathfrak{p}}(\cdot)\|_{L^1((\tilde{t},t))}+\varepsilon (t-\tilde{t}) \Big),
\end{split}
\end{eqnarray}
where we have the estimate
\begin{equation*}
|\chi_{\mathfrak{p},1}(\tilde{t}) - \chi_{\mathfrak{p},2}(\tilde{t})|\leq \mathcal{O}(1)\int_{\chi^{\textrm{m}}_{\mathfrak{p}}(t)}^{\chi^{\textsc{m}}_{\mathfrak{p}}(t)}  |{U}(\tilde{t},x)- {V}(\tilde{t},x)| dx.
\end{equation*}

\emph{Proof of Theorem \ref{thm:1.1} for the $L^1$-stability under condition $\rm{(C1)}$}.
Suppose ${U}^\mu, {U}^\nu$ are two approximate solutions to (\emph{IBVP}) under condition $\rm{(C1)}$. Particularly, we take
\begin{eqnarray*}
\varepsilon = \max\{\mu,\nu\}, \quad ({U}^\varepsilon , u_{\mathfrak{p}}^\varepsilon) = ({U}^\mu, u_{\mathfrak{p}}^\mu), \quad
({V}^\varepsilon , v_{\mathfrak{p}}^\varepsilon) = ({U}^\nu, u_{\mathfrak{p}}^\nu),
\end{eqnarray*}
and let $\tilde{t}=0$ in (\ref{estimate stability for Euler approximation}). Then, it yields that
\begin{align*}
&\quad \|{U}^\mu(t,\cdot)- {U}^\nu(t,\cdot)\|_{L^1(\mathbb{R})} \\[5pt]
& \leq \mathcal{O}(1) \Big( \|{U}_0^\mu- {U}_0^\nu \|_{L^1(\mathbb{R}_{+})}
	+\|u^{\mu}_{\mathfrak{p}}(\tau) -u^{\nu}_{\mathfrak{p}}(\tau)\|_{L^1([0,t])} \Big)+\mathcal{O}(1)\max\{\mu,\nu\} t\\[5pt]
	&\leq \mathcal{O}(1) (\mu+\nu) +\mathcal{O}(1)\max\{\mu,\nu\} t.
\end{align*}
Passing to the limits as $\mu,\ \nu \rightarrow 0$, we see that
\begin{eqnarray*}
\|{U}^\mu(t,\cdot)- {U}^\nu(t,\cdot)\|_{L^1(\mathbb{R})} \rightarrow 0.
\end{eqnarray*}
This means the approximations $\{{U}^\varepsilon \}$ constructed by fractional-step front tracking algorithm is indeed a Cauchy sequence which has a unique limit ${U}$.
In terms of the strong convergence, we suppose that
\begin{eqnarray*}
{U}^\varepsilon \to  {U}, \quad
{V}^\varepsilon \to {V}, \quad
{u}^\varepsilon_{\mathfrak{p}} \to  {u}_{\mathfrak{p}}, \quad
{v}^\varepsilon_{\mathfrak{p}} \to {v}_{\mathfrak{p}} \qquad\mbox{as}\quad \varepsilon \rightarrow 0.
\end{eqnarray*}
Then estimate (\ref{estimate stability for Euler approximation}) obviously implies that
\begin{eqnarray*}
\|{U}(t,\cdot)- {V}(t,\cdot)\|_{L^1(\mathbb{R})}\leq \mathcal{O}(1) \Big( \|{U}(\tilde{t},\cdot)- {V}(\tilde{t},\cdot) \|_{L^1(\mathbb{R})}
+ \|u_{\mathfrak{p}}(\cdot)  - v_{\mathfrak{p}}(\cdot)\|_{L^1([\tilde{t},t])}\Big)
\end{eqnarray*}
for any $t>\tilde{t}\geq0$. This eventually established the $L^1$-stability of non-reacting flow under condition $\textrm{(C1)}$.
\hfill$\Box$

\subsection{$L^1$-stability for reacting flow}

Consider the $\varepsilon$-approximate solutions ${U}$ and ${V}$ to the (\emph{IBVP})  under condition $\textrm{(C2)}$ or $\textrm{(C3)}$ constructed by fractional-step wave front tracking scheme as stated in section 3.
Then, for any $t\in (t_{k-1}, t_k)$ with $k\geq 1$, the analysis of $\dot{\mathscr{L}}({U}(t), {V}(t))$ in section 5.2 remains valid by only slight modifications on $\ell_3$ distribution in
Table \ref{ell distribution (C2) or (C3)}.
More precisely, since $T.V.\{Y_{1}(0,\cdot); (0,+\infty)\}+T.V.\{Y_{2}(0,\cdot); (0,+\infty)\} < 2\epsilon \ll 1$, we can use $\mathcal{K} |q_3| <1$ instead of $\mathcal{K} \ell_3 = \ell$ in (\ref{estimate E3 in case 6.1.3}) and deduce again that
\begin{equation}\label{Phi for non-reaction}
\dot{\mathscr{L}}({U}(t), {V}(t)) \leq \mathcal{O}(1) (\varepsilon + |\Delta u^{\textsc{m}}_{\mathfrak{p}}| ), \qquad \mbox{for}\ \mbox{a.e.}\quad t\in (t_{k-1}, t_k).
\end{equation}
So our purpose now is to show the changes of $\mathscr{L}$ when crossing the time $t=t_k$. Denote the states before reaction by
\begin{eqnarray*}
{U}^* = (\rho_{1}^*, u_{1}^*, p_{1}^*, Y_{1}^*)^{\top},\quad {V}^* = (\rho_{2}^*, u_{2}^*, p_{2}^*, Y_{2}^*)^{\top},
\end{eqnarray*}
and states after reaction by
\begin{eqnarray*}
 {U} = (\rho_{1}, u_{1}, p_{1}, Y_{1})^{\top}, \quad {V} = (\rho_{2}, u_{2}, p_{2}, Y_{2})^{\top}.
\end{eqnarray*}
They satisfy the relations \eqref{reaction step}. Besides, we always use the notations
\begin{align*}
&q_i = q_i(t_k,x),  \qquad \quad \mathcal{W}_{i} = \mathcal{W}_{i}(t_k,x), \\
&q_i^* = q_i(t_k-0,x), \quad\mathcal{W}_{i}^{*} = \mathcal{W}_{i}(t_k-0,x)
\end{align*}
for brevity.
Above all, we consider $\Delta \mathscr{L}(U(t_{k}), V(t_{k}))$ for partially ignited flow in three cases as below.

\textbf{Case 5.3.1}. \emph{Difference $\Delta \mathscr{L}$ at $x\in (\chi^{\textsc{m}}_{\mathfrak{p}}(t_k), \chi^{\rm{m}}_{\rm{s}}(t_k))$}.
We begin with the estimates on distance indices in combustion process.

\begin{lemma}\label{Lemma qi for reaction step}
For sufficiently small $\varepsilon$, there hold
\begin{equation}\label{estimate combustion q0}
|q_3| - |q_3^*| \leq -|Y_{1}^* -Y_{2}^*| \underline{\phi}\varepsilon + \mathcal{O}(1)\sum_{j\neq 3}  |q_j^*| (Y_{1}^* + Y_{2}^*) \varepsilon,
\end{equation}
and
\begin{equation}\label{estimate combustion qi}
|q_i| - |q_i^*| = \mathcal{O}(1)\Big( |Y_{1}^* -Y_{2}^*| +  \sum_{j\neq 3}   |q_j^*| (Y_{1}^* + Y_{2}^*)  \Big) \varepsilon \qquad\mbox{for}\  i \neq 3,
\end{equation}
\end{lemma}

\begin{proof}
First, by \eqref{reaction step}, one figures out that
\begin{eqnarray*}
Y_{1}  -Y_{2}=(Y_{1}^* -Y_{2}^*) \big(1-\phi(T_{1}^*)\varepsilon\big) + Y_{2}^* \big(\phi(T_{2}^*)-\phi(T_{1}^*)\big) \varepsilon.
\end{eqnarray*}
Then, by condition \eqref{eq:4.20} on $\phi$, we obtain
\begin{eqnarray*}
\begin{split}
|q_3| - |q_3^*|&=|Y_{1}  -Y_{2} | - |Y_{1}^* -Y_{2}^*|\\[5pt]
&\leq  -|Y_{1}^* -Y_{2}^*| \underline{\phi}\varepsilon + \mathcal{O}(1)Y_{2}^* \varepsilon  \sum_{j\neq 3} |q_j^*|,
\end{split}
\end{eqnarray*}
or
\begin{eqnarray*}
\begin{split}
|q_3| - |q_3^*|&=|Y_{1}  -Y_{2} | - |Y_{1}^* -Y_{2}^*|\\[5pt]
&\leq  -|Y_{1}^* -Y_{2}^*| \underline{\phi} \varepsilon + \mathcal{O}(1)Y_{1}^*\varepsilon  \sum_{j\neq 3}   |q_j^*|.
\end{split}
\end{eqnarray*}
Thus above two estimates give (\ref{estimate combustion q0}).

Set $\Upsilon_1 \doteq |Y_{1} - Y_{1}^*|, \ \ \Upsilon_2 \doteq |Y_{2} - Y_{2}^*|$. Then we have
\begin{eqnarray}\label{estimate r1-r2}
\begin{split}
|\Upsilon_1 -\Upsilon_2| & \leq |(Y_{1}^*- Y_{2}^*) \phi(T_{1}^*)  + Y_{2}^* ( \phi(T_{1}^*) -\phi(T_{2}^*)  )| \varepsilon \\[5pt]
		&  \leq \mathcal{O}(1) \Big(|Y_{1}^*- Y_{2}^*| + Y_{2}^* \sum_{j\neq 3}  |q_j^*|\Big)\varepsilon .
\end{split}
\end{eqnarray}
Consider the function $q_i = f_i (\mathbf{q}^* ,\Upsilon_1 ,\Upsilon_2)$ where $\mathbf{q}^* = (q_1^*, q_2^*, q_4^*)$. It follows from (\ref{estimate r1-r2}) that
\begin{eqnarray*}
\begin{split}
q_i  & = f_i (\mathbf{q}^* ,\Upsilon_1 ,\Upsilon_1) +  \mathcal{O}(1) |\Upsilon_1 -\Upsilon_2|\\[5pt]
	& = q_i^* +  \mathcal{O}(1) \big( |\mathbf{q}^*| \cdot|\Upsilon_1| +   |\Upsilon_1 -\Upsilon_2|\big) \\[5pt]
	& = q_i^* +  \mathcal{O}(1) \Big( |Y_{1}^*- Y_{2}^*|\varepsilon + \sum_{j\neq 3}   |q_j^*| (Y_{1}^* + Y_{2}^*)\varepsilon  \Big),
\end{split}
\end{eqnarray*}
which leads to the estimate (\ref{estimate combustion qi}).
\end{proof}

Now we continue to study $\Delta \mathscr{L}(U(t_{k}), V(t_{k}))$. Recall the estimate \eqref{eq:4.30} on $\Delta \mathcal{G} (t_k)$. We obtain
\begin{eqnarray*}
\mathcal{W}_{i}  - \mathcal{W}_{i}^*
\leq -\frac{1}{4} \mathcal{K}\Big(  K\Upsilon^{\rm{s}}_k+ \theta_3 \underline{\phi}  \sum_{x_\alpha \leq \chi_{\rm{s},1}(t_{k})}|\alpha_3^*| \varepsilon
\Big)_{U}
-\frac{1}{4} \mathcal{K}\Big( K\Upsilon^{\rm{s}}_k + \theta_3 \underline{\phi}  \sum_{x_\beta \leq \chi_{\rm{s},2}(t_{k})}|\beta_3^*|\varepsilon
\Big)_{V} .
\end{eqnarray*}
Thus

\begin{align}
\notag
\sum_{ i \neq 3} \ell_i \big( \mathcal{W}_{i} |q_i| - \mathcal{W}_{i}^* |q_i^*|\big)
&= \sum_{ i \neq 3} \ell_i \Big( \mathcal{W}_{i} (|q_i| - |q_i^*|) + ( \mathcal{W}_{i}  - \mathcal{W}_{i}^*) |q_i^*| \Big) \\[5pt]
\notag
&\leq   \mathcal{O}(1) \mathcal{W } \sum_{ i \neq 3} \ell_i \Big( |Y_{1}^* -Y_{2}^*| +  \sum_{j\neq 3}   |q_j^*| (Y_{1}^* + Y_{2}^*)  \Big) \varepsilon  \\[5pt]
\notag
& \quad-\frac{1}{4} \mathcal{K}  \sum_{ i \neq 3} \ell_i |q_i^*| \Big( K\Upsilon_k^{\rm{s}} + \theta_3 \underline{\phi}  \sum_{x_\alpha \leq \chi_{\rm{s},1}(t_{k})}|\alpha_3^*| \varepsilon\Big)_{U}  \\[5pt]
\label{estimate li Wi qi 1-3}
& \quad -\frac{1}{4} \mathcal{K}  \sum_{ i \neq 3} \ell_i |q_i^*| \Big( K\Upsilon_k^{\rm{s}} + \theta_3 \underline{\phi}  \sum_{x_\beta \leq \chi_{s,2}(t_{k})}|\beta_3^*|\varepsilon\Big)_{V}.
\end{align}
By (\ref{estimate combustion q0}), we deduce that
\begin{eqnarray}\label{estimate l0 W0 q0}
\begin{split}
\ell_3 \big( \mathcal{W}_{3} |q_3| - \mathcal{W}_{3}^* |q_3^*|\big)
& = \ell_3 \Big( \big(\mathcal{W}_{3} -\mathcal{W}_{3}^*\big) |q_3|  +  \mathcal{W}_{3}^* \big( |q_3| - |q_3^*| \big) \Big) \\[5pt]
& < -\ell_3 |Y_{1}^* -Y_{2}^*| \underline{\phi}\varepsilon   +  \mathcal{O}(1)\ell_3 \mathcal{W} \sum_{j\neq 3}   |q_j^*| \big(Y_{1}^* + Y_{2}^*\big)\varepsilon.
\end{split}
\end{eqnarray}
Therefore, it follows from the estimates (\ref{estimate li Wi qi 1-3})(\ref{estimate l0 W0 q0}) and the coefficients condition (\ref{condition kappa ell}) that
\begin{eqnarray}\label{estimate 1 of Phic for reaction step}
\begin{split}
&\quad \sum_{i=1}^4 \ell_i \big(\mathcal{ W}_{i} |q_i| - \mathcal{W}^*_{i} |q_i^*|\big)\\[5pt]
& \leq\Big( \mathcal{O}(1) \sum_{ i \neq 3} \ell_i   - \ell_3 \underline{\phi} \Big) |Y_{1}^* -Y_{2}^*| \varepsilon  \\[5pt]
& \quad +  \mathcal{O}(1)  \sum_{i=1}^4 \ell_i\sum_{j\neq 3}|q_j^*|
\bigg( \Big( \Upsilon_k^{\rm{s}} +  \sum_{x_\alpha \leq \chi_{\rm{s},1}(t_{k})}|\alpha_3^*|  \varepsilon \Big)_{U}+\Big(\Upsilon_k^{\rm{s}} +   \sum_{x_\beta \leq \chi_{\rm{s},2}(t_{k})}|\beta_3^*|  \varepsilon \Big)_{V} \bigg)\\[5pt]
&\quad -\frac{1}{4} \mathcal{K}  \sum_{i \neq 3} \ell_i |q_i^*| \Big( K\Upsilon_k^{\rm{s}} + \theta_3 \underline{\phi}  \sum_{x_\alpha \leq \chi_{\rm{s},1}(t_{k})}|\alpha_3^*| \varepsilon\Big)_{U} \\[5pt]
&\quad  -\frac{1}{4} \mathcal{K} \sum_{ i \neq 3} \ell_i |q_i^*| \Big( K\Upsilon_k^{\rm{s}} + \theta_3 \underline{\phi}  \sum_{x_\beta \leq \chi_{\rm{s},2}(t_{k})}|\beta_3^*|\varepsilon\Big)_{V} \\[5pt]
&\leq   \Big( \mathcal{O}(1)\sum_{ i \neq 3} \ell_i   - \ell_3 \underline{\phi} \Big) |Y_{1}^* -Y_{2}^*| \varepsilon  \\[5pt]
& \quad + \Big( \mathcal{O}(1)  \sum_{i=1}^4 \ell_i  - \frac{1}{4} \mathcal{K} \Big) \sum_{i \neq 3} |q_i^*| \Big( K\Upsilon_k^{\rm{s}}
+\theta_3 \underline{\phi}  \sum_{x_\alpha \leq \chi_{\rm{s},1}(t_{k})}|\alpha_3^*|\varepsilon\Big)_{U}  \\[5pt]
& \quad  +  \Big( \mathcal{O}(1)  \sum_{i=1}^4 \ell_i  - \frac{1}{4} \mathcal{K} \Big)  \sum_{i \neq 3} |q_i^*| \Big( K\Upsilon_k^{\rm{s}}+ \theta_3 \underline{\phi}  \sum_{x_\beta \leq \chi_{\rm{s},2}(t_{k})}|\beta_3^*| \varepsilon
	\Big)_{V} \\[5pt]
& <0.
\end{split}
\end{eqnarray}

\textbf{Case 5.3.2}. \emph{Difference $\Delta \mathscr{L}$ at $x \in (\chi^{\rm{m}}_{\rm{s}}(t_k) ,  \chi^{\textsc{m}}_{\rm{s}}(t_k))$}. Without loss of generality, we assume that $T_{1} > T_{\text{i}} > T_{2}$ for this case.
Hence
\begin{eqnarray*}
\phi(T_{1}) >0 = \phi(T_{2}),\quad \Upsilon_1 = |Y_{1} - Y_{1}^*| \neq 0, \quad \Upsilon_2 = |Y_{2} - Y_{2}^*|=0.
\end{eqnarray*}
By direct computation, we know that
\begin{align*}
&\Upsilon_{1}  = Y_{1}^* \phi(T_{1}^*)\varepsilon \leq \mathcal{O}(1) \Big( \sum_{x_\alpha<\chi_{s,1}(t_k)} |\alpha_3^*|\varepsilon + \Upsilon_k^s \Big)_{U} , \\[5pt]
&q_i - q_i^* =  \mathcal{O}(1)\Upsilon_{1}  \quad \mbox{for } i\neq3, \quad |q_3| - |q_3^*| = |Y_{1}- Y_{2}| - |Y_{1}^*- Y_{2}| \leq \Upsilon_{1},
\end{align*}
and
\begin{eqnarray*}
&\mathcal{W}_{i}  -\mathcal{W}_{i}^*  <  -\frac{1}{4} \mathcal{K}\Big( K\Upsilon_k^s + \theta_3 \underline{\phi} \displaystyle\sum_{x_\alpha \leq \chi_{s,1}(t_k)}|\alpha_3^*| \varepsilon\Big)_{U}.
\end{eqnarray*}
These imply that
\begin{equation*}
\begin{split}
\sum_{i\neq 3} \ell_i \big(\mathcal{W}_{i} |q_i| - \mathcal{W}_{i}^* |q_i^*|\big)
&= \sum_{i\neq 3} \ell_i \Big( \mathcal{W}_{i} \big(|q_i| - |q_i^*|\big) + \big( \mathcal{W}_{i}  -\mathcal{W}_{i}^*\big) |q_i^*| \Big) \\[5pt]
& \leq \mathcal{O}(1) \mathcal{W} \sum_{i\neq 3} \ell_i  \Upsilon_{1}
- \frac{1}{4} \mathcal{K} \sum_{i\neq 3} |q_i^*| \Big(  K\Upsilon_k^s + \theta_3 \underline{\phi}  \sum_{x_\alpha \leq \chi_{s,1}(t_k)}|\alpha_3^*| \varepsilon\Big)_{U},
\end{split}
\end{equation*}
and
\begin{eqnarray*}
\begin{split}
\ell_3 \big( \mathcal{W}_{3} |q_3| - \mathcal{W}_{3}^* |q_3^*|\big) & = \ell_3 \Big( \mathcal{W}_{3} \big(|q_3| - |q_3^*|\big) + \big(\mathcal{W}_{3}  - \mathcal{W}_{3}^*\big) |q_3^*| \Big)\\[5pt]
	& \leq \ell_3 \mathcal{W} \Upsilon_{1}.
\end{split}
\end{eqnarray*}
Thus we conclude that
\begin{eqnarray}\label{estimate 2 of Phic for reaction step}
\begin{split}
&\quad \sum_{i=1}^4 \ell_i ( \mathcal{W}_{i} |q_i| -\mathcal{W}_{i}^* |q_i^*|) \\[5pt]
&  \leq \Big(\mathcal{O}(1) \sum_{i=1}^4 \ell_i
- \frac{1}{4} \mathcal{K} \sum_{i\neq 3} |q_i^*|\Big) \cdot \Big(  K\Upsilon_k^{\rm{s}}  + \theta_3 \underline{\phi}  \sum_{x_\alpha \leq \chi_{\rm{s},1}(t_k)}|\alpha_3^*| \varepsilon\Big)_{U} \\[5pt]
		& <0,
\end{split}
\end{eqnarray}
because $|q_4^*|$ is approximate to strength of large $4$-shock and $\mathcal{K}$ is large enough.

\textbf{Case 5.3.3}. \emph{Difference $\Delta \mathscr{L}$ at $x\in (\chi^{\textsc{m}}_{\rm{s}} (t_k), +\infty )$}. In this case, we know that both $T_{1}$ and $T_{2}$ are less than $T_{\text{i}}$.
This leads to $\Upsilon_{1} =\Upsilon_{2} =0$, \emph{i.e.}, the chemical reaction is absent in the region $(\chi^{\textsc{m}}_{\rm{s}} (t_k) , +\infty)$.
By construction, no wave interaction takes place at $t=t_k$. Hence the Lipschitz continuity yields
\begin{equation}\label{estimate 3 of Phic for reaction step}
\int_{\chi^{\textsc{m}}_{\rm{s}} (t_k)} ^{+\infty}  \sum_{i=1}^4 \ell_i \big( \mathcal{W}_{i} |q_i| -\mathcal {W}_{i}^*|q_i^*|\big) dx =0.
\end{equation}

Combining (\ref{estimate 3 of Phic for reaction step}) with (\ref{estimate 1 of Phic for reaction step})(\ref{estimate 2 of Phic for reaction step}), we conclude that at $t=t_k$, there holds
\begin{equation}\label{Phi for reaction Tr<Ti}
\Delta \mathscr{L} (t_k)  = \bigg(\int_{\chi^{\textsc{m}}_{\mathfrak{p}}}^{\chi^{\textrm{m}}_{\rm{s}}}+ \int_{\chi^{\textrm{m}}_{\rm{s}}}^{\chi^{\textsc{m}}_{\rm{s}}}+\int_{\chi^{\textsc{m}}_{\rm{s}}} ^{+\infty} \bigg)
	\sum_{i=1}^4 \ell_i\big( \mathcal{W}_{i} |q_i| - \mathcal{W}_{i}^* |q_i^*|\big) dx <0.
\end{equation}

Finally, we are concerned with the completely ignited flow. It follows from \eqref{eq:4.31} that
\begin{eqnarray*}
\mathcal{W}_{i}  - \mathcal{W}_{i}^*\leq -\frac{1}{4} \mathcal{K}\theta_3 \underline{\phi} \sum_{\alpha_3^* \in \mathcal{J}^{\textsc{y}}_3}|\alpha_3^*| \varepsilon \qquad \mbox{for}\quad 1\leq i \leq4.
\end{eqnarray*}
Then, according to Lemma \ref{Lemma qi for reaction step}, there holds
\begin{eqnarray*}
\begin{split}
&\quad\ \sum_{i=1}^4 \ell_i \big( \mathcal{W}_{i} |q_i| - \mathcal{W}_{i}^* |q_i^*|\big)\\[5pt]
&= \sum_{i=1}^4 \ell_i \Big( \mathcal{W}_{i} \big(|q_i| -  |q_i^*|\big)+ \big(\mathcal{W}_{i} - \mathcal{W}_{i}^*\big)|q_i^*| \Big)\\[5pt]
		&  \leq   \Big( \mathcal{O}(1)   \sum_{i\neq3} \ell_i   - \ell_3 \underline{\phi} \Big) |Y_{1}^* -Y_{2}^*| \varepsilon
		+  \mathcal{O}(1)  \sum_{i=1}^4 \ell_i\sum_{j\neq3}  |q_j^*| (Y_{1}^* + Y_{2}^*) \varepsilon  \\[5pt]
		&\quad-\frac{1}{4} \mathcal{K}\theta_3 \underline{\phi} \sum_{j\neq3}  |q_j^*|\cdot \sum_{\alpha_3^* \in \mathcal{J}^{\textsc{y}}_3}|\alpha_3^*| \varepsilon \\[5pt]
		&  \leq  \Big( \mathcal{O}(1)  \sum_{i\neq3} \ell_i   - \ell_3 \underline{\phi} \Big) |Y_{1}^* -Y_{2}^*| \varepsilon  \\[5pt]
		& \quad+ \Big( \mathcal{O}(1)   \sum_{i=1}^4 \ell_i  - \frac{1}{4} \mathcal{K}\theta_3 \underline{\phi}   \Big)  \sum_{j\neq3}  |q_j^*|  \cdot \sum_{\alpha_3^* \in \mathcal{J}^{\textsc{y}}_3}|\alpha_3^*| \varepsilon \\[5pt]
		& <0 .
\end{split}
\end{eqnarray*}
Therefore at $t=t_k$, we still have
\begin{equation}\label{Phi for reaction Tr>Ti}
\Delta \mathscr{L}(t_k)  = \int_{\chi^{\textsc{m}}_{\mathfrak{p}}}^{+\infty}\sum_{i=1}^4 \ell_i\big( \mathcal{W}_{i} |q_i| - \mathcal{W}_{i}^* |q_i^*|\big)dx <0.
\end{equation}

As a consequence, estimates (\ref{Phi for non-reaction})(\ref{Phi for reaction Tr<Ti})(\ref{Phi for reaction Tr>Ti}) give
\begin{equation*}\label{estimate Phic epsilon}
\mathscr{L}({U}(t), {V}(t)) \leq\mathscr{L}({U}(\tilde{t}),{V}(\tilde{t})) + \mathcal{O}(1) \varepsilon (t-\tilde{t})+\mathcal{O}(1) \int_{\tilde{t}}^t |\Delta u^{\textsc{m}}_{\mathfrak{p}}(\tau)|d\tau
\end{equation*}
for every $t>\tilde{t}\geq 0$, which implies that
\begin{eqnarray*}
\begin{split}
&\quad\int_{\chi^{\textsc{m}}_{\mathfrak{p}}(t)}^\infty |{U}(t,\cdot)- {V}(t,\cdot)| dx\\[5pt]
&\leq \mathcal{O}(1) \bigg(\int_{\chi^{\textsc{m}}_{\mathfrak{p}}(\tilde{t})}^\infty|{U}(\tilde{t},\cdot)- {V}(\tilde{t},\cdot)|dx
+ \int_{\tilde{t}}^t |\Delta u^{\textsc{m}}_{\mathfrak{p}}(\tau)| d\tau +\varepsilon (t-\tilde{t}) \bigg).
\end{split}
\end{eqnarray*}

\emph{Proof of Theorem \ref{thm:1.1} for the $L^1$-stability under condition $\rm{(C2)}$ or $\rm{(C3)}$}.
By Proposition \ref{prop:5.1} for the velocities comparison, and apply the similar argument for non-reacting flow in section 5.2 to reacting flow, we deduce that
\begin{eqnarray}\label{estimate stability for ZND approximation}
\begin{split}
&\quad \|{U}^\varepsilon(t,\cdot)-{V}^\varepsilon(t,\cdot)\|_{L^1(\mathbb{R})} \\[5pt]
&\leq \mathcal{O}(1) \Big( \|{U}^\varepsilon(\tilde{t},\cdot)-{V}^\varepsilon(\tilde{t},\cdot) \|_{L^1(\mathbb{R})}
	+ \|u^\varepsilon_{\mathfrak{p}}(\cdot) - v^\varepsilon_{\mathfrak{p}}(\cdot)\|_{L^1([\tilde{t},t])}
+\varepsilon(t-\tilde{t}) \Big).
\end{split}
\end{eqnarray}
Take $\tilde{t}=0$ in \eqref{estimate stability for ZND approximation}. Extracting a subsequence $\{\varepsilon_{k}\}$ if necessary, then as $\varepsilon_{k} \to 0$, we see that
\begin{equation*}
\|{U}(t,\cdot)- {V}(t,\cdot)\|_{L^1(\mathbb{R})}\leq \mathcal{O}(1) \Big(\|{U}_0(\cdot)- {V}_0(\cdot)\|_{L^1(\mathbb{R})} + \|u_{\mathfrak{p}}(\cdot) - v_{\mathfrak{p}}(\cdot)\|_{L^1([0,t])}\Big),
\end{equation*}
 which establishes the $L^1$-stability for combustion flow.
\hfill$\Box$

\section{Uniqueness of limit solution to (\emph{IBVP}) for combustion flow}

In this section, we will consider the uniqueness of the entropy solutions obtained by fractional-step wave front tracking scheme for the combustion reacting flow under condition $\textrm{(C2)}$ or $\textrm{(C3)}$.
To complete it, we need to establish the error estimates on distinct trajectories. 
For this purpose, we first rewrite the approximate solutions to the (\emph{IBVP}) with the initial-boundary data $(U^\varepsilon_{0}, u^\varepsilon_{\mathfrak{p}})$ constructed in section 3 in the operator form:
\begin{equation*}
{U}^\varepsilon (t,x) = \mathcal{P}_t ^\varepsilon ({U}^\varepsilon_0(x),u^\varepsilon_{\mathfrak{p}}(t))\doteq
\mathcal{S}_{t-k\varepsilon} ^\varepsilon \circ(\mathcal{T}_{\varepsilon}\circ\mathcal{S}_\varepsilon ^\varepsilon )^k({U}^\varepsilon _0(x), u^\varepsilon_{\mathfrak{p}}(t)),
\end{equation*}
where integer $k=\text{min}\{\lceil t/\varepsilon \rceil, N\}$. Here, $\mathcal{S} ^\varepsilon$ denotes the approximate solution operator
to the (\emph{IBVP}) for non-reaction flow
\begin{equation}\label{Eq: non reaction step}
\left\{
\begin{array}{llll}
\partial_t E({U}) + \partial_x F({U}) =0, \qquad &(t,x)\in \Omega^\varepsilon_k ,\\[5pt]
{U}(t_{k-1},x)  ={U}^\varepsilon (t_{k-1},x),\qquad & x >\chi^\varepsilon_{\mathfrak{p}}(t_{k-1}), \\[5pt]
u(t, \chi^\varepsilon_{\mathfrak{p}}(t))= u^\varepsilon _{\mathfrak{p}}(t),\qquad & t \in [t_{k-1}, t_{k}).
\end{array}
\right.
\end{equation}
Moreover, the solution to (\ref{Eq: non reaction step}) is defined by
\begin{eqnarray*}\label{non reaction solution}
{U}^\varepsilon (t,x)\doteq \mathcal{S}_{t-t_{k-1}} ^\varepsilon ({U}^\varepsilon (t_{k-1},x), u_{\mathfrak{p}}^\varepsilon(t)),\qquad (t,x)\in \Omega^\varepsilon_k \quad \mbox{for}\quad k\geq 1.
\end{eqnarray*}
Let $\mathcal{T}$ be the approximate solution operator to the following Cauchy problem of ordinary differential equations:
\begin{equation*}\label{Eq2: reaction ODE}
\left\{
\begin{array}{llll}
\partial_t E({U})= G({U}), \qquad  & t \in [0,\varepsilon],\\[5pt]
{U}(0,x)= {U}^{\varepsilon}(t_k-0,x), \qquad & x >\chi^\varepsilon_{\mathfrak{p}}(t_k),
\end{array}
\right.
\end{equation*}
whose approximate solution can be given by
\begin{eqnarray*}\label{reaction solution}
\mathcal{T}_{t} ({U}(0,x)) \doteq {U}(0,x) + t G_{0}({U}(0,x)).
\end{eqnarray*}

%

In the sequel, we always use the notations ${U}_0, {V}_0, {U} _{\tau}, {V} _{\tau} $ (or $u_{\mathfrak{p}}, v_{\mathfrak{p}}$) etc.
to represent the piecewise constant approximations to the initial (or boundary) values. Then, by \eqref{estimate stability for ZND approximation}, we have

\begin{proposition}\label{prop:7.1}
Given two $\varepsilon$-approximate initial-boundary data $({U}_0 , u_{\mathfrak{p}})$ and $(V_0 , v_{\mathfrak{p}} )$ satisfying condition $\textrm{(C2)}$, there holds
\begin{eqnarray*}
\begin{split}
&\quad\ \|\mathcal{P}_t^\varepsilon ({U}_0 , u_{\mathfrak{p}}) - \mathcal{P}_t^\varepsilon  ({V}_0 , v_{\mathfrak{p}})\|_{L^1(\mathbb{R})}\\[5pt]
&\leq \mathcal{O}(1) \Big( \|\mathcal{P}_{\tilde{t}}^\varepsilon ({U}_0 , u_{\mathfrak{p}})- \mathcal{P}_{\tilde{t}}^\varepsilon ({V}_0 , v_{\mathfrak{p}})\|_{L^1(\mathbb{R})}+\|u_{\mathfrak{p}}-v_{\mathfrak{p}}\|_{L^1([\tilde{t},t])} \Big) \\[5pt]
&\quad+ \mathcal{O}(1)  \varepsilon (t-\tilde{t}),
\end{split}
\end{eqnarray*}
for any $t>\tilde{t}\geq 0$. Moreover, the above flow estimate also holds under condition $\textrm{(C3)}$.
\end{proposition}

Next, we turn to discuss the evolution of reacting flows after shifting the initial time to $\tau>0$. Thus the initial position of large leading shock must be separated from that of piston.
But we need not be concerned with what the backward trajectory is and whether the large shock front coincides with the piston in the past.
Suppose that ${U} _{\tau} = (U^{\textsc{e}} _{\tau}, Y_{1,\tau} )$ and ${V} _{\tau} =(V^{\textsc{e}} _{\tau}, Y_{2,\tau} )$ are new initial data for reacting flow,
which are the small perturbation of the background states $\bar{U}_{\mathfrak{b}}(\tau,x)$.
Moreover, the boundary data $u _\mathfrak{p}$ and  $v _\mathfrak{p}$ are constructed by small perturbations of reference velocity $\bar{u}_{\mathfrak{p}}$ on the boundary.

Let $\chi_{\mathfrak{p},1}(\tau)$ and ${ \chi_{\rm{s},1}}(\tau)$ be respectively the positions of the piston and large shock corresponding to $U_\tau$, while
$\chi_{\mathfrak{p},2}(\tau)$ and ${ \chi_{\rm{s},2}}(\tau)$ corresponding to $V_\tau $. Then define the initial intervals at time $\tau$ by
\begin{eqnarray*}
\begin{split}
&I_{j}(\tau)= (\chi_{\mathfrak{p},j}(\tau), { \chi_{{\rm{s}},j}}(\tau)), \quad J_{j}(\tau)= ({ \chi_{{\rm{s}},j}}(\tau), + \infty), \qquad \mbox{for}\quad j=1,2.
\end{split}
\end{eqnarray*}
It is required  that $I_1(\tau) \cap I_2(\tau) \neq \emptyset$. We also demand that the new initial-boundary data satisfy
\begin{eqnarray}\label{shift initial boundary condition}
\begin{split}
&\|U^{\textsc{e}}_\tau  - \bar{U}^{\textsc{e}}_{\mathfrak{b}}\|_{L^1(\mathbb{R})}+\|V^{\textsc{e}}_\tau  - \bar{U}^{\textsc{e}}_{\mathfrak{b}}\|_{L^1(\mathbb{R})}
+\|u _\mathfrak{p}-\bar{u}_{\mathfrak{p}}\|_{L^1(\mathbb{R}_{+})}+\|v _\mathfrak{p}-\bar{u}_{\mathfrak{p}}\|_{L^1(\mathbb{R}_{+})} < +\infty, \\[5pt]
&T.V.\{U^{\textsc{e}}_\tau -\bar{U}^{\textsc{e}}_{\mathfrak{b},\rm{l}}; I_1(\tau)\} + T.V.\{U^{\textsc{e}}_\tau -\bar{U}^{\textsc{e}}_{\mathfrak{b},\rm{r}}; J_1(\tau)\} \\[5pt]
&\qquad +T.V.\{u_{\mathfrak{p}} ; (\tau,+\infty)\} +T.V.\{Y_{1,\tau} ; I_1(\tau)\cup J_1(\tau)\}\ll 1, \\[5pt]
&T.V.\{V^{\textsc{e}}_\tau -\bar{U}^{\textsc{e}}_{\mathfrak{b},\rm{l}}; I_2(\tau)\} + T.V.\{V^{\textsc{e}}_\tau -\bar{U}^{\textsc{e}}_{\mathfrak{b},\rm{r}}; J_2(\tau)\}\\[5pt]
&\qquad +T.V.\{v_{\mathfrak{p}} ; (\tau,+\infty)\} +T.V.\{Y_{2,\tau} ; I_2(\tau)\cup J_2(\tau)\}\ll 1,\\[5pt]
&\bar{u}_{\mathfrak{p}} ^{-1} \big(\|Y_{1,\tau} \|_{L^1(\mathbb{R})} + \|Y_{2,\tau}\|_{L^1(\mathbb{R})}\big)\ll 1.
\end{split}
\end{eqnarray}
We make the following assumptions:\\
$\widetilde{\rm{(C1)}}$\ conditions $(\ref{shift initial boundary condition})_1$-$(\ref{shift initial boundary condition})_3$ hold if $\bar{T}_{\rm{r}}> T_\text{i}$;\\
$\widetilde{\rm{(C2)}}$\ conditions $(\ref{shift initial boundary condition})_1$-$(\ref{shift initial boundary condition})_4$ hold if $\bar{T}_{\rm{l}}> T_\text{i}> \bar{T}_{\rm{r}}$.

Upon Remark \ref{rem:4.1} and previous stability argument in sections 5.2-5.3, we get the following conclusions on operators $\mathcal{P}^\varepsilon_t $ and $\mathcal{S}^\varepsilon_t $.
\begin{proposition}\label{prop:7.2}
Suppose initial-boundary data $({U}_{\tau} , u_{\mathfrak{p}} )$ and $({V} _{\tau}, v_{\mathfrak{p}})$ satisfy the { assumption} $\widetilde{\rm{(C1)}}$ or $\widetilde{\rm{(C2)}}$.
Then the reacting flows constructed by fractional-step front tracking method satisfy
\begin{eqnarray}\label{P_t approximate continuity 1}
\begin{split}
&\quad\|\mathcal{P}_t^\varepsilon  ({U} _{\tau}, u_{\mathfrak{p}}) - \mathcal{P}_t^\varepsilon ({V} _{\tau}, v_{\mathfrak{p}} )\|_{L^1(\mathbb{R})}\\[5pt]
&\leq \mathcal{O}(1) \Big( \|\mathcal{P}_{\tilde{t}}^\varepsilon ({U} _{\tau}, u_{\mathfrak{p}})- \mathcal{P}_{\tilde{t}}^\varepsilon ({V} _{\tau}, v_{\mathfrak{p}}) \|_{L^1(\mathbb{R})}
	+ \|u_{\mathfrak{p}}- v_{\mathfrak{p}} \|_{L^{1}([\tilde{t},t])}  \Big) \\[5pt]
&\quad + \mathcal{O}(1)\varepsilon (t-\tilde{t}),
\end{split}
\end{eqnarray}
while the non-reacting flows satisfy
\begin{eqnarray}\label{S_t approximate continuity 1}
\begin{split}
&\quad \|\mathcal{S}_t^\varepsilon ({U} _{\tau}, u_{\mathfrak{p}}) - \mathcal{S}_t^\varepsilon ({V} _{\tau}, v_{\mathfrak{p}})\|_{L^1(\mathbb{R})}\\[5pt]
&\leq \mathcal{O}(1) \Big(\|\mathcal{S}_{\tilde{t}}^\varepsilon({U}_{\tau}, u_{\mathfrak{p}})- \mathcal{S}_{\tilde{t}}^\varepsilon ({V}_{\tau}, v_{\mathfrak{p}})\|_{L^1(\mathbb{R})}
+\|u_{\mathfrak{p}}  - v_{\mathfrak{p}}\|_{L^1([\tilde{t},t])}\Big)\\[5pt]
&\qquad + \mathcal{O}(1) \varepsilon (t-\tilde{t}),
\end{split}
\end{eqnarray}
and
\begin{align}
\label{S_t approximate continuity 2}
&\|\mathcal{S}_t^\varepsilon ({U} _{\tau}, u_{\mathfrak{p}}) - \mathcal{S}_{\tilde{t}}^\varepsilon ({U} _{\tau}, u_{\mathfrak{p}}) \|_{L^1(\mathbb{R})}\leq  \mathcal{O}(1)(t-\tilde{t}), \\[5pt]
\label{S_t approximate continuity 3}
&\|\mathcal{S}_t^\mu ({U}_{\tau} , u_{\mathfrak{p}}) - \mathcal{S}_t^\nu ({U}_{\tau},  u_{\mathfrak{p}})\|_{L^1(\mathbb{R})}\leq  \mathcal{O}(1)\max\{\mu,\nu\}t,
\end{align}
for any $t>\tilde{t}\geq0$ and $\mu, \nu \leq \varepsilon$.

\end{proposition}

According to the equivalence of total variation and Glimm-type functional, there exists a constant $C_5>0$ dependent of $\bar{U}_{\mathfrak{b}}$ such that
condition $T.V.\{{U}_0(\cdot); \mathbb{R}_+\} +T.V.\{u_{\mathfrak{p}}(\cdot); \mathbb{R}_+\}<\epsilon$ implies
\begin{eqnarray*}
 T.V.\{\mathcal{S}^\varepsilon_t ({U}_0 , u_\mathfrak{p}); (\chi_{\mathfrak{p},1},+\infty)\} < C_5 \epsilon \qquad \forall\ t>0.
\end{eqnarray*}
Thus we can choose a positive $\bar{t}$ suitably small, so that
\begin{eqnarray*}
T.V.\big\{\mathcal{S}^\varepsilon_t ({U}_0 , u_{\mathfrak{p}}) + { \varepsilon} G_{0}\big(\mathcal{S}^\varepsilon_t ({U}_0 , u_{\mathfrak{p}}) \big); (\chi_{\mathfrak{p}, 1},+\infty) \big\} < 2C_5 \epsilon , \qquad
\forall\ \varepsilon\in (0,\bar{t}).
\end{eqnarray*}
The smallness of variation ensures the global stability after perturbation. So we can utilize all the results proved in section 5.3 to establish the following lemma on operators commutation.

\begin{lemma}[Commutation estimate]\label{commutation estimate}
Assume $({U}_\tau, u_\mathfrak{p})$ is the initial-boundary data under condition $\widetilde{\rm{(C1)}}$ or $\widetilde{\rm{(C2)}}$. Then the operators $\mathcal{S}^\varepsilon _t$ and $\mathcal{T}_\varepsilon$ satisfy
\begin{eqnarray*}
\| \mathcal{S}^\varepsilon _t\circ \mathcal{T}_\varepsilon ({U}_\tau , u_{\mathfrak{p}})- \mathcal{T}_\varepsilon \circ\mathcal{S}^\varepsilon _t ({U}_\tau , u_{\mathfrak{p}})\|_{L^1(\mathbb{R})}\leq \mathcal{O}(1) \varepsilon t,
\end{eqnarray*}
for any $t>0$ and $0<\varepsilon\ll \bar{t}$.
\end{lemma}

\begin{proof}
Set $\hat{U}(t)=\mathcal{S}^\varepsilon_t \circ\mathcal{T}_\varepsilon ({U}_\tau, u_{\mathfrak{p}})$, $\hat{V}(t)= \mathcal{S}^\varepsilon _t ({U}_\tau, u_{\mathfrak{p}})$ and
$\hat{w}(t)= \varepsilon G_0( \mathcal{S}^\varepsilon _t ({U}_\tau, u_{\mathfrak{p}}))$. Obviously, $\hat{U}$, $\hat{V}$ and $\hat{V}+\hat{w}$ have the same boundary curve $x=\chi_{\mathfrak{p},1}(t)$
and boundary velocity $u_{\mathfrak{p}}$. Connect $\hat{U}$ with $\hat{V}+\hat{w}$ by Hugoniot curves according to the orientation rule in (\ref{orientation rule}).
Then there is a unique vector $\mathbf{q}=(q_1,q_2,q_3,q_4)$ such that $\hat{V}+\hat{w}= \mathcal{H}(\mathbf{q}) (\hat{U})$, or conversely $\hat{U}=\mathcal{H}(\mathbf{q})(\hat{V}+\hat{w})$. Here the mapping $\mathcal{H}$ is given by Definition \ref{Def: distance index}.
We further define a perturbed Lyapunov functional
\begin{eqnarray*}
\mathscr{L}_{\textrm{p}} (\hat{U}(t), \hat{V}(t)+ \hat{w}(t)) \doteq \sum_{i=1}^4 \int_{\chi_{\mathfrak{p},1}} ^{+\infty} \ell_i \mathcal{W}_i|q_i|dx,
\end{eqnarray*}
where
\begin{eqnarray*}
\mathcal{W}_i= 1+ \mathcal{K} \Big(A_i(\hat{U}, \hat{V}) + K\big(\mathcal{Q}(\hat{U})+ \mathcal{Q}(\hat{V})\big) + K\mathcal{L}_{\hat{\textsc{b}}}(\hat{U}, \hat{V}) \Big)
\end{eqnarray*}
follows the definition in (\ref{Def Wi}). Compute the time derivative of $\mathscr{L}_{\textrm{p}}$ to get
\begin{eqnarray*}
\begin{split}
\dot{\mathscr{L}}_{\textrm{p}} (\hat{U}, \hat{V} + \hat{w})
&=\sum_{\alpha \in  \mathcal{J} }\sum_{i=1}^4\big(|q_i^{\alpha+}| \ell_i^{\alpha+} \mathcal{W}_i ^{\alpha+}(\lambda_i^{\alpha+}-\dot{x}_\alpha )
		- |q_i^{\alpha-}| \ell_i^{\alpha-} \mathcal{W}_i ^{\alpha-}(\lambda_i^{\alpha-}-\dot{x}_\alpha )\big) \\[5pt]
&\quad + \sum_{i=1}^4 |q_i^{\mathfrak{p}+}| \ell_i^{\mathfrak{p}+} \mathcal{W}_i ^{\mathfrak{p}+}(\lambda_i^{\mathfrak{p}+}- u_{\mathfrak{p}})\\[5pt]
& \doteq \sum_{\alpha \in  \mathcal{J} } \sum _{i=1}^4 \widetilde{\mathscr{E}}_{i,\alpha}+\sum_{i=1}^4 \widetilde{\mathscr{E}}_{i,\mathfrak{p}}.
\end{split}
\end{eqnarray*}
By identical argument in Case 5.1.5, it is readily derived that
\begin{eqnarray*}
\sum_{i=1}^4 \widetilde{\mathscr{E}}_{i,\mathfrak{p}}<0.
\end{eqnarray*}
We only need  to evaluate $\widetilde{\mathscr{E}}_{i,\alpha}$.

For every wave $\alpha$ of $\hat{U}$, sum $\sum_{i=1}^4 \widetilde{\mathscr{E}}_{i,\alpha} $ satisfies the estimates $(\ref{eq:conclusion Ei})_1$-$(\ref{eq:conclusion Ei})_2$
by the standard argument in section 5.2. Now it remains to consider the wave $\alpha$ of $\hat{V}$, which is  certainly associated with the front of $\hat{V}+ \hat{w}$.
If $1 \leq k_\alpha \leq 4$, one has $\hat{V}^+ = S_{k_\alpha} (\alpha)(\hat{V}^-)$ and the front speed
\begin{equation*}
\dot{x}_\alpha = \left\{
\begin{aligned}
&\lambda_{k_\alpha} (\hat{V}^-) &&\qquad  \mbox{if }  \alpha >0 ,\\
&\lambda_{k_\alpha} (\hat{V}^- , \hat{V}^+) &&\qquad  \mbox{if}\ \alpha<0.
\end{aligned}\right.
\end{equation*}
Then we introduce an auxiliary state
\begin{equation*}
(\hat{V} + \hat{w})^\oplus  \doteq S_{k_\alpha}(\alpha) (\hat{V}^- + \hat{w}^-).
\end{equation*}
Connect state $\hat{U}$ with $(\hat{V} + \hat{w})^\oplus $ by Hugoniot curves according to (\ref{orientation rule}). Denote the successive states from left to right by $\omega_0^\oplus, \omega_1^\oplus, \cdots, \omega_4^\oplus$. Then they satisfy $\omega_i^\oplus = S_i (q_i^\oplus )(\omega_{i-1}^\oplus)$ for some index $q_i^\oplus$. Accordingly, define
\begin{eqnarray*}\label{def q^oplus for uniqueness}
\lambda_i^\oplus   \doteq \lambda_i (\omega_{i-1}^\oplus, \omega_{i}^\oplus),  \quad
\dot{x}^\oplus  \doteq
\left\{
\begin{aligned}
&\lambda_{k_\alpha} (\hat{V}^- + \hat{w}^-) &&\ \mbox{if}\ \alpha >0, \\
&\lambda_{k_\alpha} (\hat{V}^- + \hat{w}^-, (\hat{V} + \hat{w})^\oplus)  &&\ \mbox{if}\ \alpha<0.
\end{aligned}
\right.
\end{eqnarray*}
So one can rewrite the term $\widetilde{\mathscr{E}}_{i,\alpha}$ as
\begin{eqnarray*}
\begin{split}
\widetilde{\mathscr{E}}_{i,\alpha}
&= \ell_i^+ \mathcal{W}_i^+ |q_i^+ | (\lambda_i^+ - \dot{x}_\alpha) -  \ell_i^- \mathcal{W}_i^- |q_i^- | (\lambda_i^- - \dot{x}_\alpha) \\[5pt]
&= \big( \ell_i^+ \mathcal{W}_{i}^+ |q_i^\oplus | (\lambda_i^\oplus - \dot{x}^\oplus) -   \ell_i^- \mathcal{W}_i^- |q_i^- | (\lambda_i^- - \dot{x}^\oplus)\big)\\[5pt]
&\quad +  \big(\ell_i^+ \mathcal{W}_i^+ |q_i^+ | (\lambda_i^+ - \dot{x}_\alpha) - \ell_i^+  \mathcal{W}_{i}^+ |q_i^\oplus | (\lambda_i^\oplus - \dot{x}^\oplus)
+ \ell_i^- \mathcal{W}_i^- |q_i^- | (\dot{x}_\alpha - \dot{x}^\oplus) \big) \\[5pt]
& \doteq \widetilde{\mathscr{E}}^{-}_{i,\alpha}+ \widetilde{\mathscr{E}}^{+}_{i,\alpha}.
\end{split}
\end{eqnarray*}
By similar argument on $\hat{V}^- + \hat{w}^-, (\hat{V} + \hat{w})^\oplus$ instead of ${V}^-, {V}^+$ in section 5.2, it turns out that
\begin{eqnarray*}
\begin{split}
\sum_{i=1}^4 \widetilde{\mathscr{E}}_{i,\alpha}^- &\leq \mathcal{O}(1) \varepsilon |\alpha|, &\qquad &\alpha  \in {\mathcal{J}^{\textsc{e}}(\hat{V})}\cup\mathcal{J}^{\textsc{y}}(\hat{V}) \setminus\mathcal{NP},\\[5pt]
\sum_{i=1}^4  \widetilde{\mathscr{E}}_{i,\alpha}^- &\leq \mathcal{O}(1)  |\alpha|, &\qquad  &\alpha  \in \mathcal{J}^{\textsc{e}} (\hat{V})\cap \mathcal{NP}, \\[5pt]
\sum_{i=1}^4 \widetilde{\mathscr{E}}_{i,\alpha}^- &<0, &\qquad &\alpha \in \mathcal{S}_{\rm{b}}(\hat{V}).
\end{split}
\end{eqnarray*}
Hence, we only need to evaluate $\widetilde{\mathscr{E}}_{i,\alpha}^+$ in the following cases.

\textbf{Case 6.1}. \emph{Wave $\alpha \in \mathcal{J}^{\textsc{e}}(\hat{V}) \setminus\mathcal{NP}$}.
To begin with, we present some elementary estimates on indices $q^\pm_i$ and $q_i^\oplus$, etc.
For $1\leq i \leq4$, there hold
\begin{eqnarray*}
\begin{split}
&q_i^+  -  q_i^\oplus =\mathcal{O}(1)\big(|\alpha| \cdot |\hat{w}^-| + |\Delta\hat{w}|\big), \quad
q_i^\oplus-q_i^- =\mathcal{O}(1) |\alpha|, \\[5pt]
&\lambda_i^+ -\lambda_i^\oplus =\mathcal{O}(1)\big(|\alpha| \cdot |\hat{w}^-| + |\Delta\hat{w}|\big), \quad
\dot{x}_\alpha  - \dot{x}^\oplus = \mathcal{O}(1)\big( |\hat{w}^-|+\varepsilon\big),
\end{split}
\end{eqnarray*}
where $\Delta\hat{w} = \hat{w}^+ - \hat{w}^-$.

If $q_i^+ q_i^\oplus \leq 0 $, it is clear that
\begin{eqnarray*}
|q_i^+|  +  |q_i^\oplus|= |q_i^+ -q_i^\oplus|= \mathcal{O}(1)\big(|\alpha| \cdot |\hat{w}^-| + |\Delta\hat{w}|\big).
\end{eqnarray*}
Thus,
\begin{eqnarray*}
\begin{split}
\sum_{i=1}^4 \widetilde{\mathscr{E}}^+_{i, \alpha}
&=\sum_{i=1}^4 \ell_i \Big(\mathcal{W}_i^+ |q_i^+ | (\lambda_i^+ - \dot{x}_\alpha) -\mathcal{W}_{i}^+ |q_i^\oplus |\big(\lambda_i^\oplus - \dot{x}^\oplus\big)
+\mathcal{W}_i^- |q_i^- | (\dot{x}_\alpha - \dot{x}^\oplus) \Big)\\[5pt]
&\leq \mathcal{O}(1)\big(|\alpha| \cdot |\hat{w}^-| + |\Delta\hat{w}|\big)
	+ \mathcal{O}(1)\big(|\alpha|\cdot|\hat{w}^-| + |\Delta\hat{w}| + |\alpha|\big)\big(|\hat{w}^-| + \varepsilon\big)\\[5pt]
&= \mathcal{O}(1) \big(|\alpha| ( |\hat{w}^-|+ \varepsilon) + |\Delta\hat{w}| \big).
\end{split}
\end{eqnarray*}

If $q_i^+ q_i^\oplus > 0$, then we have $\mathcal{W}_i^+  |q_i^\oplus | - \mathcal{W}_i^- |q_i^-|= \mathcal{O}(1)|\alpha|$ and
\begin{eqnarray*}
\begin{split}
\sum_{i=1}^4 \widetilde{\mathscr{E}}_{i,\alpha}^+
&=\sum_{i=1}^4 \ell_i  \Big(\mathcal{W}_i^+ \big(|q_i^+|-|q_i^\oplus|\big) \big( \lambda_i^+ - \dot{x}_\alpha\big)+\mathcal{W}_{i}^+  |q_i^\oplus|\big(\lambda_i^+ - \lambda_i^\oplus\big)\\[5pt]
&\quad + \big(\mathcal{W}_i^+|q_i^\oplus|- \mathcal{W}_i^-  |q_i^-|\big) \big(\dot{x}^\oplus-\dot{x}_\alpha\big)\Big)\\[5pt]
&\leq \mathcal{O}(1) \big(|\alpha| \cdot |\hat{w}^-| + |\Delta\hat{w}|\big)+\mathcal{O}(1)|\alpha|\big(|\hat{w}^-| + \varepsilon\big)\\[5pt]
&=\mathcal{O}(1) \Big( |\alpha|\big(|\hat{w}^-|+ \varepsilon\big) + |\Delta\hat{w}|\Big).
\end{split}
\end{eqnarray*}

\textbf{Case 6.2}. \emph{Wave $\alpha \in \mathcal{J}^{\textsc{y}}_3(\hat{V})$}.
For $k_\alpha =3$, one has $\alpha= Y^+ - Y^-$ and the front speed $\dot{x}_\alpha =\lambda_3(\hat{V})$. The corresponding auxiliary quantities satisfy $\lambda_i^\oplus  = \lambda_i ^-$ and
$\dot{x}^\oplus = \lambda_3 (\hat{V}^- + \hat{w}^-)$. Moreover, it is clear that
\begin{eqnarray*}
q_i^+  -  q_i^\oplus =\mathcal{O}(1)|\Delta\hat{w}| , \quad \lambda_i^+  -  \lambda_i^\oplus =\mathcal{O}(1) |\Delta\hat{w}|,\quad \dot{x}_\alpha  - \dot{x}^\oplus = 0,
\end{eqnarray*}
which imply
\begin{eqnarray*}
\begin{split}
\sum_{i=1}^4\widetilde{\mathscr{E}}_{i,\alpha}^+
&=\sum_{i=1}^4 \ell_i \mathcal{W}_i^+ \Big(|q_i^+ | (\lambda_i^+ - \dot{x}_\alpha) -   |q_i^\oplus | (\lambda_i^\oplus - \dot{x}_\alpha) \Big)\\[5pt]
&=\sum_{i=1}^4  \ell_i\mathcal{W}_i^+ \Big( \big(|q_i^+ | -   |q_i^\oplus | \big)\big(\lambda_i^+ - \dot{x}_\alpha\big) + |q_i^\oplus | \big( \lambda_i^+ - \lambda_i^\oplus\big) \Big)\\[5pt]
&=\mathcal{O}(1) |\Delta\hat{w}|.
\end{split}
\end{eqnarray*}

\textbf{Case 6.3}. \emph{Wave $\alpha \in \mathcal{S}_{\rm{b}}(\hat{V})$}.
For the strong $4$-shock $\alpha$ of $\hat{V}$, there hold
\begin{eqnarray*}
q_i^+ - q_i^\oplus=  \mathcal{O}(1)\big(|\hat{w}^-| +|\hat{w}^+|\big),\quad \lambda_i^+ - \lambda_i^\oplus=\mathcal{O}(1)\big(|\hat{w}^-| +|\hat{w}^+| \big),
\end{eqnarray*}
and
\begin{eqnarray*}
\dot{x}_\alpha - \dot{x}^\oplus = \mathcal{O}(1)\big(|\hat{w}^-| +|\hat{w}^+| \big).
\end{eqnarray*}
Then we have
\begin{eqnarray*}
\begin{split}
\quad\sum_{i=1}^4 \widetilde{\mathscr{E}}_{i,\alpha} ^+
&=\sum_{i=1}^4 \Big(\ell_i^+ \mathcal{W}_i^+ \big( (|q_i^+ | - |q_i^\oplus |) (\lambda_i^+ - \dot{x}_\alpha) + |q_i^\oplus | (\lambda_i^+ - \lambda_i^\oplus + \dot{x}^\oplus -\dot{x}_\alpha) \big) \\
& \hspace{1.3cm} + \ell_i^- \mathcal{W}_i^- |q_i^- | (\dot{x}_\alpha - \dot{x}^\oplus) \Big)\\[5pt]
&= \mathcal{O}(1)\big(|\hat{w}^-| +|\hat{w}^+|\big)\\[5pt]
&= \mathcal{O}(1) \varepsilon .
\end{split}
\end{eqnarray*}

\textbf{Case 6.4}. \emph{Wave $\alpha \in \mathcal{J}^{\textsc{e}}(\hat{V})\cap\mathcal{NP}$}.
Recall that $|\alpha|= |\hat{V}^+ - \hat{V}^-|$. Then from
\begin{eqnarray*}
\mathcal{W}_i^+= \mathcal{W}_i^-, \qquad q_i^+  -  q_i^- = \mathcal{O}(1)|\alpha|,\qquad  \lambda_i^+  - \lambda_i^- =\mathcal{O}(1)|\alpha|,
\end{eqnarray*}
it follows that
\begin{align*}
\sum_{i=1}^4 \widetilde{\mathscr{E}}_{i,\alpha}
&=\sum_{i=1}^4 \Big(\ell_i \mathcal{W}_i^+ |q_i^+ |\big(\lambda_i^+ - \dot{x}_\alpha\big) -  \ell_i \mathcal{W}_i^- |q_i^- |\big(\lambda_i^- - \dot{x}_\alpha\big) \Big)\\[5pt]
	& =  \sum_{i=1}^4  \ell_i \mathcal{W}_i^+ \Big( (|q_i^+ | -|q_i^- |) (\lambda_i^- - \dot{x}_\alpha)+  |q_i^- |(\lambda_i^+ - \lambda_i^- )\Big) \\[5pt]
	& = \mathcal{O}(1) |\alpha| .
\end{align*}

Combining the above four cases altogether, we eventually obtain
\begin{eqnarray*}
\begin{split}
\dot{\mathscr{L}}_{\textrm{p}}(\hat{U}, \hat{V} + \hat{w} )
&=\sum_{\alpha} \sum_{i=1}^4 \widetilde{\mathscr{E}}_{i,\alpha} + \sum_{i=1}^4 \widetilde{\mathscr{E}}_{i,\mathfrak{p}}\\[5pt]
	& \leq \mathcal{O}(1) \Big(\varepsilon+ \|{ \hat{w}}\|_{L^{\infty}} + T.V.\{\hat{w}; (\chi_{\mathfrak{p}}(\tau),+\infty) \} \Big) \\[5pt]
	& = \mathcal{O}(1) \big(\varepsilon+  (\gamma-1)\varepsilon T.V.\{Y\phi(T);(\chi_{\mathfrak{p}}(\tau),+\infty)\} \big) \\[5pt]
	& = \mathcal{O}(1) \varepsilon.
\end{split}
\end{eqnarray*}
Therefore
\begin{eqnarray*}
\mathscr{L}_{\textrm{p}}(\hat{U}(t), (\hat{V} + \hat{w})(t)) - \mathscr{L}_{\textrm{p}}(\hat{U}(\tilde{t}), (\hat{V} + \hat{w})(\tilde{t}))
\leq \mathcal{O}(1) \varepsilon (t-\tilde{t}), \quad \forall\ t>\tilde{t} \geq 0.
\end{eqnarray*}
Taking $\tilde{t}=0$ and using the equivalence between $\mathscr{L}_{\textrm{p}}$ and $L^1$ metric, we conclude that
\begin{eqnarray*}
\|\mathcal{S}^\varepsilon _t \circ\mathcal{T}_\varepsilon ({U}_\tau, u_{\mathfrak{p}})-\mathcal{T}_\varepsilon \circ\mathcal{S}^\varepsilon _t ({U}_\tau, u_{\mathfrak{p}})\|_{L^1(\mathbb{R})}
=\|\hat{U}- \hat{V} -\hat{w}\|_{L^1(\mathbb{R})}
\leq \mathcal{O}(1) \varepsilon t.
\end{eqnarray*}

\end{proof}

Next we proceed to establish the tangent estimate which exhibits that $\mathcal{P}^\varepsilon _t$ and $\mathcal{S}^\varepsilon _t\circ \mathcal{T}_t $ have an identical tangent operator at $({U}_\tau, u_{\mathfrak{p}})$.

\begin{lemma}[Tangent estimate] \label{Tangent estimate}
Given initial-boundary data $({U}_\tau, u_{\mathfrak{p}})$ under condition $\widetilde{\rm{(C1)}}$ or $\widetilde{\rm{(C2)}}$, there holds
\begin{eqnarray*}
\|\mathcal{P}^\varepsilon _t ({U}_\tau, u_{\mathfrak{p}})-\mathcal{S}^\varepsilon _t\circ\mathcal{T}_t ({U}_\tau, u_{\mathfrak{p}}) \|_{L^1(\mathbb{R})}=\mathcal{O}(1)\big(1+\|{ Y_{1,\tau}}\|_{L^1(\mathbb{R})} \big) t^2
\end{eqnarray*}
for every $t<\bar{t}$ and $ \varepsilon\ll \min\{t, t^2\}$.
\end{lemma}

\begin{proof}
Since operator $\mathcal{T}_t$ does not affect the velocity of piston, we claim that $\mathcal{P}^\varepsilon _t ({U}_\tau, u_{\mathfrak{p}})$,
$\mathcal{S}^\varepsilon _t \circ \mathcal{T}_t ({U}_\tau, u_{\mathfrak{p}}) $ have the same boundary $\chi_{\mathfrak{p},1}$.
Set $n=\lceil t/\varepsilon\rceil$. We first consider the partially ignited flow under condition $\widetilde{\textrm{(C2)}}$.
Recall that ${ \chi_{\rm {s},1}}$ is the flame front in ${U}_\tau$. Then the fluid temperature dramatically exceeds its ignition temperature behind the front, \emph{i.e.},
\begin{eqnarray*}
T>T_{\text{i}} \quad \mbox{for}\  x\in (\chi_{\mathfrak{p},1},{ \chi_{\rm {s},1}}),\qquad
T<T_{\text{i}} \quad \mbox{for}\  x\in ({ \chi_{\rm {s},1}}, +\infty).
\end{eqnarray*}
This implies
\begin{eqnarray*}
G_{0}\big((\mathcal{T}_\varepsilon)^n ({U}_\tau)\big)\neq0, \quad G_{0}(\mathcal{T}_{\varepsilon n} ({U}_\tau))\neq 0 \quad \mbox{for}\ \  x\in(\chi_{\mathfrak{p},1},{ \chi_{\rm {s},1}}),
\end{eqnarray*}
and
\begin{eqnarray*}
G_{0}\big((\mathcal{T}_\varepsilon)^n ({U}_\tau)\big)=G_0(\mathcal{T}_{\varepsilon n}({U}_\tau))=0 \quad \mbox{for}\ \  x\in ({ \chi_{\rm {s},1}}, +\infty),
\end{eqnarray*}
provided $n\varepsilon$ is small, where $(\mathcal{T}_\varepsilon)^n$ stands for composition of $n$ operators $\mathcal{T}_\varepsilon$. Furthermore,
\begin{eqnarray*}
\begin{split}
\|G_{0}\big((\mathcal{T}_\varepsilon)^n ({U}_\tau)\big)- G_0(\mathcal{T}_{n\varepsilon} ({U}_\tau))\|_{L^1(\mathbb{R})}
&\leq L_{G_0}  \int_{\chi_{\mathfrak{p},1}} ^{{ \chi_{\rm {s},1}}}| (\mathcal{T}_\varepsilon)^n ({U}_\tau)-\mathcal{T}_{n\varepsilon} ({U}_\tau)|dx \\[5pt]
&\leq L_{G_0}\|(\mathcal{T}_\varepsilon)^n ({U}_\tau)- \mathcal{T}_{n\varepsilon} ({U}_\tau)\|_{L^1(\mathbb{R})}
\end{split}
\end{eqnarray*}
where $L_{G_0}$ denotes the local Lipschitz constant of function $G_0$. By induction, we derive that
\begin{equation}\label{induction on T^n}
\|(\mathcal{T}_\varepsilon)^n({U}_\tau)- \mathcal{T}_{n\varepsilon} ({U}_\tau)\|_{L^1(\mathbb{R})}=\mathcal{O}(1) (n-1)^2 \varepsilon^2 \| Y_{1,\tau}\|_{L^1(\mathbb{R})} .
\end{equation}
The above equality also holds under condition $\widetilde{\textrm{(C1)}}$. The flow estimates (\ref{P_t approximate continuity 1})(\ref{S_t approximate continuity 1})(\ref{induction on T^n}) and Lemma \ref{commutation estimate} yield that
\begin{eqnarray}\label{Estimate Pn-SnTn}
\begin{split}
&\| \mathcal{P}_{n\varepsilon}^\varepsilon ({U}_\tau, u_{\mathfrak{p}})- \mathcal{S}_{n\varepsilon}^\varepsilon\circ \mathcal{T}_{n\varepsilon}({U}_\tau, u_{\mathfrak{p}})\|_{L^1(\mathbb{R})} \\[5pt]
\leq& \| \mathcal{P}_{n\varepsilon}^\varepsilon ({U}_\tau, u_{\mathfrak{p}})
		- \mathcal{S}_{n\varepsilon}^\varepsilon \circ (\mathcal{T}_\varepsilon)^n ({U}_\tau, u_{\mathfrak{p}})\|_{L^1(\mathbb{R})} \\
	&	+ \|\mathcal{S}_{n\varepsilon}^\varepsilon \circ (\mathcal{T}_\varepsilon)^n ({U}_\tau, u_{\mathfrak{p}}) - \mathcal{S}_{n\varepsilon}^\varepsilon \circ\mathcal{T}_{n\varepsilon}({U}_\tau, u_{\mathfrak{p}})\|_{L^1(\mathbb{R})}\\[5pt]
\leq & \sum_{k=1}^n \|\mathcal{P}_{(n-k)\varepsilon}^\varepsilon \circ\mathcal{T}_\varepsilon\circ\mathcal{S}_{k\varepsilon}^\varepsilon \circ (\mathcal{T}_\varepsilon)^{k-1}({U}_\tau, u_{\mathfrak{p}})
		- \mathcal{P}_{(n-k)\varepsilon}^\varepsilon \circ\mathcal{S}_{k\varepsilon}^\varepsilon \circ(\mathcal{T}_\varepsilon)^{k}({U}_\tau, u_{\mathfrak{p}})\|_{L^1(\mathbb{R})}\\[5pt]
	&+\|\mathcal{S}_{n\varepsilon}^\varepsilon\circ (\mathcal{T}_\varepsilon)^n ({U}_\tau, u_{\mathfrak{p}}) - \mathcal{S}_{n\varepsilon}^\varepsilon \circ \mathcal{T}_{n\varepsilon} ({U}_\tau, u_{\mathfrak{p}})\|_{L^1(\mathbb{R})} \\[5pt]
\leq & \mathcal{O}(1) \mathop{\sum}_{k=1}^n \big(\| \mathcal{T}_\varepsilon\circ\mathcal{S}_{k\varepsilon}^\varepsilon\circ (\mathcal{T}_\varepsilon)^{k-1}({U}_\tau, u_{\mathfrak{p}})
		-\mathcal{S}_{k\varepsilon}^\varepsilon \circ(\mathcal{T}_\varepsilon)^{k}({U}_\tau, u_{\mathfrak{p}})\|_{L^1(\mathbb{R})}+(n-k)\varepsilon^2 \big)  \\[5pt]
& + \mathcal{O}(1) \|(\mathcal{T}_\varepsilon)^n ({U}_\tau) -\mathcal{T}_{n\varepsilon}  ({U}_\tau)\|_{L^1(\mathbb{R})}+ \mathcal{O}(1) n\varepsilon^2  \\[5pt]
\leq & \mathcal{O}(1) \sum_{k=1}^n \big( k\varepsilon^2+(n-k)\varepsilon^2   \big)  +  \mathcal{O}(1) (n-1)^2\varepsilon^2\|Y_{1,\tau}\|_{L^1(\mathbb{R})}  + \mathcal{O}(1) n\varepsilon^2  \\[5pt]
= & \mathcal{O}(1)\big(1+\| Y_{1,\tau}\|_{L^1(\mathbb{R})}\big) n^2\varepsilon^2.
\end{split}
\end{eqnarray}
We consequently deduce that
\begin{eqnarray*}
\begin{split}
&\|\mathcal{P}_t^\varepsilon ({U}_\tau, u_{\mathfrak{p}})- \mathcal{S}_t^\varepsilon\circ\mathcal{T}_t ({U}_\tau, u_{\mathfrak{p}})\|_{L^1(\mathbb{R})}\\[5pt]
\leq & \|\mathcal{P}_t^\varepsilon ({U}_\tau, u_{\mathfrak{p}})-   \mathcal{P}_{n\varepsilon}^\varepsilon ({U}_\tau, u_{\mathfrak{p}})\|_{L^1(\mathbb{R})}
	+ \|\mathcal{P}_{n\varepsilon}^\varepsilon({U}_\tau, u_{\mathfrak{p}})- \mathcal{S}_{n\varepsilon}^\varepsilon \circ\mathcal{T}_{n\varepsilon}({U}_\tau, u_{\mathfrak{p}})\|_{L^1(\mathbb{R})}\\[5pt]
&+\|\mathcal{S}_{n\varepsilon}^\varepsilon\circ\mathcal{T}_{n\varepsilon}({U}_\tau, u_{\mathfrak{p}}) - \mathcal{S}_{t}^\varepsilon \circ\mathcal{T}_{t} ({U}_\tau, u_{\mathfrak{p}})\|_{L^1(\mathbb{R})} \\[5pt]
\leq & \mathcal{O}(1)(t-n\varepsilon) +  \mathcal{O}(1)\big(1+\| Y_{1,\tau}\|_{L^1(\mathbb{R})}\big) n^2\varepsilon^2 \\[5pt]
&+\mathcal{O}(1)\|\mathcal{T}_{n\varepsilon} ({U}_\tau) - \mathcal{T}_{t} ({U}_\tau)\|_{L^1(\mathbb{R})}   + \mathcal{O}(1)n\varepsilon^2 \\[5pt]
\leq & \mathcal{O}(1)\big(1+\| Y_{1,\tau}\|_{L^1(\mathbb{R})}\big) t^2
\end{split}
\end{eqnarray*}
by estimates (\ref{Estimate Pn-SnTn}) and the Lipschitz continuity of $\mathcal{S}_t^\varepsilon $ in (\ref{S_t approximate continuity 1})(\ref{S_t approximate continuity 2}).

\end{proof}

Finally, we establish the time-additivity of reacting flow $\mathcal{P}^\varepsilon _t (U_\tau, u_{\mathfrak{p}})$.

\begin{lemma}[Additivity estimate]\label{semigroup estimate}
Given initial-boundary data $(U_\tau, u_{\mathfrak{p}})$ under condition $\widetilde{\textsc{(C1)}}$ or $\widetilde{\textsc{(C2)}}$, there holds
\begin{eqnarray*}
\|\mathcal{P}^\varepsilon _{\tilde{t}}\circ\mathcal{P}^\varepsilon _{t}(U_\tau, u_{\mathfrak{p}})- \mathcal{P}^\varepsilon _{{ t+\tilde{t}}}(U_\tau, u_{\mathfrak{p}})\|_{L^1(\mathbb{R})}
=\mathcal{O}(1) \big(1+\varepsilon+t+\tilde{t}+ \| Y_{1,\tau}\|_{L^1(\mathbb{R})}\big)\varepsilon,
\end{eqnarray*}
for every $t,\ \tilde{t}>0$ and $\varepsilon \ll \bar{t}$.
\end{lemma}

\begin{proof}
Fix $\varepsilon \ll\bar{t}$. There exist two integers $k$ and $\textrm{K}$ such that
\begin{eqnarray*}
(k-1)\varepsilon < t \leq k\varepsilon, \quad  \textrm{K}\varepsilon \leq t +\tilde{t} < (\textrm{K}+1)\varepsilon,
\end{eqnarray*}
which lead to
\begin{eqnarray*}
\lceil \tilde{t}/\varepsilon \rceil = \textrm{K}-k  \quad\mbox{or}\quad  \textrm{K}-k+1.
\end{eqnarray*}
It suffices to discuss the case $\lceil \tilde{t}/\varepsilon \rceil= \textrm{K}-k$. Set $\tilde{U}\doteq \mathcal{P}_{t}^\varepsilon (U_\tau, u_{\mathfrak{p}})$. Then the Lipschitz continuity
(\ref{P_t approximate continuity 1})-(\ref{S_t approximate continuity 2}) and Proposition \ref{prop:4.1} give that
\begin{eqnarray*}
\begin{split}
&\|\mathcal{P}_{\tilde{t}}^\varepsilon \circ\mathcal{P}_{t}^\varepsilon (U_{\tau},u_{\mathfrak{p}})- \mathcal{P}_{t +\tilde{t}}^\varepsilon (U_{\tau},u_{\mathfrak{p}})\|_{L^1(\mathbb{R})}\\[5pt]
\leq & \|\mathcal{S}_{t+\tilde{t}-\textrm{K}\varepsilon}^\varepsilon \circ\mathcal{S}_{k\varepsilon- t}^\varepsilon\circ\mathcal{P}_{(\textrm{K}-k)\varepsilon}^\varepsilon (\Tilde{U},u_{\mathfrak{p}})
-\mathcal{S}_{t+\tilde{t}-\textrm{K}\varepsilon}^\varepsilon \circ \mathcal{P}_{(\textrm{K}-k)\varepsilon}^\varepsilon \circ\mathcal{T}_\varepsilon\circ \mathcal{S}_{k\varepsilon- t}^\varepsilon (\Tilde{U},u_{\mathfrak{p}})\|_{L^1(\mathbb{R})}  \\[5pt]
\leq & \mathcal{O}(1)\| \mathcal{S}_{k\varepsilon- t}^\varepsilon \circ\mathcal{P}_{(\textrm{K}-k)\varepsilon}^\varepsilon (\Tilde{U},u_{\mathfrak{p}})
	-\mathcal{P}_{(\textrm{K}-k)\varepsilon}^\varepsilon \circ\mathcal{T}_\varepsilon\circ \mathcal{S}_{k\varepsilon- t}^\varepsilon (\Tilde{U},u_{\mathfrak{p}})\|_{L^1(\mathbb{R})} \\[5pt]
&+\mathcal{O}(1) (t+\tilde{t}-\textrm{K}\varepsilon)\varepsilon\\[5pt]
\leq & \mathcal{O}(1)\| \mathcal{P}_{(\textrm{K}-k)\varepsilon}^\varepsilon \circ\mathcal{T}_\varepsilon \circ\mathcal{S}_{k\varepsilon- t}^\varepsilon (\Tilde{U},u_{\mathfrak{p}})
	-\mathcal{P}_{(\textrm{K}-k)\varepsilon}^\varepsilon (\Tilde{U},u_{\mathfrak{p}})\|_{L^1(\mathbb{R})} \\[5pt]
&+\mathcal{O}(1)\|\mathcal{S}_{k\varepsilon- t}^\varepsilon \circ\mathcal{P}_{(\textrm{K}-k)\varepsilon}^\varepsilon (\Tilde{U},u_{\mathfrak{p}})
	- \mathcal{P}_{(\textrm{K}-k)\varepsilon}^\varepsilon (\Tilde{U},u_{\mathfrak{p}})\|_{L^1(\mathbb{R})} + \mathcal{O}(1) \varepsilon^2 \\[5pt]
\leq & \mathcal{O}(1)\|\mathcal{T}_\varepsilon\circ \mathcal{S}_{k\varepsilon- t}^\varepsilon(\Tilde{U},u_{\mathfrak{p}}) - \tilde{{U}}\|_{L^1(\mathbb{R})}+ \mathcal{O}(1) (\textrm{K}-k)\varepsilon^2 
 +\mathcal{O}(1) (k\varepsilon - t) + \mathcal{O}(1) \varepsilon^2 \\[5pt]
\leq & \mathcal{O}(1) \big(\| \mathcal{S}_{k\varepsilon- t}^\varepsilon (\Tilde{U},u_{\mathfrak{p}}) - \tilde{{U}}\|_{L^1(\mathbb{R})}
    + \varepsilon\|G_0(\mathcal{S}_{k\varepsilon- t}^\varepsilon (\Tilde{U},u_{\mathfrak{p}}))\|_{L^1(\mathbb{R})} \big) 
 + \mathcal{O}(1) \big(\tilde{t}\varepsilon + \varepsilon+ \varepsilon^2\big)  \\[5pt]
\leq & \mathcal{O}(1)(k\varepsilon - t)+ \mathcal{O}(1) \| Y_{1,\tau}\|_{L^1(\mathbb{R})}\varepsilon+ \mathcal{O}(1)\big(\tilde{t}\varepsilon + \varepsilon+ \varepsilon^2\big)  \\[5pt]
\leq & \mathcal{O}(1) \big(1+\varepsilon+t+\tilde{t}+ \|Y_{1,\tau}\|_{L^1(\mathbb{R})}\big)\varepsilon.
\end{split}
\end{eqnarray*}

\end{proof}

\emph{Proof of Theorem \ref{thm:1.1} for the uniqueness under condition $\rm{(C2)}$ or $\rm{(C3)}$}.
Remember that $({U}_0^\varepsilon ,u_{\mathfrak{p}}^\varepsilon)$ is an approximate initial-boundary data under condition $\textrm{(C2)}$ or $\textrm{(C3)}$.
Taking $\varepsilon= \mu, \nu$ respectively for $\mathcal{P}_t^\varepsilon ({U}_0^\varepsilon ,u_{\mathfrak{p}}^\varepsilon)$, we have
\begin{eqnarray*}
\begin{split}
\|\mathcal{P}_t^\mu ({U}_0^\mu, u_{\mathfrak{p}}^\mu)-\mathcal{P}_t^\nu  ({U}_0^\nu ,u_{\mathfrak{p}}^\nu)\|_{L^1(\mathbb{R})}
&\leq \|\mathcal{P}_t^\mu (U^\mu_0, u_{\mathfrak{p}}^\mu)-\mathcal{P}_t^\nu (U^\mu_0, u_{\mathfrak{p}}^\mu)\|_{L^1(\mathbb{R})}\\[5pt]
&\quad+ \|\mathcal{P}_t^\nu (\mathbf{U}^\mu_0, u_{\mathfrak{p}}^\mu) -\mathcal{P}_t^\nu (\mathbf{U}^\nu_0, u_{\mathfrak{p}}^\nu)\|_{L^1(\mathbb{R})}.
\end{split}
\end{eqnarray*}
To prove convergence, we need a local error estimate of distinct trajectories $\mathcal{P}_t^\mu (U^\mu_0, u_{\mathfrak{p}}^\mu)$ and $\mathcal{P}_t^\nu (U^\mu_0, u_{\mathfrak{p}}^\mu)$ in advance.
As shifting the initial time to $\tau>0$, we use the notations ${U}^\mu ={U}^\mu(\tau,x)$ and $u_{\mathfrak{p}}^\mu =u_{\mathfrak{p}}^\mu(t+\tau)$ in this paragraph for convenience. Then it is deduced from Lemma \ref{Tangent estimate} and (\ref{S_t approximate continuity 3}) that
\begin{eqnarray}\label{lemma local error Pt}
\begin{split}
&\|\mathcal{P}_t^\mu ({U}^\mu , u_{\mathfrak{p}}^\mu)- \mathcal{P}_t^\nu ({U}^\mu , u_{\mathfrak{p}}^\mu)\|_{L^1(\mathbb{R})} \\[5pt]
\leq & \|\mathcal{P}_t^\mu ({U}^\mu , u_{\mathfrak{p}}^\mu)- \mathcal{S}_t^\mu \circ\mathcal{T}_t({U}^\mu , u_{\mathfrak{p}}^\mu)\|_{L^1(\mathbb{R})}\\[5pt]
&+ \|\mathcal{P}_t^\nu ({U}^\mu , u_{\mathfrak{p}}^\mu)- \mathcal{S}_t^\nu \circ\mathcal{T}_t({U}^\mu, u_{\mathfrak{p}}^\mu)\|_{L^1(\mathbb{R})} \\[5pt]
& +\|\mathcal{S}_t^\mu \circ\mathcal{T}_t ({U}^\mu , u_{{\mathfrak{p}}}^\mu)- \mathcal{S}_t^\nu \circ\mathcal{T}_t({U}^\mu , u_{\mathfrak{p}}^\mu)\|_{L^1(\mathbb{R})} \\[5pt]
\leq & \mathcal{O}(1)\big(1+\|Y^\mu(\tau,\cdot) \|_{L^1(\mathbb{R})}\big) t^2 + \mathcal{O}(1)\max\{\mu, \nu\}t \\[5pt]
\leq & \mathcal{O}(1)\big(1+\|Y^\mu(\tau,\cdot) \|_{L^1(\mathbb{R})}\big) t^2,
\end{split}
\end{eqnarray}
for any positive $\mu,\ \nu \ll \min\{t,t^2\}$ and  $t<\bar{t}$.

Now we proceed to discuss the uniqueness of limit solution. Assume that $({U}^\mu_0, u_{\mathfrak{p}}^\mu)$ and $({U}^\nu_0, u_{\mathfrak{p}}^\nu)$ are two couples of piecewise constant functions, which satisfy that
\begin{eqnarray*}
\begin{split}
&T.V.\{({U}^{\varepsilon}_0,u^{\varepsilon}_{\mathfrak{p}}); \mathbb{R}_+\}\leq T.V.\{({U}_0,u_{\mathfrak{p}}); \mathbb{R}_+\}
\qquad \mbox{for}\quad \varepsilon=\mu,\ \nu,
\end{split}
\end{eqnarray*}
and
\begin{eqnarray*}
\begin{split}
&\|{U}^{\varepsilon}_0- {U}_0\|_{L^1 (\mathbb{R}_+)} < {\varepsilon}, \qquad\|u^{\varepsilon}_{\mathfrak{p}}- u_{\mathfrak{p}}\|_{L^1 (\mathbb{R}_+)} <{\varepsilon},\qquad \mbox{for}\quad \varepsilon=\mu,\ \nu.
\end{split}
\end{eqnarray*}
 We claim that for every $t\geq 0$, there holds
\begin{equation}\label{Cauchy sequence convergence}
\| \mathcal{P}^\mu _{t} ({U}^\mu_0, u_{\mathfrak{p}}^\mu)- \mathcal{P}^\nu _{t} ({U}^\nu_0, u_{\mathfrak{p}}^\nu)\|_{L^1(\mathbb{R})}\rightarrow 0\qquad \mbox{as}\quad \mu,\ \nu \rightarrow 0.
\end{equation}
To prove the claim, let's divide the interval $[0,t]$ equally into $m$ parts such that $\Delta t= t/m$ is small.
Let the points of such division satisfy $0=\tilde{t}_0 < \tilde{t}_1 < \cdots < \tilde{t}_m = t$.
If $\mu, \nu < \text{min}\{\bar{t},(\Delta t) ^2\}$,  it is deduced from Lemma \ref{semigroup estimate} that
\begin{eqnarray}\label{Estimate 4}
\begin{split}
&\quad\|\mathcal{P}_t^\mu ({U}^\mu_0, u^\mu_{\mathfrak{p}})-\mathcal{P}_t^\nu ({U}^\mu_0, u^\mu_{\mathfrak{p}})\|_{L^1(\mathbb{R})}\\[5pt]
&\leq \sum_{i=1}^m \|\mathcal{P}_{t-\tilde{t}_i}^\mu \circ\mathcal{P}_{\tilde{t}_i}^\nu({U}^\mu_0, u^\mu_{\mathfrak{p}})-\mathcal{P}_{t-\tilde{t}_{i-1}}^\mu\circ \mathcal{P}_{\tilde{t}_{i-1}}^\nu({U}^\mu_0, u^\mu_{\mathfrak{p}})\|_{L^1(\mathbb{R})} \\[5pt]
&\leq  \sum_{i=1}^m \|\mathcal{P}_{t-\tilde{t}_i}^\mu\circ \mathcal{P}_{\tilde{t}_i}^\nu({U}^\mu_0, u^\mu_{\mathfrak{p}})- \mathcal{P}_{t-\tilde{t}_{i}}^\mu\circ\mathcal{P}_{\tilde{t}_i-\tilde{t}_{i-1}}^\mu \circ\mathcal{P}_{\tilde{t}_{i-1}}^\nu({U}^\mu_0, u^\mu_{\mathfrak{p}})\|_{L^1(\mathbb{R})}\\[5pt]
&\quad +  \mathcal{O}(1) \sum_{i=1}^m\big(1+\mu+ t-\tilde{t}_{i-1} + \|Y_0^\mu\|_{L^1(\mathbb{R})}\big)\mu \\[5pt]
&\doteq   \Sigma_1 + \Sigma_2.
\end{split}
\end{eqnarray}
Using flow estimates (\ref{P_t approximate continuity 1})(\ref{lemma local error Pt}), Proposition \ref{prop:4.1} and Lemma \ref{semigroup estimate},
we derive
\begin{align}
\notag
\Sigma_1 &= \sum_{i=1}^m \|\mathcal{P}_{t-\tilde{t}_i}^\mu \circ\mathcal{P}_{\tilde{t}_i}^\nu({U}^\mu_0, u^\mu_{\mathfrak{p}})
- \mathcal{P}_{t-\tilde{t}_{i}}^\mu\circ\mathcal{P}_{\tilde{t}_i-\tilde{t}_{i-1}}^\mu\circ \mathcal{P}_{\tilde{t}_{i-1}}^\nu({U}^\mu_0, u^\mu_{\mathfrak{p}})\|_{L^1(\mathbb{R})} \\
\notag
&\leq \mathcal{O}(1)  \sum_{i=1}^m \| \mathcal{P}_{\tilde{t}_i}^\nu({U}^\mu_0, u^\mu_{\mathfrak{p}})-\mathcal{P}_{\tilde{t}_i-\tilde{t}_{i-1}}^\mu \circ\mathcal{P}_{\tilde{t}_{i-1}}^\nu({U}^\mu_0, u^\mu_{\mathfrak{p}})\|_{L^1(\mathbb{R})}
	+ \mathcal{O}(1) \sum_{i=1}^m (t-\tilde{t}_i)\mu \\	
\notag
&\leq \mathcal{O}(1) \sum_{i=1}^m || \mathcal{P}_{\tilde{t}_i-\tilde{t}_{i-1}}^\nu\circ\mathcal{P}_{\tilde{t}_{i-1}}^\nu({U}^\mu_0, u^\mu_{\mathfrak{p}})-\mathcal{P}_{\tilde{t}_i-\tilde{t}_{i-1}}^\mu \circ\mathcal{P}_{\tilde{t}_{i-1}}^\nu({U}^\mu_0, u^\mu_{\mathfrak{p}})\|_{L^1(\mathbb{R})}  \\
\notag
& \quad + \mathcal{O}(1)\sum_{i=1}^m (1+ \nu + \tilde{t}_i +\| Y_0^\mu\|_{L^1(\mathbb{R})}) \nu+\mathcal{O}(1)\sum_{i=1}^m (t-\tilde{t}_i)\mu \\[5pt]
\notag
&\leq \mathcal{O}(1)\sum_{i=1}^m \big(1+\| Y_0^\mu\|_{L^1(\mathbb{R})} \big)(\Delta t)^2 \\
\notag
& \quad +\mathcal{O}(1)\big(1+ \nu +  t +\| Y_0^\mu\|_{L^1(\mathbb{R})} \big )m \nu+\mathcal{O}(1) tm\mu \\
\notag
&\leq \mathcal{O}(1)\big(1+ \mu + \|Y_0\|_{L^1(\mathbb{R})} \big)t\Delta t \\
\label{Estimate sigma1}
&\quad + \mathcal{O}(1) \big(1+\mu+ \nu + t +\|Y_0 \|_{L^1(\mathbb{R})}\big)m \nu+\mathcal{O}(1) tm\mu,
\end{align}
and
\begin{eqnarray}\label{Estimate sigma2}
\begin{split}
\Sigma_2&= \mathcal{O}(1) \sum_{i=1}^m \big(1+\mu+ t-\tilde{t}_{i-1} +\| Y_0^\mu\|_{L^1(\mathbb{R})}\big)\mu \\[5pt]
& \leq \mathcal{O}(1)\big(1+\mu  + t + \|Y_0 \|_{L^1(\mathbb{R})}\big)m \mu .
\end{split}
\end{eqnarray}
The above estimates (\ref{Estimate 4})-(\ref{Estimate sigma2}) yield
\begin{eqnarray*}
\begin{split}
\|\mathcal{P}_t^\mu ({U}^\mu_0, u^\mu_{\mathfrak{p}})-\mathcal{P}_t^\nu ({U}^\mu_0, u^\mu_{\mathfrak{p}})\|_{L^1(\mathbb{R})}
	\leq \mathcal{O}(1) \big (1+ \mu +\nu +t+  \|Y_0\|_{L^1(\mathbb{R})} \big) \big(t\Delta t +m (\mu+\nu) \big).
\end{split}
\end{eqnarray*}
As a result, it follows that
\begin{eqnarray*}
\begin{split}
&\quad \|\mathcal{P}_t^\mu ({U}^\mu_0, u^\mu_{\mathfrak{p}})-\mathcal{P}_t^\nu ({U}^\nu_0, u^\nu_{\mathfrak{p}})\|_{L^1(\mathbb{R})} \\[5pt]
&\leq  \|\mathcal{P}_t^\mu ({U}^\mu_0, u^\mu_{\mathfrak{p}})-\mathcal{P}_t^\nu ({U}^\mu_0, u^\mu_{\mathfrak{p}})\|_{L^1(\mathbb{R})}
+ \|\mathcal{P}_t^\nu ({U}^\mu_0, u^\mu_{\mathfrak{p}})-\mathcal{P}_t^\nu ({U}^\nu_0, u^\nu_{\mathfrak{p}})\|_{L^1(\mathbb{R})} \\[5pt]
&\leq \mathcal{O}(1) \big (1+ \mu +\nu +t+ \|Y_0\|_{L^1(\mathbb{R})} \big) \big(t\Delta t +m (\mu+\nu) \big)\\[5pt]
&\quad+ \mathcal{O}(1) \|{U}^\mu_0 -{U}^\nu_0\|_{L^1(\mathbb{R})} + \mathcal{O}(1)t\nu \\[5pt]
&\leq   \mathcal{O}(1)\big(1+ \mu +\nu +t+\|Y_0\|_{L^1(\mathbb{R})} \big) \big(t\Delta t +m (\mu+\nu)\big).
\end{split}
\end{eqnarray*}
Passing to limits as $\mu,\ \nu \rightarrow 0$, one has
\begin{eqnarray*}
\mathop{\overline{\lim}}_{\mu, \nu \rightarrow 0}
\|\mathcal{P}_t^\mu ({U}^\mu_0, u^\mu_{\mathfrak{p}})-\mathcal{P}_t^\nu ({U}^\nu_0, u^\nu_{\mathfrak{p}})||_{L^1(\mathbb{R})}
\leq  \mathcal{O}(1)\big(1+t+\|Y_0\|_{L^1(\mathbb{R})}\big)t\Delta t,
\end{eqnarray*}
which implies (\ref{Cauchy sequence convergence}) owing to the arbitrariness of $\Delta t$.
Subsequently, applying this strong convergence to estimate
\eqref{estimate stability for ZND approximation}, we establish \eqref{eq:1.16}, namely the Lipschitz continuity of combustion solution
$U(t,x)\doteq \mathcal{P}_t ({U}_0, u_{\mathfrak{p}})$ with respect to initial-boundary data.
\hfill$\Box$

\bigskip
\section*{Acknowledgements}
The research of Kai Hu was supported in part by NSFC Project No.11901475 and China Postdoctoral Science Foundation No.2019M653815XB.
The research of Jie Kuang was supported in part by the NSFC Project No.11801549, NSFC Project No.11971024, the Start-Up Research Grant Project No.Y8S001104 from
 Wuhan Institute of Physics and Mathematics, and the Multidisciplinary Interdisciplinary Cultivation Project No.S21S6401 from
Innovation Academy for Precision Measurement Science and Technology, Chinese Academy of Sciences.



\begin{thebibliography}{10}

\bibitem{Amadori-1997} D.~Amadori,
{Initial-boundary value problem for nonlinear systems of conservation laws},
Nonlinear Differ. Equ. Appl., 4(1997), 1-42.


\bibitem{Amadori-Goss-Guerra-2002} D.~Amadori, L.~Gosse and G.~Guerra,
{Global BV entropy solutions and uniqueness for hyperbolic systems of balance laws},
Arch. Rational Mech. Anal., 162(2002), 327-366.


\bibitem{Amadori-Guerra-2002} D.~Amadori and G.~Guerra,
{Uniqueness and continuous dependence for systems of balance laws with dissipation},
Nonlinear Anal., 49(2002), 987-1014.


\bibitem{Bressan-2000} A.~Bressan,
{Hyperbolic Systems of Conservation Laws. The One-Dimensional Cauchy Problem},
Oxford University Press, Oxford, 2000.


\bibitem{Bressan-Liu-Yang-1999} A.~Bressan, T.-P.~Liu and T.~Yang,
{$L^1$ stability estimates for $n\times n$ conservation laws},
Arch. Rational Mech. Anal., 149(1999), 1-22.


\bibitem{Chen-Wang-2002} G.-Q.~Chen and D.-H.~Wang,
{The Cauchy problem for the Euler equations for compressible fluids},
North-Holland, Elsevier, Amsterdam, Handbook of Mathematical Dynamics, 2002, 421-543.


\bibitem{Chen-Wagner-2003} G.-Q.~Chen and D.~Wagner,
{Global entropy solutions to exothermically reacting compressible Euler equations},
J. Differential Equations, 191(2003), 277-322.


\bibitem{Chen-Chen-Wang-Wang-2005}G.-Q.~Chen, S.~Chen, Z.~Wang and D.~Wang,
{A multidimensional piston problem for the Euler equations for compressible flow},
Discrete Contin. Dyn. Syst., 13(2005), 361-383.




\bibitem{Chen-Wang-Zhang-2004} S.~Chen, Z.~Wang and Y.~Zhang,
{Global existence of shock front solutions to the axially symmetric piston problem for compressible fluids},
J. Hyperbolic Differ. Equ., 1(2004), 51-84.


\bibitem{Chen-Wang-Zhang-2008} S.~Chen, Z.~Wang and Y.~Zhang,
{Global existence of shock front solution to axially symmetric piston problem in compressible flow},
 Z. Angew. Math. Phys., 59(2008), 434-456.


\bibitem{Christoforou-2006} C.~Christoforou,
{Hyperbolic systems of balance laws via vanishing viscosity},
J. Differential Equations, 221(2006), 470-541.


\bibitem{Colombo-Guerra-2010} R.~M.~Colombo and G.~Guerra,
{On general balance laws with boundary},
J. Differential Equations, 248(2010), 1017-1043.


\bibitem{Courant-Friedrichs-1948} R.~Courant and K.~O.~Friedrichs,
{Supersonic flow and shock waves}, Interscience Publishers Inc., New York, 1948.


\bibitem{Dafermos-Hsiao-1982} C.~Dafermos and L.~Hsiao,
{Hyperbolic systems of balance laws with inhomogeneity and dissipation},
Indiana Univ. Math. J., 31(1982), 471-491.


\bibitem{Donadello-Marson-2007} C.~Donadello and A.~Marson,
{Stability of front tracking solutions to the initial and boundary value problem for systems of conservation laws},
Nonlinear Differ. Equ. Appl., 14(2007), 569-592.


\bibitem{Ding-2017} M.~Ding,
{Stability of rarefaction wave to the 1-D piston problem for exothermically reacting Euler equations},
Calc. Var. Partial Differential Equations, 56(2017), Paper No.78, 49pp.


\bibitem{Ding-2018} M.~Ding,
{Global existence of shock front solution to 1-D piston problem for compressible Euler equations},
J. Math. Fluid Mech., 20(2018), 2053-2071.


\bibitem{Ding-Kuang-Zhang-2017} M.~Ding, J.~Kuang and Y.~Zhang,
{Global stability of rarefaction wave to the 1-D piston problem for the compressible full Euler equations},
J. Math. Anal. Appl., 448(2017), 1228-1264.


\bibitem{Hu-2018} K.~Hu,
{Existence of global BV solutions to a model of reacting Euler fluid with variable thermodynamics parameters},
J. Math. Anal. Appl., 457(2018), 890-921.


\bibitem{Hu-2019} K.~Hu,
{Stability and uniqueness of global solutions to Euler equations with exothermic reaction},
Nonlinear Anal. Real World Appl., 48(2019), 362-382.


\bibitem{Kuang-Zhao-2020}J.~Kuang and Q.~Zhao,
{Global existence and stability of shock front solution to 1-D piston problem for exothermically reacting Euler equations},
J. Math. Fluid Mech., 22(2020), Paper No.22, 42 pp.




\bibitem{Lax-1957} P.~Lax,
{Hyperbolic systems of conservation laws II},
Comm. Pure Appl. Math., 10(1957), 537-566.


\bibitem{Lewicka-2004} M.~Lewicka,
{Well-posedness for hyperbolic systems of conservation laws with large BV data},
Arch. Rational Mech. Anal., 173(2004), 415-445.


\bibitem{Lewicka-Trivisa-2002} M.~Lewicka and K.~Trivisa,
{On the $\mathrm{L}^1$ well posedness of systems of conservation laws near solutions containing two large shocks},
Journal of Differential Equations, 179(2002), 133-177.




\bibitem{Liu-1977} T.-P.~Liu,
{Initial-boundary value problems for gas dynamics},
Arch. Ration. Mech. Anal., 64(1977), 137-168.


\bibitem{Liu-1979} T.-P.~Liu,
{Quasilinear hyperbolic systems},
Comm. Math. Phys., 68(1979), 141-172.



\bibitem{Smoller-1994} J.~Smoller,
{Shock Waves and Reaction-Diffusion Equations},
Second Edition, Springer-Verlag, Inc., New York, 1994.


\bibitem{Schochet-1991} S.~Schochet,
{Sufficient conditions for local existence via Glimm scheme for large BV data},
J. Differential Equations, 89(1991), 317-354.


\bibitem{Wang-2005} Z.~Wang,
{Global existence of shock front solution to 1-dimensional piston problem (Chinese)}.
Chinese Ann. Math. Ser. A,  26(2005), 549-560.


\bibitem{Williams-1985}F.~A.~Williams,
{Combustion Theory}, 2nd ed., The Benjamin/Cummings Publishing Company,1985.











%



\end{thebibliography}
\end{document}